\documentclass[10pt]{amsart}
\usepackage{amsmath}
\usepackage{amsxtra}
\usepackage{amscd}
\usepackage{amsthm}
\usepackage{amsfonts}
\usepackage{amssymb}
\usepackage{eucal}

\usepackage{latexsym}

\newtheorem{cor}[subsubsection]{Corollary}
\newtheorem{lem}[subsubsection]{Lemma}

\newtheorem{prop}[subsubsection]{Proposition}

\newtheorem{conj}[subsubsection]{Conjecture}

\newtheorem{thm}[subsubsection]{Theorem}
\newtheorem{mainthm}{Main Theorem}

\theoremstyle{remark}

\theoremstyle{definition}

\numberwithin{equation}{section}

\newcommand{\thmref}[1]{Theorem~\ref{#1}}
\newcommand{\mainthmref}[1]{Main Theorem~\ref{#1}}
\newcommand{\secref}[1]{Sect.~\ref{#1}}
\newcommand{\lemref}[1]{Lemma~\ref{#1}}
\newcommand{\propref}[1]{Proposition~\ref{#1}}
\newcommand{\corref}[1]{Corollary~\ref{#1}}
\newcommand{\conjref}[1]{Conjecture~\ref{#1}}

\newcommand{\nc}{\newcommand}
\nc{\renc}{\renewcommand}
\nc{\ssec}{\subsection}
\nc{\sssec}{\subsubsection}
\nc{\on}{\operatorname}

\nc\ol{\overline}
\nc\ul{\underline}
\nc\wt{\widetilde}
\nc\tboxtimes{\wt{\boxtimes}}
\nc{\wh}{\widehat}
\nc{\mc}{\mathcal}

\nc{\CM}{{\mathcal M}}
\nc{\CN}{{\mathcal N}}
\nc{\CF}{{\mathcal F}}
\nc{\D}{{\mathcal D}}
\nc{\CQ}{{\mathcal Q}}
\nc{\CY}{{\mathcal Y}}
\nc{\CX}{{\mathcal X}}
\nc{\CG}{{\mathcal G}}
\nc{\CE}{{\mathcal E}}
\nc{\CC}{{\mathcal C}}
\nc{\CO}{{\mathcal O}}
\renc{\CC}{{\mathcal C}}
\nc{\CT}{{\mathcal T}}
\nc{\CK}{{\mathcal K}}
\nc{\CS}{{\mathcal S}}
\nc{\CH}{{\mathcal H}}
\nc{\CU}{{\mathcal U}}
\nc{\CV}{{\mathcal V}}
\nc{\CA}{{\mathcal A}}
\nc{\CB}{{\mathcal B}}
\nc{\CW}{{\mathcal W}}
\nc{\CL}{{\mathcal L}}
\nc{\CP}{{\mathcal P}}
\nc{\CI}{{\mathcal I}}
\nc{\CJ}{{\mathcal J}}
\nc{\CR}{{\mathcal R}}
\nc{\CZ}{{\mathcal Z}}

\nc{\BA}{{\mathbb{A}}}
\nc{\BC}{{\mathbb{C}}}
\nc{\BG}{{\mathbb{G}}}
\nc{\BM}{{\mathbb{M}}}
\nc{\BN}{{\mathbb{N}}}
\nc{\BP}{{\mathbb{P}}}
\nc{\BR}{{\mathbb{R}}}
\nc{\BZ}{{\mathbb{Z}}}
\nc{\BV}{{\mathbb{V}}}
\nc{\BW}{{\mathbb{W}}}
\nc{\BS}{{\mathbb{S}}}
\nc{\BD}{{\mathbb{D}}}
\nc{\BQ}{{\mathbb{Q}}}
\nc{\BL}{{\mathbb{L}}}
\renc{\BW}{{\mathbb{W}}}

\nc{\fa}{{\mathfrak{a}}}
\nc{\fb}{{\mathfrak{b}}}
\nc{\fg}{{\mathfrak{g}}}
\nc{\fq}{{\mathfrak{q}}}
\nc{\fgl}{{\mathfrak{gl}}}
\nc{\fh}{{\mathfrak{h}}}
\nc{\fj}{{\mathfrak{j}}}
\nc{\fm}{{\mathfrak{m}}}
\nc{\fl}{{\mathfrak{l}}}
\nc{\fn}{{\mathfrak{n}}}
\nc{\fu}{{\mathfrak{u}}}
\nc{\fp}{{\mathfrak{p}}}
\nc{\fr}{{\mathfrak{r}}}
\nc{\fs}{{\mathfrak{s}}}
\nc{\fsl}{{\mathfrak{sl}}}

\nc{\hsl}{{\widehat{\mathfrak{sl}}}}
\nc{\hgl}{{\widehat{\mathfrak{gl}}}}
\nc{\hg}{{\widehat{\mathfrak{g}}}}
\nc{\hb}{{\widehat{\mathfrak{b}}}}
\nc{\hn}{{\widehat{\mathfrak{n}}}}

\nc{\fA}{{\mathfrak{A}}}
\nc{\fB}{{\mathfrak{B}}}
\nc{\fO}{{\mathfrak{O}}}
\nc{\fD}{{\mathfrak{D}}}
\nc{\fE}{{\mathfrak{E}}}
\nc{\fF}{{\mathfrak{F}}}
\nc{\fG}{{\mathfrak{G}}}
\nc{\fK}{{\mathfrak{K}}}
\nc{\fL}{{\mathfrak{L}}}
\nc{\fC}{{\mathfrak{C}}}
\nc{\fM}{{\mathfrak{M}}}
\nc{\fN}{{\mathfrak{N}}}
\nc{\fH}{{\mathfrak{H}}}
\nc{\fP}{{\mathfrak{P}}}
\nc{\fU}{{\mathfrak{U}}}
\nc{\fV}{{\mathfrak{V}}}
\nc{\fZ}{{\mathfrak{Z}}}
\nc{\fz}{{\mathfrak{z}}}

\nc{\bc}{{\mathbf{c}}}
\nc{\ba}{{\mathbf{a}}}
\nc{\bd}{{\mathbf{d}}}
\nc{\bh}{{\mathbf{h}}}
\nc{\be}{{\mathbf{e}}}
\nc{\bj}{{\mathbf{j}}}
\nc{\bn}{{\mathbf{n}}}
\nc{\bp}{{\mathbf{p}}}
\nc{\bg}{{\mathbf{g}}}
\nc{\bq}{{\mathbf{q}}}
\nc{\bs}{{\mathbf{s}}}
\nc{\bbt}{{\mathbf{t}}}
\nc{\bu}{{\mathbf{u}}}
\nc{\bv}{{\mathbf{v}}}
\nc{\bx}{{\mathbf{x}}}
\nc{\by}{{\mathbf{y}}}
\nc{\bw}{{\mathbf{w}}}
\nc{\bA}{{\mathbf{A}}}
\nc{\bK}{{\mathbf{K}}}
\nc{\bB}{{\mathbf{B}}}
\nc{\bC}{{\mathbf{C}}}
\nc{\bD}{{\mathbf{D}}}
\nc{\bF}{{\mathbf{F}}}
\nc{\bH}{{\mathbf{H}}}
\nc{\bM}{{\mathbf{M}}}
\nc{\bN}{{\mathbf{N}}}
\nc{\bV}{{\mathbf{V}}}
\nc{\bW}{{\mathbf{W}}}
\nc{\bL}{{\mathbf{L}}}
\nc{\bU}{{\mathbf{U}}}
\nc{\bX}{{\mathbf{X}}}
\nc{\bY}{{\mathbf{Y}}}
\nc{\bI}{{\mathbf{I}}}
\nc{\bZ}{{\mathbf{Z}}}
\nc{\bS}{{\mathbf{S}}}

\nc{\sA}{{\mathsf{A}}}
\nc{\sB}{{\mathsf{B}}}
\nc{\sC}{{\mathsf{C}}}
\nc{\sD}{{\mathsf{D}}}
\nc{\sF}{{\mathsf{F}}}
\nc{\sH}{{\mathsf{H}}}
\nc{\sG}{{\mathsf{G}}}
\nc{\sK}{{\mathsf{K}}}
\nc{\sM}{{\mathsf{M}}}
\nc{\sO}{{\mathsf{O}}}
\nc{\sQ}{{\mathsf{Q}}}
\nc{\sR}{{\mathsf{R}}}
\nc{\sP}{{\mathsf{P}}}
\nc{\sV}{{\mathsf{V}}}
\nc{\sZ}{{\mathsf{Z}}}
\nc{\sfp}{{\mathsf{p}}}
\nc{\sr}{{\mathsf{r}}}
\nc{\sg}{{\mathsf{g}}}
\nc{\sk}{{\mathsf{k}}}
\nc{\ssf}{{\mathsf{f}}}
\nc{\ssh}{{\mathsf{h}}}
\nc{\sse}{{\mathsf{e}}}
\nc{\sfb}{{\mathsf{b}}}
\nc{\sfc}{{\mathsf{c}}}
\nc{\sd}{{\mathsf{d}}}

\nc{\Av}{\on{Av}}
\nc{\act}{\on{act}}
\nc{\Hom}{\on{Hom}}
\nc{\Ho}{\on{Ho}}
\nc{\End}{\on{End}}
\nc{\Lie}{\on{Lie}}
\nc{\Loc}{\on{Loc}}
\nc{\IC}{\on{IC}}
\nc{\Aut}{\on{Aut}}
\nc{\rk}{\on{rk}}
\nc{\Sh}{\on{Sh}}
\nc{\Perv}{\on{Perv}}
\nc{\pos}{{\on{pos}}}
\nc{\Conv}{\on{Conv}}
\nc{\Sph}{\on{Sph}}
\nc{\Sym}{\on{Sym}}
\nc{\Rep}{\on{Rep}}
\nc{\RepH}{{\mc R}ep(H)}
\nc{\Fun}{\on{Fun}}
\nc{\Res}{\on{Res}}
\nc{\Ind}{\on{Ind}}
\nc{\Id}{\on{Id}}
\nc{\id}{\on{id}}
\renc{\mod}{\on{--mod}}

\nc{\crit}{{\on{crit}}}
\nc{\reg}{{\on{reg}}}
\nc{\nilp}{{\on{nilp}}}
\nc{\ord}{\on{ord}}
\nc{\nil}{\wt{\on{reg}}}
\nc{\mb}{\mathbf}
\nc{\ren}{\on{ren}}
\nc{\res}{\on{res}}
\nc{\RS}{{\on{RS}}}
\nc{\Dist}{\on{Dist}}
\nc{\semiinf}{{\frac{\infty}{2}}}
\nc{\semiinfi}{{\frac{\infty}{2}+i}}
\nc{\semiinfb}{{\frac{\infty}{2}+\bullet}}
\nc{\torsemiinf}{{\overset{\semiinf}\otimes}}
\nc{\Hitch}{\on{Hitch}}

\nc{\hl}{\overset{\leftarrow}h}
\nc{\hr}{\overset{\rightarrow}h}
\nc\Dh{\widehat{\D}}
\nc{\Gr}{\on{Gr}}
\nc{\Fl}{{\on{Fl}}}
\nc{\Flt}{\wt{\Fl}{}}
\nc{\Pic}{\on{Pic}}
\nc{\Bun}{\on{Bun}}

\nc{\bDR}{\mathbf {DR}}
\nc{\uV}{\underline{V}}
\nc{\arrowtimes}{\underset{\to}\times}
\nc{\arrowotimes}{\underset{\to}\otimes}
\nc{\hattimes}{\widehat\otimes}
\nc{\larrowtimes}{\overset{\leftarrow}\otimes}
\nc{\shriektimes}{\overset{!}\otimes}
\nc{\startimes}{\overset{*}\otimes}
\nc{\sCliff}{\mathsf {Cliff}}
\nc{\sSpin}{\mathsf {Spin}}

\nc{\one}{{\mathbf{1}}}

\nc\Spec{\on{Spec}}
\nc{\Pro}{\on{Pro}}
\nc{\Coh}{\on{Coh}}
\nc{\QCoh}{\on{QCoh}}
\nc{\uHom}{\underline{\on{Hom}}}
\nc{\RHom}{\on{RHom}}
\nc{\uRHom}{\underline{\on{RHom}}}
\nc{\CHom}{{\mathcal Hom}}
\nc{\uCHom}{\underline{{\mathcal Hom}}}
\nc{\uCRHom}{\underline{{\mathcal R}{\mathcal Hom}}}

\nc{\cg}{\check \fg}
\nc{\cT}{\check {T}}
\nc{\Op}{{\on{Op}}}
\nc{\nOp}{{\on{Op}^{\nilp}}}
\nc{\rOp}{{\on{Op}^{\reg}}}
\nc{\nMOp}{\on{MOp}^{\nilp}_{\cg}}

\nc{\tg}{\wt{\check \fg}}
\nc{\cn}{\check \fn}
\nc{\tn}{\wt{\cn}}
\nc{\cG}{{{\check G}}}
\nc{\cB}{{{\check B}}}
\nc{\cb}{\check \fb}
\nc{\MOp}{\on{MOp}_{\check \cg}}
\nc{\cN}{\check\CN}
\nc{\tN}{\wt{\check{\CN}}}
\nc{\dIsom}{{\mathsf{Isom}}_{\Op}}
\nc{\disom}{{\mathsf{isom}}_{\Op}}
\nc{\Kdv}{{\mathsf{Isom}}_{\rOp}}
\nc{\kdv}{{\mathsf{isom}}_{\rOp}}
\nc{\Isom}{{\mathsf{Isom}}_{\on{Op}_\fg}}
\nc{\isom}{{\mathsf{isom}}_{\on{Op}_\fg}}

\nc{\wcosta}{j_{\wt{w},*}}
\nc{\wsta}{j_{\wt{w},!}}
\nc{\wcost}{j_{w,*}}
\nc{\wst}{j_{w,!}}

\nc{\epsi}{{\mathbf e}^\psi}
\nc{\epsip}{{\mathbf e}^{\psi'}}

\nc{\Ppi}{{\mathbf \Pi}}

\nc{\hCO}{{\hat{\CO}}}
\nc{\hCK}{{\hat{\CK}}}

\nc{\CPreg}{\CP_{G,\on{Op}^\reg}}
\nc{\CPBreg}{\CP_{B,\on{Op}^\reg}}
\nc{\CPnilp}{\CP_{G,\on{Op}^\nilp}}
\nc{\CPBnilp}{\CP_{B,\on{Op}^\nilp}}
\nc{\CPla}{\CP_{G,\on{Op}_{\cla}}}
\nc{\CPBla}{\CP_{B,\on{Op}_{\cla}}}

\nc{\Cat}{\hg_\crit\mod^{I,m}_\nilp}
\nc{\Catf}{{}^f\hg_\crit\mod^{I,m}_\nilp}
\nc{\DCat}{\bD^f(\hg_\crit\mod_\nilp)^{I^0}}
\nc{\DCatf}{{}^f \bD^f(\hg_\crit\mod_\nilp)^{I^0}}
\nc{\Catr}{\hg_\crit\mod^{I,m}_\reg}
\nc{\Catrf}{{}^f\hg_\crit\mod^{I,m}_\reg}
\nc{\DCatr}{\bD^f(\hg_\crit\mod_\reg)^{I^0}}
\nc{\DCatrf}{{}^f \bD^f(\hg_\crit\mod_\reg)^{I^0}}

\nc{\CatFl}{\bD^f(\fD(\Fl^{\on{aff}}_G)_\crit\mod)}

\nc{\ppart}{(\!(t)\!)}

\nc{\crho}{{\check\rho}}
\nc{\cla}{{\check\lambda}}
\nc{\cmu}{{\check\mu}}
\nc{\cnu}{{\check\nu}}
\nc{\cLambda}{\check\Lambda}

\nc{\fInd}{\mathsf{Ind}}

\nc{\psHom}{{\mathcal PsHom}}
\nc{\Vect}{\on{Vect}}
\renc{\Bar}{\on{Bar}}

\nc{\ua}{\underset{\to}}
\nc{\Heart}{\on{Heart}}

\nc{\Comp}{{\mathbf{Comp}}}

\nc{\DGCat}{{\mathbf {DGCat}}}
\nc{\TrCat}{{\mathbf {TrCat}}}
\nc{\TrMod}{{\mathbf {Tr}}{\mathbf {mod}}}
\nc{\DGMod}{{\mathbf {DG}}{\mathbf {mod}}}
\nc{\DGMonCat}{{\mathbf{DGMonCat}}}
\nc{\TrMonCat}{{\mathbf{TrMonCat}}}
\nc{\obC}{{\overset{\circ}\bC}{}}
\nc{\obA}{{\overset{\circ}\bA}{}}

\nc{\sT}{{\mathsf T}}
\nc{\sS}{{\mathsf S}}

\begin{document}

\title[D-modules on the affine flag variety]
{D-modules on the affine flag variety and representations of affine Kac-Moody algebras} 

\author[Edward Frenkel]{Edward Frenkel$^1$}
\thanks{$^1$Supported by DARPA and AFOSR through the grant FA9550-07-1-0543.}

\address{Department of Mathematics, University of California,
Berkeley, CA 94720, USA}

\email{frenkel@math.berkeley.edu}

\author[Dennis Gaitsgory]{Dennis Gaitsgory$^2$}
\thanks{$^2$Supported by NSF grant 0600903.}

\address{Department of Mathematics, Harvard University,
Cambridge, MA 02138, USA}

\email{gaitsgde@math.harvard.edu}

\date{December 2007; Revised: July 2009. Version: Sept. 29, 2009}

\maketitle

\setcounter{tocdepth}{1}

\tableofcontents

\newpage

\section*{Introduction}

\ssec{}   \label{first paragraph}

Let $\hg$ be the affine Kac-Moody algebra corresponding to a finite-dimensional
semi-simple Lie algebra $\fg$. Let $\hg_\crit\mod$ denote the category of 
(continuous) $\hg$-modules at the critical level (see \cite{FG2} for the precise 
definition).

\medskip

It is often the case in representation theory that in order to gain a good understanding
of a category of modules of some sort, one has to reinterpret it in more geometric terms,
by which we mean either as the category of D-modules on an algebraic 
variety, or as the category of quasi-coherent sheaves on some (usually, different) 
algebraic variety. This is what we do in this paper for a certain subcategory 
of $\hg_\crit\mod$, thereby proving two conjectures from \cite{FG2}.

\medskip

There is another angle under which this paper can be viewed: the results concerning
$\hg_\crit\mod$ fit into the framework of the geometric local Langlands correspondence.
We refer the reader to the introduction to \cite{FG2} where this view point is
explained in detail.

\ssec{Localization}

Let us first describe the approach via D-modules. This pattern is known as {\it localization},
a prime example of which is the equivalence of \cite{BB} between the category of 
$\fg$-modules with a given central character and the category of (twisted) D-modules 
on the flag variety $G/B$.

\medskip

The affine analog of $G/B$ is the affine flag scheme $G\ppart/I$, where $I\subset G\ppart$
is the Iwahori subgroup. By taking sections of (critically twisted) D-modules, we obtain a
functor $$\Gamma_\Fl: \fD(\Fl^{\on{aff}}_G)_\crit\mod \to \hg_\crit\mod.$$

However, as in the finite-dimensional case, one immediately observes that the 
$\hg_\crit$-modules that one obtains in this way are not arbitrary, but belong to
a certain subcategory singled out by a condition on the action of the center 
$Z(\wt{U}(\hg)_\crit)$, where $\wt{U}(\hg)_\crit$ is the (completed, reduced) universal 
enveloping algebra at the critical level. 

\medskip

Namely, $\fZ_\fg:=Z(\wt{U}(\hg)_\crit)$ is a topological commutative algebra,
which according to \cite{FF}, admits the following explicit description in terms of 
the Langlands dual group $\cG$: the ind-scheme $\Spec(\fZ_\fg)$ is isomorphic 
to the ind-scheme $\Op(\D^\times)$ of $\cG$-opers on the formal punctured disc. 
This ind-scheme of opers was introduced in \cite{BD}, and it 
contains a closed subscheme denoted $\nOp$ which
corresponds to opers with a nilpotent singularity, introduced in \cite{FG2}.

\medskip

It is rather straightforward to see that the image
of the functor $\Gamma_\Fl$ lands in the subcategory $\hg_\crit\mod_\nilp\subset
\hg_\crit\mod$ consisting of modules, whose support over $\Spec(\fZ_\fg)\simeq \Op(\D^\times)$
is contained in $\nOp$. Thus, we can consider $\Gamma_\Fl$ as a functor
\begin{equation}  \label{Gamma nilp}
\fD(\Fl^{\on{aff}}_G)_\crit\mod \to \hg_\crit\mod_\nilp.
\end{equation}

\medskip

We should remark that it is here that the assumption that we work at the critical
level becomes crucial:

\medskip

For any level
$\kappa$ one can consider the corresponding functor 
$\Gamma_\Fl: \fD(\Fl^{\on{aff}}_G)_\kappa\mod \to \hg_\kappa\mod$, and it is again relatively easy to see that this functor cannot be 
essentially surjective. However, in the non-critical case the image
category is much harder to describe: we cannot do this by imposing a condition on
the action of the center $Z(\wt{U}(\hg)_\kappa)$ (as we did in the finite-dimensional
case, or in the affine case at the critical level) since the latter is essentially trivial.
This fact prevents one from proving (or even formulating)
a localization type equivalence in the non-critical case.

\ssec{Non-exactness}

Returning to the analysis of the functor \eqref{Gamma nilp} we observe two phenomena
that distinguish the present situation from the finite-dimensional case of \cite{BB}. 

\medskip

First, unlike the case of the finite-dimensional flag variety,
the functor $\Gamma_{\Fl}$ is not exact (and cannot be made exact by any additional
twisting). This compels us to leave the hopes of staying within the realm of abelian categories, 
and pass to the corresponding derived ones. I.e., from now on we will be considering 
the derived functor of $\Gamma_\Fl$, denoted by a slight abuse of notation by the same
character 
\begin{equation}  \label{Gamma nilp der}
\Gamma_\Fl:\bD^b(\fD(\Fl^{\on{aff}}_G)_\crit\mod) \to \bD^b(\hg_\crit\mod_\nilp).
\end{equation}

The necessity do work with triangulated categories as opposed to abelian ones
accounts for many of the technical issues in this paper, and ultimately, for
its length. 

\medskip

That said, we should remark that in \secref{new t structure} we define a new t-structure
on the category $\bD^b(\fD(\Fl^{\on{aff}}_G)_\crit\mod)$ and make a conjecture that in this
new t-structure the functor $\Gamma_\Fl$ is exact. 

\ssec{Base change}

The second new phenomenon present in the case of affine Kac-Moody
algebras is that the (derived) functor 
$\Gamma_{\Fl}$ is not fully faithful. The reason is very simple: the center of the category 
$\fD(\Fl^{\on{aff}}_G)_\crit\mod$ is essentially trivial, whereas that of $\hg_\crit\mod_\nilp$ is the algebra
of functions on the scheme $\nOp$. 

\medskip

I.e., by setting the level to critical we have gained
the center, which allows to potentially describe the image of $\Gamma_\Fl$, but we
have gained too much: instead of just one central character as in the finite-dimensional
case, we obtain a $\nOp$-worth of those.

\medskip

However, the non-fully faithfulness of $\Gamma_\Fl$ can be accounted for. 

\medskip

Let $\tN:=T^*(\on{Fl}^{\cG})$ be the Springer variety corresponding to the Langlands dual group, where $\Fl^\cG$ denotes the flag variety of $\cG$.
\footnote{We emphasize that $\Fl^\cG$ denotes the finite-dimensional flag
variety $\cG/\cB$ of the Langlands dual group $\cG$, and in this paper we
will consider quasi-coherent sheaves on it. It should not be confused with the affine
flag scheme $G\ppart/I$ of $G$, denoted $\Fl^{\on{aff}}_G$, on which we will consider D-modules.}
Consider the stack $\tN/\cG$. A crucial piece of structure is that the 
monoidal triangulated category $\bD^b(\Coh(\tN/\cG))$ (i.e., the 
$\cG$-equivariant derived category of coherent sheaves on $\tN$) acts on the
triangulated category $\bD^b(\fD(\Fl^{\on{aff}}_G)_\crit\mod)$. This action was constructed
in the paper \cite{AB}, and we denote it here by $\star$.

\medskip

Another observation is that there is a natural map $\fr_\nilp:\nOp\to \tN/\cG$,
and we show that these structures are connected as follows. For
$\CF\in \bD^b(\fD(\Fl^{\on{aff}}_G)_\crit\mod)$, $\CM\in \bD^b(\Coh(\tN/\cG))$ we have
a canonical isomorphism
$$\Gamma_\Fl(\CF\star \CM)\simeq \Gamma_\Fl(\CF)\underset{\CO_{\nOp}}\otimes 
\fr_\nilp^*(\CM).$$
I.e., the effect of acting on $\CF$ by $\CM$ and then taking sections 
is the same as that of first taking sections and then tensoring over the algebra
of functions on $\nOp$ by the pull-back of $\CM$ by means of $\fr_\nilp$.

\medskip

This should be viewed as a categorical analog of the following situation in
linear algebra. Let $V_1$ be a vector space, acted on by an algebra $A_1$
by endomorphisms (i.e., $V_1$ is a $A_1$-module). Let $(V_2,A_2)$ be
another such pair; let $r_A:A_1\to A_2$ be a homomorphism of algebras,
and $r_V:V_1\to V_2$ a map of vector spaces, compatible with the actions.

In this case, we obtain a map $A_2\underset{A_1}\otimes V_1\to V_2$.

\medskip

We would like to imitate this construction, where instead of vector spaces
we have categories: 
$$V_1\mapsto \bD^b(\fD(\Fl^{\on{aff}}_G)_\crit\mod),\,\, V_2\mapsto 
 \bD^b(\hg_\crit\mod_\nilp),$$
instead of algebras we have monoidal categories
$$A_1\mapsto \bD^b(\Coh(\tN/\cG)),\,\, A_2\mapsto \bD^b(\Coh(\nOp)),$$
and instead of maps we have functors:
$$r_A\mapsto \fr^*_\nilp,\,\, r_V\mapsto \Gamma_\Fl.$$

\medskip

Therefore, it is a natural idea to try to define a categorical tensor product
\begin{equation} \label{base change intr}
\bD^b(\Coh(\nOp))\underset{\bD^b(\Coh(\tN/\cG))}\otimes \bD^b(\fD(\Fl^{\on{aff}}_G)_\crit\mod),
\end{equation}
which can be viewed as a base change of $\bD^b(\fD(\Fl^{\on{aff}}_G)_\crit\mod)$ with
respect to the morphism $\fr_\nilp:\nOp\to \tN/\cG$, and a functor from 
\eqref{base change intr} to $\bD^b(\hg_\crit\mod_\nilp)$, denoted 
$\Gamma_{\Fl,\nOp}$, 
compatible with the action of $\bD^b(\Coh(\nOp))$. Unlike $\Gamma_\Fl$, the new
functor $\Gamma_{\Fl,\nOp}$ has a chance of being an equivalence. 

\ssec{Localization results}

The above procedure of taking the tensor product can indeed be carried out, and
is the subject of Part III of this paper. I.e., we can define the category in 
\eqref{base change intr} as well as the functor $\Gamma_{\Fl,\nOp}$. 

\medskip

We conjecture that $\Gamma_{\Fl,\nOp}$ is an 
equivalence, which would be a complete localization result in the context of $\hg_\crit\mod$.
Unfortunately, we cannot prove it at the moment. We do prove, however, that 
$\Gamma_{\Fl,\nOp}$ is fully faithful; this is one of the four main results of this
paper, \mainthmref{Gamma nilp fully faithful}.  

\medskip

In addition, we prove that $\Gamma_{\Fl,\nOp}$ does induce an equivalence
between certain subcategories on both sides, 
namely the subcategories consisting of Iwahori-monodromic objects. This is the second
main result of the present paper, \mainthmref{main}. 
The Iwahori-monodromic subcategory on the RHS, denoted $\bD^b(\hg_\crit\mod_\nilp)^{I^0}$,
can be viewed as a critical level version of the category $\CO$. Thus, \mainthmref{main},
provides a localization description at least for this subcategory. 

\medskip

We should remark that our inability to prove the fact that $\Gamma_{\Fl,\nOp}$
is an equivalence for the ambient categories
stems from our lack of any explicit information about objects of 
$\hg_\crit\mod$ other than the Iwahori-monodromic ones. 

\medskip

The above results that relate the categories $\bD^b(\fD(\Fl^{\on{aff}}_G)_\crit\mod)$ and
$\bD^b(\hg_\crit\mod_\nilp)$ have analogues, when instead of $\Fl^{\on{aff}}_G$ 
we consider the affine Grassmannian $\Gr^{\on{aff}}_G$, and instead of
$\nOp\subset \Op(\D^\times)$ we consider the sub-scheme of regular opers
$\rOp \subset \Op(\D^\times)$. These results will be recalled below in 
\secref{intr rel to gr}.

\ssec{The quasi-coherent picture} Let us now pass to the description of a 
subcategory of $\hg_\crit\mod_\nilp$ in terms of quasi-coherent sheaves, mentioned
in \secref{first paragraph}. The subcategory in question, or rather its triangulated
version, is $\bD^b(\hg_\crit\mod_\nilp)^{I^0}$ that has appeared above.

\medskip

In \cite{FG2} we proposed (see Conjecture 6.2 of {\it loc. cit.}) 
that $\bD^b(\hg_\crit\mod_\nilp)^{I^0}$
should be equivalent to the category $\bD^b(\QCoh(\nMOp))$, where $\nMOp$ is the
scheme classifying {\it Miura opers with a nilpotent singularity}, introduced in \cite{FG2}, Sect. 3.14.

\medskip

By definition, $\nMOp$ is the fiber product
$$\nOp\underset{\tN/\cG}\times \check{\on{St}}/\cG,$$
where $\check{\on{St}}$ is the Steinberg scheme. In other words,
$\nMOp$ is the moduli space of pairs: an oper $\chi$ on the formal punctured
disc $\D^\times$ with a nilpotent singularity, and its reduction to the Borel subgroup
$\cB\subset \cG$ as a local system. 

\medskip

The motivation for the above conjecture was that for any point $\wt\chi\in \nMOp$
one can attach a specific object $\BW_{\wt\chi}\in \hg_\crit\mod_\nilp^{I^0}$,
called the Wakimoto module, and the conjecture can be viewed as saying that
any object of $\bD^b(\hg_\crit\mod_\nilp)^{I^0}$ is canonically a "direct integral"
of Wakimoto modules.

\medskip

In this paper we prove this conjecture by combining our \mainthmref{main} and
one of the main results of Bezrukavnikov's theory (Theorem 4.2 of \cite{Bez}),
which provides an equivalence
\begin{equation} \label{Bezr eq}
\bD^b(\fD(\Fl^{\on{aff}}_G)_\crit\mod)\simeq \bD^b(\QCoh(\check{\on{St}}/\cG)).
\end{equation}
The proof is obtained by essentially base-changing both sides of \eqref{Bezr eq}
with respect to $\fr_\nilp$. This is the third main result of this
paper, \mainthmref{Miura}.

\medskip

We refer the reader to the introduction to \cite{FG2} for the explanation
how the above corollary can be viewed as a particular case of the 
local geometric Langlands correspondence.

\medskip

As a corollary we obtain the following result: let $\chi\in \nOp$ be an oper
with a nilpotent singularity.
On the one hand, let us consider $\bD^b(\hg_\crit\mod_{\chi})^{I^0}$,
which is the full subcategory of $\bD^b(\hg_\crit\mod_\chi)$--the
derived category of $\hg$-modules with central character given by $\chi$,
consisting of $I^0$-integrable objects. 

\medskip

On the other hand, let $n$ be an element of $\cn$, whose image in 
$\cn/\cB\simeq \tN/\cG$ equals that of $\fr_\nilp(\chi)$. Let $\on{Spr}_{n}$
be the derived Springer fiber over $n$, i.e., the Cartesian product
$\on{pt}\underset{\cg}\times \tN$,
taken in the category of DG-schemes. We obtain:

\medskip

\noindent{\bf Corollary.}
{\it There is an equivalence of categories
$$\bD^b(\hg_\crit\mod_{\chi})^{I^0}\simeq \bD^b(\QCoh(\on{Spr}_{n})).$$}

\medskip

In this paper we do not prove, however, that the functor thus obtained is compatible 
with the Wakimoto module construction. Some particular cases of this assertion 
have been established in \cite{FG5}. We expect that the general case could be
established by similar, if more technically involved, methods. 

\ssec{Relation to the affine Grassmannian}  \label{intr rel to gr}

The main representation-theoretic ingredient in the proof of the main results
of this paper is the fully faithfulness assertion, \mainthmref{Gamma nilp fully faithful}. Its proof
is based on comparison between the functors $\Gamma_\Fl$ and 
$\Gamma_{\Fl,\nOp}$ mentioned above, and the corresponding functors
when the affine flag scheme $\Fl^{\on{aff}}_G$ is replaced by the affine Grassmannian $\Gr^{\on{aff}}_G$.

\medskip

Let us recall that in \cite{FG4} we considered the category 
$\hg_\crit\mod_\reg$, corresponding to modules whose
support over $\Spec(\fZ_\fg)\simeq \Op(\D^\times)$ is contained
in the subscheme $\rOp$ of regular opers. 

\medskip

We also considered the category $\fD(\Gr^{\on{aff}}_G)_\crit\mod$ and a functor
$\Gamma_{\Gr}:\fD(\Gr^{\on{aff}}_G)_\crit\mod\to \hg_\crit\mod_{\reg}$, which
by contrast with the case of $\Fl^{\on{aff}}_G$, was exact. In addition,
the category $\fD(\Gr^{\on{aff}}_G)_\crit\mod$ was naturally acted upon by
$\Rep^{f.d.}(\cG)\simeq \Coh(\on{pt}/\cG)$, and we considered the base
changed category
$$\Coh(\rOp)\underset{\Coh(\on{pt}/\cG)}\otimes \fD(\Gr^{\on{aff}}_G)_\crit\mod,$$
and the functor $\Gamma_{\Gr,\rOp}$ from it to $\hg_\crit\mod_\reg$.

\medskip

In {\it loc. cit.} it was shown that on the level of derived categories,
the corresponding functor
\begin{multline*}
\Gamma_{\Gr,\rOp}:\bD^b\left(
\Coh(\rOp)\underset{\Coh(\on{pt}/\cG)}\otimes \fD(\Gr^{\on{aff}}_G)_\crit\mod\right)\simeq  \\
\simeq
\bD^b(\Coh(\rOp))\underset{\bD^b(\Coh(\on{pt}/\cG))}\otimes \bD^b(\fD(\Gr^{\on{aff}}_G)_\crit\mod)\to
\bD^b(\hg_\crit\mod_\reg)
\end{multline*}
was fully faithful.

\medskip

The relation between the $\Fl^{\on{aff}}_G$ and the $\Gr^{\on{aff}}_G$ pictures is provided by our
fourth main result, \mainthmref{thm relation to Grassmannian}. This theorem asserts
that the base-changed category
$$\bD^b(\Coh(\on{pt}/\cB))\underset{\bD^b(\Coh(\tN/\cG))}\otimes \bD^b(\fD(\Fl^{\on{aff}}_G)_\crit\mod)$$
with respect to the map
$$\on{pt}/\cB\simeq \on{Fl}^{\cG}/\cG\hookrightarrow \tN/\cG,$$
given by the 0-section $\on{Fl}^{\cG}\hookrightarrow \tN$
is essentially equivalent to the base-changed category
$$\bD^b(\Coh(\on{pt}/\cB))\underset{\bD^b(\Coh(\on{pt}/\cG))}\otimes \bD^b(\fD(\Gr^{\on{aff}}_G)_\crit\mod).$$

This equivalence makes it possible to write down a precise relationship between the functors
$\Gamma_{\Fl,\nOp}$ and $\Gamma_{\Gr,\rOp}$ (see \thmref{rel to aff gr and sections}),
and deduce \mainthmref{Gamma nilp fully faithful} from the fully-faithfulness of
$\Gamma_{\Gr,\rOp}$ (see \secref{proof of Gamma nilp ff}).

\ssec{Contents}

This paper consists of four parts: 

\medskip

In Part I we perform the representation-theoretic and geometric constructions
and formulate the main results. 

\medskip

In \secref{AB action naive} we show how the constructions of the paper \cite{AB}
define an action of the triangulated category $\bD^b(\Coh(\tN/\cG))$ on
$\bD^b(\fD(\Fl^{\on{aff}}_G)_\crit\mod)$. In fact, the action at the level of triangulated 
categories comes naturally from a finer structure at a DG (differential graded)
level. The latter fact allows to introduce the category \eqref{base change intr},
which is one of the main players in this paper.

\medskip

In \secref{new t structure} we introduce a new t-structure on the category
$\bD^b(\fD(\Gr^{\on{aff}}_G)_\crit\mod)$. It will turn out that this t-structure has 
a better behavior than the usual one with respect to the constructions that we perform
in this paper.

\medskip

In \secref{functor to the critical level} we combine some results of \cite{FG3}
and \cite{FG4} and show how the functor $\Gamma_\Fl$ extends to a functor
$$\Gamma_{\Fl,\nOp}:
\bD^b(\Coh(\nOp))\underset{\bD^b(\Coh(\tN/\cG))}\otimes \bD^b(\fD(\Fl^{\on{aff}}_G)_\crit\mod)\to
\bD^b(\hg_\crit\mod_\nilp),$$
We formulate the first main result of this paper,
\mainthmref{Gamma nilp fully faithful} that asserts that the functor $\Gamma_{\Fl,\nOp}$
is fully faithful. As was explained above, the latter result is as
close as we are currently able to get to {\it localization} at the critical level. 
We also formulate \conjref{main conj} to the effect that the functor $\Gamma_{\Fl,\nOp}$
is an equivalence.

\medskip

In \secref{Iwahori monodromic} we consider the Iwahori-monodromic subcategories
in the framework of \conjref{main conj}. We formulate the second main
result of this paper, \mainthmref{main}, which asserts that the functor $\Gamma_{\Fl,\nOp}$
induces an equivalence
$$\bD^b(\Coh(\nOp))\underset{\bD^b(\Coh(\tN/\cG))}\otimes \bD^b(\fD(\Fl^{\on{aff}}_G)_\crit\mod)^{I^0}
\to\bD^b(\hg_\crit\mod_\nilp)^{I^0}.$$
In addition, we formulate \mainthmref{Miura}, which sharpens the
description of the category $\bD^b(\hg_\crit\mod_\nilp)^{I^0}$ in terms of quasi-coherent sheaves
on the scheme of Miura opers, proposed in \cite{FG2}.

\medskip

In \secref{relation to Grassmannian} we formulate 
\mainthmref{thm relation to Grassmannian}, which essentially
expresses the category $\bD^b(\fD(\Gr^{\on{aff}}_G)_\crit\mod)$
in terms of $\bD^b(\fD(\Fl^{\on{aff}}_G)_\crit\mod)$ and the action of 
$\bD^b(\Coh(\tN/\cG))$ on it. More precisely, we construct a functor
\begin{multline}   \label{intr Ups}
\Upsilon:
\bD^b(\Coh(\on{pt}/\cB))\underset{\bD^b(\Coh(\tN/\cG))}\otimes \bD^b(\fD(\Fl^{\on{aff}}_G)_\crit\mod)\to \\
\bD^b(\Coh(\on{pt}/\cB))\underset{\bD^b(\Coh(\on{pt}/\cG))}\otimes \bD^b(\fD(\Gr^{\on{aff}}_G)_\crit\mod),
\end{multline}
which will turn out to be "almost" an equivalence.

\medskip

In \secref{compat with sect} we formulate a theorem that connects 
the functors $\Gamma_{\Fl,\nOp}$ and $\Gamma_{\Gr,\rOp}$ by
means of the functor $\Upsilon$ of \secref{relation to Grassmannian}.

\bigskip

In Part II we prove the theorems formulated in Part I, assuming a number
of technical results, which will be the subject of Parts III and IV.

\medskip

In \secref{ups adj} we prove a number of adjunction properties related to 
the functor $\Upsilon$ of \secref{relation to Grassmannian}, and reduce the fully 
faithfulness result of \mainthmref{thm relation to Grassmannian} to a certain isomorphism 
in $\bD^f(\fD(\Gr^{\on{aff}}_G)_\crit\mod)$, given by \thmref{thm zero section}.

\medskip

In \secref{zero section} we prove \thmref{thm zero section}. We give two
proofs, both of which use Bezrukavnikov's theory. One proof uses
some still unpublished results of \cite{Bez}, while another proof
uses only \cite{AB}.

\medskip

In \secref{t structure and Grass} we study the interaction between the
new t-structure on $\bD^f(\fD(\Fl^{\on{aff}}_G)_\crit\mod)$ and the usual
t-structure on $\bD^f(\fD(\Gr^{\on{aff}}_G)_\crit\mod)$.

\medskip

In \secref{calc of endomorphisms} we use the results of \secref{t structure and Grass}
to complete the proof of \thmref{thm zero section}.

\medskip

In \secref{turning Ups} we show how to modify the LHS of \eqref{intr Ups} and the
functor $\Upsilon$ to turn it into an equivalence. We should remark that the
results of this section are not needed for the proofs of the main theorems of
this paper.

\medskip

In \secref{Ups and Gamma} we prove the theorem announced in \secref{compat with sect}
on how the functor $\Upsilon$ intertwines between the functors $\Gamma_{\Fl,\nOp}$
and $\Gamma_{\Gr,\rOp}$.

\medskip

In \secref{proof of Gamma nilp ff} we use the fact that $\Upsilon$ is fully faithful
to deduce fully faithfulness of the functor $\Gamma_{\Fl,\nOp}$, which
is our \mainthmref{Gamma nilp fully faithful}.

\medskip

In \secref{proof of I eq} we prove \mainthmref{main} and \mainthmref{Miura}.

\bigskip

In Part III we develop the machinery used in Parts I and II that has to
do with the notion of tensor product of (triangulated) categories over a 
(triangulated) monoidal category. As is often the case in homotopy
theory, the structure of triangulated category is not rigid enough for
the constructions that we perform. For that reason we will have to 
deal with DG categories  rather than with the
triangulated ones. \footnote{Our decision to work in the framework of DG categories rather
than in a better behaved world of quasi-categories stems from two reasons. One is the fact
that we have not yet learnt the latter theory well enough to apply it. The other is that 
we are still tempted to believe that when working with linear-algebraic objects over a field of 
characteristic zero, on can construct a homotopy-theoretic framework based on DG categories, 
which will avoid some of the combinatorial machinery involved in dealing with simplicial
sets when proving foundational results on quasi-categories.}

\medskip 
 
In \secref{DG reminder} we recall the basics of DG categories and
their relation to triangulated categories. Essentially, we summarize some of the results
of \cite{Dr}. 

\medskip

In \secref{homotopy monoidal categories and actions} we review the notion of
homotopy monoidal structure on a DG category. Our approach amounts to considering
a pseudo-monoidal structure, which yields a monoidal structure on the homotopy level;
this idea was explained to us by A.~Beilinson. We should note that one could consider
a more flexible definition following the prescription of \cite{Lu}; however, as was explained to us
by J.~Lurie, the two approaches are essentially equivalent.
 
\medskip

\secref{tensor products of categories} deals with the tensor product 
of categories, which is a central object for all the constructions in Part I. 
Given two DG categories $\bC_1$ and $\bC_2$ acted on by a monoidal DG category 
$\bA$ on the left and on the right, respectively, we define a new DG category 
$\bC_1\underset\bA\otimes \bC_2$. This construction was explained to us by J.~Lurie.
It essentially consists of taking the absolute tensor product $\bC_1\otimes \bC_2$ 
and imposing the isomorphisms $(\bc_1\cdot \ba)\otimes \bc_2\simeq
\bc_1\otimes (\ba\cdot \bc_2)$, where $\bc_i\in \bC_i$ and $\ba\in \bA$. 

\medskip

In \secref{adjunctions and tightness} we study the properties of the tensor
product construction which can be viewed as generalizations of the projection
formula in algebraic geometry.

\medskip

In \secref{t-structures: a reminder} we recollect some facts related to the notion 
of t-structure on a triangulated category. 

\medskip

In \secref{tensor products and t-structures} we study how the tensor product construction
interacts with t-structures. In particular, we study the relationship between tensor products
at the triangulated and abelian levels.

\medskip

In \secref{categories over stacks} we apply the constructions of Sects. 
\ref{homotopy monoidal categories and actions}, 
\ref{tensor products of categories}, \ref{adjunctions and tightness},
\ref{t-structures: a reminder} and \ref{tensor products and t-structures}
in the particular case when the monoidal triangulated category is the
perfect derived category of coherent sheaves on an algebraic stack.
In this way we obtain the notion of triangulated category over a stack,
and that of base change with respect to a morphism of stacks.

\bigskip

Part IV is of technical nature: we discuss the various triangulated categories
arising in representation theory.

\medskip

\secref{a renormalization procedure} contains a crucial ingredient
needed to make the constructions in Part I work. It turns out that when dealing with
infinite-dimensional objects such as $\Fl^{\on{aff}}_G$ or $\hg$, the usual triangulated categories
associated to them, such as the derived category $\bD(\fD(\Fl^{\on{aff}}_G)_\crit\mod)$
of D-modules in the former case, and the derived category $\bD(\hg\mod)$ of 
$\hg$-representations in the latter case, are not very convenient to work with. 
The reason is that these categories have too few compact objects. 
We show how to modify such categories "at $-\infty$" (i.e., keeping the corresponding
$\bD^+$ subcategories unchanged), so that the resulting categories are still
co-complete (i.e., contain arbitrary direct sums), but become compactly generated.

\medskip

In \secref{rep theory} we apply the discussion of \secref{a renormalization procedure}
in the two examples mentioned above, i.e., $\bD(\hg_\crit\mod)$ and
$\bD(\fD(\Fl^{\on{aff}}_G)_\crit\mod)$, and study the resulting categories, denoted
$\bD_{ren}(\hg_\crit\mod)$ and $\bD_{ren}(\fD(\Fl^{\on{aff}}_G)_\crit\mod)$, respectively.

\medskip

Having developed the formalism of monoidal actions, tensor products,
and having defined the desired representation-theoretic categories equipped
with DG models, in \secref{DG model for the AB action} we upgrade to the DG
level the constructions from Part I, which were initially carried out at the triangulated 
level.

\medskip

In \secref{the I-equivariant situation} we show that imposing the condition of
$I$-monodromicity (which in our case coincides with that of $I^0$-equivariance)
survives the manipulations of Sects. \ref{rep theory} and \ref{DG model for the AB action}.

\ssec{Notation}   \label{notation}

Notation and conventions in this paper follow closely those of \cite{FG2}. 

\medskip

We fix $G$ to be a semi-simple simply connected group over a ground field,
which is algebraically closed and has characteristic zero. We shall denote by
$\cLambda$ the lattice of coweights corresponding to $G$ and by 
$\cLambda^+$ the semi-group of dominant co-weights.

\medskip

We let $\cG$ denote the Langlands dual group of $G$. Let $\cg$ be the Lie
algebra of $\cG$. Let $\cB\subset \cN$ be the Borel subgroup and its unipotent
radical, and $\cb\subset \cn$ be their Lie algebras, respectively. 

\medskip

Let $\on{Fl}^{\cG}$ be the flag variety of $\cG$, thought of as a scheme parameterising  
Borel subalgebras in $\cg$. For $\cla\in \cLambda$ we denote by $\CL^\cla$
the corresponding line bundle on $\on{Fl}^{\cG}/\cG\simeq \on{pt}/\cB$. Our
normalization is such that $\CL^\cla$ is ample if $\cla$ is dominant,
and $\Gamma(\on{Fl}^{\cG},\CL^\cla)=V^\cla$, the irreducible representation
of highest weight $\cla$.

\medskip

We denote by $\tg$ Grothendieck's alteration. This is the 
{\it tautological} sub-bundle in the trivial vector bundle $\on{Fl}^{\cG}\times \cg$. 
Let $\tN\subset \tg$ be the Springer resolution, i.e., it is the variety of pairs
$\{(x\in \cg,\cb'\in \on{Fl}^{\cG})\, |\, x\in \cn'\}$. We denote by $\pi$ the natural projection
$\tN\to \on{Fl}^{\cG}$, and by $\iota$ the zero section $\on{Fl}^{\cG}\to \tN$.

\medskip

When discussing opers or Miura opers, we will mean these objects with respect to
the group $\cG$ (and never $G$), so the subscript "$\cG$" will be omitted. 

\medskip

We will consider the affine Grassmannian $\Gr^{\on{aff}}_G:=G\ppart/G[[t]]$ and
the affine flag scheme $\Fl^{\on{aff}}_G:=G\ppart/I$. We will denote by $p$ the natural 
projection $\Fl^{\on{aff}}_G\to \Gr^{\on{aff}}_G$.

\medskip

For an element $\wt{w}$ in the affine Weyl group, we shall denote by $j_{\wt{w},!}$
(resp., $j_{\wt{w},*}$) the corresponding standard (resp., co-standard) $I$-equivariant
objects of $\fD(\Fl^{\on{aff}}_G)\mod$. For $\cla\in \cLambda$ we denote by $J_\cla\in 
\fD(\Fl^{\on{aff}}_G)\mod$ the corresponding Mirkovic-Wakimoto D-module, which is characterized
by the property that $J_{\cla}=j_{\cla,*}$ for $\cla\in \cLambda^+$, 
$J_{\cla}=j_{\cla,!}$ for $\cla\in -\cLambda^+$ and $J_{\cla_1+\cla_2}=J_{\cla_1}\star 
J_{\cla_2}$.

\medskip

The geometric Satake equivalence (see \cite{MV}) defines a functor from the 
category $\Rep(\cG)$ to that of $G[[t]]$-equivariant objects in 
$\fD(\Gr^{\on{aff}}_G)\mod$. We denote this functor by
$V\mapsto \CF_V$. The construction of \cite{Ga} defines for every $V$ as above
an object $Z_V\in \fD(\Fl^{\on{aff}}_G)\mod$, which is $I$-equivariant, and is {\it central}, a 
property that will be reviewed in the sequel.

\ssec{Acknowledgments} We would like to thank R.~Bezrukavnikov for teaching 
us how to work with the geometric Hecke algebra, 
i.e., Iwahori-equivariant sheaves on $\Fl^{\on{aff}}_G$. In particular he has explained to us
the theory, 
developed by him and his collaborators, of the relationship between this category
and that of coherent sheaves on geometric objects related to the Langlands dual group,
such as the Steinberg scheme $\check{\on{St}}$.

\medskip

We would like to thank J.~Lurie for explaining to us how to resolve a multitude of issues 
related to homotopy theory of DG categories, triangulated categories and t-structures. 
This project could not have been completed without his help.

\medskip

We would like to thank A.~Neeman for helping us prove a key result in
\secref{a renormalization procedure}. 

\medskip

Finally, we would like to thank A.~Beilinson for numerous illuminating discussions 
related to this paper. 

\newpage



\centerline{\bf \large Part I: Constructions}

\bigskip

\section{The Arkhipov-Bezrukavnikov action}   \label{AB action naive}

Let $\bD^f(\fD(\Fl^{\on{aff}}_G)_\crit\mod)$ be the bounded derived category
of finitely generated critically twisted D-modules on $\Fl^{\on{aff}}_G$.
It is well-defined since $\Fl^{\on{aff}}_G$ is a strict ind-scheme of ind-finite
type.

\medskip

The goal of this section is to endow $\bD^f(\fD(\Fl^{\on{aff}}_G)_\crit\mod)$
with a structure of triangulated category over the stack $\tN/\cG$
(see  \secref{over} for the precise definition of what this means). I.e., we will make 
the triangulated monoidal category 
of perfect complexes on $\tN/\cG$ act on $\bD^f(\fD(\Fl^{\on{aff}}_G)_\crit\mod)$.
The action must be understood in the sense of triangulated categories,
equipped with DG models (see \secref{DG models for action}).
In the present section, we will perform the construction at
the triangulated level only, and refer
the reader to \secref{AB action: model}, where it is upgraded to the DG
level.

\ssec{}

Let $\bD^{perf}(\on{Coh}(\tN/\cG))$ be the perfect derived category 
on the stack $\tN/\cG$, as introduced in \secref{perf on stack}, i.e.,
$$\bD^{perf}(\on{Coh}(\tN/\cG)):=\Ho\left(\bC^b(\on{Coh}^{loc.free}(\tN/\cG))\right)/
\Ho\left(\bC^b_{acycl}(\on{Coh}^{loc.free}(\tN/\cG))\right),$$
where $\bC^{b}(\on{Coh}^{loc.free}(\tN/\cG)$ is the DG category of
bounded complexes of locally free coherent sheaves on $\tN/\cG$,
and $\bC^{b}_{acycl}(\on{Coh}^{loc.free}(\tN/\cG)$ is the subcategory
of acyclic complexes. The former has a natural structure of
DG monoidal category, and the latter is a monoidal ideal,
making the quotient $\bD^{perf}(\on{Coh}(\tN/\cG))$ into a
triangulated monoidal category. 

\medskip

In order to define the
action of $\bD^{perf}(\on{Coh}(\tN/\cG))$ on $\bD^{perf}(\on{Coh}(\tN/\cG))$,
following \cite{AB}, we will use a different realization of 
$\bD^{perf}(\on{Coh}(\tN/\cG))$ as a quotient of an explicit triangulated 
monoidal category $\Ho\left(\bC^b(\on{Coh}^{free}(\tN/\cG))\right)$ by a 
monoidal ideal.

\medskip

Namely, let $\bC^b(\on{Coh}^{free}(\tN/\cG))$ be the monoidal DG subcategory of
$\bC^b(\on{Coh}^{loc.free}(\tN/\cG))$, consisting of complexes, whose terms 
are direct sums of coherent sheaves of the form $\pi^*(\CL^\cla)\otimes V$, 
where $V$ is a finite dimensional representation of $\cG$, and 
$\CL^\cla$ is the line bundle on
$\on{Fl}^{\cG}/\cG\simeq \on{pt}/\cB$ corresponding to 
$\cla\in \cLambda$, and $\pi$ denotes the projection $\tN\to \on{Fl}^{\cG}$. 
(We remind that our normalization is that for $\cla\in \cLambda^+$,
$\Gamma(\on{Fl}^{\cG},\CL^\cla)\simeq V^\cla$, 
the representation of highest weight $\cla$.)

\medskip

Let $\bC^b_{acycl}(\on{Coh}^{free}(\tN/\cG))$ be the monoidal ideal
$$\bC^b(\on{Coh}^{free}(\tN/\cG))\cap 
\bC^b_{acycl}(\on{Coh}^{loc.free}(\tN/\cG)).$$ 
By \cite{AB}, Lemma 20, the natural functor 
\begin{multline*}
\Ho\left(\bC^b(\on{Coh}^{free}(\tN/\cG))\right)/
\Ho\left(\bC^b_{acycl}(\on{Coh}^{free}(\tN/\cG))\right)\to \\
\to \Ho\left(\bC^b(\on{Coh}^{loc.free}(\tN/\cG))\right)/
\Ho\left(\bC^b_{acycl}(\on{Coh}^{loc.free}(\tN/\cG))\right)
\bD^{perf}(\on{Coh}(\tN/\cG))
\end{multline*}
is an equivalence. 

\ssec{}

We claim, following \cite{AB}, that there exists a natural action of 
$\Ho\left(\bC^b(\on{Coh}^{free}(\tN/\cG))\right)$ on $\bD^f(\fD(\Fl^{\on{aff}}_G)_\crit\mod)$,
with $\Ho\left(\bC^b_{acycl}(\on{Coh}^{free}(\tN/\cG))\right)$ acting trivially.

\sssec{}   \label{functor sF}

First, we construct a DG monoidal functor $\sF$ from 
$\bC^b(\on{Coh}^{free}(\tN/\cG))$ to a DG monoidal subcategory of the 
DG category of finitely generated $I$-equivariant D-modules on $\Fl^{\on{aff}}_G$.
The functor $\sF$ will have the property that if $\CF_1$ and $\CF_2$
appear as terms of complexes of some $\sF(\CM^\bullet_1)$
and $\sF(\CM^\bullet_2)$, respectively, for $\CM^\bullet_1,
\CM^\bullet_2\in \bC^b(\on{Coh}^{free}(\tN/\cG))$, the convolution
$\CF_1\star \CF_2\in \bD^f(\fD(\Fl^{\on{aff}}_G)_\crit\mod)$ is acyclic
off cohomological degree $0$.

\medskip

\noindent{\it Remark.} In \secref{new t structure} a new t-structure 
on (an ind-completion of) $\bD^f(\fD(\Fl^{\on{aff}}_G)_\crit\mod)$ will be defined, and 
we will prove that the functor $\sF$ is exact in this t-structure, see \corref{F exact}.

\medskip

The functor $\sF$ is characterized uniquely by the following conditions:

\begin{itemize}

\item For $\cla\in \cLambda$, $\sF(\pi^*(\CL^\cla)):=J_\cla$.
The isomorphism $\CL^{\cla_1}\otimes \CL^{\cla_2}\simeq
\CL^{\cla_1+\cla_2}$ goes under $\sF$ to the natural isomorphism
$J_{\cla_1}\star J_{\cla_2}\simeq J_{\cla_1+\cla_2}$.

\medskip

\item For $V\in \Rep^{f.d.}(\cG)$, $\sF(\CO_{\tN/\cG}\otimes V):=Z_V$,
where $Z_V$ is the corresponding central sheaf. We have
commutative diagrams
$$
\CD
\sF(\CO_{\tN/\cG}\otimes V^1)\star \sF(\CO_{\tN/\cG}\otimes V^2)
@>>>  Z_{V^1}\star Z_{V^2}  \\
@VVV   @VVV   \\
\sF\left((\CO_{\tN/\cG}\otimes (V^1\otimes V^2)\right) @>>>
Z_{V^1\otimes V^2}
\endCD
$$
and
$$
\CD
\sF(\CO_{\tN/\cG}\otimes V)\star \sF(\pi^*(\CL^\cla)) @>>>
Z_V\star J_\cla \\
@VVV   @VVV  \\
\sF(\pi^*(\CL^\cla))\star \sF(\CO_{\tN/\cG}\otimes V) @>>>
J_\cla\star Z_V,
\endCD
$$
where the right vertical maps are the canonical morphisms of \cite{Ga}, Theorem 1(c) and (b),
respectively.

\item The tautological endomorphism $N^{taut}_V$ of the object 
$\CO_{\tN/\cG}\otimes V$ goes over under $\sF$ to the
monodromy endomorphism $N_{Z_V}$ of $Z_V$
(the latter morphism is given by Theorem 2 of \cite{Ga} 
and is denoted there by $M$).

\item
For $\cla\in \cLambda^+$ the canonical map
$\CO_{\tN/\cG}\otimes V^\cla\to \pi^*(\CL^\cla)$ goes
over to the map $Z_{V^\cla}\to J_{\cla}$, given by Lemma 9 of \cite{AB}.
\end{itemize}

\medskip

Thus, we obtain the desired functor $\sF$. It is shown in \cite{AB},
Lemma 20(a), that if $\CM^\bullet\in \bC^b_{acycl}(\on{Coh}^{free}(\tN/\cG))$,
then $\sF(\CM^\bullet)$ is acyclic as a complex of D-modules on $\Fl^{\on{aff}}_G$.

\sssec{} The assignment

$$\CM\in \Ho\left(\bC^b(\on{Coh}^{free}(\tN/\cG))\right), \CF\in 
\bD^f(\fD(\Fl^{\on{aff}}_G)_\crit\mod)\mapsto 
\CF\star \sF(\CM)\in \bD^f(\fD(\Fl^{\on{aff}}_G)_\crit\mod)
$$
defines a functor
$$\Ho\left(\bC^b(\on{Coh}^{loc.free}(\tN/\cG))\right)\times
\bD^f(\fD(\Fl^{\on{aff}}_G)_\crit\mod)\to \bD^f(\fD(\Fl^{\on{aff}}_G)_\crit\mod),$$
which is the sought-for monoidal action of $\bD^{perf}(\on{Coh}(\tN/\cG))$
on $\bD^f(\fD(\Fl^{\on{aff}}_G)_\crit\mod)$. For $\CM$ and $\CF$ as above,
we denote the resulting object of $\bD^f(\fD(\Fl^{\on{aff}}_G)_\crit\mod)$
by
\begin{equation} \label{notation for action}
\CM\underset{\CO_{\tN/\cG}}\otimes \CF.
\end{equation}

As was mentioned above, in 
\secref{AB action: model} we will upgrade this action to the DG level.

\ssec{}

Recall the scheme classifying $\cG$-opers on the formal punctured disc
$\D^\times$ with nilpotent singularities, introduced in \cite{FG2}, Sect. 2.13. 
We denote this scheme by $\nOp$. This is an affine scheme
of infinite type, isomorphic to the infinite-dimensional affine space.
By \cite{FG2}, Sect. 2.18, there exists a canonical smooth map
$$\fr_\nilp:\nOp\to \tN/\cG,$$
that corresponds to taking the residue of the connection corresponding
to an oper.

\medskip

By \secref{base change}, we have a well-defined base-changed
triangulated category
\begin{equation} \label{base changed category}
\nOp\underset{\tN/\cG}\times \bD^f(\fD(\Fl^{\on{aff}}_G)_\crit\mod),
\end{equation}
equipped with a DG model, which carries an action of the monoidal
triangulated category $\bD^{perf}(\on{Coh}(\nOp))$, where the latter
category and the action are also equipped with DG models. 

The category \eqref{base changed category} is the main character
of this paper.

\sssec{}   \label{univ ppty nilp}

As was explained in the introduction, the base change construction
is a categorical version of the tensor product construction for
modules over an associative algebra. In particular, it satisfies a 
certain universal property (see \secref{univ ppty base change}),
which when applied to our situation reads as follows:

\medskip

Let $\bD'$ be a triangulated category {\it over} the scheme $\nOp$,
i.e., $\bD'$ is acted on by $\bD^{perf}(\on{Coh}(\nOp))$ and 
both the category and the action are equipped with DG models.
Then functors 
$$\nOp\underset{\tN/\cG}\times \bD^f(\fD(\Fl^{\on{aff}}_G)_\crit\mod)\to \bD',$$
compatible with the action of $\bD^{perf}(\on{Coh}(\nOp))$ (where the
compatibility data is also equipped with a DG model) are in bijection
with functors
$$\bD^f(\fD(\Fl^{\on{aff}}_G)_\crit\mod)\to \bD',$$
that are compatible with the action of $\bD^{perf}(\on{Coh}(\tN/\cG))$,
where the latter acts on $\bD'$ via the monoidal functor
$\fr_\nilp^*:\bD^{perf}(\on{Coh}(\tN/\cG))\to \bD^{perf}(\on{Coh}(\nOp))$.

\sssec{}

We have the tautological pull-back functor, denoted by a slight abuse of
notation by the same character
$$\fr_\nilp^*:\bD^f(\fD(\Fl^{\on{aff}}_G)_\crit\mod)\to 
\nOp\underset{\tN/\cG}\times \bD^f(\fD(\Fl^{\on{aff}}_G)_\crit\mod).$$

An additional piece of information on the category \eqref{base changed category}
is that we know how to calculate $\Hom$ in it between objects of the form
$\fr_\nilp^*(\CF_i)$, $\CF_i\in \bD^f(\fD(\Fl^{\on{aff}}_G)_\crit\mod)$, $i=1,2$. Namely,
this is given by \corref{tight mon}(2), and in our situation it reads as follows:
\begin{equation} \label{hom in base change nilp}
\Hom\left(\fr_\nilp^*(\CF_1),\fr_\nilp^*(\CF_2)\right)\simeq
\Hom(\CF_1,(\fr_\nilp)_*(\CO_{\nOp})\underset{\CO_{\tN/\cG}}\otimes \CF_2),
\end{equation}
where $(\fr_\nilp)_*(\CO_{\nOp})$ (resp., 
$(\fr_\nilp)_*(\CO_{\nOp})\underset{\CO_{\tN/\cG}}\otimes \CF_2)$) 
is regarded as an object of the ind-completion of $\bD^{perf}(\on{Coh}(\tN/\cG))$
(resp., the ind-completion of the category \eqref{base changed category}).

\medskip

The universal property of \secref{univ ppty nilp} and \eqref{hom in base change nilp}
is essentially all the information that we have about the category \eqref{base changed category},
but it will suffice to prove a number of results.

\section{The new $\bbt$-structure}   \label{new t structure}

As a tool for the study of the category $\bD^f(\fD(\Fl^{\on{aff}}_G)_\crit\mod)$,
we shall now introduce a new t-structure on the ind-completion of this
category. Its main property will be that the functors of convolution $?\star J_\cla$
become exact in this new t-structure. 

\ssec{}

As the triangulated category $\bD^f(\fD(\Fl^{\on{aff}}_G)_\crit\mod)$ is equipped
with a DG model, it has a well-defined ind-completion 
(see \secref{big and small, triang}), which we denote 
$\bD_{ren}(\fD(\Fl^{\on{aff}}_G)_\crit\mod)$. This is a co-complete triangulated category,
which is generated by the subcategory of its compact objects, the 
latter being identified with $\bD^f(\fD(\Fl^{\on{aff}}_G)_\crit\mod)$ itself.

\sssec{}

By \secref{cat on Fl}, the usual t-structure on $\bD(\fD(\Fl^{\on{aff}}_G)_\crit\mod)$ induces
a t-structure on $\bD_{ren}(\fD(\Fl^{\on{aff}}_G)_\crit\mod)$; moreover, the latter 
is compactly generated (see \secref{compact gener t} where this notion is introduced).
We have an exact functor
$$\bD_{ren}(\fD(\Fl^{\on{aff}}_G)_\crit\mod)\to \bD(\fD(\Fl^{\on{aff}}_G)_\crit\mod),$$
which induces an equivalence
$$\bD^+_{ren}(\fD(\Fl^{\on{aff}}_G)_\crit\mod)\to \bD^+(\fD(\Fl^{\on{aff}}_G)_\crit\mod).$$

In what follows we will refer to this t-structure on 
$\bD_{ren}(\fD(\Fl^{\on{aff}}_G)_\crit\mod)$ as
the "old" one. 

\sssec{}

Recall the general framework of defining a t-structure on a co-complete
triangulated category given by \lemref{t general}.

\medskip

We define a new t-structure on $\bD_{ren}(\fD(\Fl^{\on{aff}}_G)_\crit\mod)$
so that $\bD^{\leq 0_{new}}_{ren}(\fD(\Fl^{\on{aff}}_G)_\crit\mod)$ is generated
by objects $\CF\in \bD^f(\fD(\Fl^{\on{aff}}_G)_\crit\mod)$ that satisfy:
\begin{equation}  \label{less than 0} 
\CF\star J_{\cla}\in \bD_{ren}^{\leq 0_{old}}(\fD(\Fl^{\on{aff}}_G)_\crit\mod) 
\text{ for any } \cla\in \cLambda. 
\end{equation}
Since for $\cla\in \cLambda^+$ the
functor $\CF\mapsto \CF\star J_\cla$ is right-exact (in the old
t-structure), condition \eqref{less than 0} is sufficient to check for
a subset of $\cLambda$ of the form $\cla_0-\cLambda^+$
for some/any $\cla_0\in \cLambda$.

\sssec{Remarks}

At the moment we do not know how to answer some basic 
questions about the new t-structure on $\bD_{ren}(\fD(\Fl^{\on{aff}}_G)_\crit\mod)$.
For example:

\medskip

\noindent{\it Question:} Is the  new t-structure compatible with the subcategory
$\bD^f(\fD(\Fl^{\on{aff}}_G)_\crit\mod)\subset \bD_{ren}(\fD(\Fl^{\on{aff}}_G)_\crit\mod)$,
i.e., is the subcategory $\bD^f(\fD(\Fl^{\on{aff}}_G)_\crit\mod)$ preserved by
the truncation functors? \footnote{Probably, that answer to this question is negative.}

\medskip

However, we propose the following conjectures:

\begin{conj} \label{t-structure conj one} Let $\CF\in \bD_{ren}(\fD(\Fl^{\on{aff}}_G)_\crit\mod)$
satisfy $\CF\star J_\cla\in \bD_{ren}^{\leq 0_{old}}(\fD(\Fl^{\on{aff}}_G)_\crit\mod)$
for any $\cla\in \cLambda$. Then $\CF\in \bD_{ren}^{\leq 0_{new}}(\fD(\Fl^{\on{aff}}_G)_\crit\mod)$.
\end{conj}

\medskip

The second conjecture has to do with the stability of the above t-structure with
respect to base change. Let $A$ be an associative algebra, and consider the
triangulated category $$\bD_{ren}(\fD(\Fl^{\on{aff}}_G)_\crit\otimes A\mod):=
\bD^f(\fD(\Fl^{\on{aff}}_G)_\crit\mod)\underset{\to}\otimes \bD^{perf}(A\mod),$$
i.e., the ind-completion of 
$\bD^f(\fD(\Fl^{\on{aff}}_G)_\crit\mod)\otimes \bD^{perf}(A\mod)$.

It is endowed with the usual, a.k.a. "old", t-structure, equal to
the tensor product of the "old" t-structure on 
$\bD_{ren}(\fD(\Fl^{\on{aff}}_G)_\crit\mod)$ and the usual t-structure on $\bD(A\mod)$,
see \secref{t structure ten prod}

\medskip

However, there are two ways to introduce a "new" t-structure on
$\bD_{ren}(\fD(\Fl^{\on{aff}}_G)_\crit\otimes A\mod)$. The t-structure
"new$_1$" is obtained by the tensor product construction of 
\secref{t structure ten prod} from the new t-structure on 
$\bD_{ren}(\fD(\Fl^{\on{aff}}_G)_\crit\mod)$ and the usual t-structure on $\bD(A\mod)$.

\medskip

The t-structure "new$_2$" is defined to be generated by compact objects,
i.e., objects $\CF\in \bD^f(\fD(\Fl^{\on{aff}}_G)_\crit\mod)\otimes \bD^{perf}(A\mod)$, such 
that $\CF\star J_{\cla}$ is $\leq 0$ in the old sense.

\begin{conj} \label{t-structure conj two}
The t-structures "new$_1$" and "new$_2$" on
$\bD_{ren}(\fD(\Fl^{\on{aff}}_G)_\crit\otimes A\mod)$ coincide.
\end{conj}

\noindent We remark, that, more generally, instead of the category $\bD(A\mod)$,
we could have taken any compactly generated triangulated category, 
equipped with a DG model, and a compactly generated t-structure.

\ssec{}

Here are some of the basic properties of the new t-structure. 

\sssec{}

First, the functors $\CF\mapsto \CF\star J_{\cla}$ are exact in the
new t-structure. 

\medskip

Further, we have:

\begin{prop} \label{F right-exact and heart}
For $\CM\in \QCoh(\tN/\cG)$ the functor
$$\CF\mapsto \CF\star \sF(\CM)$$ is
right-exact in the new t-structure. If $\CM$ is
flat, it is exact.
\end{prop}

\begin{proof}
Any $\CM$ as in the proposition is quasi-isomorphic to a direct summand
of a complex $\CM^\bullet$, situated in non-positive cohomological degrees, 
such that each $\CM^\bullet$ is a direct sum of line bundles $\pi^*(\CL^\cla)$.
This readily implies the first point of the proposition.

\medskip

The second assertion follows from the first one, see \secref{flat exact}.

\end{proof}

\sssec{}

The next lemma shows that the new t-structure induces the old
one on the finite-dimensional flag variety $G/B\subset \Fl^{\on{aff}}_G$:

\begin{lem}  \label{finite flags}
Any D-module $\CF\in \fD(G/B)\mod\subset \fD(\Fl^{\on{aff}}_G)_\crit\mod$
belongs to the heart of the new t-structure. 
\end{lem}

The proof follows from the fact that the map defining the convolution
$\CF\star J_{-\cla}$ with $\cla\in \cLambda^+$ is one-to-one over the supports
of the corresponding D-modules.

\ssec{}

The identity functor on $\bD_{ren}(\fD(\Fl^{\on{aff}}_G)_\crit\mod)$ is tautologically
left-exact when viewed as a functor from the old t-structure to the new one.
We claim, however, that the deviation is by a finite amount:

\begin{prop}  \label{finite deviation}
Any $\CF\in \bD_{ren}(\fD(\Fl^{\on{aff}}_G)_\crit\mod)$, which is
$\geq 0_{new}$, is $\geq -\dim(G/B)_{old}$.
\end{prop}

\begin{proof}

Let $\CF$ be as in the proposition. We have to show that 
$\Hom(\CF',\CF)=0$ for any $\CF'\in \bD^f(\fD(\Fl^{\on{aff}}_G)_\crit\mod)$,
which is $<-\dim(G/B)_{old}$.

\medskip

It would be sufficient to show that for any such $\CF'$, the objects
$\CF'\star J_{-\cla}$ are $< 0_{old}$ for $\cla\in \cLambda^+$. I.e.:

\begin{lem}  \label{J la bdd}
For $\cla\in \cLambda^+$, 
the functor $?\star J_{-\cla}$ has a cohomological amplitude
(in the old t-structure) bounded by $\dim(G/B)$.
\end{lem}

\end{proof}

\sssec{Proof of \lemref{J la bdd}}   \label{proof of J la bdd}

Consider the object $$\CO_{\Delta_{\on{pt}/\cB}}\in 
\Coh(\on{pt}/\cB\underset{\on{pt}/\cG}\times \on{pt}/\cB)\simeq
\Coh(\on{Fl}^{\cG}\times \on{Fl}^{\cG}/\cG).$$ It can be realized as a direct
summand of a complex, concentrated in non-positive degrees
and of length $\dim(\on{Fl}^{\cG})=\dim(G/B)$, whose terms are of the form
$\CL^{\cmu_i}\underset{\CO_{\on{pt}/\cG}}\boxtimes \CL^{\cla_i}$,
where with no restriction of generality we can assume that 
$\cmu_i\in \cLambda^+$.

\medskip

This implies that for any $\cla$, the object $\CL^{-\cla}\in \bD^{perf}
(\Coh(\on{pt}/\cB))$
is a direct summand of an object that
can be written as a successive extension of objects of the form
$\CL^{\cmu}\otimes V[-k]$ with $k\leq \dim(\on{Fl}^{\cG})$
and $\cmu\in \cLambda^+$ and $V\in \Rep^{f.d.}(\cG)$.

\medskip

Hence, $\CF\star J_{-\cla}$ is a direct summand of an object which is
a successive extension of objects
of the form $\CF\star J_{\cmu}\star Z_{V}[-k]$. However, the functor
$?\star J_{\cmu}$ is right-exact, and $?\star Z_{V}$ is exact.

\qed

\section{Functor to modules at the critical level}   \label{functor to the critical level}

Let $\hg_\crit\mod_\nilp$ be the abelian category of
$\hg_\crit\mod$, on which the center $\fZ_\fg$ acts 
through its quotient $\fZ^\nilp_\fg$ (see \cite{FG2}, Sect. 7.1).
Let $\bD(\hg_\crit\mod_\nilp)$ be its derived category.
In this section we will consider the functor of global sections
$\Gamma_\Fl:\bD^f(\fD(\Fl^{\on{aff}}_G)_\crit\mod)\to \bD(\hg_\crit\mod_\nilp)$,
and using \secref{univ ppty nilp} we will extend it to a functor
$$\Gamma_{\Fl,\nOp}:\nOp\underset{\tN/\cG}\times \bD^f(\fD(\Fl^{\on{aff}}_G)_\crit\mod)\to
\bD(\hg_\crit\mod_\nilp).$$

\ssec{}

Being the derived category of an abelian category, $\bD(\hg_\crit\mod_\nilp)$
is equipped with a natural DG model. Moreover, the abelian category
$\hg_\crit\mod_\nilp$ has $\fZ^\nilp_\fg\simeq \CO_{\nOp}$ mapping 
to its center. This defines on $\bD(\hg_\crit\mod_\nilp)$ a structure
of triangulated category over the (affine) scheme $\nOp$. In
particular, we have a monoidal action of the monoidal triangulated
category $\bD^{perf}(\on{Coh}(\nOp))$ on it.

\sssec{}   \label{nilpotent category}

In \secref{good category sZ} we will introduce another
triangulated category (also equipped with a DG model
and acted on by $\bD^{perf}(\on{Coh}(\nOp))$),
denoted $\bD_{ren}(\hg_\crit\mod_\nilp)$. This category
is also co-complete and has a t-structure, but unlike 
$\bD(\hg_\crit\mod_\nilp)$, the new category is
generated by its subcategory of compact objects
denoted $\bD^f_{ren}(\hg_\crit\mod_\nilp)$. 

\medskip

In addition, we have a functor (equipped with a DG model,
and compatible with the action of $\bD^{perf}(\on{Coh}(\nOp))$)
$$\bD_{ren}(\hg_\crit\mod_\nilp)\to \bD(\hg_\crit\mod_\nilp),$$
which is exact and induces an equivalence 
$$\bD^+_{ren}(\hg_\crit\mod_\nilp)\to \bD^+(\hg_\crit\mod_\nilp).$$

\medskip

We have
$$\bD^f_{ren}(\hg_\crit\mod_\nilp)\subset \bD^+_{ren}(\hg_\crit\mod_\nilp),$$
so we can identify of $\bD^f_{ren}(\hg_\crit\mod_\nilp)$ with its essential
image in $\bD(\hg_\crit\mod_\nilp)$, denoted $\bD^f(\hg_\crit\mod_\nilp)$.

\medskip

By \secref{action on subcat}, all of the above triangulated categories
inherit DG models and the action of $\bD^{perf}(\on{Coh}(\nOp))$,
so they are triangulated categories over the scheme $\nOp$.

\ssec{}

Our present goal is to construct a functor
$$\Gamma_{\Fl,\nOp}:\nOp\underset{\tN/\cG}\times \bD^f(\fD(\Fl^{\on{aff}}_G)_\crit\mod)
\to \bD^f(\hg_\crit\mod_\nilp),$$
compatible with an action of $\bD^{perf}(\on{Coh}(\nOp))$. Both the
functor, and the compatibility isomorphisms will be equipped with DG
models. 

\sssec{}

By \secref{univ ppty base change}, a functor as above would be defined once
we define a functor
\begin{equation} \label{R Gamma to finite}
\Gamma_{\Fl}:\bD^f(\fD(\Fl^{\on{aff}}_G)_\crit\mod)\to 
\bD^f(\hg_\crit\mod_\nilp),
\end{equation}
which is compatible with the action of $\bD^{perf}(\on{Coh}(\tN/\cG))$
(where the action on the RHS is via $\fr_\nilp$),
such that again both the functor and the compatibility isomorphisms 
are equipped with DG models. 

\medskip

We define the functor 
\begin{equation} \label{R Gamma}
\Gamma_{\Fl}:\bD^f(\fD(\Fl^{\on{aff}}_G)_\crit\mod)\to 
\bD(\hg_\crit\mod_\nilp),
\end{equation}
to be $\Gamma(\Fl,?)$, i.e., the derived functor of global sections 
of a critically twisted D-module.

\medskip

In this section we will
check the required compatibility on the triangulated level.
In \secref{sections: model} we will upgrade this
construction to the DG level. In \secref{sections to finite}
we will show that the image of
\eqref{R Gamma to finite} belongs to
$\bD^f(\hg_\crit\mod_\nilp)$, thereby constructing the functor
\eqref{R Gamma to finite} with the required properties.

\ssec{}  \label{a-e}

Let us introduce some notations:
$$\CL_{\nOp}^\cla:=\fr_\nilp^*(\pi^*(\CL^\cla)), \cla\in \cLambda \text{  and  }
\CV_{\nOp}:=\fr_\nilp^*\left(\CO_{\tN/\cG}\otimes V\right), V\in \Rep^{f.g.}(\cG).$$

\medskip

The data of compatibility of $\Gamma_{\Fl}$
with the action of $\bD^{perf}(\on{Coh}(\tN/\cG))$ would follow from the corresponding data for $\Ho(\bC^b(\on{Coh}^{free}(\tN/\cG)))$ (see \lemref{action on quotient}).

\medskip

By \cite{AB}, Proposition 4, the latter amounts to constructing
the following isomorphisms for an object $\CF\in 
\bD^f(\fD(\Fl^{\on{aff}}_G)_\crit\mod)$:

\begin{itemize}

\item(i)
For $\cla\in \cLambda$,
$$\Gamma_\Fl(\CF\star J_{\cla})\overset{\gamma_\cla}\simeq 
\Gamma_\Fl(\CF)\underset{\CO_{\nOp}}\otimes \CL^\cla_\nOp.$$

\item(ii)
For $V\in \Rep^{f.d.}(\cG)$, an isomorphism
$$\Gamma_\Fl(\CF\star Z_V)\overset{\gamma_V}\simeq 
\Gamma_\Fl(\CF)\underset{\CO_{\nOp}}\otimes 
\CV_\nOp,$$

\end{itemize}
and such that the following conditions hold:

\begin{itemize}

\item(a) For $\cla_1,\cla_2\in \cLambda$, the diagram
$$
\CD
\Gamma_\Fl(\CF\star J_{\cla_1}\star J_{\cla_2}) @>{\gamma_{\cla_2}}>>
\Gamma_\Fl(\CF\star J_{\cla_1})\underset{\CO_{\nOp}}\otimes 
\CL^{\cla_2}_\nOp \\
@VVV   @V{\gamma_{\cla_1}}VV  \\
\Gamma_\Fl(\CF\star J_{\cla_1+\cla_2}) @>{\gamma_{\cla_1+\cla_2}}>>
\Gamma_\Fl(\CF)\underset{\CO_{\nOp}}\otimes \CL^{\cla_1+\cla_2}_\nOp
\endCD
$$
commutes.

\item(b)
The endomorphism of the object $\Gamma_\Fl(\CF\star Z_V)$, induced by the
monodromy endomorphism $N_{Z_V}$ of $Z_V$ goes over by means
of $\gamma_V$ to the endomorphism of the object $\Gamma_\Fl(\CF)\underset{\CO_{\nOp}}\otimes 
\CV_\nOp$, induced by the tautological 
endomorphism of $\CO_{\tN/\cG}\otimes V$.

\item(c)
For $V^1,V^2\in \Rep^{f.d.}(\cG)$ the diagram 
$$
\CD
\Gamma_\Fl(\CF\star Z_{V^1}\star Z_{V^2})  @>{\gamma_{V^2}}>>
\Gamma_\Fl(\CF\star Z_{V^1})\underset{\CO_{\nOp}}\otimes 
\CV^2_\nOp  \\
@VVV    @V{\gamma_{V^1}}VV  \\
\Gamma_\Fl(\CF\star Z_{V^1\otimes V^2}) @>{\gamma_{V^1\otimes V^2}}>>
\Gamma_\Fl(\CF)\underset{\CO_{\nOp}}\otimes (\CV^1_\nOp
\underset{\CO_{\nOp}}\otimes \CV^2_\nOp)
\endCD
$$
commutes.

\item(d)
$\cla\in \cLambda$ and $V\in \Rep^{f.d.}(\cG)$, the diagram
$$
\CD
\Gamma_\Fl(\CF\star J_{\cla}\star Z_V) @>{\gamma_V}>> 
\Gamma_\Fl(\CF\star J_{\cla})\underset{\CO_{\nOp}}\otimes \CV_\nOp \\
@VVV   @V{\gamma_\cla}VV  \\
\Gamma_\Fl(\CF\star Z_V\star J_{\cla}) & & 
\Gamma_\Fl(\CF)\underset{\CO_{\nOp}}\otimes 
\CL^\cla_\nOp\underset{\CO_{\nOp}}\otimes \CV_\nOp \\
@V{\gamma_\cla}VV    @VVV  \\
\Gamma_\Fl(\CF\star Z_V)
\underset{\CO_{\nOp}}\otimes \CL^\cla_\nOp @>{\gamma_V}>>
\Gamma_\Fl(\CF)\underset{\CO_{\nOp}}\otimes 
\CV_\nOp  \underset{\CO_{\nOp}}\otimes 
\CL^\cla_\nOp
\endCD
$$
commutes, where the first left vertical arrow is the isomorphism
of \cite{Ga}, Theorem 1(b).

\medskip

\item(e)
For $\cla\in \cLambda^+$ the canonical map $Z_{V^\cla}\to J_\cla$
makes the following diagram commutative:
$$
\CD
\Gamma_\Fl(\CF\star Z_{V^\cla}) @>>> \Gamma_\Fl(\CF\star J_{\cla}) \\
@V{\gamma_{V^\cla}}VV    @V{\gamma_\cla}VV   \\
\Gamma_\Fl(\CF) \underset{\CO_{\nOp}}\otimes 
\CV^\cla_\nOp @>>>
\Gamma_\Fl(\CF) \underset{\CO_{\nOp}}\otimes 
\CL^\cla_\nOp,
\endCD
$$
where the bottom horizontal arrow comes from the canonical map
$V^\cla\otimes \CO_{\on{Fl}^{\cG}}\to \CL^\cla$.

\end{itemize}

\ssec{}

To construct the above isomorphisms we will repeatedly use the fact that
for $\CF'\in \fD(\Fl^{\on{aff}}_G)_\crit\mod^I$, $\CF\in \bD^f(\fD(\Fl^{\on{aff}}_G)_\crit\mod)$
\begin{equation} \label{convolution principle}
\Gamma_\Fl(\CF\star \CF')\simeq \CF\star \Gamma_\Fl(\CF')\in \bD(\hg_\crit\mod_\nilp).
\end{equation}

\medskip

By the definition of the critical twisting on $\Fl^{\on{aff}}_G$, we have
\begin{equation} \label{Gamma of delta}
\Gamma_\Fl(\delta_{1,\Fl^{\on{aff}}_G})\simeq \BM_{\crit,-2\rho}.
\end{equation}
Here $\delta_{1,\Fl^{\on{aff}}_G}$ is the $\delta$-function twisted
D-module on $\Fl^{\on{aff}}_G$ at the point $1\in \Fl^{\on{aff}}_G$,
and $\BM_{\crit,-2\rho}$ denotes the Verma module at the critical level 
with highest weight $-2\rho$. 

\medskip

From \eqref{Gamma of delta}, we obtain that for any 
$\CF\in \bD^f(\fD(\Fl^{\on{aff}}_G)_\crit\mod)$,
$$\Gamma_\Fl(\CF)\simeq \CF\star \BM_{\crit,-2\rho}.$$

\sssec{}

Thus, to construct isomorphisms $\gamma_\cla$ as in (i),
it is enough to construct an isomorphism
$$J_{\cla}\star \BM_{\crit,-2\rho}\simeq \BM_{\crit,-2\rho}
\underset{\CO_{\nOp}}\otimes \CL^\cla_\nOp.$$

\medskip

By the definition of the map $\fr_\nilp$ (see \cite{FG2}, Sect. 2.18),
$$\CL^\cla_\nOp\simeq \CO_{\nOp}\otimes (\fl')^\cla,$$
where $\cla\to (\fl')^\cla$ is a $\cT$-torsor.
In addition, by \cite{FG2}, Corollary 13.12
$$J_{\cla}\star \BM_{\crit,-2\rho}\simeq \BM_{\crit,-2\rho}\otimes 
(\fl'')^\cla,$$
where $\cla\to (\fl'')^\cla$ is also a $\cT$-torsor. Moreover, there exists
a canonical isomorphism of $\cT$-torsors
$$(\fl')^\cla \simeq (\fl'')^\cla,$$
given by \cite{FG5}, equations (6.3), (6.4) and Proposition 6.11(2)
of {\it loc. cit.}

\medskip

This gives rise to the isomorphism $\gamma_\cla$. Condition (a)
follows from the construction.

\sssec{}

Isomorphism $\gamma_V$ results via \eqref{convolution principle}
from the isomorphism established in \cite{FG3}, Theorem 5.4. Namely,
this theorem asserts that for every $\CM\in \hg_\crit\mod_\nilp^I$
and $V\in \Rep(\cG)$ there exists a canonical isomorphism
\begin{equation} \label{FG3}
Z_V\star \CM\simeq \CM\underset{\CO_{\nOp}}\otimes \CV_{\nOp}.
\end{equation}

The fact that conditions (b) and (c) are satisfied is included in the formulation
of the above result in {\it loc. cit.}

\sssec{}

Condition (d) follows from the functoriality
of the isomorphism \eqref{FG3} and the fact that
the following diagram is commutative for any $\CF\in 
\bD^f(\fD(\Fl^{\on{aff}}_G)_\crit\mod)$ and $\CM\in \hg_\crit\mod_\nilp^I$:
\begin{equation} \label{com and fusion}
\CD
(Z_V\star \CF)\star \CM @>{\sim}>> (\CF\star Z_V)\star \CM \\
@VVV   @VVV \\
Z_V\star (\CF\star \CM) & &  \CF\star (Z_V\star \CM) \\ 
@V{\sim}VV            @V{\sim}VV    \\
\CV_{\nOp}
\underset{\CO_\nOp}\otimes
(\CF\star \CM)  @>{\sim}>>
\CF\star (\CV_{\nOp} \underset{\CO_\nOp}\otimes \CM),
\endCD
\end{equation}
The commutativity of the above diagram follows, in turn, from the construction of the isomorphism 
$Z_V\star \CF\simeq \CF\star Z_V$ and that of 
Theorem 5.4 of \cite{FG3} via nearby cycles.

\sssec{}

Finally, let us prove the commutativity of the diagram in (e).

\medskip

As before, it suffices to consider the case of $\CF=\delta_{1,\Fl^{\on{aff}}_G}$.
Let us choose a coordinate on a formal disc, and consider
the resulting grading on the corresponding modules. We obtain that
the space of degree $0$ morphisms 
$$\BM_{\crit,-2\rho}\underset{\CO_{\nOp}}\otimes \CV^\cla_{\nOp}\to
\BM_{\crit,-2\rho}\otimes \fl^\cla$$
is one-dimensional. Therefore, it is enough to show the commutativity
of the following diagram instead:
\begin{equation}   \label{crucial diag}
\CD
Z_{V^\cla}\star \BM_{\crit,-2\rho,\reg}  @>>> J_{\cla}\star \BM_{\crit,-2\rho,\reg} \\
@V{\sim}VV    @V{\sim}VV   \\
\BM_{\crit,-2\rho,\reg}\underset{\CO_\nOp}\otimes \CV^\cla_{\nOp} @>>>
\BM_{\crit,-2\rho,\reg}\underset{\CO_\nOp}\otimes \CL^\cla_{\nOp},
\endCD
\end{equation}
where
$$\BM_{\crit,-2\rho,\reg}:=\BM_{\crit,-2\rho}\underset{\CO_{\nOp}}\otimes \CO_{\rOp},$$
and the left vertical arrow is given by \eqref{FG3}, and the right vertical arrow
is given by Proposition 6.11(2) of \cite{FG5}.

\medskip

Further, it is enough to show the commutativity of the following diagram, obtained
from \eqref{crucial diag} by composing with the canonical morphism 
$\BV_\crit\otimes \fl^{-2\crho}\to \BM_{\crit,-2\rho,\reg}$ of equation (6.6) of \cite{FG5},
where $\BV_\crit$ is the vacuum module at the critical level:
$$
\CD
Z_{V^\cla}\star \BV_\crit \otimes \fl^{-2\crho} @>>> J_\cla\star \BV_\crit \otimes \fl^{-2\crho} \\
@VVV    @VVV  \\
\BM_{\crit,-2\rho,\reg}\underset{\CO_\nOp}\otimes \CV^\cla_{\nOp} @>>>
\BM_{\crit,-2\rho,\reg}\underset{\CO_\nOp}\otimes \CL^\cla_{\nOp}.
\endCD
$$

However, the latter diagram coincides with the commutative diagram (6.10)
of \cite{FG5}.

\ssec{}

We can now state the first main result of the present paper:

\begin{mainthm}  \label{Gamma nilp fully faithful}
The functor 
$$\Gamma_{\nOp}:
\nOp\underset{\tN/\cG}\times \bD^f(\fD(\Fl^{\on{aff}}_G)_\crit\mod)\to 
\bD^f(\hg_\crit\mod_\nilp)$$
is fully faithful.
\end{mainthm}

\ssec{}

Consider the ind-completion of 
$\nOp\underset{\tN/\cG}\times \bD^f(\fD(\Fl^{\on{aff}}_G)_\crit\mod)$, which we will
denote 
\begin{equation} \label{main cat}
\nOp\underset{\tN/\cG}\arrowtimes \bD^f(\fD(\Fl^{\on{aff}}_G)_\crit\mod).
\end{equation}

The functor $\Gamma_{\Fl,\nOp}$ extends to a functor
\begin{equation} \label{main functor}
\Gamma_{\Fl,\nOp}:\nOp\underset{\tN/\cG}\arrowtimes \bD^f(\fD(\Fl^{\on{aff}}_G)_\crit\mod)
\to\bD_{ren}(\hg_\crit\mod_\nilp),
\end{equation}
which commutes with the formation of direct sums. \mainthmref{Gamma nilp fully faithful}
implies that the latter functor is also fully faithful.

\sssec{}   \label{t str on base change}

By \secref{t structure ten prod}, the {\it new} t-structure on
$\bD_{ren}(\fD(\Fl^{\on{aff}}_G)_\crit\mod)$ induces a t-structure on \eqref{main cat}.
Namely, the corresponding $\leq 0$ category is generated
by objects of the form
$$\CM\underset{\CO_{\tN/\cG}}\otimes \CF,$$
where $\CM\in \bD^{perf,\leq 0}(\Coh(\nOp))$ and
$\CF\in \bD^f(\fD(\Fl^{\on{aff}}_G)_\crit\mod)\cap \bD^{\leq 0}_{ren}(\fD(\Fl^{\on{aff}}_G)_\crit\mod)$.
Since $\nOp$ is affine, it is in fact enough to take objects just of the form
$\fr_\nilp^*(\CF)$ for $\CF$ as above, where $\fr_\nilp^*$ denotes the
tautological pull-back functor
\begin{equation} \label{fr nilp base changed}
\fr_\nilp^*:\bD_{ren}(\fD(\Fl^{\on{aff}}_G)_\crit\mod)\to 
\nOp\underset{\tN/\cG}\arrowtimes \bD^f(\fD(\Fl^{\on{aff}}_G)_\crit\mod).
\end{equation}

\sssec{}

We propose:

\begin{conj}   \label{main conj}
The functor \eqref{main functor} is an equivalence of categories,
and is exact.
\end{conj}

There are two pieces of evidence in favor of this conjecture.
One is \mainthmref{Gamma nilp fully faithful} which says that the
functor in question is fully faithful. Another is given by 
\mainthmref{main} (see \secref{Iwahori monodromic}), which
says that the conclusion of the conjecture
holds when we restrict ourselves to the corresponding $I^0$-equivariant
categories on both sides, where $I^0$ is the unipotent radical of the 
Iwahori subgroup $I$.

\sssec{}

Suppose that \conjref{main conj} is true. We would then obtain an equivalence of abelian categories
$$\Heart\left(\nOp\underset{\tN/\cG}\arrowtimes \bD^f(\fD(\Fl^{\on{aff}}_G)_\crit\mod)\right)\to
\hg_\crit\mod_\nilp.$$

As the RHS, i.e., $\hg_\crit\mod_\nilp$ is of prime interest for representation theory,
let us describe the LHS more explicitly.

\medskip

By \propref{heart of base change}, the category 
$\Heart\left(\nOp\underset{\tN/\cG}\arrowtimes \bD^f(\fD(\Fl^{\on{aff}}_G)_\crit\mod)\right)$
is equivalent to the abelian base change
$$\QCoh(\nOp)\underset{\QCoh(\tN/\cG)}\otimes \Heart^{new}\left(\bD^f(\fD(\Fl^{\on{aff}}_G)_\crit\mod)\right).$$

\medskip

As in \cite{Ga1}, Sect. 22, the latter category can be described as follows. 

\medskip

Its objects
are $\CF\in \Heart^{new}\left(\bD^f(\fD(\Fl^{\on{aff}}_G)_\crit\mod)\right)$ endowed with
an action of the algebra $\CO_{\nOp}\simeq \fZ^\nilp_\fg$ by endomorphisms together
with a system of isomorphisms
$$\gamma_V:\CV_{\nOp}\underset{\CO_\nOp}\otimes \CF\simeq \CF\star Z_V, 
\text{ for every }V\in \Rep(\cG)$$
and
$$\gamma_\cla:\CL^\cla_\nOp \underset{\CO_\nOp}\otimes \CF\simeq \CF\star J_\cla, \text{ for every }
\cla\in \cLambda,$$
which satisfy the conditions parallel to (a)-(e) of \secref{a-e}.

\medskip

Morphisms in this category between $(\CF^1,\{\gamma^1_V\},\{\gamma^1_\cla\})$ 
and $(\CF^2,\{\gamma^2_V\},\{\gamma^2_\cla\})$ are morphisms in 
$\Heart^{new}\left(\bD^f(\fD(\Fl^{\on{aff}}_G)_\crit\mod)\right)$ that intertwine the actions of
$\CO_{\nOp}$ and the data of $\gamma_V$, $\gamma_\cla$.

\sssec{}

As another corollary of \conjref{main conj}, we obtain:
\begin{conj}
The functor
$\Gamma_\Fl:\bD_{ren}(\fD(\Fl^{\on{aff}}_G)_\crit\mod)\to
\bD_{ren}(\hg_\crit\mod_\nilp)$
is exact for the {\it new} t-structure on the left-hand side.
\end{conj}

Indeed, the functor $\Gamma_\Fl$ is the composition of 
$\Gamma_{\Fl,\nOp}$ and the functor $\fr_\nilp^*$
of \eqref{fr nilp base changed}. 
However, by \propref{flat exact}, the functor $\fr_\nilp^*$ is exact.

\section{The Iwahori-monodromic subcategory} \label{Iwahori monodromic}

In this section we will consider the restriction of the functor $\Gamma_{\Fl,\nOp}$,
introduced in \secref{functor to the critical level} to the corresponding
$I^0$-equivariant subcategories.

\ssec{}  \label{intr Iw category}

Let $\bD^+(\fD(\Fl^{\on{aff}}_G)_\crit\mod)^{I^0}$ be the full subcategory
of $\bD^+(\fD(\Fl^{\on{aff}}_G)_\crit\mod)$ consisting of $I^0$-equivariant objects, 
as defined, e.g., in \cite{FG2}, Sect. 20.11.

\medskip

Let $\bD_{ren}(\fD(\Fl^{\on{aff}}_G)_\crit\mod)^{I^0}$ be the full subcategory of
$\bD_{ren}(\fD(\Fl^{\on{aff}}_G)_\crit\mod)$ generated by $\bD^+(\fD(\Fl^{\on{aff}}_G)_\crit\mod)^{I^0}$
under the identification 
$$\bD^+(\fD(\Fl^{\on{aff}}_G)_\crit\mod)\simeq \bD^+_{ren}(\fD(\Fl^{\on{aff}}_G)_\crit\mod).$$

\begin{lem}  \label{I^0 monodromic D-modules nice} 
The category $\bD_{ren}(\fD(\Fl^{\on{aff}}_G)_\crit\mod)$ is generated by the set of
its compact objects, which identifies with
$$\bD^f(\fD(\Fl^{\on{aff}}_G)_\crit\mod)^{I^0}:=\bD^+(\fD(\Fl^{\on{aff}}_G)_\crit\mod)^{I^0}
\cap \bD^f(\fD(\Fl^{\on{aff}}_G)_\crit\mod).$$
\end{lem}

This lemma will be proved in \secref{proof I lemma b}.

\sssec{}

The subcategory $\bD^f(\fD(\Fl^{\on{aff}}_G)_\crit\mod)^{I^0}\subset
\bD^f(\fD(\Fl^{\on{aff}}_G)_\crit\mod)$ is preserved by the action of 
$\bD^{perf}(\Coh(\tN/\cG))$ on $\bD^f(\fD(\Fl^{\on{aff}}_G)_\crit\mod)$,
and inherits a DG model. Hence, we have a well-defined category
\begin{equation} \label{base changed category, Iwahori}
\nOp\underset{\tN/\cG}\times \bD^f(\fD(\Fl^{\on{aff}}_G)_\crit\mod)^{I^0},
\end{equation}
and its ind-completion
\begin{equation} \label{base changed category, Iwahori ind}
\nOp\underset{\tN/\cG}\arrowtimes \bD^f(\fD(\Fl^{\on{aff}}_G)_\crit\mod)^{I^0}.
\end{equation}

\medskip

The following lemma will be proved in \secref{proof I lemma d}:
\begin{lem} \label{Iw fully faithful}  \hfill

\smallskip

\noindent{(1)}
The natural functor
$$\nOp\underset{\tN/\cG}\arrowtimes \bD^f(\fD(\Fl^{\on{aff}}_G)_\crit\mod)^{I^0}\to
\nOp\underset{\tN/\cG}\arrowtimes \bD^f(\fD(\Fl^{\on{aff}}_G)_\crit\mod)$$
is fully faithful.

\smallskip

\noindent{(2)} The image of the functor in (1) consist of all objects,
whose further image under
$$(\fr_\nilp)_*:\nOp\underset{\tN/\cG}\arrowtimes \bD^f(\fD(\Fl^{\on{aff}}_G)_\crit\mod)
\to \bD_{ren}(\fD(\Fl^{\on{aff}}_G)_\crit\mod)$$
belongs to $\bD_{ren}(\fD(\Fl^{\on{aff}}_G)_\crit\mod)^{I^0}\subset
\bD_{ren}(\fD(\Fl^{\on{aff}}_G)_\crit\mod)$. (Here $(\fr_\nilp)_*$ is the functor right
adjoint to $(\fr_\nilp)^*$, see Sects \ref{lower star} and \ref{upper and lower star}.)

\end{lem}

\sssec{}

Let 
$\bD^+(\hg_\crit\mod_\nilp)^{I^0}\subset \bD^+(\hg_\crit\mod_\nilp)$
be the subcategory of $I^0$-equivariant objects, as defined, e.g.,
in \cite{FG2}, Sect. 20.11. 

As above, we define $\bD_{ren}(\hg_\crit\mod_\nilp)^{I^0}$ to be the full subcategory
of $\bD_{ren}(\hg_\crit\mod_\nilp)$, generated by $\bD^+(\hg_\crit\mod_\nilp)^{I^0}$
under the identification
$$\bD^+(\hg_\crit\mod_\nilp)\simeq \bD^+_{ren}(\hg_\crit\mod_\nilp).$$

Set also:
$$\bD^f(\hg_\crit\mod_\nilp)^{I^0}:=\bD^+(\hg_\crit\mod_\nilp)^{I^0}\cap \bD^f(\hg_\crit\mod_\nilp)
\subset \bD^+(\hg_\crit\mod_\nilp)^{I^0},$$
which we can also regard as a subcategory
\begin{multline*}
\bD^f(\hg_\crit\mod_\nilp)^{I^0}=:\bD^f_{ren}(\hg_\crit\mod_\nilp)^{I^0}\subset 
\bD^+_{ren}(\hg_\crit\mod_\nilp)^{I^0}\subset \\
\subset \bD_{ren}(\hg_\crit\mod_\nilp)^{I^0}.
\end{multline*}

It is clear that $\bD^f_{ren}(\hg_\crit\mod_\nilp)^{I^0}$ consists of compact objects of
$\bD_{ren}(\hg_\crit\mod_\nilp)^{I^0}$.

\medskip

\noindent{\it Remark.} It is easy to see that the subcategory $\bD^f_{ren}(\hg_\crit\mod_\nilp)^{I^0}$ equals
the subcategory of compact objects in $\bD_{ren}(\hg_\crit\mod_\nilp)^{I^0}$.
It will {\it a posteriori} true, see \corref{I^0 monodromic g-modules also nice}, that it actually 
generates $\bD_{ren}(\hg_\crit\mod_\nilp)^{I^0}$. However, we will be able to deduce this
as a corollary of our \mainthmref{main}. Unfortunately, at the moment we are unable to
give an {\it a priori} proof of this statement, parallel to that of
\lemref{I^0 monodromic D-modules nice}.

\sssec{}

It is clear from the definitions that the functor $\Gamma_\Fl$
sends
$$\bD^f(\fD(\Fl^{\on{aff}}_G)_\crit\mod)^{I^0}\to \bD^f(\hg_\crit\mod_\nilp)^{I^0}.$$
Thus, it extends to a functor
$$\nOp\underset{\tN/\cG}\times \bD^f(\fD(\Fl^{\on{aff}}_G)_\crit\mod)^{I^0}\to
\bD^f(\hg_\crit\mod_\nilp)^{I^0},$$
making the diagram
\begin{equation}  \label{Iw and not}
\CD
\nOp\underset{\tN/\cG}\times \bD^f(\fD(\Fl^{\on{aff}}_G)_\crit\mod)^{I^0} @>>>
\nOp\underset{\tN/\cG}\times \bD^f(\fD(\Fl^{\on{aff}}_G)_\crit\mod) \\
@V{\Gamma_{\Fl,\nOp}}VV    @V{\Gamma_{\Fl,\nOp}}VV \\
\bD^f(\hg_\crit\mod_\nilp)^{I^0}  @>>> \bD^f(\hg_\crit\mod_\nilp)
\endCD
\end{equation}
commute.

\sssec{}

Recall now that the categories
$$\bD_{ren}(\hg_\crit\mod_\nilp);\,\,\,
\bD_{ren}(\fD(\Fl^{\on{aff}}_G)_\crit\mod);\,\,\,
\nOp\underset{\tN/\cG}\arrowtimes \bD^f(\fD(\Fl^{\on{aff}}_G)_\crit\mod)$$
are equipped with t-structures. The t-structure that we consider 
on $\bD_{ren}(\fD(\Fl^{\on{aff}}_G)_\crit\mod)$ can be either the old or the new one, whereas the t-structure on 
$\nOp\underset{\tN/\cG}\arrowtimes \bD^f(\fD(\Fl^{\on{aff}}_G)_\crit\mod)$
is that introduced in \secref{t str on base change}, using the {\it new} t-structure
on $\bD_{ren}(\fD(\Fl^{\on{aff}}_G)_\crit\mod)$.

\medskip

We have:

\begin{prop}   \label{prop Iw and t} \hfill

\smallskip

\noindent{\em(a)} The subcategory
$\bD_{ren}(\hg_\crit\mod_\nilp)^{I^0}\subset \bD_{ren}(\hg_\crit\mod_\nilp)$
is compatible with the t-structure.

\smallskip

\noindent{\em(b)} The subcategory
$\bD_{ren}(\fD(\Fl^{\on{aff}}_G)_\crit\mod)^{I^0}\subset
\bD_{ren}(\fD(\Fl^{\on{aff}}_G)_\crit\mod)$  
is compatible with the old t-structure.

\smallskip

\noindent{\em(c)} The subcategory
$\bD_{ren}(\fD(\Fl^{\on{aff}}_G)_\crit\mod)^{I^0}\subset
\bD_{ren}(\fD(\Fl^{\on{aff}}_G)_\crit\mod)$ is compatible also
with the {\it new} t-structure.

\smallskip

\noindent{\em(d)}
The subcategory
$$\nOp\underset{\tN/\cG}\arrowtimes \bD^f(\fD(\Fl^{\on{aff}}_G)_\crit\mod)^{I^0} 
\subset
\nOp\underset{\tN/\cG}\arrowtimes \bD^f(\fD(\Fl^{\on{aff}}_G)_\crit\mod)$$
is compatible with the t-structure.
\end{prop}

Point (a) of the proposition follows from the fact that embedding functor
$$\bD_{ren}(\hg_\crit\mod_\nilp)^{I^0}\hookrightarrow \bD_{ren}(\hg_\crit\mod_\nilp)$$
admits a right adjoint, which fits into the commutative diagram:
$$
\CD
\bD_{ren}(\hg_\crit\mod_\nilp)^{I^0}  @<<<  \bD_{ren}(\hg_\crit\mod_\nilp) \\
@AAA    @AAA  \\
\bD^+(\hg_\crit\mod_\nilp)^{I^0}  @<{\on{Av}_*^{I^0}}<< \bD^+(\hg_\crit\mod_\nilp),
\endCD 
$$
where $\on{Av}_*^{I^0}$ is the averaging functor of \cite{FG2}, Sect. 20.10; in particular,
the right adjoint in question is left-exact.

\medskip

Point (b), (c) and (d) of the proposition will be proved in Sections \ref{proof I lemma b},
\ref{proof I lemma c} and \ref{proof I lemma d}, respectively.

\ssec{}

The second main result of this paper is the following:

\begin{mainthm} \label{main}
The functor
$$\Gamma_{\Fl,\nOp}:\nOp\underset{\tN/\cG}\arrowtimes 
\bD^f(\fD(\Fl^{\on{aff}}_G)_\crit\mod)^{I^0} 
\to \bD_{ren}(\hg_\crit\mod_\nilp)^{I^0}$$
is an equivalence of categories, and is exact. 
\end{mainthm}

\sssec{}

As a consequence, we obtain: 
\begin{cor}   \label{G exact iw}
The functor $\Gamma_\Fl:\bD_{ren}(\fD(\Fl^{\on{aff}}_G)_\crit\mod)^{I^0} 
\to \bD_{ren}(\hg_\crit\mod_\nilp)^{I^0}$ is exact (in the new t-structure
on the LHS).
\end{cor}

As another corollary we obtain:

\begin{cor}  \label{F exact}
The functor $\sF:\bD(\QCoh(\tN/\cG))\to \bD_{ren}(\fD(\Fl^{\on{aff}}_G)_\crit\mod)$
is exact (in the new t-structure on the RHS).
\end{cor}

\begin{proof}

By construction, the image of $\sF$ lies in 
$\bD_{ren}(\fD(\Fl^{\on{aff}}_G)_\crit\mod)^{I^0}$. First, we claim that the
composed functor 
$$\Gamma_\Fl\circ \sF\simeq \Gamma_{\Fl,\nOp}\circ \fr_\nilp^*\circ \sF:
\bD(\QCoh(\tN/\cG))\to \bD_{ren}(\hg_\crit\mod_\nilp)^{I^0}$$
is exact. Indeed, the above functor is given by
$$\CM\mapsto \CM\underset{\CO_{\nOp}}\otimes\BM_{\crit,-2\rho},$$
and the claim follows from the fact that $\BM_{\crit,-2\rho}$ is
$\CO_{\nOp}$-flat.

\medskip

Since the functor $\Gamma_{\Fl,\nOp}$ is an equivalence, the
assertion of the corollary follows from the fact that the functor
$$\fr_\nilp^*: \bD_{ren}(\fD(\Fl^{\on{aff}}_G)_\crit\mod)\to
\nOp\underset{\tN/\cG}\arrowtimes \bD_{ren}(\fD(\Fl^{\on{aff}}_G)_\crit\mod)$$
is exact and conservative on the heart.

\end{proof}

\sssec{}

Finally, we remark that since the LHS appearing in \mainthmref{main} is compactly
generated, we obtain:

\begin{cor}  \label{I^0 monodromic g-modules also nice}
The category $\bD_{ren}(\hg_\crit\mod_\nilp)^{I^0}$ is
generated by $\bD^f_{ren}(\hg_\crit\mod_\nilp)^{I^0}$.
\end{cor}

\ssec{}

We will now show how \mainthmref{main} implies a conjecture from
\cite{FG2} about a description of the category $\bD_{ren}(\hg_\crit\mod_\nilp)^{I^0}$
is terms of quasi-coherent sheaves.

\sssec{}

Let $\check{\on{St}}$ denote the Steinberg scheme of $\cG$, i.e.,
$$\check{\on{St}}:=\tN\underset{\cg}\times \tg,$$
where $\tg$ denoted Grothendieck's alteration (see \secref{notation}).
The diagonal map
$\tN\to \tN\times \tg$ defines a map of stacks
$$\tN/\cG\hookrightarrow \check{\on{St}}/\cG.$$

\medskip

Let us now recall the main result of Bezrukavnikov's theory 
\cite{Bez}, Theorem 4.2(a):

\begin{thm} \label{Bezr}
The functor $\sF$ extends to an equivalence 
$$\sF_{\check{\on{St}}}:\bD^b(\on{Coh}(\check{\on{St}}/\cG))\to \bD^f(\fD(\Fl^{\on{aff}}_G)_\crit\mod)^{I^0}$$
as categories over the stack $\tN/\cG$.
\end{thm}

\sssec{}

Let us recall now the scheme of Miura opers with nilpotent
singularities, introduced in \cite{FG2}, Sect. 3.14, and denoted
$\nMOp$. By definition,
$$\nMOp:=\nOp\underset{\tN/\cG}\times \check{\on{St}}/\cG.$$

\medskip

The following conjecture was stated in \cite{FG2}, Conjecture 6.2: 

\begin{conj}  \label{qc conj}
There exists an equivalence 
$$\bD^b(\QCoh(\nMOp))\simeq \bD^b(\hg_\crit\mod_\nilp)^{I^0}$$
of categories over the scheme $\nOp$.
\end{conj}

We will deduce this conjecture from \mainthmref{main} by base-changing the 
equivalence of \thmref{Bezr} by means of the morphism
$\fr_\nilp:\nOp\to \tn/\cG$. In fact, we will prove a slightly stronger statement:

\begin{mainthm}  \label{Miura}
There exists an equivalence
$$\bD^+(\QCoh(\nMOp))\simeq \bD^+(\hg_\crit\mod_\nilp)^{I^0}$$
of cohomological amplitude bounded by $\dim(G/B)$.
\end{mainthm}

\section{Relation to the affine Grassmannian}   \label{relation to Grassmannian}

As was mentioned in the introduction, the main tool that will eventually
allow us to prove \mainthmref{Gamma nilp fully faithful} and, consequently, 
\mainthmref{main}, is an explicit connection between the category 
$\bD^f(\fD(\Fl^{\on{aff}}_G)_\crit\mod)$ and the category of D-modules
on the affine Grassmannian, $\Gr^{\on{aff}}_G$.

\ssec{}

Let $\bD^f(\fD(\Gr^{\on{aff}}_G)_\crit\mod)$ be the derived category
of finitely generated critically twisted D-modules on $\Gr^{\on{aff}}_G$. This
is a triangulated category equipped with a natural DG model.

\sssec{}

As in the case of $\Fl^{\on{aff}}_G$, we can form the ind-completion
of $\bD^f(\fD(\Gr^{\on{aff}}_G)_\crit\mod)$, denoted
$\bD_{ren}(\fD(\Fl^{\on{aff}}_G)_\crit\mod)$. This is a co-complete triangulated category,
which is generated by the subcategory of its compact objects, the 
latter being identified with $\bD^f(\fD(\Gr^{\on{aff}}_G)_\crit\mod)$ itself. Moreover,
$\bD_{ren}(\fD(\Gr^{\on{aff}}_G)_\crit\mod)$ is equipped 
with a compactly generated t-structure.  
We have an exact functor
$$\bD_{ren}(\fD(\Gr^{\on{aff}}_G)_\crit\mod)\to \bD(\fD(\Gr^{\on{aff}}_G)_\crit\mod),$$
which induces an equivalence
$$\bD^+_{ren}(\fD(\Gr^{\on{aff}}_G)_\crit\mod)\to \bD^+(\fD(\Gr^{\on{aff}}_G)_\crit\mod).$$

\sssec{}

The geometric Satake equivalence endows the abelian
category $\fD(\Gr^{\on{aff}}_G)_\crit\mod$ with an action of the
tensor category $\Rep^{f.d.}(G)\simeq \Coh(\on{pt}/\cG)$
by exact functors.

\medskip

This defines on $\bD^f(\fD(\Gr^{\on{aff}}_G)_\crit\mod)$ an action 
of $\bD^{perf}(\Coh(\on{pt}/\cG))$, equipped with an
(evident) DG model. I.e., $\bD^f(\fD(\Gr^{\on{aff}}_G)_\crit\mod)$
is a triangulated category over the stack $\on{pt}/\cG$
(see \secref{categories over stacks} where the terminology
is introduced).

\sssec{}

Consider the base-changed category 
$$\on{pt}/\cB \underset{\on{pt}/\cG}\times \bD^f(\fD(\Gr^{\on{aff}}_G)_\crit\mod),$$
which is a triangulated category over the stack $\on{pt}/\cB$. 
(See \secref{base change} for the definition of base change.)

\medskip

Consider now another triangulated category over $\on{pt}/\cB$,
namely,
$$\on{pt}/\cB \underset{\tN/\cG}\times \bD^f(\fD(\Fl^{\on{aff}}_G)_\crit\mod),$$
where 
$$\on{pt}/\cB \simeq \on{Fl}^{\cG}/\cG\hookrightarrow \tN/\cG$$ is the embedding of the
zero section, denoted $\iota$. 

\medskip

Our current goal is to construct a functor, denoted $\Upsilon$,
\begin{equation} \label{desired Ups}
\Upsilon:\on{pt}/\cB \underset{\tN/\cG}\times \bD^f(\fD(\Fl^{\on{aff}}_G)_\crit\mod)\to 
\on{pt}/\cB \underset{\on{pt}/\cG}\times \bD^f(\fD(\Gr^{\on{aff}}_G)_\crit\mod),
\end{equation}
as categories over the stack $\on{pt}/\cB$.

\ssec{}     \label{tilde ups}

By the universal property of the tensor product construction
(see \secref{univ ppty nilp}), in order to construct the functor
\begin{equation}   \label{ups}
\Upsilon: \on{pt}/\cB \underset{\tN/\cG}\times \bD^f(\fD(\Fl^{\on{aff}}_G)_\crit\mod)\to
\on{pt}/\cB \underset{\on{pt}/\cG}\times \bD^f(\fD(\Gr^{\on{aff}}_G)_\crit\mod),
\end{equation}
by \secref{univ ppty base change}, we have to produce a functor
\begin{equation} \label{wt ups}
\wt\Upsilon: \bD^f(\fD(\Fl^{\on{aff}}_G)_\crit\mod)\to 
\on{pt}/\cB \underset{\on{pt}/\cG}\times \bD^f(\fD(\Gr^{\on{aff}}_G)_\crit\mod),
\end{equation}
as categories over the stack $\tN/\cG$.

\sssec{}

Since the action of $\Rep(\cG)$ on $\fD(\Gr^{\on{aff}}_G)_\crit\mod$ is
given by right-exact (and, in fact, exact) functors, we can consider
the base-changed abelian category:
$$\Rep(\cB)\underset{\Rep(\cG)}\otimes \fD(\Gr^{\on{aff}}_G)_\crit\mod=:
\on{pt}/\cB \underset{\on{pt}/\cG}\times \fD(\Gr^{\on{aff}}_G)_\crit\mod,$$
see \secref{ten prod abelian} for the definiton (the existence of
this category is easily established; see, e.g., in \cite{Ga1}, Sect. 10).

\medskip

By construction, $\on{pt}/\cB \underset{\on{pt}/\cG}\times \fD(\Gr^{\on{aff}}_G)_\crit\mod$
is an abelian category acted on by $\Rep(\cB)$ by exact functors. Hence,
the derived category 
$\bD(\on{pt}/\cB \underset{\on{pt}/\cG}\times \fD(\Gr^{\on{aff}}_G)_\crit\mod)$
has a natural structure of category over the stack $\on{pt}/\cB$.

\sssec{}

By \secref{univ ppty base change}, the tautological functor
$$
\bD^f(\fD(\Gr^{\on{aff}}_G)_\crit\mod)\to 
\bD(\on{pt}/\cB \underset{\on{pt}/\cG}\times \fD(\Gr^{\on{aff}}_G)_\crit\mod)$$
gives rise to a functor
\begin{equation} \label{two versions of base chaged grassmannian}
\on{pt}/\cB \underset{\on{pt}/\cG}\times \bD^f(\fD(\Gr^{\on{aff}}_G)_\crit\mod)\to
\bD(\on{pt}/\cB \underset{\on{pt}/\cG}\times \fD(\Gr^{\on{aff}}_G)_\crit\mod).
\end{equation}

Moreover, by \lemref{abelian vs derived base change}, the functor
\eqref{two versions of base chaged grassmannian} is fully faithful.

\medskip

Hence, in order to construct the functor $\wt\Upsilon$,
it suffices to construct a functor
\begin{equation} \label{coarse ups}
\bD^f(\fD(\Fl^{\on{aff}}_G)_\crit\mod)\to 
\bD(\on{pt}/\cB \underset{\on{pt}/\cG}\times \fD(\Gr^{\on{aff}}_G)_\crit\mod),
\end{equation}
as categories over $\tN/\cG$, such that its image belongs to
the essential image of \eqref{two versions of base chaged grassmannian}.

\medskip

By a slight abuse of notation, we will denote the functor \eqref{coarse ups}
also by $\wt\Upsilon$. In this section we will construct $\wt\Upsilon$ as a 
functor between triangulated categories, compatible with the action of
$\bD^{perf}(\on{Coh}(\tN/\cG))$. In \secref{ups: model} we
will upgrade the construction to the DG level.

\sssec{}

Let $\CW$ be the object of 
$\on{pt}/\cB \underset{\on{pt}/\cG}\times \fD(\Gr^{\on{aff}}_G)_\crit\mod^I$, introduced 
in \cite{FG5}, Sect. 3.15 (under the name $\CW_{w_0}$), as well as in \cite{FG4},
Sect. 4.1 (under the name $\CF_{w_0}$),  and in \cite{ABBGM}, Sect. 3.2.13
(under the name $\CM^1$).

\medskip

We define the functor $\wt\Upsilon$ of \eqref{coarse ups} by
\begin{equation} \label{def wt Ups}
\wt\Upsilon(\CF):= \CF\star J_{2\crho}\star\CW.
\end{equation}

\begin{lem}
For a finitely generated D-module $\CF$ on $\Fl^{\on{aff}}_G$, the object
$$\CF\star J_{2\crho}\star\CW\in 
\bD(\on{pt}/\cB \underset{\on{pt}/\cG}\times \fD(\Gr^{\on{aff}}_G)_\crit\mod)$$
belongs to the essential image of 
\eqref{two versions of base chaged grassmannian}.
\end{lem}

\begin{proof}

Recall (see, e.g., \cite{ABBGM}, Corollary 1.3.10 and Proposition 3.2.6) that $\CW$,
as an object of $\on{pt}/\cB \underset{\on{pt}/\cG}\times \fD(\Gr^{\on{aff}}_G)_\crit\mod^I$,
has a finite filtration with subquotients of the form
$$V\underset{\CO_{\on{pt}/\cG}}\otimes \CF'\in \on{pt}/\cB 
\underset{\on{pt}/\cG}\times\fD(\Gr^{\on{aff}}_G)_\crit\mod^I$$ with $V\in \Rep^{f.d.}(\cB)$ and $\CF'\in 
\fD(\Gr^{\on{aff}}_G)_\crit\mod^I$ is finitely generated as a D-module.

\medskip

Since $\Fl^{\on{aff}}_G$ is proper, for any such $\CF'$ and $\CF$ as in the statement
of the lemma, $\CF\underset{I}\star \CF'$ belongs to 
$\bD^f(\fD(\Gr^{\on{aff}}_G)_{crit}\mod)$. 

\end{proof}

\ssec{}

In order to endow $\wt\Upsilon$ with the data of compatibility
with respect to the action of $\bD^{perf}(\on{Coh}(\tN/\cG))$ 
(at the triangulated level), by \lemref{action on quotient},
it is enough to do so with respect to the 
action of $\Ho\left(\bC^b(\on{Coh}^{free}(\tN/\cG))\right)$.

\medskip

This amounts to constructing isomorphisms

\begin{itemize}
\item(i)
$$\wt\Upsilon(\CF\star J_\cla)\simeq \CL^\cla\underset{\CO_{\on{pt}/\cB}}\otimes \wt\Upsilon(\CF),\,\, \cla\in \cLambda$$

\item(ii)
$$\wt\Upsilon(\CF\star Z_V)\simeq (\CO_{\on{pt}/\cB}\otimes V)
\underset{\CO_{\on{pt}/\cB}}\otimes \wt\Upsilon(\CF),\,\, V\in \Rep^{f.d.}(\cG)$$
\end{itemize}

so that the conditions, parallel to (a)-(e) of \secref{a-e} hold.

\sssec{}

Isomorphism (i) above follows from the basic isomorphism
\begin{equation} \label{j-inv Wak}
J_{\cla}\star \CW\simeq 
\CL^\cla\underset{\CO_{\on{pt}/\cB}}\otimes \CW,
\end{equation}
established in  \cite{FG4}, Corollary 4.5, or \cite{FG5}, Proposition 3.19,
or \cite{ABBGM}, Corollary 3.2.2.

\sssec{}

To construct isomorphism (ii) we observe that for any $\CF'\in \fD(\Gr^{\on{aff}}_G)_\crit\mod^I$
and $V\in \Rep^{f.d.}(\cG)$ we have a canonical isomorphism
\begin{equation} \label{left central}
Z_V\underset{I}\star \CF'\simeq \CF'\underset{G[[t]]}\star \CF_V,
\end{equation}
where $\CF_V$ is the spherical D-module on $\Gr^{\on{aff}}_G$, corresponding to $V\in \Rep(\cG)$.
This isomorphism is a particular case of Theorem 1(b) of \cite{Ga}, combined with
point (d) of the same theorem.

\medskip

Thus, isomorphism (ii) is obtained from
$$\CF\star Z_V\star J_{2\crho}\star \CW\simeq 
\CF\star J_{2\crho}\star (Z_V\star \CW)\simeq (\CF\star J_{2\crho}\star \CW)\underset{G[[t]]}
\star \CF_V,$$
where we observe that the operation $\CF'\mapsto \CF'\underset{G[[t]]}\star \CF_V$ corresponds 
by definition to
$$\CF'\mapsto (\CO_{\on{pt}/\cB}\otimes V)
\underset{\CO_{\on{pt}/\cB}}\otimes \CF'$$ for $\CF'\in 
\on{pt}/\cB \underset{\on{pt}/\cG}\times \bD(\fD(\Gr^{\on{aff}}_G)_\crit\mod)$.

\sssec{}

Conditions (a) and (c) follow by definition. Condition (d) follows
from the functoriality of the isomorphism \eqref{left central} and the
associativity property of the isomorphism of Theorem 1(b) of \cite{Ga},
established as Property 1 in \cite{Ga'}.

\medskip

Condition (b) follows from the fact that for any $\CF'\in 
\fD(\Gr^{\on{aff}}_G)_\crit\mod^I$, 
the isomorphism, induced
by $N_{Z_V}$ on the left-hand side of \eqref{left central},
is zero.

\sssec{}

To prove condition (e), we need to check the commutativity 
of the following diagram:
$$
\CD
Z_{V^\cla}\star \CW & @>>> & J_\cla\star \CW \\
@V{\sim}VV  & &  @V{\sim}VV  \\
\CW\underset{G[[t]]}\star \CF_{V^\cla} @>{\sim}>> 
\on{Res}^{\cG}_{\cB}(V^\cla)\underset{\CO_{\on{pt}/\cB}}\otimes \CW
@>>>
\CL^\cla\underset{\CO_{\on{pt}/\cB}}\otimes \CW.
\endCD
$$

Since $\on{End}(\CW)\simeq \BC$ (see \cite{ABBGM}, Proposition 3.2.6(1)),
to prove the commutativity of the latter diagram,
it is enough to establish the commutativity of the following one, obtained from the
map $\delta_{1,\Gr^{\on{aff}}_G}\to \CW$:
$$
\CD
Z_{V^\cla}\star \delta_{1,\Gr^{\on{aff}}_G}  & @>>> & J_\cla\star \delta_{1,\Gr^{\on{aff}}_G} \\
@VVV     &   &    @VVV  \\   
Z_{V^\cla}\star \CW & & & &  J_{\cla}\star \CW \\
@V{\sim}VV & & @V{\sim}VV \\
\CW\underset{G[[t]]}\star \CF_{V^\cla} @>{\sim}>> 
\on{Res}^{\cG}_{\cB}(V^\cla)\underset{\CO_{\on{pt}/\cB}}\otimes \CW
@>>>
\CL^\cla\underset{\CO_{\on{pt}/\cB}}\otimes \CW.
\endCD
$$
However, this last diagram is equivalent to the following one.
$$
\CD
\CF_{V^\cla}  @>>> J_\cla\star \delta_{1,\Gr^{\on{aff}}_G} @>>> J_\cla\star \CW \\
@VVV & & @VVV \\
\CW\underset{G[[t]]}\star \CF_{V^\cla} @>{\sim}>> 
\on{Res}^{\cG}_{\cB}(V^\cla)\underset{\CO_{\on{pt}/\cB}}\otimes \CW
@>>> \CL^\cla\underset{\CO_{\on{pt}/\cB}}\otimes \CW.
\endCD
$$
The latter diagram is commutative, as it is a
version of the commutative diagram of Lemma 5.3 of \cite{FG5}
(or, which is the same, commutative diagram (28) of \cite{FG4},
or Corollary 3.2.3 of \cite{ABBGM}).

\ssec{}

Thus, the functor $\wt\Upsilon$, and hence the functor $\Upsilon$,
have been constructed. We are now ready to state the fourth main result
of this paper:

\begin{mainthm}   \label{thm relation to Grassmannian}
The functor
$$\Upsilon:\on{pt}/\cB \underset{\on{pt}/\cG}\times \bD^f(\fD(\Gr^{\on{aff}}_G)_\crit\mod)\to
\on{pt}/\cB \underset{\tN/\cG}\times \bD^f(\fD(\Fl^{\on{aff}}_G)_\crit\mod)$$
is fully faithful.
\end{mainthm}

\noindent{\it Remark.} As will be explained in the sequel, the functor
$\Upsilon$, as it is, is not an equivalence of categories: it fails to be
essentially surjective.  In \secref{turning Ups} we will show 
how to modify the LHS by adding certain colimits so that the resulting 
functor becomes an equivalence. In addition, the latter functor will
be exact with respect to the natural t-structures.

\section{Sections over $\Fl^{\on{aff}}_G$ vs. $\Gr^{\on{aff}}_G$}  \label{compat with sect}

In this section we will study how the functor $\Upsilon$ of 
\secref{relation to Grassmannian} intertwines
between the functors of taking global sections of critically twisted 
D-modules on $\Fl^{\on{aff}}_G$ and $\Gr^{\on{aff}}_G$.

\ssec{}

Let $\hg_\crit\mod_\reg$ be the abelian category of
$\hg_\crit$-modules, on which the center $\fZ_\fg$ acts 
through its quotient $\fZ^\reg_\fg\simeq \CO_{\rOp}$.
Let $\bD(\hg_\crit\mod_\reg)$ be its derived category.
It has the same properties as those discussed in 
\secref{nilpotent category} with "$\nilp$" replaced
by "$\reg$". 

\medskip

In particular we have a diagram of categories
$$
\CD
\bD^f_{ren}(\hg_\crit\mod_\reg)  @>>> \bD^+_{ren}(\hg_\crit\mod_\reg)  \\
@V{\sim}VV   @V{\sim}VV   \\
\bD^f(\hg_\crit\mod_\reg)  @>>> \bD^+(\hg_\crit\mod_\reg)
\endCD
$$ 
with the rows being fully faithful functors.

\sssec{}

Recall from \cite{FG2} that we have an exact functor
$$\Gamma_{\Gr}:\fD(\Gr^{\on{aff}}_G)_\crit\mod\to \hg_\crit\mod_\reg,$$
compatible with the action of $\Rep(\cG)$, where the action
on the RHS is given via the morphism
$$\fr_\reg:\rOp\to \on{pt}/\cG.$$

\medskip

This gives rise to a functor
\begin{equation}  \label{sections on Gr}
\Gamma_{\Gr}:\bD^f(\fD(\Gr^{\on{aff}}_G)_\crit\mod)\to \bD(\hg_\crit\mod_\reg)
\end{equation}
as categories over the stack $\on{pt}/\cG$.

\medskip

Repeating the argument of \secref{sections to finite}, one shows that
the image of the functor \eqref{sections on Gr} belongs to the
subcategory $\bD^f(\hg_\crit\mod_\reg)\subset \bD(\hg_\crit\mod_\reg)$.
So we have a functor 
\begin{equation}  \label{sections on Gr finite}
\Gamma_{\Gr}:\bD^f(\fD(\Gr^{\on{aff}}_G)_\crit\mod)\to \bD^f(\hg_\crit\mod_\reg)
\end{equation}
as categories over the stack $\on{pt}/\cG$.

\sssec{}

Consider the base-changed category
$$\rOp \underset{\on{pt}/\cG}\times \bD^f(\fD(\Gr^{\on{aff}}_G)_\crit\mod).$$

By \secref{univ ppty base change}, the functor $\Gamma_{\Gr}$
gives rise to a functor
\begin{equation}  \label{sections on Gr op}
\Gamma_{\Gr,\nOp}:
\rOp \underset{\on{pt}/\cG}\times \bD^f(\fD(\Gr^{\on{aff}}_G)_\crit\mod)\to 
\bD^f(\hg_\crit\mod_\reg)
\end{equation}
as triangulated categories over $\rOp$.

\medskip

The following has been established in \cite{FG2}, Theorem 8.17:
\begin{thm}  \label{faithfulness on Gr}
The functor $\Gamma_{\Gr,\nOp}$ is fully faithful.
\end{thm}

In addition, in {\it loc. cit.} we formulated the conjecture to the effect
that the functor $\Gamma_{\Gr,\nOp}$ is an equivalence. 

\ssec{}

Recall now that the map $\fr_\reg$ canonically factors as
$$\rOp\overset{\fr'_\reg}\to \on{pt}/\cB \to \on{pt}/\cG.$$

In fact, we have a Cartesian square with vertical arrows being smooth morphisms:
\begin{equation} \label{iotas}
\CD
\rOp  @>{\iota_\Op}>> \nOp \\
@V{\fr'_\reg}VV  @V{\fr_\nilp}VV \\
\on{pt}/\cB @>{\iota}>> \tN/\cG,
\endCD
\end{equation}
see \cite{FG2}, Lemma 2.19.

\sssec{}

Hence, we can regard $\bD^f(\hg_\crit\mod_\reg)$ is a triangulated
category over the stack $\on{pt}/\cB$, and by the universal property 
(see \secref{univ ppty base change}),
the functor \eqref{sections on Gr op} gives rise to a functor
\begin{equation}  \label{sections on Gr B}
\Gamma_{\Gr,\on{pt}/\cB}:
\on{pt}/\cB \underset{\on{pt}/\cG}\times \bD^f(\fD(\Gr^{\on{aff}}_G)_\crit\mod)\to 
\bD^f(\hg_\crit\mod_\reg),
\end{equation}
as triangulated categories over $\on{pt}/\cB$.

\sssec{}

Recall now that we have a functor
$$\iota^*_\hg:\bD(\hg_\crit\mod_\nilp)\to \bD(\hg_\crit\mod_\reg),$$
obtained as a derived functor of 
$$\CM\mapsto \CO_{\rOp}\underset{\CO_{\nOp}}\otimes \CM:
\hg_\crit\mod_\nilp\to \hg_\crit\mod_\reg.$$
By construction, the above functor has a natural DG model,
and as such is a functor between categories over $\nOp$. 

\medskip

In \secref{restr central char 1} we will show that the above functor
sends the subcategory $\bD^f(\hg_\crit\mod_\nilp)$ to $\bD^f(\hg_\crit\mod_\reg)$,
thereby giving rise to a functor between the above categories, as categories
over $\nOp$. By \secref{univ ppty base change}, we obtain a functor
\begin{equation} \label{base change central character}
(\rOp\underset{\nOp}\times \iota^*_\hg):
\rOp\underset{\nOp}\times \bD^f(\hg_\crit\mod_\nilp)\to
\bD^f(\hg_\crit\mod_\reg),
\end{equation}
as categories over $\rOp$.

\medskip

In \secref{restr central char 2} we will prove:

\begin{prop}  \label{prop base change central character}
The functor $(\rOp\underset{\nOp}\times \iota^*_\hg)$ of
\eqref{base change central character} is fully faithful.
\end{prop}

\noindent{\it Remark.} The functor in \propref{prop base change central character}
fails to be essentially surjective for the same reasons as the functor
$\Upsilon$ of \mainthmref{thm relation to Grassmannian}. In 
\secref{restr central char 3} it will be shown how to modify the LHS to turn
it into an equivalence.

\ssec{}

Consider the base change of the functor $\Gamma_{\Fl}$
with respect to $\iota$, and obtain a functor
$$\Gamma_{\Fl,\on{pt}/\cB}:
\on{pt}/\cB\underset{\tN/\cG}\times \bD^f(\fD(\Fl^{\on{aff}}_G)_\crit\mod)\to
\on{pt}/\cB\underset{\tN/\cG}\times \bD^f(\hg_\crit\mod_\nilp).$$

We will prove:

\begin{thm}  \label{rel to aff gr and sections}
The diagram of functors between categories over $\on{pt}/\cB$
$$
\CD
\on{pt}/\cB\underset{\tN/\cG}\times \bD^f(\fD(\Fl^{\on{aff}}_G)_\crit\mod)  
@>{\Gamma_{\Fl,\on{pt}/\cB}}>>
\on{pt}/\cB\underset{\tN/\cG}\times \bD^f(\hg_\crit\mod_\nilp) \\
@V{\Upsilon}VV    @V{(\on{pt}/\cB\underset{\tN/\cG}\times \iota^*_\hg)}VV   \\
\on{pt}/\cB \underset{\on{pt}/\cG}\times \bD^f(\fD(\Gr^{\on{aff}}_G)_\crit\mod)
@>{\Gamma_{\Gr,\on{pt}/\cB}}>> 
\bD^f(\hg_\crit\mod_\reg)
\endCD
$$
is commutative.
\end{thm}

By \propref{univ ppty base change},
another way to formulate the above theorem is that the diagram
\begin{equation}   \label{eq rel to aff gr and sections}
\CD
\bD^f(\fD(\Fl^{\on{aff}}_G)_\crit\mod) @>{\Gamma_{\Fl}}>> \bD^f(\hg_\crit\mod_\nilp) \\
@VVV    @V{\iota^*_\hg}VV   \\
\on{pt}/\cB \underset{\on{pt}/\cG}\times \bD^f(\fD(\Gr^{\on{aff}}_G)_\crit\mod)
@>{\Gamma_{\Gr,\on{pt}/\cB}}>> 
\bD^f(\hg_\crit\mod_\reg)
\endCD
\end{equation}
commutes, as functors between categories over $\tN/\cG$, where the
left vertical arrow is the composition
$$\bD^f(\fD(\Fl^{\on{aff}}_G)_\crit\mod) \overset{\iota^*}\to
\on{pt}/\cB\underset{\tN/\cG}\times \bD^f(\fD(\Fl^{\on{aff}}_G)_\crit\mod) 
\overset{\Upsilon}\to \on{pt}/\cB 
\underset{\on{pt}/\cG}\times \bD^f(\fD(\Gr^{\on{aff}}_G)_\crit\mod).$$

\medskip

Base-changing the diagram in the theorem with respect to $\fr'_\reg$,
we obtain:

\begin{cor}  \label{cor rel to aff gr and sections}
We have the following commutative diagram of
functors between categories over $\rOp$:
$$ 
\CD
\rOp\underset{\nOp}\times \left(\nOp\underset{\tN/\cG}\times 
\bD^f(\fD(\Fl^{\on{aff}}_G)_\crit\mod)\right)  @>{\Gamma_{\Fl,\nOp}}>>
\rOp\underset{\nOp}\times \bD(\hg_\crit\mod_\nilp)   \\ 
@V{\sim}VV   \\
\rOp\underset{\tN/\cG}\times
\bD^f(\fD(\Fl^{\on{aff}}_G)_\crit\mod)  & & @VVV \\
@V{\sim}VV   \\
\rOp\underset{\on{pt}/\cB}\times \left(\on{pt}/\cB\underset{\tN/\cG}\times
\bD^f(\fD(\Fl^{\on{aff}}_G)_\crit\mod)\right)  & &
@V{\rOp\underset{\nOp}\times \iota^*_\hg}VV  \\
@V{\Upsilon}VV \\
\rOp\underset{\on{pt}/\cB}\times \left(\on{pt}/\cB\underset{\on{pt}/\cG}\times
\bD^f(\fD(\Gr^{\on{aff}}_G)_\crit\mod)\right)  & & @VVV \\
@V{\sim}VV   \\
\rOp\underset{\on{pt}/\cG}\times
\bD^f(\fD(\Gr^{\on{aff}}_G)_\crit\mod) @>{\Gamma_{\Gr,\rOp}}>> \bD(\hg_\crit\mod_\reg).
\endCD
$$
\end{cor}

\newpage



\centerline{\bf \large Part II: Proofs}

\bigskip

The strategy of the proofs of the four main theorems is as follows.
We shall first prove \mainthmref{thm relation to Grassmannian},
which is a purely geometric assertion, i.e., involves only D-modules,
but not representations of the Kac-Moody algebra. 

\medskip

We shall then combine \mainthmref{thm relation to Grassmannian}
with \thmref{faithfulness on Gr} to prove 
\thmref{Gamma nilp fully faithful}.

\medskip

\mainthmref{main} will follow from \mainthmref{Gamma nilp fully faithful}
using some additional explicit information about the category of 
$I^0$-equivariant objects in $\hg_\crit\mod$. Namely, the latter category
is essentially generated by Verma modules.

\medskip

Finally, \mainthmref{Miura} will follow from \mainthmref{main}
by combining the latter with the results of \cite{Bez}.

\section{Fully faithfulness of $\Upsilon$}   \label{ups adj}

In this section we will reduce the proof of \mainthmref{thm relation to Grassmannian}
to another assertion, \thmref{thm zero section}.

\ssec{}

As will be shown in \secref{closed imm 1}, we have the functors
$$\iota^*:\bD^{perf}(\tN/\cG)\leftrightarrows 
\bD^{perf}(\on{pt}/\cB):\iota_*$$
which give rise to a pair of mutually adjoint functors
$$\bD^f(\fD(\Fl^{\on{aff}}_G)_\crit\mod)\leftrightarrows
\on{pt}/\cB \underset{\on{pt}/\cG}\times \bD^f(\fD(\Gr^{\on{aff}}_G)_\crit\mod),$$
which we denote by $(\iota_\Fl)^*$ and $(\iota_\Fl)_*$, respectively.

\sssec{}

Recall the functor
$$\wt\Upsilon:\bD^f(\fD(\Fl^{\on{aff}}_G)_\crit\mod)\to 
\on{pt}/\cB \underset{\on{pt}/\cG}\times \bD^f(\fD(\Gr^{\on{aff}}_G)_\crit\mod).$$
of \eqref{wt ups}. In this section, we will use an alternative notation
for it, namely, $(\wt\iota_\Fl)^*$. 
Our present goal is to construct the right adjoint of $(\wt\iota_\Fl)^*$, that
we will denote $(\wt\iota_\Fl)_*$. 

\sssec{}   \label{constr wt i}

Consider the functor $\bD^f(\fD(\Gr^{\on{aff}}_G)_\crit\mod)\to 
\bD^f(\fD(\Fl^{\on{aff}}_G)_\crit\mod)$
$$\CF\mapsto p^{!*}(\CF):=p^{*}(\CF)[\dim(G/B)]\simeq
p^{!}(\CF)[-\dim(G/B)].$$

Since the action of $\Rep(\cG)$ on $\fD(\Gr^{\on{aff}}_G)_\crit\mod$
and $\fD(\Fl^{\on{aff}}_G)_\crit\mod$ is given by exact functors, 
the above functor has an evident structure of functor
between triangulated categories over $\on{pt}/\cG$.

\medskip

We regard $\bD^f(\fD(\Fl^{\on{aff}}_G)_\crit\mod)$ as a category over
$\on{pt}/\cB$ via the map $\pi:\tN/\cG\to \on{pt}/\cB$. 

\medskip

By the universal property of base change (see \secref{univ ppty base change}), 
from $p^{!*}$ we obtain a functor
$$\on{pt}/\cB\underset{\on{pt}/\cG}\times \bD^f(\fD(\Gr^{\on{aff}}_G)_\crit\mod)\to \bD^f(\fD(\Fl^{\on{aff}}_G)_\crit\mod),$$
between triangulated categories over $\on{pt}/\cB$. We denote it
$(\wt\iota_\Fl)_*$.

\begin{prop} \label{i star adj}
There exists a natural transformation
$(\wt\iota_\Fl)^*\circ (\wt\iota_\Fl)_*\to \on{Id}$,
as functors between categories over $\on{pt}/\cB$,
which makes $(\wt\iota_\Fl)_*$ a right adjoint of 
$(\wt\iota_\Fl)^*$ at the level triangulated categories.
\end{prop}

\medskip

\noindent{\it Remark.}
Although both sides of 
$$(\wt\iota_\Fl)_*:\on{pt}/\cB \underset{\on{pt}/\cG}\times \bD^f(\fD(\Gr^{\on{aff}}_G)_\crit\mod)
\to \bD^f(\fD(\Fl^{\on{aff}}_G)_\crit\mod)$$
are categories over $\tN/\cG$, the construction of $(\wt\iota_\Fl)_*$ only
makes it a functor between categories over $\on{pt}/\cB$. The additional
structure of functor between categories over $\tN/\cG$
is not be obvious from the construction.

\ssec{}

For the proof of \propref{i star adj} we will need to review two other adjunction
constructions.

\sssec{}  \label{G B adjunctions}

Let $\fq$ denote the (proper) map of stacks $\on{pt}/\cB\to \on{pt}/\cG$,
and consider the functor 
$$\fq^*:\bD^{perf}(\Coh(\on{pt}/\cG))\to 
\bD^{perf}(\Coh(\on{pt}/\cB)).$$
We will denote this functor also by $\Res^\cG_\cB$
as it corresponds to the restriction functor
$\Rep(\cG)\to \Rep(\cB)$.  It has an evident structure of
functor between categories over $\on{pt}/\cG$. 

\medskip

We claim that there exist the functors
$$\Ind^{\cG}_{\cB},\on{co-Ind}^\cG_\cB:
\bD^{perf}(\Coh(\on{pt}/\cB))\to 
\bD^{perf}(\Coh(\on{pt}/\cG)),$$
\begin{equation} \label{ind res adj}
\on{Id}\to \Ind^\cG_\cB\circ \Res^\cG_\cB 
\end{equation}
and
\begin{equation} \label{coind res adj}
\on{co-Ind}^\cG_\cB\circ \Res^\cG_\cB\to \on{Id},
\end{equation}
as categories and functors over the stack $\on{pt}/\cG$. 

\medskip

This follows from Sections \ref{coherent direct image} and \ref{proper morphism adjunctions},
as 
$$\Ind^{\cG}_{\cB}\simeq \fq_* \text{ and } \on{co-Ind}^\cG_\cB\simeq \fq_?,$$
in the notation of {\it loc. cit.} 
 
Moreover, by \eqref{Serre expected}, we have
an isomorphism of functors over $\on{pt}/\cG$:
\begin{equation} \label{Serre on Fl}
\on{Ind}^\cG_\cB\simeq \on{co-Ind}^\cG_\cB\circ 
(\CL^{2\crho}\underset{\CO_{\on{pt}/\cB}}\otimes ?)[-\dim(\on{Fl}^{\cG})].
\end{equation}

\medskip

Hence, for any triangulated category $\bD$ over $\on{pt}/\cG$
we have a similar set of functors and adjunctions between
$\bD$ and $\on{pt}/\cB\underset{\on{pt}/\cG}\times \bD\leftrightarrows
\bD$ as categories and functors over the stack $\on{pt}/\cG$. 

\sssec{}

Consider the direct image functor
$$p_!=p_*:\bD^f(\fD(\Fl^{\on{aff}}_G)_\crit\mod)\to \bD^f(\fD(\Gr^{\on{aff}}_G)_\crit\mod).$$
The construction of \secref{cat on Fl: fine DG model} endows it 
with a structure of functor between categories over $\on{pt}/\cG$. 
We leave the following assertion to the reader, as it
repeats the constructions carried out in 
\secref{DG model for the AB action}:

\begin{lem}  \label{p ! *}
The adjunction maps
$$\on{Id}\to p_*\circ p^{!*}[-\dim(G/B)] \text{ and } p^{!*}\circ p_![-\dim(G/B)] \to
\on{Id}$$
and 
$$p_!\circ p^{!*}[\dim(G/B)]\to \on{Id} \text{ and }
\on{Id}\to p^{!*}\circ p^![\dim(G/B)]$$
can be endowed with a structure of natural stransformations
between functors over $\on{pt}/\cG$.
\end{lem}

\sssec{Proof of \propref{i star adj}}

By the universal property of base change (see
\secref{univ ppty base change}), constructing a map of functors
$$(\wt\iota_\Fl)^*\circ (\wt\iota_\Fl)_*\to \on{Id},$$
compatible with the structure of functors over the stack $\on{pt}/\cB$ is equivalent 
to constructing a map 
$$(\wt\iota_\Fl)^*\circ (\wt\iota_\Fl)_*\circ \Res^\cG_\cB\to 
\Res^\cG_\cB:\bD^f(\fD(\Gr^{\on{aff}}_G)_\crit\mod)\to
\on{pt}/\cB\underset{\on{pt}/\cG}\times \bD^f(\fD(\Gr^{\on{aff}}_G)_\crit\mod)$$
as functors over $\on{pt}/\cG$. 

\medskip

By \secref{G B adjunctions}, the latter is in turn equivalent to
constructing a map
\begin{multline} \label{constr i star adj}
\on{co-Ind}^\cG_\cB\circ
(\wt\iota_\Fl)^*\circ (\wt\iota_\Fl)_*\circ \Res^\cG_\cB\simeq  
\on{co-Ind}^\cG_\cB\circ (\wt\iota_\Fl)^*\circ p^{!*}\to 
\on{Id}: \\
\bD^f(\fD(\Gr^{\on{aff}}_G)_\crit\mod)\to \bD^f(\fD(\Gr^{\on{aff}}_G)_\crit\mod),
\end{multline}
as functors over $\on{pt}/\cG$. 

\medskip

The construction of the arrow in \eqref{constr i star adj} follows
now from \lemref{p ! *} and the following:

\begin{lem}   \label{wt i p}
The functor
$$\on{co-Ind}^\cG_\cB\circ (\wt\iota_\Fl)^*:
\bD^f(\fD(\Fl^{\on{aff}}_G)_\crit\mod)\to \bD^f(\fD(\Gr^{\on{aff}}_G)_\crit\mod)$$
is canonically isomorphic to $p_![\dim(G/B)]$, as functors between categories
over the stack $\on{pt}/\cG$. 
\end{lem}

\begin{proof}
We will construct an isomorphism as a functor between triangulated
categories, compatible with the action of $\Rep^{f.d.}(\cG)$. The
upgrading of the isomorphism to the DG level will be commented on
in \secref{wt i p: DG model}.

\medskip

We have:
$$\on{co-Ind}^\cG_\cB\circ (\wt\iota_\Fl)^*(\CF)\simeq
\on{co-Ind}^\cG_\cB(\CF\star J_{2\rho}\star \CW)\simeq \CF\star
\on{co-Ind}^\cG_\cB(J_{2\rho}\star \CW).$$

\medskip

The isomorphism follows now from \cite{ABBGM}, Lemma 3.2.22
(or, using \eqref{Serre on Fl}, from \cite{FG5}, Proposition 3.18)
where it is shown that there exists a canonical isomorphism
\begin{equation} \label{av of baby}
\on{co-Ind}^\cG_\cB(J_{2\rho}\star \CW)\simeq
\delta_{1,\Gr^{\on{aff}}_G}[\dim(G/B)].
\end{equation}

\medskip

The compatibility with the action of $\Rep^{f.d.}(\cG)$ follows
from the construction tautologically.

\end{proof}

Thus, to prove the proposition, it remains to show that the
map 
$$\Hom
\left(\CF,(\wt\iota_\Fl)_*(\CF')\right) 
\to \Hom
\left((\wt\iota_\Fl)^*(\CF),(\wt\iota_\Fl)^*\circ (\wt\iota_\Fl)_*(\CF')\right) 
\to \Hom
\left((\wt\iota_\Fl)^*(\CF),\CF'\right).
$$
is an isomorphism for any
$\CF\in \bD^f(\fD(\Fl^{\on{aff}}_G)_\crit\mod)$, $\CF'\in 
\on{pt}/\cB\underset{\on{pt}/\cG}\times \bD^f(\fD(\Gr^{\on{aff}}_G)_\crit\mod)$.

\medskip

By the definition of the latter category, we can take $\CF'$ of the form
$\CL^\cla\underset{\CO_{\on{pt}/\cB}}\otimes \Res^\cG_\cB(\CF'_1)$ with
$\CF'_1\in \bD^f(\fD(\Gr^{\on{aff}}_G)_\crit\mod)$. However, since the functor of
tensor product with $\CL^\cla$ is an equivalence on both categories,
we can replace $\CF$ by $\CL^{-\cla}\underset{\CO_{\on{pt}/\cB}}\otimes \CF$,
and so reduce the assertion to the case when $\CF'=\Res^\cG_\cB(\CF'_1)$.

\medskip

In the latter case, the assertion reduces to the $(p_!,p^!)$ adjunction
by \lemref{wt i p} above.

\qed

\ssec{}  \label{sect map i to it}

We claim now that \propref{i star adj} gives rise to a natural
transformation of functors at the level of triangulated categories:
\begin{equation} \label{basic nat trans star}
(\iota_\Fl)_*\to (\wt\iota_\Fl)_*\circ \Upsilon,
\end{equation}
where both sides are functors
$$\on{pt}/\cB\underset{\tN/\cG}\times \bD^f(\fD(\Fl^{\on{aff}}_G)_\crit\mod)\to
\bD^f(\fD(\Fl^{\on{aff}}_G)_\crit\mod).$$

Indeed, the natural transformation in question comes by adjunction from
$$(\wt\iota_\Fl)^*\circ (\iota_\Fl)_*\simeq\Upsilon\circ 
(\iota_\Fl)^*\circ (\iota_\Fl)_*\to \Upsilon.$$

\medskip

Composing the natural transformation \eqref{basic nat trans star} with the functor
$(\iota_\Fl)^*$, we obtain
a natural transformation
\begin{equation}   \label{map i to it}
(\iota_\Fl)_*\circ (\iota_\Fl)^*\to
(\wt\iota_\Fl)_*\circ (\wt\iota_\Fl)^*,
\end{equation}
also at the level of triangulated categories, which makes the following
diagram commute:
\begin{equation}  \label{barr beck}
\CD
\Hom((\iota_\Fl)^*(\CF_1),(\iota_\Fl)^*(\CF_2))  @>{\sim}>>
\Hom(\CF_1,(\iota_\Fl)_*\circ (\iota_\Fl)^*(\CF_2)) \\
@VVV   @VVV  \\
\Hom(\Upsilon\circ (\iota_\Fl)^*(\CF_1),\Upsilon\circ (\iota_\Fl)^*(\CF_2)) 
@>{\sim}>> \Hom(\CF_1,(\wt\iota_\Fl)_*\circ (\wt\iota_\Fl)^*(\CF_2))
\endCD
\end{equation}
for $\CF_1,\CF_2\in \bD^f(\fD(\Fl^{\on{aff}}_G)_\crit\mod)$.

\medskip

We will prove:

\begin{thm}  \label{thm zero section}
The natural transformation \eqref{map i to it} is an isomorphism.
\end{thm}

\sssec{}

Let us show how \thmref{map i to it} implies fully-faithfulness of $\Upsilon$.

\begin{proof}

We have to show that the map 
$$\Hom(\CF'_1,\CF'_2)\to \Hom(\Upsilon(\CF'_1),\Upsilon(\CF'_2))$$
is an isomorphism for any $\CF'_1,\CF'_2\in 
\on{pt}/\cB\underset{\tN/\cG}\times\bD^f(\fD(\Fl^{\on{aff}}_G)_\crit\mod)$.

\medskip

Since the functor $\iota^*:\bD^{perf}(\Coh(\tN/\cG))\to
\bD^{perf}(\Coh(\on{pt}/\cB))$ is affine (see \secref{affine functor}),
it is sufficient to take $\CF'_i$, $i=1,2$ of the form 
$(\iota_\Fl)^*(\CF_i)$ with $\CF_i\in \bD^f(\fD(\Fl^{\on{aff}}_G)_\crit\mod)$.

\medskip

In the latter case the required isomorphism follows from 
\thmref{thm zero section} via the commutative diagram
\eqref{barr beck}.

\end{proof}

The same argument shows that \thmref{thm zero section} implies
that the natural transformation \eqref{basic nat trans star}
is also an isomorphism.

\ssec{}

In \secref{zero section} we will give two proofs of \thmref{thm zero section}: 
a shorter one, which relies on some unpublished results, announced in \cite{Bez},
and a slightly longer one, which only uses \cite{AB}.

\medskip

As a preparation for the latter argument
we will now describe a functor, which is the right adjoint to $(\wt\iota_\Fl)_*$. 
We will use the notation, $(\wt\iota_\Fl)_!:=(\wt\iota_\Fl)_*$, and 
the right adjoint in question will be denoted $(\wt\iota_\Fl)^!$.

\sssec{}

Set
$$(\wt\iota_\Fl)^!:=\CL^{-2\crho}\underset{\CO_{\on{pt}/\cB}}\otimes
(\wt\iota_\Fl)^*[-\dim(\cG/\cB)].$$

\medskip

In other words, the construction of $(\wt\iota_\Fl)^!$ is the same as that
of $(\wt\iota_\Fl)^*$, where instead of the object
$$J_{2\rho}\star \CW\in \on{pt}/\cB 
\underset{\on{pt}/\cG}\times \fD(\Gr^{\on{aff}}_G)_\crit\mod^I$$
we use $\CW[-\dim(\cG/\cB)]$.

\medskip

By \cite{FG5}, Proposition 3.18 (or \cite{ABBGM}, Proposition 3.2.16)
we have 
\begin{equation} \label{av of co-baby}
\on{Ind}^\cG_\cB(\CW)\simeq
\delta_{1,\Gr^{\on{aff}}_G}.
\end{equation}

\medskip

Repeating the proof of \propref{i star adj}, we obtain that there exists
a natural transformation
$$\on{Id}\to (\wt\iota_\Fl)^!\circ (\wt\iota_\Fl)_!$$
as functors between categories over $\on{pt}/\cB$, which makes
$(\wt\iota_\Fl)^!$ the right adjoint of $(\wt\iota_\Fl)_!$ at the
level of triangulated categories.

\sssec{}

According to \secref{closed imm 3}, we have the 
functors 
$$\iota_!:=\iota_*:\bD^{perf}(\on{pt}/\cB)\leftrightarrows \bD^{perf}(\tN/\cG):\iota^!$$
$$\iota^!\simeq \CL^{-2\crho}\underset{\CO_{\on{pt}/\cB}}\otimes
\iota^*[-\dim(\cG/\cB)],$$
 over $\tN/\cG$, and the adjunctions 
$$\iota_!\circ \iota^! \to \on{Id} \text{ and }
\on{Id}\to \iota^!\circ \iota_!,$$
defined at the triangulated level. 

\medskip

Hence, by {\it loc. cit.}, we have the corresponding functors and adjunctions between 
$$(\iota_\Fl)_!\simeq (\iota_\Fl)_*:\on{pt}/\cB \underset{\on{\tN}/\cG}\times \bD^f(\fD(\Fl^{\on{aff}}_G)_\crit\mod)
\leftrightarrows \bD^f(\fD(\Fl^{\on{aff}}_G)_\crit\mod):(\iota_\Fl)^!.$$
and
$$\iota_\Fl^!\simeq \CL^{-2\crho}\underset{\CO_{\on{pt}/\cB}}\otimes
\iota_\Fl^*[-\dim(\cG/\cB)],$$
at the triangulated level.

\sssec{}

By construction, we have:
$$(\wt\iota_\Fl)^!\simeq \Upsilon\circ (\iota_\Fl)^!,$$
which as in \secref{sect map i to it} defines a natural transformation
\begin{equation}  \label{basic nat trans shriek}
\Upsilon\circ (\wt\iota_\Fl)_!\to (\iota_\Fl)_!.
\end{equation}

Since
$$(\wt\iota_\Fl)_!\simeq (\wt\iota_\Fl)_* \text{ and }
(\iota_\Fl)_!\simeq (\iota_\Fl)_*,$$
\eqref{basic nat trans shriek} gives a map in the direction opposite to
that of \eqref{basic nat trans star}:
\begin{equation}  \label{basic nat trans star another}
\Upsilon\circ (\wt\iota_\Fl)_*\to (\iota_\Fl)_*.
\end{equation}

Composing with $(\iota_\Fl)^*$ we obtain a natural transformation
\begin{equation}  \label{map it to i}
(\wt\iota_\Fl)_*\circ (\wt\iota_\Fl)^*\to 
(\iota_\Fl)_*\circ (\iota_\Fl)^*.
\end{equation}

\thmref{thm zero section} follows from the next assertion:

\begin{prop}  \label{comp zero section}
Each of the two compositions
$$(\iota_\Fl)_*\circ (\iota_\Fl)^* \overset{\text{\eqref{map i to it}}}\to
(\wt\iota_\Fl)_*\circ (\wt\iota_\Fl)^*
\overset{\text{\eqref{map it to i}}}\to (\iota_\Fl)_*\circ (\iota_\Fl)^*$$
and 
$$(\wt\iota_\Fl)_*\circ (\wt\iota_\Fl)^*
\overset{\text{\eqref{map it to i}}}\to (\iota_\Fl)_*\circ (\iota_\Fl)^*
\overset{\text{\eqref{map i to it}}}\to
(\wt\iota_\Fl)_*\circ (\wt\iota_\Fl)^*$$
is a non-zero scalar multiple of the identity map.
\end{prop}

\section{Description of the zero section}  \label{zero section}

In this section will deduce \thmref{thm zero section} from 
Bezrukavnikov's theory. We first give an argument, using
the still unpublished results announced in \cite{Bez}.

\ssec{}  \label{delta is enough}

It is clear from the construction of the map \eqref{map i to it} that
for $\CF\in  \bD^f(\fD(\Fl^{\on{aff}}_G)_\crit\mod)$ we have
a commutative diagram of functors:
\begin{equation} \label{map i to it vs delta}
\CD
(\iota_\Fl)_*\circ (\iota_\Fl)^*(\CF)   @>{\text{\eqref{map i to it}}}>>  
(\wt\iota_\Fl)_*\circ (\wt\iota_\Fl)^*(\CF)   \\
@V{\sim}VV    @V{\sim}VV   \\
\CF\star \left((\iota_\Fl)_*\circ (\iota_\Fl)^*(\delta_{1,\Fl^{\on{aff}}_G})\right) @>>>
\CF\star \left((\wt\iota_\Fl)_*\circ (\wt\iota_\Fl)^*(\delta_{1,\Fl^{\on{aff}}_G})\right).
\endCD
\end{equation}

\medskip

Hence, the assertion of \thmref{thm zero section} is equivalent to the
fact that the map
\begin{equation} \label{map i to it delta}
(\iota_\Fl)_*\circ (\iota_\Fl)^*(\delta_{1,\Fl^{\on{aff}}_G}) \overset{\text{\eqref{map i to it}}}
\longrightarrow
(\wt\iota_\Fl)_*\circ (\wt\iota_\Fl)^*(\delta_{1,\Fl^{\on{aff}}_G})
\end{equation}
is an isomorphism.

\sssec{}

Let $\bD^f(\fD(\Gr^{\on{aff}}_G)_\crit\mod)^{I^0}\subset \bD^f(\fD(\Gr^{\on{aff}}_G)_\crit\mod)$ 
be the categories, defined as in the case of $\Fl^{\on{aff}}_G$,
see \secref{intr Iw category}.

\medskip

The functor $\Upsilon$ induces a functor
\begin{equation} \label{ups I}
\on{pt}/\cB\underset{\tN/\cG}\times 
\bD^f(\fD(\Fl^{\on{aff}}_G)_\crit\mod)^{I^0} \to \on{pt}/\cB\underset{\on{pt}/\cG}\times 
\bD^f(\fD(\Gr^{\on{aff}}_G)_\crit\mod)^{I^0},
\end{equation}
and since $\delta_{1,\Fl^{\on{aff}}_G}\in \bD^f(\fD(\Gr^{\on{aff}}_G)_\crit\mod)^{I^0}$, it
suffices to show that the functor \eqref{ups I} is fully faithful.

\sssec{}

Applying \thmref{Bezr}, we obtain that the LHS in \eqref{ups I}
identifies with
$$\on{pt}/\cB\underset{\tN/\cG}\times \bD^b(\Coh(\check{\on{St}}/\cG)).$$

\medskip

Consider the Cartesian product 
$\on{pt}/\cB\underset{\tN/\cG}\times \check{\on{St}}/\cG$,
understood in the DG sense. By an argument similar to that of 
\secref{perv vs coh}, we have a fully faithful functor
\begin{equation} \label{perf vs coh steinberg}
\on{pt}/\cB\underset{\tN/\cG}\times \bD^b(\Coh(\check{\on{St}}/\cG))\to
\bD^b(\Coh(\on{pt}/\cB\underset{\tN/\cG}\times \check{\on{St}}/\cG)).
\end{equation}

\sssec{}

Consider now the DG scheme
$$\wt\Fl{}^\cG:=\on{pt}\underset{\cg}\times \tg.$$
By applying the Koszul duality to the equivalence of \cite{ABG} we
obtain an equivalence 
\begin{equation} \label{ABG}
\bD^f(\fD(\Gr^{\on{aff}}_G)_\crit\mod)^{I^0} \simeq 
\bD^b(\Coh(\wt\Fl{}^\cG/\cG)),
\end{equation}
as categories over $\on{pt}/\cG$.

\medskip

By an argument similar to that of \secref{perv vs coh}, there exists a fully faithful functor
$$\on{pt}/\cB\underset{\on{pt}/\cG}\times \bD^b(\Coh(\wt\Fl{}^\cG/\cG))\to
\bD^b(\Coh(\on{pt}/\cB\underset{\on{pt}/\cG}\times \wt\Fl{}^\cG/\cG)).$$

\sssec{}

Note now that 
$$\on{Fl}^{\cG}\underset{\tN}\times \check{\on{St}}\simeq
\on{Fl}^{\cG}\underset{\cg}\times \tg\simeq
\on{Fl}^{\cG}\times \wt\Fl{}^\cG,$$
and hence
$$\on{pt}/\cB\underset{\tN/\cG}\times \check{\on{St}}/\cG\simeq
\left(\on{Fl}^{\cG}\underset{\tN}\times \check{\on{St}}\right)/\cG\simeq
\left(\on{Fl}^{\cG}\times \wt{\on{Fl}}{}^{\cG}\right)/\cG\simeq 
\on{pt}/\cB\underset{\on{pt}/\cG}\times \wt\Fl{}^\cG/\cG,$$
and by the construction of the functors involved we have 
a commutative diagram 
$$
\CD
\on{pt}/\cB\underset{\tN/\cG}\times \bD^f(\fD(\Fl^{\on{aff}}_G)_\crit\mod)^{I^0} 
@>{\text{\eqref{ups I}}}>> \on{pt}/\cB\underset{\on{pt}/\cG}\times 
\bD^f(\fD(\Gr^{\on{aff}}_G)_\crit\mod)^{I^0}  \\
@V{\sim}VV    @V{\sim}VV   \\
\on{pt}/\cB\underset{\tN/\cG}\times \bD^b(\Coh(\check{\on{St}}/\cG))  & & 
\on{pt}/\cB\underset{\on{pt}/\cG}\times \bD^b(\Coh(\wt\Fl{}^\cG/\cG)) \\
@VVV    @VVV  \\
\bD^b(\Coh(\on{pt}/\cB\underset{\tN/\cG}\times \check{\on{St}}/\cG))
@>{\sim}>>
\bD^b(\Coh(\on{pt}/\cB\underset{\on{pt}/\cG}\times \wt\Fl{}^\cG/\cG)),
\endCD
$$
proving that \eqref{ups I} is fully faithful.

\medskip

\noindent{\it Remark.} The above interpretation in terms of
coherent sheaves makes it explicit why the functor $\Upsilon$
is not an equivalence. The reason is that the functor
\eqref{perf vs coh steinberg} is not an equivalence.

\ssec{}

We shall now give an alternative argument, proving \thmref{thm zero section},
which avoids the reference to \cite{Bez}. Namely, we are going to prove
\propref{comp zero section}.  

\medskip

By the same argument as in \secref{delta is enough}, it is enough 
to show that the compositions
\begin{equation} \label{comp 1 for delta}
(\iota_\Fl)_*\circ (\iota_\Fl)^*(\delta_{1,\Fl^{\on{aff}}_G}) \to
(\wt\iota_\Fl)_*\circ (\wt\iota_\Fl)^*(\delta_{1,\Fl^{\on{aff}}_G})
\to (\iota_\Fl)_*\circ (\iota_\Fl)^*(\delta_{1,\Fl^{\on{aff}}_G})
\end{equation}
and 
\begin{equation} \label{comp 2 for delta}
(\wt\iota_\Fl)_*\circ (\wt\iota_\Fl)^*(\delta_{1,\Fl^{\on{aff}}_G})
\to (\iota_\Fl)_*\circ (\iota_\Fl)^*(\delta_{1,\Fl^{\on{aff}}_G})\to
(\wt\iota_\Fl)_*\circ (\wt\iota_\Fl)^*(\delta_{1,\Fl^{\on{aff}}_G})
\end{equation}
are non-zero multiples of the identity map.

\medskip

This, in turn, breaks into three assertions:

\begin{prop} \label{comp non-zero}
The compositions \eqref{comp 1 for delta} and \eqref{comp 2 for delta}
are non-zero.
\end{prop}

\begin{prop}   \label{from delta to delta}
The composition \eqref{comp 1 for delta} is the image by means of
$(\iota_\Fl)_*$ of a map 
$$(\iota_\Fl)^*(\delta_{1,\Fl^{\on{aff}}_G}) \to (\iota_\Fl)^*(\delta_{1,\Fl^{\on{aff}}_G})$$
in $\on{pt}/\cB\underset{\tN/\cG}\times 
\bD^f(\fD(\Fl^{\on{aff}}_G)_\crit\mod)$.
\end{prop}

\begin{prop} \label{endomorphisms}   \hfill

\smallskip

\noindent(1) 
$\End\left((\wt\iota_\Fl)_*\circ (\wt\iota_\Fl)^*(\delta_{1,\Fl^{\on{aff}}_G})\right)\simeq \BC$.

\smallskip

\noindent(2) 
$\End\left((\iota_\Fl)^*(\delta_{1,\Fl^{\on{aff}}_G})\right)\simeq \BC$.

\end{prop}

We shall now prove \propref{comp non-zero}. Propositions
\ref{from delta to delta} and \ref{endomorphisms} will be
proved in \secref{calc of endomorphisms}.

\ssec{}

To prove \propref{comp non-zero}, we will consider a
non-degenerate character $\fn\to \BG_a$, and denote
by $\psi_0$ the corresponding character of $I^0$:
$$\psi:I^0\twoheadrightarrow \fn\to \BG_a.$$

\medskip

Consider the corresponding equivariant subcategories
$$\bD^f(\fD(\Gr^{\on{aff}}_G)_\crit\mod)^{'I^0,\psi}\subset \bD^f(\fD(\Gr^{\on{aff}}_G)_\crit\mod)$$
and
$$\bD^f(\fD(\Fl^{\on{aff}}_G)_\crit\mod)^{'I^0,\psi}\subset \bD^f(\fD(\Fl^{\on{aff}}_G)_\crit\mod),$$
where $'I^0$ is the conjugate of $I^0$ by means of the element
$t^\crho\in T\ppart$.

\medskip

The functor $\Upsilon$ induces a functor
\begin{equation} \label{ups whit}
\on{pt}/\cB\underset{\tN/\cG}\times 
\bD^f(\fD(\Fl^{\on{aff}}_G)_\crit\mod)^{'I^0,\psi}\to 
\on{pt}/\cB\underset{\on{pt}/\cG}\times 
\bD^f(\fD(\Gr^{\on{aff}}_G)_\crit\mod)^{'I^0,\psi}.
\end{equation}

\medskip

To prove \propref{endomorphisms} it is enough to show that
the compositions \eqref{comp 1 for delta} and \eqref{comp 2 for delta}
do not vanish, as natural transformations between functors from
$\bD^f(\fD(\Fl^{\on{aff}}_G)_\crit\mod)^{'I^0,\psi}$ to itself. 

\medskip

The latter would follow, once we show that the functor
\eqref{ups whit} is fully faithful. We claim that the latter functor
is in fact an equivalence.

\sssec{}

Indeed, by \cite{FG2}, Theorem 15.8 (or \cite{ABBGM}, Corollary 2.2.3)
we have a canonical equivalence
$$\bD^f(\fD(\Gr^{\on{aff}}_G)_\crit\mod)^{'I^0,\psi}\simeq \bD^{perf}(\Coh(\on{pt}/\cG)),$$
hence, 
$$\on{pt}/\cB\underset{\on{pt}/\cG}\times 
\bD^f(\fD(\Gr^{\on{aff}}_G)_\crit\mod)^{'I^0,\psi} \simeq \bD^{perf}(\Coh(\on{pt}/\cB)).$$

\medskip

The main result of \cite{AB} asserts that the category 
$\bD^f(\fD(\Fl^{\on{aff}}_G)_\crit\mod)^{'I^0,\psi}$, viewed as a triangulated category
over $\tN/\cG$, is canonically equivalent to $\bD^{perf}(\Coh(\tN/\cG))$.
Therefore,
$$\on{pt}/\cB\underset{\tN/\cG}\times 
\bD^f(\fD(\Fl^{\on{aff}}_G)_\crit\mod)^{'I^0,\psi}\simeq \bD^{perf}(\Coh(\on{pt}/\cB)).$$

\medskip

Moreover, from the construction of the functor of \cite{AB} and 
\cite{ABBGM}, Proposition 3.2.6(1),
the functor \eqref{ups whit} corresponds to the identity functor
on $\bD^{perf}(\Coh(\on{pt}/\cB))$, implying our assertion.

\section{New $\bbt$-structure and the affine Grassmannian}  \label{t structure and Grass}

In this section we will study how the new t-structure on $\bD_{ren}(\fD(\Fl^{\on{aff}}_G)_\crit\mod)$
behaves with respect to the functors $(\wt\iota_\Fl)^*$, $(\wt\iota_\Fl)_*$ and $\Upsilon$.

\ssec{}

Let us recall that 
\begin{equation}  \label{B G Grasmmannian ind}
\on{pt}/\cB\underset{\on{pt}/\cG} \arrowtimes 
\bD^f(\fD(\Gr^{\on{aff}}_G)_\crit\mod)
\end{equation} denotes the ind-completion
of the category 
$$\on{pt}/\cB\underset{\on{pt}/\cG} \times 
\bD^f(\fD(\Gr^{\on{aff}}_G)_\crit\mod)$$
(see \secref{base change}). By \secref{t structure ten prod},
the category \eqref{B G Grasmmannian ind} acquires a 
compactly generated t-structure. It is characterized by the property
that the $\leq 0$ subcategory is generated by objects of the form
$$V\underset{\CO_{\on{pt}/\cG}}\otimes \CF,$$
$V\in \bD^{perf,\leq 0}(\Coh(\on{pt}/\cB))$,
$\CF\in \bD^{f,\leq 0}(\fD(\Gr^{\on{aff}}_G)_\crit\mod)$.

\sssec{}

Let us denote by the same characters $(\wt\iota_\Fl)^*,(\wt\iota_\Fl)_*=(\wt\iota_\Fl)_!$,
$(\wt\iota_\Fl)^!$
the ind-extensions of the functors from \secref{ups adj}
$$\bD_{ren}(\fD(\Fl^{\on{aff}}_G)_\crit\mod)\leftrightarrows
\on{pt}/\cB\underset{\on{pt}/\cG} \arrowtimes 
\bD^f(\fD(\Gr^{\on{aff}}_G)_\crit\mod).$$
They satisfy the same adjunction properties as the original functors.

\medskip

\begin{prop}  \label{iota ex}   With respect to the new t-structure on
$\bD_{ren}(\fD(\Fl^{\on{aff}}_G)_\crit\mod)$ we have:

\smallskip

\noindent{\em(a)}
$(\wt\iota_\Fl)^*$ is right-exact,

\smallskip

\noindent{\em(b)} $(\wt\iota_\Fl)_*$ is exact.

\smallskip

\noindent{\em(c)} 
$(\wt\iota_\Fl)^!$ is left-exact.
\end{prop}

\begin{proof}

Let us first show that $(\wt\iota_\Fl)_*$ is right-exact. This amounts to the next:

\begin{lem}
The functor 
$$p^{!*}:\bD(\fD(\Gr^{\on{aff}}_G)_\crit\mod)\to \bD_{ren}(\fD(\Fl^{\on{aff}}_G)_\crit\mod)$$
is right-exact (and, in fact, exact) in the new t-structure.
\end{lem}

\medskip

\noindent{\it Remark.} Note that $p^{!*}$ is exact and hence left-exact
in the old t-structure, and hence is left-exact in the new t-structure. Hence,
the essential image of $p^{!*}$ provides a collection of objects that
belong to the hearts of both t-structures. We remind that another such collection 
was given by \lemref{finite flags}. 

\begin{proof}

The assertion of the lemma is equivalent to the fact that
$$p^{!*}(\CF)\star J_{-\cla}\in \bD_{ren}^{\leq 0_{old}}(\fD(\Fl^{\on{aff}}_G)_\crit\mod)$$
for any $\cla\in \cLambda^+$. We have:
$$p^{!*}(\CF)\star J_{-\cla}\simeq 
\CF\underset{G[[t]]}\star 
\left(p^{!*}(\delta_{1,\Gr^{\on{aff}}_G})\underset{I}\star J_{-\cla}\right).$$
Since $p^{!*}(\delta_{1,\Gr^{\on{aff}}_G})\underset{I}\star J_{-\cla}\in \fD(\Fl^{\on{aff}}_G)_\crit\mod^{G[[t]]}$,
the assertion follows from the next:

\begin{lem}   \label{conv with Gr}
For $\CF_1\in \fD(\Gr^{\on{aff}}_G)_\crit\mod$ and $\CF_2\in \fD\left(G\ppart\right)_\crit\mod^{G[[t]]}$,
the convolution $\CF_1\underset{G[[t]]}\star \CF_2\in \bD(\fD\left(G\ppart\right)_\crit\mod)$
is acylcic off cohomological degree $0$.
\end{lem}

The lemma is proved by repeating the argument of \cite{Ga}, Theorem 1(a), or \cite{ABBGM}, Sect. 2.1.3.

\end{proof}

\medskip

The right-exactness of $(\wt\iota_\Fl)_*=(\wt\iota_\Fl)_!$ implies
by adjunction the left-exactness of $(\wt\iota_\Fl)^!$. Hence, 
it remains to show that $(\wt\iota_\Fl)^*$ is right-exact; this would imply 
the left-exactness of $(\wt\iota_\Fl)_*$ also by adjunction.

\medskip

\begin{lem}
An object 
$$\CF'\in \on{pt}/\cB\underset{\on{pt}/\cG}\times \bD^f(\fD(\Gr^{\on{aff}}_G)_\crit\mod)$$
is $\leq 0$, as an object of
$\on{pt}/\cB\underset{\on{pt}/\cG}\arrowtimes \bD^f(\fD(\Gr^{\on{aff}}_G)_\crit\mod)$
if and only if for all $\cla$,
$$\on{co-Ind}^{\cG}_\cB(\CL^\cla\otimes \CF) \in \bD^f(\fD(\Gr^{\on{aff}}_G)_\crit\mod)$$
is $\leq 0$.
\end{lem}

\begin{proof}

This follows from the fact that for $\CF'$ as in the lemma,
which is $\leq k$ for some $k$, there exists
$\cla$ deep inside the dominant chamber so that
$$\on{Cone}\left(\Res^\cG_\cB\circ \on{co-Ind}^{\cG}_\cB(\CL^{-\cla}\otimes \CF)
\to \CL^{-\cla}\otimes \CF\right)$$
is $\leq k-1$.

\end{proof}

Hence, to prove that $(\wt\iota_\Fl)^*$ is right-exact, it suffices to
show that the composition 
$$\on{co-Ind}^{\cG}_\cB\circ (\wt\iota_\Fl)^*:
\bD_{ren}(\fD(\Fl^{\on{aff}}_G)_\crit\mod)\to \bD_{ren}(\fD(\Fl^{\on{aff}}_G)_\crit\mod)$$
is right-exact. The latter holds due to \lemref{wt i p}.

\end{proof}

\ssec{}   \label{t structure on B N Flags ind}

Consider the category
\begin{equation} \label{B N Flags ind}
\on{pt}/\cB\underset{\tN/\cG} \arrowtimes 
\bD^f(\fD(\Fl^{\on{aff}}_G)_\crit\mod),
\end{equation} 
which is the ind-completion of the category 
$$\on{pt}/\cB\underset{\tN/\cG} \times 
\bD^f(\fD(\Fl^{\on{aff}}_G)_\crit\mod).$$

By \secref{t structure ten prod}, the new t-structure on
$\bD_{ren}(\fD(\Fl^{\on{aff}}_G)_\crit\mod)$ gives rise to a t-structure
on \eqref{B N Flags ind}. It is characterized by the property
that the $\leq 0$ category is generated by the objects
of the form
$$\CM\underset{\CO_{\tN/\cG}}\otimes \CF$$
for $\CM\in \bD^{perf,\leq 0}(\Coh(\on{pt}/\cB))$
and $\CF\in \bD^f(\fD(\Fl^{\on{aff}}_G)_\crit\mod)\cap 
\bD^{\leq 0}_{ren}(\fD(\Fl^{\on{aff}}_G)_\crit\mod)$.

\medskip

The ind-extension of $(\iota_\Fl)^*$:
$$(\iota_\Fl)^*:\bD_{ren}(\fD(\Fl^{\on{aff}}_G)_\crit\mod)\to
\on{pt}/\cB\underset{\tN/\cG} \arrowtimes 
\bD^f(\fD(\Fl^{\on{aff}}_G)_\crit\mod)$$
is tautologically right-exact, and its right adjoint
$$(\iota_\Fl)_*:\on{pt}/\cB\underset{\tN/\cG} \arrowtimes 
\bD^f(\fD(\Fl^{\on{aff}}_G)_\crit\mod)\to \bD_{ren}(\fD(\Fl^{\on{aff}}_G)_\crit\mod)$$
is exact, by \propref{affine exact}.

\sssec{}

As a formal corollary of \propref{iota ex}(a), we obtain:

\begin{cor} \label{Ups right-exact}
The functor
$$\Upsilon:\on{pt}/\cB\underset{\tN/\cG} \arrowtimes 
\bD^f(\fD(\Fl^{\on{aff}}_G)_\crit\mod)\to 
\on{pt}/\cB\underset{\on{pt}/\cG}\arrowtimes \bD^f(\fD(\Gr^{\on{aff}}_G)_\crit\mod)$$
is right-exact.
\end{cor}

\section{Calculation of endomorphisms and the functor $\Omega$}  
\label{calc of endomorphisms}

In this section we will prove Propositions \ref{from delta to delta}
and \ref{endomorphisms}. This section can 
be considered redundant by a reader willing to 
accept the proof of \thmref{thm zero section}, based on \cite{Bez}
given in \secref{zero section}.

\ssec{Proof of \propref{endomorphisms}(1)}

We have:
\begin{multline*}
\Hom\left((\wt\iota_\Fl)_*\circ (\wt\iota_\Fl)^*(\delta_{1,\Fl^{\on{aff}}_G}),
(\wt\iota_\Fl)_*\circ (\wt\iota_\Fl)^*(\delta_{1,\Fl^{\on{aff}}_G})\right)\simeq \\
\simeq \Hom\left((\wt\iota_\Fl)^*\circ (\wt\iota_\Fl)_*\circ (\wt\iota_\Fl)^*(\delta_{1,\Fl^{\on{aff}}_G}),
(\wt\iota_\Fl)^*(\delta_{1,\Fl^{\on{aff}}_G})\right).
\end{multline*}

By construction, $(\wt\iota_\Fl)^*(\delta_{1,\Fl^{\on{aff}}_G})\simeq J_{2\rho}\star \CW$
belongs to the heart of the t-structure on the category
$\on{pt}/\cB\underset{\on{pt}/\cG}\arrowtimes \bD^f(\fD(\Gr^{\on{aff}}_G)_\crit\mod)$.

\medskip

We have the following assertion

\begin{prop}  \label{comp i* i*}
For
$\CF'\in \on{pt}/\cB\underset{\on{pt}/\cG}\arrowtimes \bD^f(\fD(\Gr^{\on{aff}}_G)_\crit\mod)$,
which is $\leq 0$, the 
adjunction map $(\wt\iota_\Fl  )^*\circ (\wt\iota_\Fl  )_*(\CF')\to \CF'$
induces an isomorphism
$$H^0\left((\wt\iota_\Fl  )^*\circ (\wt\iota_\Fl  )_*(\CF')\right)\to \CF'.$$
\end{prop}

From the proposition we obtain that the map
$$\Hom\left((\wt\iota_\Fl)^*(\delta_{1,\Fl^{\on{aff}}_G}),
(\wt\iota_\Fl)^*(\delta_{1,\Fl^{\on{aff}}_G})\right)\to
\Hom\left((\wt\iota_\Fl)^*\circ (\wt\iota_\Fl)_*\circ (\wt\iota_\Fl)^*(\delta_{1,\Fl^{\on{aff}}_G}),
(\wt\iota_\Fl)^*(\delta_{1,\Fl^{\on{aff}}_G})\right)$$
is surjective (in fact, an isomorphism). Hence, it suffices to show that
$$\Hom\left((\wt\iota_\Fl)^*(\delta_{1,\Fl^{\on{aff}}_G}),
(\wt\iota_\Fl)^*(\delta_{1,\Fl^{\on{aff}}_G})\right):=
\End(J_{2\rho}\star \CW) \simeq \BC.$$

This, however, follows from the fact that $\End(\CW)\simeq \BC$, 
which is easy to deduce from the definitions (or, alternatively,
from \cite{ABBGM}, Proposition 3.2.6(1)).

\sssec{Proof of \propref{comp i* i*}}  

By \secref{heart of base change}, the heart of the t-structure
on the category
$\on{pt}/\cB\underset{\on{pt}/\cG}\arrowtimes \bD^f(\fD(\Gr^{\on{aff}}_G)_\crit\mod)$
is the abelian category
$$\on{pt}/\cB\underset{\on{pt}/\cG}\times \fD(\Gr^{\on{aff}}_G)_\crit\mod.$$

\medskip

Since the functor $(\wt\iota_\Fl  )_*$  is exact and $(\wt\iota_\Fl  )^*$ is
right-exact, it sufficient to take $\CF'$ to be one of the generators of
the category, i.e., of the form 
$\CL^\cla\underset{\CO_{\on{pt}/\cB}}\otimes \Res^\cG_\cB(\CF'')$
with $\CF''\in \fD(\Gr^{\on{aff}}_G)_\crit\mod$. However, since the functors
$(\wt\iota_\Fl  )_*$  and $(\wt\iota_\Fl  )^*$ are compatible with the action
of $\Coh(\on{pt}/\cB)$, we can assume $\cla=0$, i.e.,
$\CF'\simeq \Res^\cG_\cB(\CF'')$.

\medskip

Thus, we are reduced to showing that for $\CF''\in \bD^f(\fD(\Gr^{\on{aff}}_G)_\crit\mod)$,
the map
$$H^0(\wt\Upsilon\circ p^{!*}(\CF''))\to \Res^\cG_\cB(\CF'')$$
is an isomorphism.

\medskip

We have:
$$\wt\Upsilon\circ p^{!*}(\CF'')\simeq p^{!*}(\CF'')\star J_{2\rho}\star \CW\simeq
\CF''\underset{G[[t]]}\star \left(p^{!*}(\delta_{1,\Gr^{\on{aff}}_G}) \star J_{2\rho}\star \CW\right).$$

By \cite{ABBGM}, Proposition 3.2.6 (or which can be otherwise
easily proved directly), 
$$H^0\left(p^{!*}(\delta_{1,\Gr^{\on{aff}}_G}) \star J_{2\rho}\star \CW\right)
\simeq \Res^\cG_\cB(\delta_{1,\Gr^{\on{aff}}_G}),$$
and 
$H^{-k}\left(\delta_{1,\Gr^{\on{aff}}_G}) \star J_{2\rho}\star \CW\right)$ is a successive 
extension of objects of the form
$$\CL^{\cla}\underset{\CO_{\on{pt}/\cB}}\otimes \Res^\cG_\cB(\delta_{1,\Gr^{\on{aff}}_G})$$
(in fact, each $\cla$ appears the number of times equal to $\dim(\Lambda^k(\cn^*)_\cla)$).

\medskip

Hence, we obtain that
$$H^0(\wt\Upsilon\circ p^{!*}(\CF''))\simeq \CF''\underset{G[[t]]}\star
\Res^\cG_\cB(\delta_{1,\Gr^{\on{aff}}_G})\simeq \Res^\cG_\cB(\CF'').$$

\qed

\ssec{Proof of \propref{endomorphisms}(2)}  \label{Kosz}

By adjunction,
$$\Hom\left((\iota_\Fl)^*(\delta_{1,\Fl^{\on{aff}}_G}),
(\iota_\Fl)^*(\delta_{1,\Fl^{\on{aff}}_G})\right)
\simeq
\Hom\left(\delta_{1,\Fl^{\on{aff}}_G},(\iota_\Fl)_*\circ (\iota_\Fl)^*(\delta_{1,\Fl^{\on{aff}}_G})\right).$$

\medskip

We have
$$(\iota_\Fl)_*\circ (\iota_\Fl)^*(\delta_{1,\Fl^{\on{aff}}_G})\simeq
\sF(\CO_{\on{pt}/\cB}),$$
and let us represent $\CO_{\on{pt}/\cB}\in \bD^{perf}(\Coh(\tN/\cG))$
by the Koszul complex $\on{Kosz}_{\on{pt}/\cB,\tN/\cG}^\bullet$, corresponding 
to the vector bundle $\tN/\cG\overset{\pi}\to \on{pt}/\cB$ and its
zero section $\on{pt}/\cB\overset{\iota}\to \tN/\cG$.

\medskip

The terms $\on{Kosz}_{\on{pt}/\cB,\tN/\cG}^k$ are of the form $\pi^*(\Lambda^k(\cn^*))$, where
$\Lambda^k(\cn^*)$ is naturally an object of $\Rep(\cB)$. 

\medskip

To prove the proposition, it suffices to show that
$$\Hom\left(\delta_{1,\Fl^{\on{aff}}_G},\on{Kosz}_{\on{pt}/\cB,\tN/\cG}^k[k]
\underset{\CO_{\tN/\cG}}\otimes \delta_{1,\Fl^{\on{aff}}_G}\right)=0$$
for $k\neq 0$.

\medskip

Note that $\Lambda^k(\cn^*)\in \Rep(\cB)$ has a filtration with
1-dimensional subquotients $\CL^\cla$, with $\cla\in \cLambda^{pos}$,
and $\cla=0$ is excluded if $k\neq 0$. Hence, the proposition follows
from the next lemma:

\qed

\begin{lem}
$R\Hom(\delta_{1,\Fl^{\on{aff}}_G},J_\cla)=0$ if $\cla\in \cLambda^{pos}-0$.
\end{lem}

\begin{proof}

For any $\cmu$ we have:
$$R\Hom(\delta_{1,\Fl^{\on{aff}}_G},J_\cla)\simeq R\Hom(J_\cmu,J_{\cla+\cmu}).$$

Let $\cmu\in \cLambda^+$ be such that $\cla+\cmu\in \cLambda^+$. In this case
$$J_\cmu=j_{\cmu,*} \text{ and } J_{\cla+\cmu}=j_{\cla+\cmu,*}.$$

The required vanishing follows now from the fact that 
$R\Hom(j_{\wt{w},*},j_{\wt{w}',*})\neq 0$ only when $\wt{w}'\leq \wt{w}$ as
elements in the affine Weyl group, in particular,
$|\wt{w}'|\leq |\wt{w}|$. However, the assumptions on $\cmu$ and $\cla$
imply that $|\cla+\cmu|>|\cmu|$.

\end{proof}

\ssec{}   \label{omega}

In order to prove \propref{from delta to delta}, we will need the
following construction, which will be useful also in the sequel.

\medskip

Namely, we will construct a functor

\medskip

\begin{equation}   \label{heart to hear}
\Omega:
\Heart\left(\on{pt}/\cB\underset{\on{pt}/\cG}\arrowtimes \bD^f(\fD(\Gr^{\on{aff}}_G)_\crit\mod)\right)\to
\Heart\left(\on{pt}/\cB\underset{\tN/\cG} \arrowtimes \bD^f(\fD(\Fl^{\on{aff}}_G)_\crit\mod)\right).
\end{equation}

Note, however, that by \secref{heart of base change}, the abelian
category on the LHS of \eqref{heart to hear} identifies with
$$\Rep(\cB)\underset{\Rep(\cG)}\otimes \fD(\Gr^{\on{aff}}_G)_\crit\mod.$$

\sssec{}

For
$$\CF\in \Heart\left(\on{pt}/\cB\underset{\on{pt}/\cG}\arrowtimes \bD^f(\fD(\Gr^{\on{aff}}_G)_\crit\mod)\right)$$
we set
$$\Omega(\CF):=H^0\left((\iota_\Fl)^*\circ (\wt\iota_\Fl)_*(\CF)\right).$$

We claim:
\begin{prop} \label{hands on Omega} 
$(\iota_\Fl)_*(\Omega(\CF))\simeq (\wt\iota_\Fl)_*(\CF)$.
\end{prop}

Since the functor $(\iota_\Fl)_*$ is exact, the above proposition 
immediately follows from the next one:

\begin{lem}   \label{star star delta}
For $\CF\in \Rep(\cB)\underset{\Rep(\cG)}\otimes \fD(\Gr^{\on{aff}}_G)_\crit\mod$,
the object
$$(\iota_\Fl)_*\circ (\iota_\Fl)^*\left((\wt\iota_\Fl)_*(\CF)\right)\in
\bD_{ren}(\fD(\Fl^{\on{aff}}_G)_\crit\mod)$$
is canonically a direct sum
$$\underset{k}\bigoplus\, \Lambda^k(\cn^*)
\underset{\CO_{\on{pt}/\cB}}\otimes (\wt\iota_\Fl)_*(\CF)[k].$$
\end{lem}

\sssec{Proof of \lemref{star star delta}}

We calculate $(\iota_\Fl)_*\circ (\iota_\Fl)^*\left((\wt\iota_\Fl)_*(\CF')\right)$
by means of
$$\on{Kosz}_{\on{pt}/\cB,\tN/\cG}^\bullet
\underset{\CO_{\tN/\cG}}\otimes (\wt\iota_\Fl)_*(\CF'),$$
where $\on{Kosz}_{\on{pt}/\cB,\tN/\cG}^\bullet$ is as in \secref{Kosz}.

\medskip

Note that the differential of the Koszul complex
has the property the maps 
$$\on{Kosz}_{\on{pt}/\cB,\tN/\cG}^k\to \on{Kosz}_{\on{pt}/\cB,\tN/\cG}^{k-1}$$
take value zero on the zero section $\on{pt}/\cB\overset{\iota}\hookrightarrow 
\tN/\cG$. Hence, the assertion of the next lemma follows from the next
one:

\begin{lem}
Let $\CN_1\to \CN_2$ be a map of vector bundles on $\tN/\cG$
whose value on $\on{pt}/\cB$ is zero. Then for any $\CF'\in 
\bD^f(\fD(\Fl^{\on{aff}}_G)_\crit\mod)$ of the form $(\wt\iota_\Fl)_*(\CF)$ for
$$\CF\in \on{pt}/\cB\underset{\on{pt}/\cG} \times 
\bD^f(\fD(\Gr^{\on{aff}}_G)_\crit\mod),$$ the map
$$\CN_1\underset{\CO_{\tN/\cG}}\otimes \CF'\to
\CN_2\underset{\CO_{\tN/\cG}}\otimes \CF'$$
is zero.
\end{lem}

\begin{proof}

Since the functor $(\wt\iota_\Fl)^*$ is compatible
with the action of $\bD^{perf}(\Coh(\tN/\cG))$, by adjunction,
we obtain that the following diagram commutes:
$$
\CD
\CN_1\underset{\CO_{\tN/\cG}}\otimes (\wt\iota_\Fl)_*(\CF) @>>> 
\CN_2\underset{\CO_{\tN/\cG}}\otimes (\wt\iota_\Fl)_*(\CF) \\
@V{\sim}VV   @V{\sim}VV  \\
(\wt\iota_\Fl)_*\left(\iota^*(\CN_1)\underset{\CO_{\on{pt}/\cB}}
\otimes \CF'\right) @>>>
(\wt\iota_\Fl)_*\left(\iota^*(\CN_2)\underset{\CO_{\on{pt}/\cB}}
\otimes \CF\right).
\endCD
$$

\end{proof}

\ssec{Proof of \propref{from delta to delta}}

Consider the object
$$\Omega\circ (\wt\iota_\Fl)^*(\delta_{1,\Fl^{\on{aff}}_G})\in 
\Heart\left(\on{pt}/\cB\underset{\tN/\cG} \arrowtimes \bD^f(\fD(\Fl^{\on{aff}}_G)_\crit\mod)\right).$$

We claim that there exist canonical maps
\begin{equation} \label{map i to om}
(\iota_\Fl)^*(\delta_{1,\Fl^{\on{aff}}_G})\to \Omega\circ (\wt\iota_\Fl)^*(\delta_{1,\Fl^{\on{aff}}_G})
\text{ and }
\Omega\circ (\wt\iota_\Fl)^*(\delta_{1,\Fl^{\on{aff}}_G})\to 
(\iota_\Fl)^*(\delta_{1,\Fl^{\on{aff}}_G}).
\end{equation}

The former map follows by adjunction from
$$\delta_{1,\Fl^{\on{aff}}_G}\to (\iota_\Fl)_*\circ \Omega\circ (\wt\iota_\Fl)^*(\delta_{1,\Fl^{\on{aff}}_G})
\simeq (\wt\iota_\Fl)_*\circ (\wt\iota_\Fl)^*(\delta_{1,\Fl^{\on{aff}}_G}),$$
whereas the latter map is obtained similarly by adjunction from
$$(\iota_\Fl)_!\circ \Omega\circ (\wt\iota_\Fl)^!(\delta_{1,\Fl^{\on{aff}}_G})
\simeq (\wt\iota_\Fl)_!\circ (\wt\iota_\Fl)^!(\delta_{1,\Fl^{\on{aff}}_G})\to
\delta_{1,\Fl^{\on{aff}}_G}.$$

By adjunction, the maps in \eqref{map i to om}
give rise to non-zero maps
$$(\iota_\Fl)_*\circ (\iota_\Fl)^*(\delta_{1,\Fl^{\on{aff}}_G})\to
(\iota_\Fl)_*\circ \Omega\circ (\wt\iota_\Fl)^*(\delta_{1,\Fl^{\on{aff}}_G})\simeq
(\wt\iota_\Fl)_*\circ (\wt\iota_\Fl)^*(\delta_{1,\Fl^{\on{aff}}_G})$$
and
$$(\wt\iota_\Fl)_*\circ (\wt\iota_\Fl)^*(\delta_{1,\Fl^{\on{aff}}_G})\simeq
(\iota_\Fl)_* \circ \Omega\circ (\wt\iota_\Fl)^*(\delta_{1,\Fl^{\on{aff}}_G})\to 
(\iota_\Fl)^*\circ (\iota_\Fl)^*(\delta_{1,\Fl^{\on{aff}}_G}).$$

To prove the proposition, it is enough to show that the above maps
coincide, up to non-zero scalars, with the maps \eqref{map i to it}
and \eqref{map it to i}, respectively.

\medskip

This, in turn, follows from the next lemma:

\begin{lem} The $\Hom$ spaces
$$\Hom\left((\iota_\Fl)_*\circ (\iota_\Fl)^*(\delta_{1,\Fl^{\on{aff}}_G}),
(\wt\iota_\Fl)_*\circ (\wt\iota_\Fl)^*(\delta_{1,\Fl^{\on{aff}}_G})\right)$$
and 
$$\Hom\left((\wt\iota_\Fl)_*\circ (\wt\iota_\Fl)^*(\delta_{1,\Fl^{\on{aff}}_G}),
(\iota_\Fl)_*\circ (\iota_\Fl)^*(\delta_{1,\Fl^{\on{aff}}_G})\right)$$
are 1-dimensional.
\end{lem}

\begin{proof}

We have:
\begin{align*}
\Hom\left((\iota_\Fl)_*\circ (\iota_\Fl)^*(\delta_{1,\Fl^{\on{aff}}_G}),
(\wt\iota_\Fl)_*\circ (\wt\iota_\Fl)^*(\delta_{1,\Fl^{\on{aff}}_G})\right)\simeq \\
\Hom\left((\wt\iota_\Fl)^*\circ (\iota_\Fl)_*\circ (\iota_\Fl)^*(\delta_{1,\Fl^{\on{aff}}_G}),
(\wt\iota_\Fl)^*(\delta_{1,\Fl^{\on{aff}}_G})\right)\simeq \\
\Hom\left(\Upsilon\circ (\iota_\Fl)^*\circ
(\iota_\Fl)_*\circ (\iota_\Fl)^*(\delta_{1,\Fl^{\on{aff}}_G}), \Upsilon\circ (\iota_\Fl)^*
(\delta_{1,\Fl^{\on{aff}}_G})\right)
\end{align*}

We have a canonical map
$$(\iota_\Fl)^*\circ
(\iota_\Fl)_*\circ (\iota_\Fl)^*(\delta_{1,\Fl^{\on{aff}}_G})\to (\iota_\Fl)^*(\delta_{1,\Fl^{\on{aff}}_G}),$$
whose cone is filtered by objects of the form 
$\CL^\cla\underset{\CO_{\on{pt}/\cB}}\otimes (\iota_\Fl)^*(\delta_{1,\Fl^{\on{aff}}_G})[k]$,
$k>0$. Hence, when we apply the functor $\Upsilon$ to this
cone, we obtain an object which is $<0$. 

\medskip

Hence, the above $\Hom$ is isomorphic
to $\End\left(\Upsilon\circ (\iota_\Fl)^*(\delta_{1,\Fl^{\on{aff}}_G})\right)$,
and the latter identifies with $\End(\CW)\simeq \BC$.

\medskip

The second assertion of the lemma follows by a similar manipulation
involving the pair $(\iota_\Fl)^!,(\iota_\Fl)_!$.

\end{proof}

\section{Turning $\Upsilon$ into an equivalence}  \label{turning Ups}

In this section we will show how to modify the functor
$\Upsilon$ to turn it into an equivalence.
The results of this section will not be used elsewhere in the paper.

\ssec{}

As a first step, we will modify the functor $\Upsilon$ so that 
the new functor defines an equivalence between the corresponding
$\bD^+$ categories. \footnote{Here and elsewhere, for a triangulated
category $\bD$ equipped with a t-structure, we denote by $\bD^b$, $\bD^+$
and $\bD^-$, respectively, the corresponding bounded (resp., bounded
from below, bounded from above) subcategories, see \secref{bdd subcat}.}

\sssec{}

We begin with the following observation:

\begin{prop}   \label{Ups up to junk}
Suppose that $\CF\in \on{pt}/\cB\underset{\tN/\cG} \arrowtimes 
\bD^f(\fD(\Fl^{\on{aff}}_G)_\crit\mod)$ belongs to $\bD^+$. Then for 
$i\ll 0$, the truncation $\tau^{< i}(\Upsilon(\CF))$ is an acyclic object of 
$\on{pt}/\cB\underset{\on{pt}/\cG}\arrowtimes \bD^f(\fD(\Gr^{\on{aff}}_G)_\crit\mod)$, i.e., is cohomologically $\leq -n$ for any $n\in \BN$.
\end{prop}

\propref{Ups up to junk} follows from the next more precise
estimate:

\begin{prop}   \label{Ups nearly exact}
If $\CF\in \on{pt}/\cB\underset{\tN/\cG} \arrowtimes 
\bD^f(\fD(\Fl^{\on{aff}}_G)_\crit\mod)$ is cohomologically $\geq 0$, then
$\tau^{< 0}(\Upsilon(\CF))\in \on{pt}/\cB\underset{\on{pt}/\cG}\arrowtimes 
\bD^f(\fD(\Gr^{\on{aff}}_G)_\crit\mod)$ is acyclic.
\end{prop}

\ssec{Proof of \propref{Ups nearly exact}}

We will deduce the proposition from the next assertion:

\begin{lem}   \label{i Fl conserve}
The functor 
$$(\wt\iota_\Fl  )_*: 
\on{pt}/\cB\underset{\on{pt}/\cG}\arrowtimes \bD^f(\fD(\Gr^{\on{aff}}_G)_\crit\mod)\to
\bD_{ren}(\fD(\Fl^{\on{aff}}_G)_\crit\mod)$$
is conservative when restricted to $\bD^+$.
\end{lem}

\begin{proof}
It is sufficient to show that for $\CF'$, belonging to the heart
of the t-structure on the category
$\on{pt}/\cB\underset{\on{pt}/\cG}\arrowtimes \bD^f(\fD(\Gr^{\on{aff}}_G)_\crit\mod)$,
the object $(\wt\iota_\Fl  )_*(\CF')$ is non-zero. For that, it is sufficient
to show that $(\wt\iota_\Fl  )^*\circ (\wt\iota_\Fl  )_*(\CF')\neq 0$. However,
the latter follows from \propref{comp i* i*}.
\end{proof}

Returning to the proof of the proposition, it is sufficient to show
that the composed functor $(\wt\iota_\Fl)_*\circ \Upsilon$ sends
an object $\CF$ as in the proposition to an object of 
$\bD_{ren}(\fD(\Fl^{\on{aff}}_G)_\crit\mod)$, which is acyclic in degrees $<0$.

However, by \thmref{thm zero section}, the map 
$(\iota_\Fl)_*\to (\wt\iota_\Fl)_*\circ \Upsilon$ of 
\eqref{basic nat trans star} is an isomorphism. So, our
assertion follows from the fact that the functor $(\iota_\Fl)_*$ is exact,
see \propref{affine exact}.

\qed

\ssec{}

Using \propref{Ups up to junk}, we define a new functor 
$$\Upsilon^+:
\left(\on{pt}/\cB\underset{\tN/\cG} \arrowtimes 
\bD^f(\fD(\Fl^{\on{aff}}_G)_\crit\mod)\right)^+\to
\left(
\on{pt}/\cB\underset{\on{pt}/\cG}\arrowtimes \bD^f(\fD(\Gr^{\on{aff}}_G)_\crit\mod)\right)^+$$
by 
$$\Upsilon^+(\CF):=\tau^{\geq i}(\Upsilon(\CF)) \text{ for some/any } i\ll 0.$$

We claim:

\begin{thm}  \label{Upsilon +}
The functor $\Upsilon^+$ is an exact equivalence of categories.
\end{thm}

Before beginning the proof of this theorem, let us make several
remarks:

\sssec{}

First, the fact that $\Upsilon^+$ is right-exact follows from the
right-exactness of $\Upsilon$ (\corref{Ups right-exact}). The
fact that $\Upsilon^+$ is left-exact follows from \propref{Ups nearly exact}.

\sssec{}

Secondly, we claim that the fully-faithfulness of $\Upsilon^+$
follows from that of $\Upsilon$ (\mainthmref{thm relation to Grassmannian}).

\medskip

Indeed, we need to show that for $\CF_1,\CF_2\in 
\left(\on{pt}/\cB\underset{\tN/\cG} \times \bD^f(\fD(\Fl^{\on{aff}}_G)_\crit\mod)\right)^+$,
the map
$$\Hom(\CF_1,\CF_2)\to \Hom(\Upsilon^+(\CF_1),\Upsilon^+(\CF_2))$$
is an isomorphism. 

\medskip

By construction, $\Upsilon^+(\CF_2)$ belongs to $\bD^+$, and 
$\on{Cone}(\Upsilon(\CF_1)\to \Upsilon^+(\CF_1))$ belongs to
$\bD^{< i}$ for any $i$. Hence,
$$\Hom(\Upsilon^+(\CF_1),\Upsilon^+(\CF_2))\to
\Hom(\Upsilon(\CF_1),\Upsilon^+(\CF_2))$$ is an isomorphism.

\medskip

We claim now that
$$\Hom(\Upsilon(\CF_1),\Upsilon(\CF_2))\to 
\Hom(\Upsilon(\CF_1),\Upsilon^+(\CF_2))$$
is an isomorphism for any $\CF_1\in 
\on{pt}/\cB\underset{\tN/\cG} \arrowtimes \fD^f(\Fl^{\on{aff}}_G)_\crit\mod$.
Indeed, with no restriction of generality we can take
$\CF_1$ from $\on{pt}/\cB\underset{\tN/\cG} \times \fD^f(\Fl^{\on{aff}}_G)_\crit\mod$,
and further of the form $\iota_\Fl^*(\CF'_1)$ for some $\CF'_1\in
\fD^f(\Fl^{\on{aff}}_G)_\crit\mod$. 

\medskip

Thus, by adjunction, it suffices to show that the map
$$(\wt\iota_\Fl)_*\circ \Upsilon\to (\wt\iota_\Fl)_*\circ \Upsilon^+$$
is an isomorphism, which follows from the construction.

\sssec{}

Thirdly, we claim that the functor $\Upsilon^+$ commutes with
colimits taken within $\bD^{\geq -i}$ for any $i$.

\medskip

Thus, to prove \thmref{Upsilon +} it suffices to show the functor $\Upsilon^+$ is essentially surjective onto
\begin{equation} \label{heart of base change Grass}
\Heart\left(\on{pt}/\cB\underset{\on{pt}/\cG}\arrowtimes \bD^f(\fD(\Gr^{\on{aff}}_G)_\crit\mod)\right).
\end{equation}
We will do so by exhibiting a right inverse functor on this
subcategory. In fact, we claim that the functor $\Omega$ constructed in 
\secref{omega} provides such an inverse. Namely, we claim:

\begin{prop}  \label{Ups Omega}
There exists a canonical isomorphism
$$\Upsilon^+(\Omega(\CF))\simeq \CF.$$
\end{prop}

By the above discussion, \propref{Ups Omega} implies \thmref{Upsilon +}.

\sssec{Proof of \propref{Ups Omega}}

Since the functor $\Upsilon^+$ is exact, it suffices to show that
$$H^0\left(\Upsilon^+\left((\iota_\Fl)^*\circ (\wt\iota_\Fl)_*(\CF)\right)\right)\simeq 
\CF.$$

By the definition of $\Upsilon^+$, the LHS of the latter expression is isomorphic
to 
$$H^0\left(\Upsilon\left((\iota_\Fl)^*\circ (\wt\iota_\Fl)_*(\CF)\right)\right)\simeq 
\CF.$$

Hence, our assertion follows from \propref{comp i* i*}.

\qed

\ssec{}

Let \begin{equation} \label{ren base change f}
\bD^f_{ren}\left(\on{pt}/\cB\underset{\tN/\cG} \times \fD(\Fl^{\on{aff}}_G)_\crit\mod\right)
\end{equation}
be the full subcategory of 
$\on{pt}/\cB\underset{\tN/\cG} \arrowtimes 
\bD^f(\fD(\Fl^{\on{aff}}_G)_\crit\mod)$,
consisting of objects $\CF$, for which
$(\iota_\Fl)_*(\CF)$ belongs to
$\bD^f(\fD(\Fl^{\on{aff}}_G)_\crit\mod) \subset \bD^f_{ren}(\fD(\Fl^{\on{aff}}_G)_\crit\mod)$.

\medskip

Let 
\begin{equation} \label{ren base change}
\bD_{ren}\left(\on{pt}/\cB\underset{\tN/\cG} \times \fD(\Fl^{\on{aff}}_G)_\crit\mod\right)
\end{equation}
be the corresponding renormalized triangulated category, given by 
the procedure of \secref{renorm without t},
i.e., the ind-completion of \eqref{ren base change f}. 
By \secref{renorm with t}, the category \eqref{ren base change}
acquires a t-structure.

\medskip

Let us denote by 
$$((\iota_\Fl)_*)_{ren}:\bD_{ren}\left(\on{pt}/\cB\underset{\tN/\cG} 
\times \fD(\Fl^{\on{aff}}_G)_\crit\mod\right) \to \bD_{ren}(\fD(\Fl^{\on{aff}}_G)_\crit\mod)$$
the ind-extension of the functor
$$(\iota_\Fl)_*:\bD^f_{ren}\left(\on{pt}/\cB\underset{\tN/\cG} 
\times \fD(\Fl^{\on{aff}}_G)_\crit\mod\right) \to \bD^f_{ren}(\fD(\Fl^{\on{aff}}_G)_\crit\mod).$$

By the construction of the t-structure on \eqref{ren base change}, the
exactness of the functor $(\iota_\Fl)_*$ (see \propref{affine exact})
implies the exactness of $((\iota_\Fl)_*)_{ren}$.

\sssec{}

We shall now construct a functor
$$\Upsilon_{ren}:
\bD_{ren}\left(\on{pt}/\cB\underset{\tN/\cG} \times \fD(\Fl^{\on{aff}}_G)_\crit\mod\right)\to
\on{pt}/\cB\underset{\on{pt}/\cG}\arrowtimes \bD^f(\fD(\Gr^{\on{aff}}_G)_\crit\mod).$$

\medskip

It will be defined as the ind-extension of a functor
$$\bD^f_{ren}\left(\on{pt}/\cB\underset{\tN/\cG} \times \fD(\Fl^{\on{aff}}_G)_\crit\mod\right)\to
\on{pt}/\cB\underset{\on{pt}/\cG}\arrowtimes \bD^f(\fD(\Gr^{\on{aff}}_G)_\crit\mod)$$
that we denote by the same character $\Upsilon_{ren}$. The latter functor is defined
as the restriction of $\Upsilon^+$ to the subcategory
$$\bD^f_{ren}\left(\on{pt}/\cB\underset{\tN/\cG} \times \fD(\Fl^{\on{aff}}_G)_\crit\mod\right)
\subset 
\left(\on{pt}/\cB\underset{\tN/\cG}\arrowtimes \bD^f(\fD(\Fl^{\on{aff}}_G)_\crit\mod)\right)^+.$$

\medskip

We are now ready to formulate the main result of this section:

\begin{thm}  \label{equiv flags vs grass}
The functor $\Upsilon_{ren}$ is an equivalence
of categories. It is exact with respect to the t-structures
defined on both sides. 
\end{thm}

The rest of this section is devoted to the proof of this theorem.

\medskip

\noindent{\it Remark 1.} Let us note that it is {\it a priori} not clear,
although ultimately true, that the restriction of the functor
$\Upsilon_{ren}$ to 
$$\left(\on{pt}/\cB\underset{\tN/\cG} \arrowtimes 
\bD^f(\fD(\Fl^{\on{aff}}_G)_\crit\mod)\right)^+\simeq
\bD_{ren}\left(\on{pt}/\cB\underset{\tN/\cG} \times 
\fD(\Fl^{\on{aff}}_G)_\crit\mod\right)^+$$
is isomorphic to $\Upsilon^+$. This is isomorphism is one of the finiteness
issues we will have to come to grips with in the proof that follows.

\medskip

\noindent{\it Remark 2.} As another manifestation of the fact that the 
original functor
$$\Upsilon:\on{pt}/\cB\underset{\tN/\cG}\times \bD^f(\fD(\Fl^{\on{aff}}_G)_\crit\mod)
\to \on{pt}/\cB\underset{\on{pt}/\cG}\times \bD^f(\fD(\Gr^{\on{aff}}_G)_\crit\mod)$$
is not an equivalence of categories (even after passing
to Karoubian envelopes) 
is that the object $\Omega(\Res^\cG_\cB(\delta_{1,\Gr^{\on{aff}}_G}))\in
\on{pt}/\cB\underset{\tN/\cG}\arrowtimes \bD^f(\fD(\Fl^{\on{aff}}_G)_\crit\mod)$
is not compact. 

\sssec{}
The main step in the proof of the theorem is the following:

\begin{prop}     \label{Ups finite}
The functor $\Upsilon_{ren}$ sends 
$\bD^f_{ren}\left(\on{pt}/\cB\underset{\tN/\cG} \times \fD(\Fl^{\on{aff}}_G)_\crit\mod\right)$
to the category of compact objects in 
$\on{pt}/\cB\underset{\on{pt}/\cG}\arrowtimes \bD^f(\fD(\Gr^{\on{aff}}_G)_\crit\mod)$.
\end{prop}

\sssec{}

Let us show how \propref{Ups finite} implies \thmref{equiv flags vs grass}.

\medskip

First, we claim that the functor $\Upsilon_{ren}$ is fully faithful. Indeed,
by \propref{Ups finite}, it is enough to check that the restriction of
$\Upsilon_{ren}$ to 
$\bD^f_{ren}\left(\on{pt}/\cB\underset{\tN/\cG} \times \fD(\Fl^{\on{aff}}_G)_\crit\mod\right)$
is fully faithful, and the latter assertion follows from the fully-faithfulness of
$\Upsilon^+$.

\medskip

Thus, to prove that $\Upsilon_{ren}$ is an equivalence, it is sufficient to
check that it is essentially surjective onto the generators of 
$\on{pt}/\cB\underset{\on{pt}/\cG}\arrowtimes \bD^f(\fD(\Gr^{\on{aff}}_G)_\crit\mod)$.
Up to tensoring with $\CL^\cla$, it is sufficient to check that an object of
the form $\Res^\cG_\cB(\CF)$ for $\CF$ a finitely generated D-module
on $\Gr^{\on{aff}}_G$, is in the image of $\Upsilon_{ren}$.

\medskip

Consider the object $$\Omega(\Res^\cG_\cB(\CF))\in 
\Heart\left(\on{pt}/\cB\underset{\tN/\cG} \arrowtimes \fD^f(\Fl^{\on{aff}}_G)_\crit\mod\right)
\subset
\on{pt}/\cB\underset{\tN/\cG} \arrowtimes \fD^f(\Fl^{\on{aff}}_G)_\crit\mod.$$
By \propref{hands on Omega}, $\Omega(\Res^\cG_\cB(\CF))$
belongs to 
$\bD^f_{ren}\left(\on{pt}/\cB\underset{\tN/\cG} \times \fD(\Fl^{\on{aff}}_G)_\crit\mod\right)$,
since $p^{!*}(\CF)$ is finitely generated. Therefore,
$$\Upsilon_{ren}(\Omega(\Res^\cG_\cB(\CF)))\simeq
\Upsilon^+(\Omega(\Res^\cG_\cB(\CF))),$$
and the latter identifies with $\Res^\cG_\cB(\CF)$, by \propref{Ups Omega}.

\medskip

It remains to show that $\Upsilon_{ren}$ is exact. The right-exactness follows
by construction. Hence, it is enough to show that the inverse functor
$(\Upsilon_{ren})^{-1}$ is also right-exact. For that it is enough to show 
that $(\Upsilon_{ren})^{-1}$ sends generators of 
$$\Heart\left(\on{pt}/\cB\underset{\on{pt}/\cG} \arrowtimes \fD^f(\Gr^{\on{aff}}_G)_\crit\mod\right)
\simeq \Rep(\cB)\underset{\Rep(\cG)} \otimes \fD(\Gr^{\on{aff}}_G)_\crit\mod$$
to objects that are $\leq 0$. Hence, it is sufficient to show that
$(\Upsilon_{ren})^{-1}(\Res^\cG_\cB(\CF))$ is $\leq 0$ for  
$\CF$ a finitely generated D-module on $\Gr^{\on{aff}}_G$. However, we have seen
above that $(\Upsilon_{ren})^{-1}(\Res^\cG_\cB(\CF))\simeq
\Omega(\Res^\cG_\cB(\CF))$, which is in the heart of the t-structure. 

\qed

\sssec{}

As a remark, let us now show that the functors $\Upsilon_{ren}$ and $\Upsilon^+$ are
canonically isomorphic when restricted to
$$\bD^+_{ren}\left(\on{pt}/\cB\underset{\tN/\cG} \times \fD(\Fl^{\on{aff}}_G)_\crit\mod\right)
\simeq 
\left(\on{pt}/\cB\underset{\tN/\cG} \arrowtimes \bD^f(\fD(\Fl^{\on{aff}}_G)_\crit\mod)\right)^+,$$
(where the latter identification of the categories is given by 
\propref{D plus}).

\begin{proof}

By construction, there is a natural transformation
$\Upsilon_{ren}\to \Upsilon^+$. Since both functors send $\bD^+$ to 
$\bD^+$ is suffices to check that the above map induces an isomorphism an individual
cohomologies. 

\medskip

Let $\Psi$ denote the canonical functor 
$$\bD_{ren}\left(\on{pt}/\cB\underset{\tN/\cG} \times \fD(\Fl^{\on{aff}}_G)_\crit\mod\right)\to
\on{pt}/\cB\underset{\tN/\cG} \arrowtimes \bD^f(\fD(\Fl^{\on{aff}}_G)_\crit\mod),$$
see \secref{renorm without t}.
By construction, we have a natural transformation
$\Upsilon\circ \Psi\to \Upsilon_{ren}$.

\medskip

The composed map $\Upsilon\circ \Psi\to \Upsilon_{ren}\to \Upsilon^+$,
applied to objects from $\bD^+$, is an isomorphism, as follows 
from the definition of $\Upsilon^+$. Hence, it is enough to check
that $H^i(\Upsilon\circ \Psi)\to H^i(\Upsilon_{ren})$ is an isomorphism. 

\medskip

Both functors are defined on the whole of 
$\bD_{ren}\left(\on{pt}/\cB\underset{\tN/\cG} \times \fD(\Fl^{\on{aff}}_G)_\crit\mod\right)$
and commute with direct limits. Hence, it is enough to prove the
assertion for their restrictions to 
$\bD^f_{ren}\left(\on{pt}/\cB\underset{\tN/\cG} \times \fD(\Fl^{\on{aff}}_G)_\crit\mod\right)$.
The isomorphism in the latter case again follows 
from the definition of $\Upsilon^+$. 

\end{proof}

\ssec{Proof of \propref{Ups finite}}

The proof is based on the following: 
\begin{lem}  \label{crit for comp}
There exists a finite collection of elements $\cla_i$, such that
an object $$\CF\in \on{pt}/\cB\underset{\on{pt}/\cG}\arrowtimes 
\bD^f(\fD(\Gr^{\on{aff}}_G)_\crit\mod)$$
is compact if and only if 
$$\Ind^\cG_\cB\left(\CL^{\cla_i}\underset{\CO_{\on{pt}/\cB}}\otimes\CF\right)\in
\bD_{ren}(\fD(\Gr^{\on{aff}}_G)_\crit\mod)$$
is compact for every $i$.
\end{lem}

\sssec{}

Let us show how to deduce \propref{Ups finite} from the lemma. By construction,
for $\CF\in \bD^f_{ren}\left(\on{pt}/\cB\underset{\tN/\cG} 
\times \fD(\Fl^{\on{aff}}_G)_\crit\mod\right)$, the object $\Upsilon_{ren}(\CF)$ belongs
to $\bD^b$. In addition, we claim that $\Upsilon_{ren}(\CF)$ is {\it almost
compact}, see \secref{almost compact}, where the latter notion is introduced.
Indeed, this follows from \thmref{Upsilon +} and the fact that the functor
$\Omega$ commutes with direct limits on the abelian category.

\medskip

Now, we claim that any almost compact object of 
$\left(\on{pt}/\cB\underset{\on{pt}/\cG}\arrowtimes 
\bD^f(\fD(\Gr^{\on{aff}}_G)_\crit\mod)\right)^b$ is compact.  Indeed, by \lemref{crit for comp},
it is sufficient to check that 
$\CF':=\Ind^\cG_\cB\left(\CL^{\cla}\underset{\CO_{\on{pt}/\cB}}\otimes\CF\right)$
is compact for any (or, in fact, a finite collection of) $\cla$. However,
$\CF'$ is also almost compact by adjunction. But since the category
$\fD(\Gr^{\on{aff}}_G)_\crit\mod$ is Noetherian, it is clear that every almost
compact object in $\bD^b(\fD(\Gr^{\on{aff}}_G)_\crit\mod)$ is compact.

\qed

\sssec{Proof of \lemref{crit for comp}} 
The argument given below belongs to Jacob Lurie:

\medskip

We will prove the lemma in the general context of a triangulated
category $\bD$ over the stack $\on{pt}/\cG$, where in our case we 
take $\bD:=\bD^f(\fD(\Gr^{\on{aff}}_G)_\crit\mod)$.

\medskip

Let us represent $\CO_{\Delta_{\on{pt}/\cB}}\in 
\Coh(\on{pt}/\cB\underset{\on{pt}/\cG}\times \on{pt}/\cB)\simeq
\Coh(\on{Fl}^{\cG}\times \on{Fl}^{\cG}/\cG)$ as a direct summand of a complex
as in \secref{proof of J la bdd}. 

\medskip

This implies that any object $\CF\in
\on{pt}/\cB\underset{\on{pt}/\cG}\times \bD$ (or,
$\on{pt}/\cB\underset{\on{pt}/\cG}\arrowtimes \bD$) is a direct
summand of an object, which is a successive extensions of
objects of the form
$$\CL^{\cmu_i}\underset{\CO_{\on{pt}/\cB}}\otimes
\Ind^\cG_\cB\left(\CL^{\cla_i}\underset{\CO_{\on{pt}/\cB}}\otimes\CF\right).$$

If all of the latter are compact, then so is $\CF$.

\qed

\section{Compatibility of $\Gamma_\Fl$ and $\Upsilon$}   \label{Ups and Gamma}

In this section we will prove \thmref{rel to aff gr and sections}.

\ssec{}

As was mentioned earlier, it suffices to show that the following
diagram commutes
\begin{equation} \label{diag sections for reg}
\CD
\bD^f(\fD(\Fl^{\on{aff}}_G)_\crit\mod) @>{\Gamma_{\Fl}}>> \bD^f(\hg_\crit\mod_\nilp) \\
@V{\wt\iota_\Fl^*}VV    @V{\iota^*_\hg}VV   \\
\on{pt}/\cB \underset{\on{pt}/\cG}\times \bD^f(\fD(\Gr^{\on{aff}}_G)_\crit\mod)
@>{\Gamma_{\Gr,\on{pt}/\cB}}>> 
\bD^f(\hg_\crit\mod_\reg),
\endCD
\end{equation}
as functors between categories over $\tN/\cG$.

\sssec{} 

Let $(\iota_\hg)_*$ be the evident functor
$$\bD(\hg_\crit\mod_\reg)\to \bD(\hg_\crit\mod_\nilp).$$
This is a functor between categories over $\nOp$, and in particular
over $\tN/\cG$. By \secref{restr central char 1}, it sends
$$\bD^f(\hg_\crit\mod_\reg)\to \bD^f(\hg_\crit\mod_\nilp),$$
making the latter also into a functor between categories over
$\nOp$ and $\tN/\cG$; it is in fact the right adjoint
of $(\iota_\hg)^*$.

\medskip

\begin{prop}  \label{adj and sections}
We have a commutative diagram of functors between categories
over $\on{pt}/\cB$:
$$
\CD
\bD^f(\fD(\Fl^{\on{aff}}_G)_\crit\mod)  @>{\Gamma_{\Fl}}>> 
\bD^f(\hg_\crit\mod_\nilp) \\
@A{(\wt\iota_\Fl)_*}AA     @A{(\iota_\hg)_*}AA  \\
\on{pt}/\cB \underset{\on{pt}/\cG}\times \bD^f(\fD(\Gr^{\on{aff}}_G)_\crit\mod)
@>{\Gamma_{\Gr,\on{pt}/\cB}}>> 
\bD^f(\hg_\crit\mod_\reg).
\endCD
$$
\end{prop}

\begin{proof}

By \secref{univ ppty base change}, it suffices to prove the commutativity
of the following diagram of functors between categories
over $\on{pt}/\cG$,
$$
\CD
\bD^f(\fD(\Fl^{\on{aff}}_G)_\crit\mod)  @>{\Gamma_{\Fl}}>> 
\bD^f(\hg_\crit\mod_\nilp) \\
@A{p^{!*}}AA    @A{(\iota_\hg)_*}AA   \\
\bD^f(\fD(\Gr^{\on{aff}}_G)_\crit\mod) @>{\Gamma_{\Gr}}>> 
\bD^f(\hg_\crit\mod_\reg),
\endCD
$$
which is manifest.

\end{proof}

\ssec{}

As the next step in establishing the commutativity of the diagram \eqref{diag sections for reg}, we will 
prove the following:

\begin{prop}  \label{ident gen Wak reg}
The composition of the functors 
$$(\iota_\hg)^*\circ \Gamma_{\Fl} \text{ and }
\Gamma_{\Gr,\on{pt}/\cB}\circ (\wt\iota_\Fl)^*:\bD^f(\fD(\Fl^{\on{aff}}_G)_\crit\mod) \rightrightarrows \bD^f(\hg_\crit\mod_\reg)$$
with $(\iota_\hg)_*$, yield isomorphic functors 
$$\bD^f(\fD(\Fl^{\on{aff}}_G)_\crit\mod) \rightrightarrows \bD^f(\hg_\crit\mod_\nilp)$$
at the triangulated level.
\end{prop}

\sssec{Proof of \propref{ident gen Wak reg}}

On the one hand, the composition 
$$(\iota_\hg)_*\circ (\iota_\hg)^*:\bD^f(\hg_\crit\mod_\nilp)\to \bD^f(\hg_\crit\mod_\nilp)$$
identifies with the functor 
$$\CM\mapsto \CO_{\rOp}\underset{\CO_{\nOp}}\otimes \CM \simeq \CO_{\on{pt}/\cB}\underset{\CO_{\tN/\cG}}\otimes \CM.$$
Hence, since the functor $\Gamma_\Fl$ respects the action of $\bD^{perf}(\Coh(\tN/\cG))$, we obtain that
the composition $(\iota_\hg)_*\circ (\iota_\hg)^*\circ \Gamma_{\Fl}$ is isomorphic to
$$\Gamma_{\Fl}\circ (\CO_{\on{pt}/\cB}\underset{\CO_{\tN/\cG}}\otimes-)\simeq \Gamma_\Fl\circ (\iota_\Fl)_*\circ (\iota_\Fl)^*.$$

On the other hand, by \propref{adj and sections}, the composition
$(\iota_\hg)_*\circ \Gamma_{\Gr,\on{pt}/\cB}\circ (\wt\iota_\Fl)^*$ identifies with
$$\Gamma_{\Fl}\circ (\wt\iota_\Fl)_*\circ (\wt\iota_\Fl)^*,$$
and the assertion follows from \thmref{thm zero section}.

\qed

\ssec{}

We are now ready to prove the commutativity of the diagram of functors given by \eqref{diag sections for reg}.

\sssec{}

Since the tautological functor $\bD^f(\hg_\crit\mod_\reg)\to \bD(\hg_\crit\mod_\reg)$ is fully
faithful, it is sufficient to construct an isomorphism between the corresponding functors
$$\bD^f(\fD(\Fl^{\on{aff}}_G)_\crit\mod) \rightrightarrows \bD(\hg_\crit\mod_\reg)$$
as functors between categories over $\tN/\cG$.

\medskip

Consider the two objects
$$(\iota_\hg)^*\circ \Gamma_{\Fl}(\delta_{1,\Fl^{\on{aff}}_G}) \text{ and }
\Gamma_{\Gr,\on{pt}/\cB}\circ (\wt\iota_\Fl)^*(\delta_{1,\Fl^{\on{aff}}_G})\in \bD^+(\hg_\crit\mod_\reg)^I.$$

\begin{prop} \label{ident of Wak reg}
The above objects belong to $\hg_\crit\mod_\nilp^I$ and are isomorphic.
\end{prop}

\begin{proof}

Since the functor 
$$(\iota_\hg)_*:\bD(\hg_\crit\mod_\reg)\to \bD(\hg_\crit\mod_\nilp)$$
is exact and conservative, and in the diagram of functors between
{\it abelian} categories 
$$
\CD
\hg_\crit\mod_\reg^I @>{(\iota_\hg)_*}>>  \hg_\crit\mod_\nilp^I \\
@VVV    @VVV  \\
\hg_\crit\mod_\reg @>{(\iota_\hg)_*}>>  \hg_\crit\mod_\nilp
\endCD
$$
all the arrows are fully faithful, it is sufficient to show that 
$$(\iota_\hg)_*\circ (\iota_\hg)^*\circ \Gamma_{\Fl}(\delta_{1,\Fl^{\on{aff}}_G}) \text{ and }
(\iota_\hg)_*\circ \Gamma_{\Gr,\on{pt}/\cB}\circ (\wt\iota_\Fl)^*(\delta_{1,\Fl^{\on{aff}}_G})$$
are isomorphic as objects of $\bD(\hg_\crit\mod_\nilp)$, and that the LHS belongs to
the abelian category $\hg_\crit\mod_\nilp$.

\medskip

The first assertion follows readily from \propref{ident gen Wak reg}. To prove the
second assertion we note that 
$$\Gamma_{\Fl}(\delta_{1,\Fl^{\on{aff}}_G})\simeq \BM_{\crit,-2\rho},$$ 
and since $\BM_{\crit,-2\rho}$ is $\CO_{\nOp}$-flat (see \cite{FG2}, Corollary 13.9),
$$(\iota_\hg)_*\circ(\iota_\hg)^*(\BM_{\crit,-2\rho})\simeq 
\CO_{\rOp}\underset{\CO_{\nOp}}\otimes \BM_{\crit,-2\rho}=:\BM_{\crit,-2\rho,\reg}$$
belongs to $\hg_\crit\mod_\reg$. 

\end{proof}

\noindent{\it Remark.}
Note that the above assertion, which amounts to the isomorphism
$$\BM_{\crit,-2\rho,\reg}\simeq \Gamma_{\Gr,\on{pt}/\cB}(J_{2\rho}\star \CW)),$$
coincides with that of Theorem 15.6 of \cite{FG2}, which was proven by a 
rather explicit and tedious calculation. Thus, \propref{ident of Wak reg}, can be
regarded as an alternative proof of this fact, in a way more conceptual. Note,
however, that the main ingredient in the proof of \propref{ident of Wak reg}
was \thmref{thm zero section}, which was far from tautological.

\sssec{}

As in Sections \ref{sections: model} and \ref{ups: model} the functors
\begin{multline*}
\CF\mapsto \CF\star \left((\iota_\hg)^*\circ \Gamma_{\Fl}(\delta_{1,\Fl^{\on{aff}}_G})\right) \text{ and }
\CF\mapsto \CF\star \left(\Gamma_{\Gr,\on{pt}/\cB}\circ (\wt\iota_\Fl)^*(\delta_{1,\Fl^{\on{aff}}_G})\right): \\
\bD^f(\fD(\Fl^{\on{aff}}_G)_\crit\mod) \to \bD(\hg_\crit\mod_\reg)
\end{multline*}
can be upgraded to a DG level and endowed with the structure of functors between categories
over $\tN/\cG$. Moreover, the above functors identify with the functors
$$(\iota_\hg)^*\circ \Gamma_{\Fl} \text{ and } \Gamma_{\Gr,\on{pt}/\cB}\circ (\wt\iota_\Fl)^*,$$
respectively.

\medskip

Thus, the commutativity of the diagram \eqref{diag sections for reg} follows from
\propref{ident of Wak reg}.

\section{Fully-faithfulness of $\Gamma_{\Fl,\nOp}$}   \label{proof of Gamma nilp ff}

In this section we will prove \mainthmref{Gamma nilp fully faithful}.

\ssec{}

By definition, the theorem says that the map
\begin{multline}   \label{need to prove 0}
\Hom_{\nOp\underset{\tN/\cG}\times \bD^f(\fD(\Fl^{\on{aff}}_G)_\crit\mod)}
(\CF'_1,\CF'_2)\to \\
\Hom_{\bD^f(\hg_\crit\mod_\nilp)}\left(\Gamma_{\Fl,\nOp}(\CF'_1),
\Gamma_{\Fl,\nOp}(\CF'_2)\right)
\end{multline}
is an isomorphism for any $\CF'_i\in \nOp\underset{\tN/\cG}\times \bD^f(\fD(\Fl^{\on{aff}}_G)_\crit\mod)$, $i=1,2$.

\medskip

By a slight abuse of notation we will denote by $\fr_\nilp^*$ 
the pull-back functor
$$\bD^f(\fD(\Fl^{\on{aff}}_G)_\crit\mod\to \nOp\underset{\tN/\cG}\times \bD^f(\fD(\Fl^{\on{aff}}_G)_\crit\mod).$$

\sssec{}   \label{announce red step}

We shall now perform a reduction step, showing that it is enough to establish isomorphism \eqref{need to prove 0} for $\CF'_i=\fr_\nilp^*(\CF_i)$, where
each $\CF'_i$ is a single finitely generated
$I$-equivariant D-module on $\Fl^{\on{aff}}_G$.

\sssec{Step 1}

By the definition of $\nOp\underset{\tN/\cG}\times \bD^f(\fD(\Fl^{\on{aff}}_G)_\crit\mod)$,
we can take $\CF'_i$ in \eqref{need to prove 0} to be of the form
$\CM_i\underset{\CO_{\tN/\cG}}\otimes \fr_\nilp^*(\CF_i)$ with
$\CM_i\in \bD^{perf}(\Coh(\nOp))$.

\medskip

Since $\CO_{\nOp}$ is a polynomial algebra, every $\CM_i$ as above
is a direct summand of a finite complex consisting of free 
$\CO_{\nOp}$-modules. So, we can assume that $\CF'_i=\fr_\nilp^*(\CF_i)$, with
$\CF'_i\in \bD^f(\fD(\Fl^{\on{aff}}_G)_\crit\mod)$. I.e, we are reduced to showing
that the map 

\begin{multline} \label{need to prove}
\Hom_{\nOp\underset{\tN/\cG}\times 
\bD^f(\fD(\Fl^{\on{aff}}_G)_\crit\mod)}(\fr^*_\nilp(\CF_1),
\fr^*_\nilp(\CF_2))\to \\
\to \Hom_{\bD^f(\hg_\crit\mod_\nilp)}\left(\Gamma_{\Fl,\nOp}(\fr^*_\nilp(\CF_1)),
\Gamma_{\Fl,\nOp}(\fr^*_\nilp(\CF_2))\right)
\end{multline}
is an isomorphism.

\sssec{Step 2}

Recall that $\Gamma_{\Fl,\nOp}(\fr^*_\nilp(\CF_i))\simeq
\Gamma_{\Fl}(\CF_i)$ for $i=1,2$. So, the RHS of \eqref{need to prove}
is isomorphic to
\begin{equation} \label{RHS}
\Hom_{\bD^f(\hg_\crit\mod_\nilp)}\left(\Gamma_\Fl(\CF_1),
\Gamma_\Fl(\CF_2)\right).
\end{equation}

By \secref{renorm without t}, the above expression 
is isomorphic to $\Hom$ taken in the usual category 
$\bD^+(\hg_\crit\mod_\nilp)$.

\bigskip

We rewrite the LHS of \eqref{need to prove} using 
\corref{tight mon}(2), and we obtain that it is isomorphic to 
\begin{equation} \label{LHS}
\Hom_{\bD_{ren}(\fD(\Fl^{\on{aff}}_G)_\crit\mod)}
\left(\CF_1,(\fr_\nilp)_*(\CO_{\nOp})\underset{\CO_{\tN/\cG}}\otimes \CF_2\right),
\end{equation}
where $(\fr_\nilp)_*(\CO_{\nOp})\in \bD(\QCoh(\tN/\cG))\simeq
\ua\bD^{perf}(\Coh(\tN/\cG))$.

\medskip

By \propref{F right-exact and heart}, the object 
$$(\fr_\nilp)_*(\CO_{\nOp})\underset{\CO_{\tN/\cG}}\otimes \CF_2\simeq
\CF\star \sF((\fr_\nilp)_*(\CO_{\nOp}))$$
belongs to $\bD^+$, {\it in the new t-structure}. By
\propref{finite deviation}, we obtain that it belongs
to $\bD^+$ also in the {\it old t-structure}. Therefore,
by \secref{renorm with t}, we can regard it as an object of the 
usual category $\bD^+(\fD(\Fl^{\on{aff}}_G)_\crit\mod)$, and the expression
in \eqref{LHS} can also be rewritten as $\Hom$ in 
$\bD^+(\fD(\Fl^{\on{aff}}_G)_\crit\mod)$.

\medskip

The map from \eqref{LHS} to \eqref{RHS} can thus be identified with 
the composition
\begin{multline}   \label{rewrite map}
\Hom_{\bD^+(\fD(\Fl^{\on{aff}}_G)_\crit\mod)}(\CF_1,(\fr_\nilp)_*(\CO_{\nOp})
\underset{\CO_{\tN/\cG}}\otimes \CF_2)\to \\
\to\Hom_{\bD^+(\hg_\crit\mod_\nilp)}\left(\Gamma_\Fl(\CF_1),
\Gamma_\Fl(\CO_{\nOp}\underset{\CO_{\tN/\cG}}\otimes \CF_2)\right)\simeq \\
\simeq \Hom_{\bD^+(\hg_\crit\mod_\nilp)}\left(\Gamma_\Fl(\CF_1),
\fr_\nilp^*((\fr_\nilp)_*(\CO_{\nOp}))\underset{\CO_{\nOp}}
\otimes \Gamma_\Fl(\CF_2)\right)\to \\
\to\Hom_{\bD^+(\hg_\crit\mod_\nilp)}\left(\Gamma_\Fl(\CF_1),
\Gamma_\Fl(\CF_2)\right),
\end{multline}
where the last arrow comes from the canonical map
$$\fr_\nilp^*((\fr_\nilp)_*(\CO_{\nOp}))\to \CO_{\nOp}.$$

\sssec{Step 3}

By the adjunction \cite{FG2}, Proposition 22.22, we have the isomorphisms:
\begin{multline*}
\Hom_{\bD^+(\fD(\Fl^{\on{aff}}_G)_\crit\mod)}(\CF_1,(\fr_\nilp)_*(\CO_{\nOp})
\underset{\CO_{\tN/\cG}}\otimes \CF_2)\simeq \\
\simeq
\Hom_{\bD^+(\fD(\Fl^{\on{aff}}_G)_\crit\mod)^I}\left(\delta_{1,\Fl^{\on{aff}}_G},
(\fr_\nilp)_*(\CO_{\nOp}) \otimes \CF\right),
\end{multline*}
and
$$\Hom_{\bD^+(\hg_\crit\mod_\nilp)}\left(\Gamma_\Fl(\CF_1),
\Gamma_\Fl(\CF_2)\right)\simeq 
\Hom_{\bD^+(\hg_\crit\mod_\nilp)^I}\left(\Gamma_\Fl(\delta_{1,\Fl^{\on{aff}}_G}),
\Gamma_\Fl(\CF)\right),$$
with $\CF\simeq \CF_1^*\star \CF_2\in \bD^b(\fD(\Fl^{\on{aff}}_G)_\crit\mod)^I$, where 
$\CF_1^*$ denotes the dual D-module on $I\backslash G\ppart$, see
\cite{FG2}, Sect. 22.21.

\medskip

Hence, it is sufficient to show that the map
\begin{multline*}
\Hom_{\bD^+(\fD(\Fl^{\on{aff}}_G)_\crit\mod)^I}(\delta_{1,\Fl^{\on{aff}}_G},
(\fr_\nilp)_*(\CO_{\nOp})\underset{\CO_{\tN/\cG}} \otimes \CF)\to \\
\to \Hom_{\bD^+(\hg_\crit\mod_\nilp)^I}(\Gamma_\Fl(\delta_{1,\Fl^{\on{aff}}_G}),
\Gamma_\Fl(\CF)),
\end{multline*}
given by a $I$-equivariant version of \eqref{rewrite map},
is an isomorphism for $\CF\in \bD^b(\fD(\Fl^{\on{aff}}_G)_\crit\mod)^I$.

\sssec{Step 4}

Finally, using the spectral sequence that expresses $\Hom$ in the $I$-equivariant
category in terms of the usual $\Hom$, we obtain that it is enough to
show that the map
\begin{multline}    \label{need to prove I}
\Hom_{\bD^+(\fD(\Fl^{\on{aff}}_G)_\crit\mod)}(\delta_{1,\Fl^{\on{aff}}_G},
(\fr_\nilp)_*(\CO_{\nOp}) \underset{\CO_{\tN/\cG}}\otimes \CF)\to \\
\to \Hom_{\bD^+(\hg_\crit\mod_\nilp)}(\Gamma_\Fl(\delta_{1,\Fl^{\on{aff}}_G}),
\Gamma_\Fl(\CF)),
\end{multline}
of \eqref{rewrite map} is an isomorphism for any $
\CF\in \bD^b(\fD(\Fl^{\on{aff}}_G)_\crit\mod)^I$.

\medskip

Using the fact that
$\Gamma_\Fl(\delta_{1,\Fl^{\on{aff}}_G})\simeq \BM_{-2\rho}$ is almost compact as an object
of $\bD^+(\hg_\crit\mod_\nilp)$ (see \secref{almost compact}),
we can use devissage on $\CF$ and assume that it consists of a 
single finitely generated $I$-equivariant D-module.

\medskip

Thus, the reduction announced in \secref{announce red step} has
been performed.

\ssec{}

Note that the categories $\bD^+(\fD(\Fl^{\on{aff}}_G)_\crit\mod)$ and 
$\bD^+(\hg_\crit\mod_\nilp)$,
appearing on the two sides of \eqref{need to prove I} carry
a weak $\BG_m$-action by loop rotations; moreover, the action 
on the former category is strong (a.k.a. of Harish-Chandra type),
see \cite{FG2}, Sect. 20, where these concepts are introduced.

\medskip

The category $\bD^{perf}(\on{Coh}(\tN/\cG))$ 
carries a weak $\BG_m$-action (via the action of $\BG_m$ on $\cg$
by dilations), and the map $\fr_\nilp:\nOp\to \tN/\cG$ is easily seen to
be $\BG_m$-equivariant.
In addition, the construction of the functor 
$$\sF:\bD^{perf}(\on{Coh}(\tN/\cG))\to \bD^+(\fD(\Fl^{\on{aff}}_G)_\crit\mod)$$
implies that it has a natural $\BG_m$-equvariant structure.

\medskip

With no restriction of generality we can assume that $\CF$ that appears
in \eqref{need to prove I} consists of a single twisted D-module, which is 
strongly equivariant with respect to $\BG_m$. 
By the $\BG_m$-equivariance of $\fr_\nilp$,
the object $(\fr_\nilp)_*(\CO_{\nOp})\in \bD^+(\QCoh(\tN/\cG))$ is 
$\BG_m$-equivariant; hence
$$(\fr_\nilp)_*(\CO_{\nOp})\underset{\CO_{\tN/\cG}}\otimes \CF\simeq \CF\star 
\sF((\fr_\nilp)_*(\CO_{\nOp})) \in \bD^+(\fD(\Fl^{\on{aff}}_G)_\crit\mod)$$
is weakly $\BG_m$-equivariant.

\sssec{}

Thus, both sides of \eqref{need to prove I} compute $\Hom$ between weakly
$\BG_m$-equivariant objects in categories endowed with weak $\BG_m$-actions.
Hence, the resulting $\Hom$ groups are acted on by $\BG_m$, i.e., are
graded complexes. Furthermore, it is easy to see that the map in
\eqref{need to prove I} preserves the gradings. We claim now that the 
grading on both cases is bounded from below. 

\medskip

\noindent {\it Right-hand side.} It was shown in \cite{FG2} that 
$\Hom_{\bD^+(\hg_\crit\mod_\nilp)}(\BM_{-2\rho},\CM)$
is related by a spectral sequence to 
$\Hom_{\bD^+(\hg_\crit\mod)}(\BM_{-2\rho},\CM)$
for $\CM\in \bD(\hg_\crit\mod_\nilp)$. Moreover, if $\CM$ is graded,
then the grading on the former $\Hom$ is bounded from below on the former 
if and only if it is so on the latter. We apply this to $\CM:=\Gamma_\Fl(\CF)$.

\medskip

\noindent {\it Left-hand side.} Since the action of $\BG_m$ on 
$\bD^+(\fD(\Fl^{\on{aff}}_G)_\crit\mod)$ is strong our assertion follows from
the next:
\begin{lem}   \label{O bounded}
The defect of strong equivariance on the 
object $\sF((\fr_\nilp)_*(\CO_{\nOp}))$
is bounded from below.
\end{lem}

\begin{proof}

Consider the following Cartesian square:
$$
\CD
\nOp @>>>  \tN/\cG\times \nOp \\
@VVV   @VVV  \\
\tN/\cG @>>>  \tN/\cG\times \tN/\cG.
\endCD
$$

We can realize $\CO_{\tN/\cG}$ as a $\BG_m$-equivariant coherent
sheaf on $\tN/\cG\times \tN/\cG$, up to quasi-isomorphism,
as a direct summand of a finite complex
$\CM^\bullet$, such that each $\CM^i$ is a direct sum of
sheaves of the form $\pi^*(\CL^{\cla})\boxtimes
\pi^*(\CL^{\cmu})(k)$, where $(k)$ indicates a twist of the grading.

We obtain that $(\fr_\nilp)_*(\nOp)$, as a $\BG_m$-equivariant
quasi-coherent sheaf on $\tN/\cG$, up to quasi-isomorphism,
is a direct summand of a finite complex $\CN^\bullet$,
such that each $\CN^i$ is a direct sum of sheaves of the form
$$\pi^*(\CL^\cla)\otimes \Gamma(\nOp,\CL^\cmu_{\nOp})(k).$$

Both assertions of the lemma follow, since $\sF(\CL^\cla)$ is
a (single) D-module, and the grading on each 
$\Gamma(\nOp,\CL^\cmu_{\nOp})$ is bounded from below.

\end{proof}

\sssec{}

Let us recall now the following observation, which is a Nakayama-type
lemma (see \cite{FG2},  Lemma 16.5):

\begin{lem}   \label{Nakayama}
Let $\phi:\CM_1\to \CM_2$ be a graded map of two complexes of $\CO_{\nOp}$-modules,
such that the grading on both sides is bounded from below. Assume that the induced map
$$\CM_1\overset{L}{\underset{\CO_{\nOp}}\otimes} \CO_{\rOp}\to
\CM_2\overset{L}{\underset{\CO_{\nOp}}\otimes} \CO_{\rOp}$$
is a quasi-isomorphism. Then $\phi$ is a quasi-isomorphism.
\end{lem}

\medskip

Thus, in order to establish that \eqref{need to prove I} is an isomorphism
(for $\CF$ assumed strongly equivariant with respect to $\BG_m$),
we need to know that the map in 
question induces an isomorphism after tensoring with $\CO_{\rOp}$ over
$\CO_{\nOp}$.  

\medskip

We will prove more generally that the map \eqref{need to prove} becomes
an isomorphism after tensoring with $\CO_{\rOp}$ over $\CO_{\nOp}$
for any $\CF_1,\CF_2\in \bD^f(\fD(\Fl^{\on{aff}}_G)_\crit\mod)$.

\sssec{}

Let $\CF_1,\CF_2$ as above, we have:
\begin{multline*}
\Hom_{\nOp\underset{\tN/\cG}\times \bD^f(\fD(\Fl^{\on{aff}}_G)_\crit\mod)}
(\fr_\nilp^*(\CF_1),\fr^*_\nilp(\CF_2))\overset{L}{\underset{\CO_{\nOp}}\otimes} \CO_{\rOp}\simeq \\
\simeq \Hom_{\nOp\underset{\tN/\cG}\times \bD^f(\fD(\Fl^{\on{aff}}_G)_\crit\mod)}
(\fr^*_\nilp(\CF_1),\CO_{\rOp}\underset{\CO_{\nOp}}
\otimes \fr^*_\nilp(\CF_2))\simeq \\
\simeq 
\Hom_{\nOp\underset{\tN/\cG}\times \bD^f(\fD(\Fl^{\on{aff}}_G)_\crit\mod)}
\left(\fr^*_\nilp(\CF_1),(\iota_\Op)_*\circ \iota_\Op^*(\fr^*_\nilp(\CF_2))\right)
\simeq \\
\simeq \Hom_{\rOp\underset{\nOp}\times \left(\nOp
\underset{\tN/\cG}\times \bD^f(\fD(\Fl^{\on{aff}}_G)_\crit\mod)\right)}
\left(\iota^*_\Op(\fr_\nilp^*(\CF_1)),\iota^*_\Op(\fr_\nilp^*(\CF_2))\right)\simeq\\
\simeq \Hom_{\rOp\underset{\on{pt}/\cB}\times \left(\on{pt}/\cB
\underset{\tN/\cG}\times \bD^f(\fD(\Fl^{\on{aff}}_G)_\crit\mod)\right)}
\left((\fr'_\reg)^*(\iota_\Fl^*(\CF_1)),(\fr'_\reg)^*(\iota_\Fl^*(\CF_2))\right)
\end{multline*}
and
\begin{multline*}
\Hom_{\bD^+(\hg_\crit\mod_\nilp)}\left(\Gamma_{\Fl,\nOp}(\fr^*_\nilp(\CF_1)),
\Gamma_{\Fl,\nOp}(\fr^*_\nilp(\CF_2))\right)
\overset{L}{\underset{\CO_{\nOp}}\otimes} \CO_{\rOp}\simeq \\
\Hom_{\bD^+(\hg_\crit\mod_\nilp)}\left(\Gamma_{\Fl}(\CF_1),
\Gamma_\Fl(\CF_2)\right)
\overset{L}{\underset{\CO_{\nOp}}\otimes} \CO_{\rOp}\simeq \\
\simeq \Hom_{\bD^+(\hg_\crit\mod_\nilp)}\left(\Gamma_\Fl(\CF_1),
\CO_{\rOp}\underset{\CO_{\nOp}}\otimes \Gamma_\Fl(\CF_2)\right)\simeq \\
\simeq
\Hom_{\bD^+(\hg_\crit\mod_\nilp)}\left(\Gamma_\Fl(\CF_1),
(\iota_\hg)_*\circ (\iota_\hg)^*(\Gamma_\Fl(\CF_2))\right) \simeq \\ \\
\simeq
\Hom_{\bD^+(\hg_\crit\mod_\reg)}\left((\iota_\hg)^*(\Gamma_\Fl(\CF_1)),
(\iota_\hg)^*(\Gamma_\Fl(\CF_2))\right),
\end{multline*}
where the last isomorphism follows from \propref{prop base change central character}.

\medskip

By \corref{cor rel to aff gr and sections}, we have:
$$(\iota_\hg)^*(\Gamma_\Fl(\CF_i))\simeq
\Gamma_{\Gr,\rOp}\left((\fr'_\reg)^*(\Upsilon\circ \iota_\Fl^*(\CF_1))\right),$$
and if we denote
$$\CF'_i=(\iota_\Fl)^*(\CF_i)\in \on{pt}/\cB
\underset{\tN/\cG}\times \bD^f(\fD(\Fl^{\on{aff}}_G)_\crit\mod),$$
the map \eqref{need to prove} after tensoring with $\CO_{\rOp}$
identifies with
\begin{multline*}
\Hom_{\rOp\underset{\on{pt}/\cB}\times \left(\on{pt}/\cB
\underset{\tN/\cG}\times \bD^f(\fD(\Fl^{\on{aff}}_G)_\crit\mod)\right)}
\left((\fr'_\reg)^*(\CF'_1),(\fr'_\reg)^*(\CF'_2)\right)\to \\
\to \Hom_{\rOp\underset{\on{pt}/\cB}\times \left(\on{pt}/\cB
\underset{\on{pt}/\cG}\times \bD^f(\fD(\Gr^{\on{aff}}_G)_\crit\mod)\right)}
\left((\fr'_\reg)^*(\Upsilon(\CF'_1)),(\fr'_\reg)^*(\Upsilon(\CF'_2))\right)\to \\
\to \Hom_{\bD^+(\hg_\crit\mod_\reg)}
\left(\Gamma_{\Gr,\rOp}\left((\fr'_\reg)^*(\Upsilon(\CF'_1))\right),
\Gamma_{\Gr,\rOp}\left((\fr'_\reg)^*(\Upsilon(\CF'_2))\right)\right).
\end{multline*}

The first arrow in the above composition
is an isomorphism by \mainthmref{thm relation to Grassmannian}
(combined with \corref{tight mon}(2)). The second arrow is
an isomorphism by \thmref{faithfulness on Gr}. This establishes 
the required isomorphism property of \eqref{need to prove}.

\section{Equivalence for the $I^0$-categories}  \label{proof of I eq}

In this section we will prove Main Theorems \ref{main} and \ref{Miura}.

\ssec{}

The proof of \mainthmref{main} will be carried out in the following 
general framework.

\medskip

Let $\bD_i$, $i=1,2$ be two triangulated categories,
equipped with DG models. 
We assume that the following conditions hold:

\begin{itemize} 

\item{Cat(i)}
$\bD_1$ and $\bD_2$ are co-complete.

\item{Cat(ii)}
$\bD_1$ is generated by the subcategory $\bD^c_1$ of compact objects.

\end{itemize}

\medskip

Let $\bD_i$ for $i=1,2$ be equipped with a t-structure. 
As usual, we denote
$$\bD_i^+=\underset{k}\cup\, \bD^{\geq k}_i.$$
We assume that the following conditions hold: 

\begin{itemize}

\item{Cat(a)}
The $t$-structures on both $\bD_1$ and $\bD_2$ are compatible with colimits (see \secref{t colim}).

\item{Cat(b)}
$\bD_2^+$ generates $\bD_2$.


\end{itemize}

\medskip

Let now $\sT:\bD_1\to \bD_2$ be a functor, equipped with a DG
model. Assume:

\begin{itemize}

\item{Funct(i)} $\sT$ commutes with direct sums.

\item{Funct(ii)} $\sT$ sends $\bD^c_1$ to
$\bD^c_2$.

\end{itemize}

Finally, assume:

\begin{itemize}

\item{Funct(a)} $\sT:\bD^c_1\to \bD^c_2$ is
fully faithful.

\item{Funct(b)} $\sT$ is right-exact.

\item{Funct(c)} For $X\in \bD_1^c$, we have
$\sT\left(\tau^{\geq 0}(X)\right)\in \bD_2^+$.

\item{Funct(d)} For any $Y\in \Heart(\bD_2)$ there exists $X\in 
\bD_1^{\leq 0}$ with a non-zero map $\sT(X)\to Y$.

\end{itemize}

\begin{prop}  \label{gen equiv}
Under the above circumstances, the functor $\sT$ is an
exact equivalence of categories.
\end{prop}

The proof will be given in \secref{proof gen equiv}.

\ssec{Proof of \mainthmref{main}}

We apply \propref{gen equiv} to
$$\bD_1:=\nOp\underset{\tN/\cG}\arrowtimes 
\bD^f(\fD(\Fl^{\on{aff}}_G)_\crit\mod)^{I^0},$$
$$\bD_2:=\bD_{ren}(\hg_\crit\mod_\nilp)^{I^0},$$
and $\sT:=\Gamma_{\Fl,\nOp}$. Thus, we need to verify that the conditions 
of \propref{gen equiv} hold.

\sssec{}

Conditions Cat(i) and Cat(ii) hold by definition. Conditions Cat(a) and
Cat(b) follow by \propref{prop Iw and t}(a,d) from the corresponding properties
of $$\nOp\underset{\tN/\cG}\arrowtimes \bD^f(\fD(\Fl^{\on{aff}}_G)_\crit\mod)
\text{ and } \bD_{ren}(\hg_\crit\mod_\nilp),$$
respectively.

\medskip

Conditions Funct(i) and Funct(ii) also follow from the constructions.

\sssec{}

Condition Funct(a) follows using the commutative diagram
\eqref{Iw and not} from \lemref{Iw fully faithful} and 
\mainthmref{Gamma nilp fully faithful}.

\sssec{Condition Funct(b)} 

By the definition of the t-structure on the LHS, our assertion
is equivalent to $\Gamma_\Fl$ being right-exact, when restricted to 
$\bD_{ren}(\fD(\Fl^{\on{aff}}_G)_\crit\mod)^{I^0}$ in the new t-structure. 
In other words, we need to show that for
$\CF\in \bD^f(\fD(\Fl^{\on{aff}}_G)_\crit\mod)^{I^0}$, which
is $\leq 0$ in the new t-structure, the higher cohomologies
$H^i\left(\Gamma_\Fl(\CF)\right)$ vanish. Suppose not, and let $i$ be
the maximal such $i$. 

\medskip

Any $\CF$ as above in monodromic with respect to the action of $\BG_m$
by loop rotations. Hence, all $H^i\left(\Gamma_\Fl(\CF)\right)$ acquire
a natural grading. Moreover, as in \cite{BD}, Sect. 9.1, the grading on
each $H^i\left(\Gamma_\Fl(\CF)\right)$ is bounded from below.

\medskip

By \lemref{Nakayama}, this implies
that $$H^i\left(\Gamma_\Fl(\CF)\right)\underset{\CO_\nOp}\otimes \CO_\rOp 
\neq 0.$$
The maximality assumption on $i$ implies that
$$H^i\left(\Gamma_\Fl(\CF)\overset{L}
{\underset{\CO_\nOp}\otimes} \CO_\rOp\right)
\neq 0.$$

\medskip

By \thmref{rel to aff gr and sections}, 
$$\Gamma_\Fl(\CF)\overset{L}
{\underset{\CO_\nOp}\otimes} \CO_\rOp \simeq
\Gamma_{\Gr,\rOp}\left((\fr'_\reg)^*\circ (\wt\iota_\Fl)^*(\CF)\right).$$
However, by \propref{iota ex}(a), $(\wt\iota_\Fl)^*(\CF)$ is $\leq 0$,
and the functor $\Gamma_{\Gr,\rOp}$ is right-exact by the main
theorem of \cite{FG1}; in fact, by \cite{FG4}, Theorem 1.7, 
the latter functor is exact when restricted to the $I^0$-equivariant
category. This is a contradiction.

\sssec{Condition Funct(c)}

We will show more generally that for $\CF\in \bD^f(\fD(\Fl^{\on{aff}}_G)_\crit\mod)$,
the object $\Gamma_{\Fl,\nOp}(\tau^{\geq 0}(\CF))$ belongs to
$\bD^+_{ren}(\hg_\crit\mod_\nilp)$.

\medskip

Since $\CO_{\nOp}$ is a polynomial algebra, 
we can represent $\nOp$ as an inverse limit of
affine schemes of finite type $S$
$$\nOp\overset{\psi}\to S\overset{\phi}\to \tN/\cG$$
with $\psi$ flat and $\phi$ smooth. Hence,
any $\CF$ as above is in the image of the functor
$$\psi^*:S\underset{\tN/\cG}\times \bD^f(\fD(\Fl^{\on{aff}}_G)_\crit\mod)\to
\nOp\underset{\tN/\cG}\times \bD^f(\fD(\Fl^{\on{aff}}_G)_\crit\mod),$$
for some such scheme $S$.

\medskip

The ind-completion of $S\underset{\tN/\cG}\times \bD^f(\fD(\Fl^{\on{aff}}_G)_\crit\mod)$,
denoted 
$S\underset{\tN/\cG}\arrowtimes \bD^f(\fD(\Fl^{\on{aff}}_G)_\crit\mod)$,
acquires a t-structure, and by \propref{flat exact} the functor
$$S\underset{\tN/\cG}\arrowtimes \bD^f(\fD(\Fl^{\on{aff}}_G)_\crit\mod)
\overset{\psi^*}\to
\nOp\underset{\tN/\cG}\arrowtimes \bD^f(\fD(\Fl^{\on{aff}}_G)_\crit\mod)$$
is exact. Hence, it is enough to show that the composed functor
$$S\underset{\tN/\cG}\arrowtimes \bD^f(\fD(\Fl^{\on{aff}}_G)_\crit\mod)\to
\bD_{ren}(\hg_\crit\mod_\nilp)$$
sends $\bD^+$ to $\bD^+$.

\medskip

The functor in question is isomorphic to the
composition
\begin{multline}  \label{S Gamma}
S\underset{\tN/\cG}\arrowtimes \bD^f(\fD(\Fl^{\on{aff}}_G)_\crit\mod)
\overset{\on{Id}_S\times \Gamma_\Fl}\to
S\underset{\tN/\cG}\arrowtimes \bD^f(\hg_\crit\mod_\nilp)\simeq \\
\simeq (S\underset{\tN/\cG}\times \nOp)\underset{\nOp}\arrowtimes 
\bD^f(\hg_\crit\mod_\nilp)\overset{(\psi\times \on{id}_\nOp)^*}\to
\nOp\underset{\nOp}\arrowtimes \bD^f_{ren}(\hg_\crit\mod_\nilp)\simeq \\
\simeq \bD_{ren}(\hg_\crit\mod_\nilp).
\end{multline}

We claim that both the first and the third arrows in \eqref{S Gamma}
send $\bD^+$ to $\bD^+$. 

\medskip

The first arrow in \eqref{S Gamma} fits into a commutative diagram
$$
\CD
S\underset{\tN/\cG}\arrowtimes \bD^f(\fD(\Fl^{\on{aff}}_G)_\crit\mod)
@>{\on{Id}_S\times \Gamma_\Fl}>>   
S\underset{\tN/\cG}\arrowtimes \bD^f(\hg_\crit\mod_\nilp) \\
@V{\phi_*}VV    @V{\phi_*}VV   \\
\bD_{ren}(\fD(\Fl^{\on{aff}}_G)_\crit\mod) @>{\Gamma_\Fl}>> 
\bD_{ren}(\hg_\crit\mod_\nilp),
\endCD
$$
where the vertical arrows are exact and conservative by 
\propref{affine exact}. Hence, it is enough to show that the functor 
$\Gamma_\Fl$ sends $\bD^{+_{new}}_{ren}(\fD(\Fl^{\on{aff}}_G)_\crit\mod)$ to
$\bD^+_{ren}(\hg_\crit\mod_\nilp)$. Clearly, $\Gamma_\Fl$ sends 
$\bD^{+_{old}}_{ren}(\fD(\Fl^{\on{aff}}_G)_\crit\mod)$ to
$\bD^+_{ren}(\hg_\crit\mod_\nilp)$. The required
assertion results now from \propref{finite deviation}.

\medskip

The assertion concerning the third arrow in \eqref{S Gamma}
follows from \secref{closed imm 2}, since the map
$$\nOp\overset{\psi\times \on{id}_\nOp}\to
S\underset{\tN/\cG}\times \nOp$$
is a regular immersion. 

\sssec{Condition Funct(d)}

We will show that for every 
object $\CM\in \hg_\crit\mod_\nilp^{I^0}$ there exists an object
$\CF\in \bD^f(\fD(\Fl^{\on{aff}}_G)_\crit\mod)^{I^0}$,
that belongs to the heart of the new t-structure, such that
$$\Hom(\Gamma_\Fl(\CF),\CM)\neq 0.$$

\medskip

Indeed, for $\CM$ as above, some Verma module $\BM_{w,\crit}$ maps to it
non-trivially. However, 
$$\BM_{w,\crit}\simeq \Gamma_\Fl(j_{w,!}).$$

(Indeed, $\BM_{w,\crit}:=\on{Ind}^{\hg_\crit}_{\fg[[t]]\oplus \BC}(M_w)$,
where $M_w$ is the Verma module over $\fg$ isomorphic to
$\Gamma(G/B,j_{w,!})$, and for any D-module $\CF$ on $G/B$ we have:
$\Gamma(\Fl_G,\CF)\simeq \on{Ind}^{\hg_\crit}_{\fg[[t]]\oplus \BC}
\left(\Gamma(G/B,\CF)\right)$.)

\medskip

The object $j_{w,!}$ evidently belongs to the heart of the new
t-structure since for $\cla\in \cLambda^+$,
$$j_{w,!}\star J_{-\cla}=j_{w,!}\star j_{-\cla,!}\simeq j_{w\cdot (-\cla),!},$$
since $\ell(w)+\ell(-\cla)=\ell(w\cdot (-\cla))$. (For a more general assertion
see \lemref{finite flags}.)

\bigskip

This concludes the proof of \mainthmref{main}.

\ssec{}   \label{proof of Miura}

We are now going to prove \mainthmref{Miura}.

\sssec{}

First, from \thmref{Bezr}, we obtain that there exists an equivalence:
$$\nOp\underset{\tN/\cG}\times \bD^b(\Coh(\check{\on{St}}/\cG))\simeq
\nOp\underset{\tN/\cG}\times \bD^f(\fD(\Fl^{\on{aff}}_G)_\crit\mod)^{I^0},$$
and hence an equivalence of the corresponding ind-completions
\begin{equation}  \label{1 step eq}
\nOp\underset{\tN/\cG}\arrowtimes \bD^b(\Coh(\check{\on{St}}/\cG))\simeq
\nOp\underset{\tN/\cG}\arrowtimes \bD^f(\fD(\Fl^{\on{aff}}_G)_\crit\mod)^{I^0}.
\end{equation}

Moreover, we claim that the above equivalence has a cohomological amplitude
(with respect the t-structures defined on both sides) bounded by $\dim(G/B)$.
Indeed, by the definition of the t-structure on a base-changed category,
it suffices to establish the corresponding property of the equivalence
$$\bD_{ren}(\QCoh(\check{\on{St}}/\cG))\simeq \bD_{ren}(\fD(\Fl^{\on{aff}}_G)_\crit\mod)^{I^0},$$
where $\bD_{ren}(\QCoh(\check{\on{St}}/\cG))$ is the ind-completion of 
$\bD^b(\Coh(\check{\on{St}}/\cG))$.

\medskip

I.e., we have to show that $\CM\in  \bD^{\leq 0}_{ren}(\QCoh(\check{\on{St}}/\cG))\cap
\bD^b(\Coh(\check{\on{St}}/\cG))$ goes over to an object of 
$\bD_{ren}(\fD(\Fl^{\on{aff}}_G)_\crit\mod)^{I^0}$, which is $\leq \dim(G/B)_{new}$. And vice
versa, that an object $\CF\in \bD^{\leq 0}_{ren}(\fD(\Fl^{\on{aff}}_G)_\crit\mod)^{I^0} \cap
\bD^f(\fD(\Fl^{\on{aff}}_G)_\crit\mod)^{I^0}$ goes over to an object of 
$\bD_{ren}(\QCoh(\check{\on{St}}/\cG))$, which is $\leq \dim(G/B)$.

\medskip

Both assertions follow from \propref{finite deviation}, combined with the
fact that the functor of \thmref{Bezr} is right-exact, when viewed with respect to 
the {\it old} t-structure on the category $\bD^f(\fD(\Fl^{\on{aff}}_G)_\crit\mod)^{I^0}$, and has a 
cohomological amplitude bounded by $\dim(G/B)$.

\medskip

Taking into account \mainthmref{main}, we obtain that 
in order to prove \mainthmref{Miura}, it remains to prove the following:

\begin{prop}  \label{comparison on MOp}
There exists an exact equivalence between the $\bD^+$ part of 
$$\nOp\underset{\tN/\cG}\arrowtimes \bD^b(\Coh(\check{\on{St}}/\cG))$$
and $\bD^+(\QCoh(\nMOp))$.
\end{prop}

\sssec{Proof of \propref{comparison on MOp}}   \label{perv vs coh}

By \propref{cart prod inf type}, since the morphism 
$\fr_\nilp:\nOp\to \tN/\cG$
is flat, we have an exact equivalence
$$\bD(\QCoh(\nMOp))\simeq \nOp\underset{\tN/\cG}\arrowtimes \bD^{perf}(\Coh(\check{\on{St}}/\cG)).$$

\medskip

Recall the functor $\Psi: \bD^b(\Coh(\check{\on{St}}/\cG))\to \bD(\QCoh(\check{\on{St}}/\cG))$, see
\secref{functor Xi}. To prove the proposition, it suffices, therefore, to show that the functor
$$\nOp\underset{\tN/\cG}\times \Psi:
\nOp\underset{\tN/\cG}\arrowtimes \bD^b(\Coh(\check{\on{St}}/\cG))\to
\nOp\underset{\tN/\cG}\arrowtimes \bD^{perf}(\Coh(\check{\on{St}}/\cG))$$
induces an exact equivalence of the corresponding $\bD^+$ categories.

\medskip

We have a commutative diagram of functors
\begin{equation} \label{Psi nOp}
\CD
\nOp\underset{\tN/\cG}\arrowtimes \bD^b(\Coh(\check{\on{St}}/\cG))  @>{\nOp\underset{\tN/\cG}\times \Psi}>>
\nOp\underset{\tN/\cG}\arrowtimes \bD^{perf}(\Coh(\check{\on{St}}/\cG))   \\
@V{(\fr_\nilp)_*}VV    @V{(\fr_\nilp)_*}VV \\
\bD_{ren}(\QCoh(\check{\on{St}}/\cG)) @>{\Psi}>> \bD(\QCoh(\check{\on{St}}/\cG)),
\endCD
\end{equation}
with the vertical arrows being exact and conservative. Since $\Psi$ is exact 
(see \secref{renorm with t}),
we obtain that so is $\nOp\underset{\tN/\cG}\times \Psi$. In particular, it sends $\bD^+$ to $\bD^+$.

\medskip

Let us now construct the inverse functor. Let us denote by $\Xi$ the tautological functor
$$\bD^{perf}(\Coh(\check{\on{St}}/\cG))\to \bD^b(\Coh(\check{\on{St}}/\cG)),$$
and its ind-extension
$$\Xi:\bD(\QCoh(\check{\on{St}}/\cG))\to \bD_{ren}(\QCoh(\check{\on{St}}/\cG)).$$
By \secref{functor Xi}, $\Xi$ is a left adjoint and right inverse of $\Psi$. 
Consider now the functor $\nOp\underset{\tN/\cG}\times \Xi$, which fits into the
commutative diagram
$$
\CD
\nOp\underset{\tN/\cG}\arrowtimes \bD^b(\Coh(\check{\on{St}}/\cG))  @<{\nOp\underset{\tN/\cG}\times \Xi}<<
\nOp\underset{\tN/\cG}\arrowtimes \bD^{perf}(\Coh(\check{\on{St}}/\cG))   \\
@V{(\fr_\nilp)_*}VV    @V{(\fr_\nilp)_*}VV \\
\bD_{ren}(\QCoh(\check{\on{St}}/\cG)) @<{\Xi}<< \bD(\QCoh(\check{\on{St}}/\cG)).
\endCD
$$

We define a functor $(\nOp\underset{\tN/\cG}\times \Xi)_{ren}$ from the $\bD^+$ part of
$\nOp\underset{\tN/\cG}\arrowtimes \bD^{perf}(\Coh(\check{\on{St}}/\cG))$ to the $\bD^+$ part
of $\nOp\underset{\tN/\cG}\arrowtimes \bD^b(\Coh(\check{\on{St}}/\cG))$ as follows. For
$\CM\in \nOp\underset{\tN/\cG}\arrowtimes \bD^{perf}(\Coh(\check{\on{St}}/\cG))$, which
is $\geq i$, we set 
$$(\nOp\underset{\tN/\cG}\times \Xi)_{ren}(\CM):=\tau^{\geq j}
\left((\nOp\underset{\tN/\cG}\times \Xi)(\CM)\right),$$
for some/any $j<i$. 

\medskip

The fact that $(\nOp\underset{\tN/\cG}\times \Xi)_{ren}$ and $\nOp\underset{\tN/\cG}\times \Psi$
are mutually inverse follows from the corresponding assertion for 
$\bD_{ren}(\QCoh(\check{\on{St}}/\cG)) \leftrightarrows \bD(\QCoh(\check{\on{St}}/\cG))$,
since the functor $(\fr_\nilp)_*$ is conservative.

\newpage

\centerline{\bf \large Part III: Tensor products of categories}

\bigskip



\section{DG categories and triangulated categories: a reminder}   \label{DG reminder}

In this section, whenever we will discuss the category of functors
between two categories, the source category will be assumed essentially small.

\ssec{DG categories and modules}

Recall that a DG category is a $k$-linear category $\bC$, 
enriched over $\Comp_k$. I.e., for every $X,Y\in \bC$, 
the vector space $\Hom(X,Y)$ is endowed with a structure
of complex, denoted $\Hom^\bullet(X,Y)$, in a way compatible
with compositions.
Unless specified otherwise, our DG categories are pre-triangulated.
In this case, the homotopy category $\Ho(\bC)$ is a triangulated
category.

\medskip

It is clear what a DG functor between DG categories is. A DG functor
$F:\bC_1\to \bC_2$ induces a triangulated functor $\Ho(F):
\Ho(\bC_1)\to \Ho(\bC_2)$. 
We say that $F$ is a quasi-equivalence if $\Ho(F)$ is an equivalence.

\medskip

It is also clear what a DG natural transformation between 
DG functors is.

\sssec{}   

For a DG category $\bC$ we let $\bC^{op}\mod$ (resp., $\bC\mod$)
denote the DG category of contravariant (resp., covariant) functors
$\bC\to \Comp_k$. 

\medskip

We will also consider the derived category $\bD(\bC^{op}\mod)$,
which is the triangulated quotient of $\Ho(\bC^{op}\mod)$ by the
subcategory of modules $M^\bullet$ such that $M^\bullet(X)\in \Comp_k$
is acyclic for every $X\in \bC$.

\sssec{}  \label{semi free, constr}

Inside $\bC^{op}\mod$ one singles out a full DG subcategory 
$\ua\bC$ of {\it semi-free} modules. Namely, its objects are
$\bC^{op}$-modules of the following form:

\medskip

Let $X_k, k\geq 0$ be objects of $\bC^{op}\mod$ that are free
(i.e., representable by a formal direct sum of objects of $\bC$),
and consider a strictly upper triangular \footnote{i.e., we have morphisms
$\phi_{i,j}:X_i\to X_j$ for $i>j$} matrix
$\Phi:\underset{k\geq 0}\oplus\, X_k\to \underset{k\geq 0}\oplus\, X_k$
of morphisms of degree $1$, such that
$2d(\Phi)+\Phi\circ \Phi=0$. The data 
$\{\underset{k\geq 0}\oplus\, X_k,\Phi\}$ defines a $\bC^{op}$-module by
$$X\mapsto \underset{k}\oplus\, \Hom^\bullet(X,X_k),$$
with the differential given by the sum of the original differential
and $\Phi$. 

\sssec{}

According to \cite{Dr}, Sect. 4.1, the functor
\begin{equation} \label{semi free}
\Ho(\ua\bC)\to \bD(\bC^{op}\mod)
\end{equation}
is an equivalence.

\medskip

Note also that the definition of $\ua\bC$ makes sense whether or
not $\bC$ is small, so when it is not, $\Ho(\ua\bC)$ serves
as a replacement for $\bD(\bC^{op}\mod)$ via \eqref{semi free}.

\sssec{}    \label{big and small}

We have fully faithful (Yoneda) embeddings $\bC\to \ua\bC$
and $\Ho(\bC)\to \Ho(\ua\bC)\simeq \bD(\bC^{op}\mod)$.
In addition, we have:

\begin{itemize}

\item $\Ho(\ua\bC)$ is co-complete (i.e., admits arbitrary direct
sums).

\item Every object of $\Ho(\bC)$, considered as an object of
$\Ho(\ua\bC)$, is compact (we remind that an object $X$ of a triangulated
category is compact if the functor of $\Hom_{\Ho(\bC)}(X,?)$ commutes with
direct sums),

\item $\Ho(\bC)$ generates $\Ho(\ua\bC)$ (i.e., the former is not
contained in a proper full co-complete triangulated subcategory
of the latter).

\end{itemize}

Such pairs of categories are a convenient setting to work in.

\ssec{Pseudo and homotopy functors}

Let $\bC_1$ and $\bC_2$ be two DG categories. A (left) DG pseudo functor
is by definition an object $M^\bullet_F\in \left(\bC_1^{op}\times \bC_2\mod\right)^{op}$, 
i.e., a bi-additive functor $\bC_1^{op}\times \bC_2\to \Comp_k$. 
We shall denote this category of pseudo-functors $\bC_1\to \bC_2$
by $\on{PFunct}(\bC_1,\bC_2)$.

\medskip

A  homotopy functor $F:\bC_1\to \bC_2$ is by definition a DG pseudo functor
that satisfies the following property:

\medskip

\noindent {\it For every $X\in \bC_1$, the object of $\bC_2\mod$, given
by $Y\mapsto M^\bullet_F(X,Y)$ is such that its image in $\bD(\bC_2\mod)$
lies in the essential image of $\Ho(\bC^{op}_2)$.}

\medskip

I.e., we require that for every object $X\in \Ho(\bC_1)$, the functor on $\Ho(\bC_2)$ given by 
$Y\mapsto H^0\left(M^\bullet_F(X,Y)\right)$ be co-representable.

\medskip

For example, a DG functor $F:\bC_1\to \bC_2$ gives rise
to a homotopy functor by setting $M^\bullet_F(X,Y)=\Hom^\bullet_{\bC_2}(F(X),Y)$.
For a homotopy functor $F$ we will sometimes use the notation
$$\Hom^\bullet_{\bC_2}("F(X)",Y):=M^\bullet_F(X,Y).$$

\medskip

By definition, a homotopy functor as above defines a triangulated functor
$\Ho(F):\Ho(\bC_1)\to \Ho(\bC_2)$.

\sssec{}   \label{lower star}

For a pseudo functor $F$ as above and $Y\in \bC_2^{op}\mod$ we can form
$$Y\underset{\bC_2}{\overset{L}\otimes}M^\bullet_F\in
\bD(\bC_1^{op}\mod).$$ This
gives rise to a triangulated functor $F_*:\bD(\bC_2^{op}\mod)\to
\bD(\bC_1^{op}\mod)$.

\medskip

In addition, $F$ naturally extends to a pseudo functor
$\ua\bC{}_1\to \ua\bC{}_2$. Indeed, for 
$X=\{\underset{i}\oplus\, X_i,\Phi\}\in \ua\bC{}_1$,
$Y=\{\underset{j}\oplus\, Y_j,\Psi\}\in \ua\bC{}_2$, we set
$$M^\bullet_F(X,Y)=\underset{i}\Pi\, \underset{j}\oplus\, M^\bullet_F(X_i,Y_j).$$

\medskip

If $F$ was a homotopy functor, then so is the above extension; thus we obtain a
functor  
$$\bD(\bC_1^{op}\mod)\simeq \Ho(\ua\bC{}_1)\to \Ho(\ua\bC{}_2)\to
\bD(\bC_2^{op}\mod),$$
which will be denoted $F^*$; it is the left adjoint of $F_*$.

\sssec{}  \label{pseudo and subcat}
Let $F:\bC_1\to \bC_2$ be a homotopy functor, and $\bC'_2\subset \bC_2$
a full DG subcategory. We can tautologically define a pseudo functor
$F':\bC_1\to \bC'_2$. It is easy to see that $F'$ is a homotopy functor
if and only if the essential image of $\Ho(F)$ belongs to $\Ho(\bC'_2)$.

\sssec{}    \label{huts}

Let us note that homotopy functors can be represented by "huts" of
DG functors, and vice versa. 

\medskip

Namely, given a diagram
\begin{equation} \label{right hut}
\bC_1\overset{\Psi}\leftarrow \wt\bC_1\overset{\wt{F}}\to \bC_2,
\end{equation}
where $\Psi$ and $\wt{F}$ are DG functors with $\Psi$ being a 
quasi-equivalence, we define a homotopy functor $F:\bC_1\to \bC_2$
up to a derived natural isomorphism as
$(\Psi\otimes \on{Id})^*(M^\bullet_{\wt{F}})$.

\sssec{}    \label{from q to f}

Vice versa, given a homotopy functor $F$, we define a diagram
such as \eqref{right hut} as follows. We set $\wt\bC_1$ to have as
objects triples $\{X\in \bC_1,Y\in \bC_2,f\in M^{0}_F(X,Y)\}$,
where $f$ is closed and such that it induces an isomorphism $Y\to F(X)\in
\bD(\bC_2\mod)$.

\medskip

Morphisms are defined by
\begin{multline*}
\Hom^\bullet_{\wt\bC_1}\left(\{X',Y',f'\},\{X'',Y'',f''\}\right):=\\
\{\alpha\in \Hom^\bullet_{\bC_1}(X',X''),
\beta\in \Hom^\bullet_{\bC_2}(Y',Y''),
\gamma\in M^\bullet_{F}(X',Y'')[-1]\}.
\end{multline*}
The differential arises from the differentials on $\Hom^\bullet$, $M^\bullet$
and $f$.

\medskip

The functors $\wt{F}$ and $\Psi$ send $\{X,Y,f\}$ as above
to $X$ and $Y$, respectively.

\sssec{}

Finally, let us note that if we have a diagram
\begin{equation} \label{left hut}
\bC_1\overset{\wt{G}}\to \wt\bC_2\overset{\Phi}\leftarrow \bC_2,
\end{equation}
with $\Phi$ a quasi-equivalence,
it gives rise to a diagram as in \eqref{right hut} by first
defining a homotopy functor $G:\bC_1\to \bC_2$, namely,
$$M^\bullet_{G}(X,Y)=\Hom^\bullet_{\wt\bC_2}(\wt{G}(X),\Phi(Y)),$$
and the applying the procedure of \secref{from q to f}.

\ssec{Natural transformations}

Let $F'$ and $F''$ be two pseudo functors, corresponding to bi-modules
$M^\bullet_{F'}$ and $M^\bullet_{F''}$, respectively. A DG
natural transformation $g:F'\Rightarrow F''$ is by definition a closed
morphism of degree zero $g:M^\bullet_{F''}\to M^\bullet_{F'}$ in
$(\bC_1^{op}\times \bC_2\mod)^{op}\simeq \on{PFunct}(\bC_1,\bC_2)$.

\medskip

A derived natural transformation between $F'$ and $F''$ is a morphism
between $M^\bullet_{F'}$ and $M^\bullet_{F''}$ in the triangulated category
$\on{PFunct}^{\Ho}(\bC_1,\bC_2):=
\bD(\bC_1^{op}\times \bC_2\mod)^{op}$. We shall denote the full
subcategory of $\on{PFunct}^{\Ho}(\bC_1,\bC_2)$ consisting of
homotopy functors by $\on{HFunct}(\bC_1,\bC_2)$.

\sssec{}

It is clear that a derived natural transformation between homotopy functors
gives rise to a natural transformation 
$$\Ho(F')\Rightarrow \Ho(F''):\Ho(\bC_1)\to \Ho(\bC_2),$$
and also
$$F''_*\Rightarrow F'_* \text{ and } F'{}^*\Rightarrow F''{}^*.$$

\sssec{}   \label{sub and nat}

Let now $\bC'_2\subset \bC_2$ be a full DG subcategory. Let
$\on{HFunct}(\bC_1,\bC_2)'$ be the full subcategory of 
$\on{HFunct}(\bC_1,\bC_2)$, consisting of those homotopy
functors $F$, such that $\Ho(F):\Ho(\bC_1)\to \Ho(\bC'_2)$.

\medskip

Restriction (see \secref{pseudo and subcat}) defines a functor
$$\on{HFunct}(\bC_1,\bC_2)'\to \on{HFunct}(\bC_1,\bC'_2).$$

\begin{lem}
The above functor is an equivalence. Its inverse is given 
by composing with the tautological object in 
$\on{HFunct}(\bC'_2,\bC_2)$ (see \secref{pseudo comp}).
\end{lem}

\ssec{Compositions}     \label{pseudo comp}

Let $\bC_1,\bC_2,\bC_3$ be three DG categories, and 
$F':\bC_1\to \bC_2$, $F'':\bC_2\to \bC_3$, $G:\bC_1\to \bC_3$
be pseudo functors. We define the complex $\Hom^\bullet_{\on{PFunct}(\bC_1,\bC_3)}(G,"F''\circ F'")$
to consist of maps
\begin{equation} \label{pseudo-comp map}
\Hom^\bullet_{\bC_2}("F'(X)",Y)\otimes 
\Hom^\bullet_{\bC_3}("F''(Y)",Z)\to \Hom^\bullet_{\bC_3}("G(X)",Z),
\end{equation}
functorial in $X\in \bC_1,Z\in \bC_3$, and functorial in $Y\in \bC_2$
in the sense that for any $Y',Y''\in \bC_2$, the composition
\begin{multline*}
\Hom^\bullet_{\bC_2}("F'(X)",Y')\otimes 
\Hom^\bullet_{\bC_3}("F''(Y'')",Z)\otimes \Hom_{\bC_2}^\bullet(Y',Y'') \to \\
\to  \Hom^\bullet_{\bC_2}("F'(X)",Y'')\otimes 
\Hom^\bullet_{\bC_3}("F''(Y'')",Z) \to \Hom^\bullet_{\bC_3}("G(X)",Z)
\end{multline*}
equals 
\begin{multline*}
\Hom^\bullet_{\bC_2}("F'(X)",Y')\otimes 
\Hom^\bullet_{\bC_3}("F''(Y'')",Z)\otimes \Hom_{\bC_2}^\bullet(Y',Y'') \to \\
\to \Hom^\bullet_{\bC_2}("F'(X)",Y')\otimes 
\Hom^\bullet_{\bC_3}("F''(Y')",Z) \to \Hom^\bullet_{\bC_3}("G(X)",Z).
\end{multline*}

\medskip

A DG natural transformation $$G\Rightarrow "F''\circ F'"$$ is by definition a 0-cycle
in $\Hom^\bullet_{\on{PFunct}(\bC_1,\bC_3)}(G,"F''\circ F'")$. We denote the set
of DG natural transformations as above by $\Hom_{\on{PFunct}(\bC_1,\bC_3)}(G,"F''\circ F'")$.

\medskip

If $F',F'',G$ are homotopy functors, then a DG natural transformation 
as above defines a natural transformation $$\Ho(G)\Rightarrow \Ho(F'')\circ \Ho(F').$$

We say that $G$ is a homotopy composition of $F'$ and $F''$ if the
latter map is an isomorphism of functors $\Ho(\bC_1)\rightrightarrows
\Ho(\bC_3)$.

\sssec{}

In a similar way one defines the complex
$$\Hom^\bullet_{\on{PFunct}(\bC_1,\bC_{n+1})}(G,"F_n\circ...\circ F_1"),$$
where $F_i,\,j=1,...,n$ are
pseudo functors $\bC_i\to \bC_{i+1}$ and $G$ is a pseudo functor
$\bC_1\to \bC_{n+1}$. As above, this allows to introduce the notion of
homotopy composition of $F_1,...,F_n$.

\medskip

A DG natural transformation 
$$G\Rightarrow "F_n\circ ...\circ F_1"$$
is by definition a 0-cycle in $\Hom^\bullet_{\on{PFunct}(\bC_1,\bC_{n+1})}(G,"F_n\circ...\circ F_1")$.
We denote the set of DG natural transformations by $\Hom_{\on{PFunct}(\bC_1,\bC_{n+1})}(G,"F_n\circ...\circ F_1")$.

\sssec{}  \label{comp as colim}

Let $\bC_1,...,\bC_{n+1}$ and $F_1,...,F_n$ be above. For 
a pseudo functor $G:\bC_1\to \bC_{n+1}$, let us denote by
$\Hom_{\on{PFunct}(\bC_1,\bC_{n+1})}(G,"F_n\circ ...\circ F_1")$ 
the set of DG natural transformations as above.

\medskip

We define a functor 
$$\Hom_{\on{PFunct}^{\Ho}(\bC_1,\bC_{n+1})}(G,"F_n\circ ...\circ F_1"):
\on{PFunct}^{\Ho}(\bC_1,\bC_{n+1})^{op}\to \on{Sets}$$ by
$$G\mapsto \underset{F_i\to F'_i,i=1,...,n}{colim}\, 
H^0\left(\Hom^\bullet_{\on{PFunct}(\bC_1,\bC_{n+1})}(G,"F'_n\circ ...\circ F'_1")\right),$$
where the colimit is taken over the index category 
of DG natural transformations $F_i\to F'_i$ that are 
quasi-isomorphisms.

\begin{lem}  \label{lem pseudo comp}
The functor $G\mapsto \Hom_{\on{PFunct}^{\Ho}(\bC_1,\bC_{n+1})}(G,"F_n\circ ...\circ F_1")$
is representable.
\end{lem}

\begin{proof}
Let $M^\bullet_{F_i}\in \bC_i^{op}\times \bC_{i+1}\mod$ be the bi-module 
representing $F_i$. Then it is easy to see that the object of
$\bD(\bC_1^{op}\times \bC_{n+1}\mod)$, given by
$$M^\bullet_{F_n}\underset{\bC_n}{\overset{L}\otimes}M^\bullet_{F_{n-1}}\otimes...
\otimes M^\bullet_{F_2}\underset{\bC_2}{\overset{L}\otimes}M^\bullet_{F_1}$$
satisfies the requirements of the lemma.
\end{proof}

Let us denote by $F_n\circ....\circ F_1$ the universal object in $\on{PFunct}^{\Ho}(\bC_1,\bC_{n+1})$;
we shall call it the pseudo composition of $F_1,...,F_n$.

\begin{lem} 
If $F_1,...,F_n$ are homotopy functors, their pseudo-composition
is their homotopy composition.
\end{lem}

\sssec{}

We shall view (essentially small) DG categories as a 2-category with 
objects being DG categories and 1-morphisms being the categories 
$\on{HFunct}(\bC_1,\bC_2)$. We shall
denote this 2-category by $\DGCat$.

\ssec{Karoubization}   \label{Karoubian}

Let $\bC$ be a DG category. A homotopy Karoubian envelope
of $\bC$ is a pair $(\bC',F)$, where $\bC'$ is another DG category
equipped with a homotopy functor $F:\bC\to \bC'$, such that $\Ho(F)$
is fully faithful and makes $\Ho(\bC')$ into the Karoubian envelope 
of $\Ho(\bC)$, i.e., $\Ho(\bC')$ contains images of all projectors,
and every object of $\Ho(\bC')$ is isomorphic
to a direct summand of an object of the form $\Ho(F)(X)$,
$X\in \Ho(\bC)$.

\medskip

By \cite{BeilVol}, 1.6.2, a homotopy Karoubian envelope of
$\bC$ is well-defined as an object of $\DGCat$. We shall denote
it by $\bC^{Kar}$. Here is an
explicit construction:

\begin{lem}(\cite{BeilVol}, 1.4.2)  \label{compact dir summand}
For a DG category $\bC$, any compact object in
$\Ho(\ua\bC)$ is isomorphic to a direct summand
of an object of $\Ho(\bC)$.
\end{lem}

Thus, $\bC^{Kar}$ can be defined as the preimage in
$\ua\bC$ of the subcategory $\Ho(\ua\bC)^c\subset
\Ho(\bC)$ consisting of compact objects.

\ssec{Quotients}   \label{DG quotients}

Let $\bC$ be a DG category, and $\bC'$ a full DG subcategory. Following
\cite{Dr}, Sect. 4.9, one defines an object of $\DGCat$, denoted $\bC/\bC'$,
equipped with a 1-morphism $p_{can}:\bC\to \bC/\bC'$, such that
the induced functor $\Ho(p_{can}):\Ho(\bC)\to \Ho(\bC/\bC')$ identifies
$\Ho(\bC/\bC')$ with the quotient of $\Ho(\bC)$ by the triangulated
subcategory $\Ho(\bC')$. Moreover, the pair 
$(\bC/\bC',p_{can})$ satisfies a natural
universal property of \cite{Dr}, Theorem 1.6.2, see also
\secref{univ ppty quotient} below.

\sssec{}   \label{realize quotient}

Here is a concrete construction of $\bC/\bC'$. Consider the triangulated
category $\Ho(\ua\bC')$, which is a full subcategory of $\Ho(\ua\bC)$. 
Let $\Ho(\ua\bC')^\perp\subset \Ho(\ua\bC)$ be its right orthogonal.
By \cite{Dr}, Proposition 4.7, the subcategory $\Ho(\ua\bC')\subset \Ho(\ua\bC)$
is right-admissible, i.e., the tautological functor
$$\Ho(\ua\bC')^\perp\to \Ho(\ua\bC)/\Ho(\ua\bC')$$ is an equivalence.

\medskip

Let $\ua\bC'{}^\perp\subset \ua\bC$ be the DG subcategory equal
to the preimage of $\Ho(\ua\bC')^\perp$. We let $\bC/\bC'$ be the
full subcategory of $\ua\bC'{}^\perp$ consisting of objects $Y$, such that
their image in the homotopy category 
$$\Ho(\ua\bC'{}^\perp)\simeq \Ho(\ua\bC')^\perp\simeq
\Ho(\ua\bC)/\Ho(\ua\bC')$$ has the property that it is isomorphic to
the image of an object $Y'\in \Ho(\bC)$ under 
$$\Ho(\bC)\to \Ho(\ua\bC)\to \Ho(\ua\bC)/\Ho(\ua\bC').$$

\medskip

The homotopy functor $p_{can}$ is defined tautologically: for $X\in \bC$
and $Y\in \bC/\bC'\subset \ua\bC'{}^\perp\subset \ua\bC$  we let
$$M^\bullet_{p_{can}}(X,Y)=\Hom^\bullet_{\ua\bC}(X,Y).$$

\sssec{}

A part of the universal property of $\bC/\bC'$ is the following
assertion (see \cite{Dr}, Proposition 4.7):

\begin{prop}
The functor $$(p_{can})^*:\bD(\ua\bC^{op}\mod)\simeq
\Ho(\ua\bC)\to \Ho(\ua{\bC/\bC'})\simeq \bD((\bC/\bC')^{op}\mod)$$
induces an equivalence
$$\Ho(\ua\bC)/\Ho(\ua\bC')\simeq \Ho(\ua{\bC/\bC'})$$
and the functor 
$(p_{can})_*:\Ho(\ua{\bC/\bC'})\simeq \bD((\bC/\bC')^{op}\mod)\to
\bD(\ua\bC^{op}\mod)\simeq \Ho(\ua\bC)$
induces an equivalence
$$\Ho(\ua{\bC/\bC'})\simeq \Ho(\ua\bC')^\perp.$$
\end{prop}

\sssec{}  \label{weak to quotient}

The following construction will be useful in the sequel. Let $\bC_1$ and
$\bC_2$ be DG categories, and $F:\bC_1\to \bC_2$ a DG pseudo functor.
Let $\bC'_1\subset \bC_1$ and $\bC'_2\subset \bC_2$ be DG 
subcategories. Assume that the following holds:

For any $X\in \bC_1$ the functor on
$\Ho(\bC_2)/\Ho(\bC'_2)$ given by
$$Y \mapsto \underset{f:Y\to Y'}{colim}\, 
H^0\left(M_F(X,Y')\right),$$
is co-representable, where the colimit is taken over the set
of morphisms $f$ with $\on{Cone}(f)\in \bC'_2$. Assume
also that the above functor is zero for $X\in \bC'_1$.

\medskip

We claim that the above data gives rise to a 1-morphism
$F':\bC_1/\bC'_1\to \bC_2/\bC'_2$. Namely, let us realize
the above categories as in \secref{realize quotient}
as full subcategories of $\ua{\bC'_i}^\perp$, $i=1,2$,
respectively. We define the sought-for quasi-functor by setting
$$M_{F'}(X,Y):=M_F(X,Y)$$
for $X\in \bC_1/\bC'_1\subset \ua{\bC'_1}^\perp,
Y\in \bC_2/\bC'_2\subset \ua{\bC'_2}^\perp$.

The required co-representability on the homotopy level
follows from the assumptions.

\ssec{DG models of triangulated categories}

Let $\TrCat$ be the 2-category of triangulated categories. We have
an evident 2-functor $\Ho:\DGCat\to \TrCat$ that sends a DG category
$\bC$ to $\Ho(\bC)$.

\medskip

Given a triangulated category $\bD$, its DG model is an object 
of the 2-category $\DGCat$ equal to the fiber of $\Ho$ over $\bD$. Similarly,
given an arrow in $\TrCat$ (i.e., a triangulated functor 
$F_{tr}:\bD_1\to \bD_2$) by a model of $F_{tr}$ we shall mean
a lift of this functor to $\DGCat$ (i.e., if $\bD_i=\Ho(\bC_i)$, then
a model for $F_{tr}$ is a homotopy functor $F_h:\bC_1\to \bC_2$).

\sssec{}   \label{big and small, triang}

Let $\bD$ be a triangulated category, equipped with a DG model.
In this case, we can form a new triangulated category $\ua\bD$
and a fully faithful triangulated functor $\bD\to \ua\bD$, both
equipped with models, such that the pair $(\bD,\ua\bD)$
satisfies the three properties of \secref{big and small}.
Namely, if $\bD=\Ho(\bC)$, we set $\ua\bD:=\Ho(\ua\bC)$. 

We will informally call $\ua\bD$ the "ind-completion" of $\bD$.

\medskip

If $F:\bD_1\to \bD_2$ are triangulated categories and a functor
between them, all equipped with DG models, we have the corresponding
functors
$$F^*,F_*:\ua\bD{}_1\leftrightarrows \ua\bD{}_2,$$
also equipped with models.

\medskip

Similarly, the Karoubian envelope $\bD^{Kar}$ of $\bD$ and the
functor $\bD\to  \bD^{Kar}$ are both equipped with models.

\sssec{}  \label{rigid quotients}

Let $\bD$ be a triangulated category equipped with a model,
and let $\bD'\subset \bD$ be a full triangulated subcategory.
Note that $\bD'$ is also naturally equipped with a model. 

Indeed, if $\bD=\Ho(\bC)$, we define $\bC'\subset \bC$
to be the full subcategory consisting of objects, whose
image in $\bD$ is isomorphic to an object from $\bD'$.

\medskip

In this case, by \secref{DG quotients}, the triangulated category
$\bD/\bD'$ and the projection functor $\bD\to \bD/\bD'$ also
come equipped with models.

\sssec{}  \label{univ ppty quotient}

Let $\bD$ and $\bD'$ be as above, and let $\bD_1$ be yet
another triangulated category equipped with a model. The following
is a version of \cite{Dr}, Theorem 1.6.2:

\begin{lem}
For $\bD'\subset \bD$ and $\bD_1$ in $\DGCat$
the following two categories are equivalent:

\medskip

\noindent(a)
1-morphisms $F:\bD\to \bD_1$ in $\DGCat$, such that $F|_{\bD'}=0$.

\medskip

\noindent(b) 1-morphisms $\bD/\bD'\to \bD_1$ in $\DGCat$.

\end{lem}

\ssec{Homotopy colimits}   \label{hocolim}

Let $\bD$ be a co-complete triangulated category equipped with a model.
In this subsection we will review the notion of homotopy colimit, following
a recipe, explained to us by J.~Lurie. 

\sssec{}   \label{simple colim}

Let us first consider the simplest case of a sequence of
objects $X_i\overset{f_{i,i+1}}\to X_{i+1}, i=1,2,...$.
In this case we define $hocolim(\{X_i\})$ as the cone of the 
map $\underset{i\geq 1}\oplus \, X_i\to \underset{i\geq 1}\oplus \, X_i$,
where the map is 
$$X_i\overset{\on{id}_{X_i}\oplus -f_{i,i+1}}\longrightarrow X_i\oplus X_{i+1}.$$

In this definition, $hocolim(\{X_i\})$ is defined up to a non-canonical
isomorphism, and one does not even need a DG model.

\medskip

If the DG model $\bC$ of $\bD$ was itself co-complete (which we can
assume, up to replacing a given model by a quasi-equivalent one),
and if we lift the morphisms $f_{i,i+1}$ to closed morphisms of degree $0$ 
in the DG model, the above construction becomes canonical. 

\sssec{}  \label{map from ua}
 
The above homotopy colimit construction implies also the following.
Let $\{\underset{k\geq 0}\oplus\, X_k,\Phi\}$ be an object of
$\ua\bC$. Consider the $\bC$-module given by
$$X\mapsto \underset{k}\Pi\, \Hom^\bullet(X_k,X),$$
with the differential given by $\Phi$.

\begin{lem}
Assume that $\Ho(\bC)$ is co-complete. Then the image of the
above $\bC$-module in $\bD(\bC\mod)$ is in the Yoneda
image of $\Ho(\bC^{op})$. 
\end{lem}

Thus, for $\bC$ with $\Ho(\bC)$ co-complete, we obtain that the identity
functor on $\Ho(\bC)$ naturally extends to a functor 
$$\bD(\bC^{op}\mod)\simeq \Ho(\ua\bC)\to \Ho(\bC)$$
that commutes with direct sums, and which is the left adjoint
to the tautological embedding.

\sssec{}

In the general case we proceed as follows. Let $I$ be a small category.
Let $I^{DG}$ be the free (non-pretriangulated) DG category, spanned
by $I$. Consider the DG category 
$$\on{POb}(I,\bC):=\on{PFunct}(I^{op}{}^{DG},\bC^{op})^{op}:=
\bD((I \times \bC^{op})\mod),$$ 
and the corresponding triangulated categories
$$\on{HOb}(I,\bC):=\on{HFunct}(I^{op}{}^{DG},\bC^{op})^{op}\subset \on{PFunct}^{\Ho}(I^{op}{}^{DG},\bC^{op})^{op}.$$

\medskip

By definition, a homotopy $I$-object of a DG category $\bC$ is
an object $X_I\in \on{HOb}(I,\bC)$. For $X_I$ as above and $i\in I$
we will denote by $X_i$ be the corresponding object of $\Ho(\bC)$.

Being a full triangulated subcategory of $\bD((I\times \bC^{op})\mod)$,
the category $\on{HOb}(I,\bC)$ acquires a natural DG model, by
\secref{rigid quotients}.

\sssec{}

For a functor $F:I_1\to I_2$ we have a pair of adjoint functors
$$(F\times \on{Id})^*:\on{POb}(I_1,\bC)\leftrightarrows
\on{POb}(I_2,\bC):(F\times \on{Id})_*,$$
and it is easy to see that $(F\times \on{Id})_*$ sends 
$\on{HOb}(I_2,\bC)$ to $\on{HOb}(I_1,\bC)$.

\medskip

Applying this to $I_1=I$ and $I_2=\on{pt}$ we recover the tautological
functor $\bC\to \on{HOb}(I,\bC)$, and its left adjoint with values
in $\Ho(\ua\bC)\simeq \on{POb}(\on{pt},\bC)$.
We denote the latter functor
$$\on{HOb}(I,\bC)\to \Ho(\ua\bC)$$
by $\underset{I}{"hocolim"}$.

\medskip

Assume now that $\Ho(\bC)$ is co-complete. We define the functor
$$\underset{I}{hocolim}:\on{HOb}(I,\bC)\to \Ho(\bC)$$
as the composition of $\underset{I}{"hocolim"}$ and the functor 
$\Ho(\ua\bC)\to \Ho(\bC)$ of \secref{map from ua}.

\medskip

By construction, $hocolim$ is the left adjoint to the above
functor $\bC\to \on{HOb}(I,\bC)$.

\sssec{}

Here are some basic properties of the homotopy colimit
construction.

\medskip

Let $\Psi:\bC_1\to \bC_2$ be a 1-morphism in $\DGCat$, such
that $\Ho(\bC_1)$ and $\Ho(\bC_2)$ are both co-complete.

\begin{lem}  \label{functor and limit}
Assume that 
$\Ho(\Psi):\Ho(\bC_1)\to \Ho(\bC_2)$ 
commutes with direct sums. Then for any $I$ the diagram
of functors
$$
\CD
\on{HOb}(I,\bC_1)   @>{\Psi\circ ?}>>
\on{HOb}(I,\bC_2)   \\
@V{hocolim}VV     @V{hocolim}VV   \\
\Ho(\bC_1) @>{\Psi}>> \Ho(\bC_2)
\endCD
$$
commutes.
\end{lem}

\medskip

In what follows, when talking about homotopy colimits,
we will always assume that $I$ is filtered. 

\begin{lem}   \label{compact and limit}  
Let $Y$ be an object of $\bC$, such that the corresponding
object of $\bD$ is compact. Then
$$\Hom_{\bD}(Y,hocolim(X_I))\simeq
\underset{i\in I}{colim}\, \Hom_{\bD}(Y,X_i).$$
\end{lem}

\sssec{}

The following assertion will be useful in the sequel. 

\begin{lem}  \label{generate colimit}
Let $\bC$ be a DG category. Assume that
$\bD:=\Ho(\bC)$ is co-complete, and let
$\bD'\subset \bD$ be a triangulated subcategory
that generates it. Then every object $X\in \bD$
can be represented as a homotopy colimit
of $X_I\in \on{HOb}(I,\bC)$ for some $I$, where 
$X_i\in \bD'$ for every $i\in I$.
\end{lem}

\section{Homotopy monoidal categories and actions}  

\label{homotopy monoidal categories and actions}

\ssec{}

Let $\bA$ be a DG category. A DG pseudo monoidal structure on $\bA$ 
is a collection of DG functors $(\bA^{\times I})^{op}\times \bA\to {\mathbf {Comp}}_k$
$$X_I,Y\mapsto \Hom^\bullet("X^\otimes_I",Y),$$
for an ordered finite set $I$. Here $X_I$ stands for an $I$-object of
$\bA$, and the symbol $"X^\otimes_I"$ stands for the a priori non-existing 
tensor product $\underset{i\in I}\otimes\, X_i$.
The above functors must be endowed with the appropriate natural
transformations, see \cite{CHA}, Sect. 1.1.1. For $I=\{1\}$ we must
be given an identification $\Hom^\bullet(X^\otimes_I,Y)\simeq \Hom^\bullet(X,Y)$.
We require $\bA$ to be homotopy unital. I.e., there
should exist an object ${\bf 1}_\bA\in \bA$ and functorial quasi-morphisms
$$\Hom^\bullet("X^\otimes_I",Y)
\to \Hom^\bullet("X^\otimes_{I\cup \on{pt}}",Y).$$
where $X_{I\cup \on{pt}}$ corresponds to the insertion of 
the unit in $\bA$ in any place in $I$ with respect to its order.

\medskip

We say that a DG pseudo monoidal structure is a homotopy
monoidal structure if the induced pseudo monoidal structure
on $\Ho(\bA)$ given by
$$H^0\left(\Hom^\bullet("X^\otimes_I",Y)\right)$$
is a monoidal.

\medskip

Evidently, a usual DG monoidal structure on $\bA$ gives rise
to a homotopy one. A homotopy monoidal structure on $\bA$
defines a structure of monoidal triangulated category on $\Ho(\bA)$.

\ssec{Functors}

Let $\bA_1$ and $\bA_2$ be two DG pseudo monoidal
categories. A DG pseudo monoidal functor 
$F$ between them is the following data:

\medskip

For a finite ordered set $I$ we must be given a complex
\begin{equation} \label{mon data}
\Hom^\bullet_{\bA_1,\bA_2}("F(X^\otimes_I)",Y)
\end{equation}
that depends functorially on both arguments. 

\medskip

The above functors must be equipped with the following system
of natural transformations. Let $I\twoheadrightarrow J\twoheadrightarrow K$
be surjections of finite ordered sets. For $k\in K$ (resp., $j\in J$) let $J^k\subset J$ 
(resp., $I^j\subset I$) denote its pre-image. Fix objects 
$X_I\in \bA_1^I,X_J\in \bA_1^J,Y_K\in \bA_2^K,Y\in \bA_2$.

We must be given a map
\begin{multline}   \label{nat trans hom mon}
\left(\underset{j\in J}\otimes
\Hom^\bullet_{\bA_1}("X^\otimes_{I^j}",X_j)\right)
\bigotimes
\left(\underset{k\in K}\otimes \Hom^\bullet_{\bA_1,\bA_2}
("F(X^\otimes_{J^k})",Y_k)\right)\bigotimes \\
\bigotimes \Hom^\bullet_{\bA_2}("Y^\otimes_K",Y)
\to \Hom^\bullet_{\bA_1,\bA_2}
("F(X^\otimes_{I})",Y).
\end{multline}
In addition, we must be given natural quasi-isomorphisms
that correspond to insertions of the unit object. These natural 
transformations must satisfy the natural axioms that we
will not spell out here.

\sssec{}

Assume now that on both $\bA_1$ and $\bA_2$ the DG 
pseudo monoidal structure is homotopy monoidal. We say that
$F$ is a homotopy monoidal functor if the 
functors
\begin{multline*}
\{X_{I}\in \bA_1^{I},Y\in \bA_2\}\mapsto
H^0\left(\Hom_{\bA_1,\bA_2}^\bullet("F(X^\otimes_{I})",Y)\right):
\Ho(\bA^{op}_1)^{I}\times \Ho(\bA_2)\to
\Vect_k
\end{multline*}
and the maps
\begin{multline*}   
H^0\left(\underset{j\in J}\otimes\Hom^\bullet_{\bA_1}("X^\otimes_{I^j}",X_j)\right)
\bigotimes
H^0\left(\underset{k\in K}\otimes \Hom^\bullet_{\bA_1,\bA_2}
("F(X^\otimes_{J^k})",Y_k)\right)\bigotimes \\
\bigotimes
H^0\left(\Hom^\bullet_{\bA_2}("Y^\otimes_K",Y)\right) \to
H^0\left(\Hom^\bullet_{\bA_1,\bA_2}
("F(X^\otimes_{I})",Y)\right).
\end{multline*}
come from a (automatically uniquely determined)
monoidal structure on the functor $\Ho(F):\Ho(\bA_1)\to \Ho(\bA_2)$.

\medskip

In the homotopy monoidal case we say that $F$ is a monoidal quasi-equivalence if 
the functor $\Ho(F)$ is an equivalence of categories.

\sssec{}

Let $F'$ and $F'$ be two
DG pseudo monoidal functors $\bA_1\to \bA_2$. A DG natural
transformation $F'\Rightarrow F''$ is a data of a functorial map of complexes
\begin{equation} \label{hom nat trans}
\phi:\Hom^\bullet_{\bA_1,\bA_2}("F''(X^\otimes_{I})",Y)\to
\Hom^\bullet_{\bA_1,\bA_2}("F'(X^\otimes_{I})",Y),
\end{equation}
defined for all finite ordered sets $I$,
(preserving the degree and commuting with the differential),
which makes the diagrams coming from \eqref{nat trans hom mon} commute.

\medskip

We say that $\phi$ is a quasi-isomorphism if maps \eqref{hom nat trans}
are quasi-isomorphisms for all $I$. In the case when
$\bA_1,\bA_2$ and both functors are homotopy monoidal, the quasi-isomorphism
condition is enough to check for $I=\{1\}$. 

\sssec{}

For two DG pseudo monoidal categories $\bA_1,\bA_2$ 
we denote the category whose objects are DG pseudo monoidal
functors $\bA_1\to \bA_2$ and arrows DG natural transformations
by $\on{PMon}(\bA_1,\bA_2)$.

\medskip

We claim that $\on{PMon}(\bA_1,\bA_2)$ has a structure of closed model category, 
with weak equivalences being quasi-isomorphisms, and cofibrations being 
those natural transformations, for which the maps 
$\phi$ that are surjective for all $I$ (in particular, all objects are cofibrant).
Let us denote the corresponding homotopy category by
$\on{PMon}^{\Ho}(\bA_1,\bA_2)$. We shall refer to maps between
objects of $\on{PMon}^{\Ho}(\bA_1,\bA_2)$ as homotopy natural
transformations.

\medskip

\noindent{\it Remark.} The category $\on{PMon}(\bA_1,\bA_2)^{op}$ is akin
to that of DG associative algebras. Indeed, an object $F\in 
\on{PMon}(\bA_1,\bA_2)^{op}$ is given by specifying a collection
of vector spaces \eqref{mon data} and multiplication maps
\eqref{nat trans hom mon}.

\sssec{}   \label{free functor}

Let us construct a supply of fibrant objects in $\on{PMon}(\bA_1,\bA_2)^{op}$
(by the above analogy, these play the role of DG associative algebras that are
free as plain associative algebras).

\medskip

By a {\it graded} (vs. DG) pseudo monoidal functor $F:\bA_1\to \bA_2$
we will understand the same data as in \eqref{mon data} and 
\eqref{nat trans hom mon}, with the difference that the 
$\Hom^\bullet_{\bA_1,\bA_2}("F(X^\otimes_{I})",Y)$'s are just graded
vector spaces, with no differential. Morphisms in the category are defined
as in \eqref{hom nat trans}. Let us denote the corresponding 
category by $\on{PMon_{gr}}(\bA_1,\bA_2)$.

\medskip

For $n> 0$ consider also the category 
$\on{PFunct_{gr}}(\bA^{\times n}_1,\bA_2)$ being the
opposite of that of multi-additive functors 
$\bA_1^{op}\times...\times \bA_1^{op}\times \bA_2\to \Vect_k^{\BZ}$
I.e., for $M^\bullet\in \on{PFunct_{gr}}(\bA^{\times n}_1,\bA_2)$, 
$X_1,...,X_n\in \bA_1,Y\in \bA_2$,
each $M^\bullet(X_1,...X_n,Y)$ is just a graded vector space, without a differential.

\medskip

The evident forgetful functor $\on{PMon_{gr}}(\bA_1,\bA_2)\to
\underset{n>0}\Pi\, \on{PFunct_{gr}}(\bA^{\times n}_1,\bA_2)$ has a
right adjoint; we denote it
$$M^\bullet\mapsto \on{Free}(M^\bullet).$$

\medskip

Suppose an object $F\in \on{PMon}(\bA_1,\bA_2)$ has the following
properties:

\medskip

\begin{itemize}

\item The image of $F$ in $\on{PMon_{gr}}(\bA_1,\bA_2)$ is isomorphic to
$\on{Free}(M^\bullet)$ for some $$M^\bullet\in \underset{n>0}\Pi\, \on{PFunct_{gr}}(\bA^{\times n}_1,\bA_2).$$

\item The above object $M^\bullet$ can be represented as a direct sum
$M^\bullet=\underset{i\geq 0}\oplus M^\bullet_i$, where each $M^\bullet_i$
is in turn a direct sum of representable functors 
$$\underset{a}\oplus\,\Hom_{\bA^{\times n_a}_1\times \bA^{op}_2}(?,
\left(X_1^a,...,X_{n_a}^a,Y^a))\right.$$

\item The discrepancy between the natural differential on every $M^\bullet_i$
and one coming by restriction from $M^\bullet_i\subset M^\bullet\subset F$,
is as a map 
$$M^\bullet_i\to \on{Free}\left(\underset{i>j\geq 0}\oplus M^\bullet_j\right)\in
\underset{n>0}\Pi\, \on{PFunct_{gr}}(\bA^{\times n}_1,\bA_2)^{op}.$$

\end{itemize}

Then such $F$ is fibrant.

\sssec{}

When both $\bA_1$ and $\bA_2$ are homotopy monoidal
we can consider the full subcategory category of
$\on{PMon}^{\Ho}(\bA_1,\bA_2)$ whose
objects are homotopy monoidal functors. We shall denote
this category by $\on{HMon}(\bA_1,\bA_2)$. 

\ssec{}

Let $\bA_1$ and $\bA_2$ be DG pseudo-monoidal categories. There is a natural
notion of lax DG monoidal functor between them. By definition, this is a DG
functor $F_{DG}:\bA_1\to \bA_2$ endowed with a system of morphisms
\begin{equation} \label{laxity}
\Hom^\bullet_{\bA_1}("X^\otimes_{I}",Y)\to
\Hom^\bullet_{\bA_2}("F_{DG}(X^\otimes_{I})",F_{DG}(Y)),
\end{equation}
compatible with the associativity constraints. 

\medskip

In other words, this is a pseudo-monoidal functor $F$, for which
there exists a DG functor $F_{DG}:\bA_1\to \bA_2$ and 
isomorphisms
$$\Hom^\bullet_{\bA_2}("F_{DG}(X^\otimes_{I})",Y)
\simeq \Hom^\bullet_{\bA_1,\bA_2}("F(X^\otimes_{I})",Y).$$

\medskip

We say that a lax DG monoidal functor is a lax DG monoidal quasi-equivalence,
if $F_{DG}$ is a quasi-equivalence at the level of plain categories and
the maps \eqref{laxity} are quasi-isomorphisms.

\sssec{}

Let $G:\bA_2\to \bA_3$ be a pseudo monoidal functor. For $F_{DG}:\bC_1\to \bC_2$, which
is a lax DG monoidal functor, one can define the composition $G\circ F_{DG}$ as the pseudo 
monoidal functor $\bC_1\to \bC_3$ given by 
$$\Hom^\bullet_{\bA_1,\bA_3}("G\circ F_{DG}(X^\otimes_I)",Z):=
\Hom^\bullet_{\bA_2,\bA_3}("G(F_{DG}(X^\otimes_{I}))",Z).$$

\medskip

This operation defines the functor $?\circ F$
$$\on{PMon}(\bA_2,\bA_3)\to \on{PMon}(\bA_1,\bA_3),$$
and if the categories $\bA_i$ and the functor $F_{DG}$ are homotopy monoidal, 
we also obtain a functor
$$\on{HMon}(\bA_2,\bA_3)\to \on{HMon}(\bA_1,\bA_3).$$

\begin{lem}   \label{mon quasi equiv}
Assume that $F_{DG}$ is a lax DG monoidal quasi-equivalence.
Then the induced functor
$$G\mapsto G\circ F_{DG}:\on{PMon}^{\Ho}(\bA_2,\bA_3)\to
\on{PMon}^{\Ho}(\bA_1,\bA_3)$$
is an equivalence.
\end{lem}

\sssec{}   \label{mon huts}

We claim that as in \secref{huts}, for any homotopy monoidal functor
$F:\bA_1\to \bA_2$ one can find a "hut"
$$\bA_1\overset{\Psi}\leftarrow \wt\bA_1 \overset{\wt{F}}\to \bA_2,$$
where $\Psi$ is lax DG monoidal and is a quasi-equivalence, $\wt{F}$ is
also lax DG monoidal, and a DG natural transformation $F\circ \Psi\to \wt{F}$,
which is a quasi-isomorphism.

\medskip

Namely, we take $\wt\bA_1$ to be the DG category from \secref{from q to f},
i.e., its objects are triples 
$$\wt{X}=\{X\in \bA_1,Y\in \bA_2,f\in \Hom^0_{\bA_1,\bA_2}("F(X)",Y)\},$$
where $f$ is a closed morphism that induces an isomorphism in $\Ho(\bA_2)$.
The pseudo monoidal structure is given by
\begin{multline*}
\Hom^\bullet_{\wt\bA_1}("\wt{X}^\otimes_I",\wt{X}):= \\
\{\alpha\in \Hom^\bullet_{\bA_1}("X^\otimes_I",X),\,
\beta\in \Hom^\bullet_{\bA_2}("Y^\otimes_I",Y),\,
\gamma\in \Hom^\bullet_{\bA_1,\bA_2}("F(X^\otimes_I)",Y)[-1].\}
\end{multline*}

The DG functors $\Psi$ and $\wt{F}$ are defined in an evident way. The DG
natural transformation $F\circ \Psi\to \wt{F}$ is
\begin{multline*}
\Hom^\bullet_{\wt\bA_1,\bA_2}("\wt{F}(\wt{X}^\otimes_I)",Y)=:
\Hom^\bullet_{\bA_2}("Y^\otimes_I",Y)\to \\
\to \Hom^\bullet_{\bA_1,\bA_2}("F(X^\otimes_I)",Y)=:
\Hom^\bullet_{\wt\bA_1,\bA_2}("F\circ \Psi(\wt{X}^\otimes_I)",Y),
\end{multline*}
where the second arrow is given via \eqref{nat trans hom mon}
by the data of $\underset{i\in I}\otimes\, f_i$.

\ssec{Compositions}   \label{comp mon}

Let us be given three pseudo monoidal categories $\bA_1,\bA_2,\bA_3$
and pseudo monoidal functors $F':\bA_1\to \bA_2$,
$F'':\bA_2\to \bA_3$ and $G:\bA_1\to \bA_3$. 

\medskip

A DG natural transformation $G\Rightarrow "F''\circ F'"$ is a collection
of morphisms defined for $I\twoheadrightarrow J\twoheadrightarrow K$,
$X_I\in \bC_1^I$, $Y_J\in \bC_2^J$, $Z\in \bC_3$
\begin{multline}
\left(\underset{j\in J}\otimes\, \Hom^\bullet_{\bA_1,\bA_2}("F'(X^\otimes_{I^j})",Y_j)\right)
\bigotimes
\left(\underset{j\in J}\otimes\, \Hom^\bullet_{\bA_2,\bA_3}("F''(Y^\otimes_J)",Z)\right)\to \\
\to \Hom^\bullet_{\bA_1,\bA_3}("G(X^\otimes_I)",Z),
\end{multline}
preserving the degree and commuting with the differential. These morphisms are required to make the corresponding diagrams commute.

\medskip

In a similar way one defines the notion of DG natural transformation 
$G\Rightarrow "F^n\circ...\circ F^1"$, where $F^i$ are pseudo monoidal functors
$\bA_{i}\to \bA_{i+1}$, and $G$ is a pseudo monoidal functor $\bA_1\to \bA_{n+1}$.

\medskip

Assume that the categories $\bA_1,...,\bA_{n+1}$ and the functors $F_1,...,F_n,G$
are homotopy monoidal. By construction, a DG natural transformation as above
gives rise to a natural transformation between the monoidal functors
$$\Ho(G)\Rightarrow \Ho(F_n)\circ...\circ \Ho(F_1).$$
We say that $G$ is a homotopy composition of $F_1,...,F_n$ if the latter natural 
transformation is an isomorphism.

\sssec{}  \label{comp mon as colim}

For $F_1,...,F_n,G$ as above let us denote by 
$\Hom_{\on{PMon}(\bA_1,\bA_{n+1})}(G,"F_n\circ...\circ F_1")$ the set
of natural transformations as above. Keeping $F_1,...,F_n$ fixed, we define
the set $\Hom_{\ol{\on{PMon}}(\bA_1,\bA_{n+1})}(G,"F_n\circ...\circ F_1")$
as a quotient of $\Hom_{\on{PMon}(\bA_1,\bA_{n+1})}(G,"F_n\circ...\circ F_1")$ 
by the equivalence relation defined by homotopy. 

\medskip

We define the functor
$$\on{Hom}_{\on{PMon}^{\Ho}(\bA_1,\bA_{n+1})}(G,"F_n\circ...\circ F_1"):
\on{PMon}^{\Ho}(\bA_1,\bA_{n+1})^{op}\to \on{Sets}$$ by
$$G\mapsto \underset{F_i\to F'_i,i=1,...,n}{colim}\, 
\Hom_{\ol{\on{PMon}}(\bA_1,\bA_{n+1})}(G,"F'_n\circ...\circ F'_1"),$$
where the colimit is taken over the index category of homotopy 
natural transformations $F_i\to F'_i, i=1,...,n$ that are quasi-isomorphisms.

\begin{lem}
The functor $G\mapsto 
\on{Hom}_{\on{PMon}^{\Ho}(\bA_1,\bA_{n+1})}(G,"F_n\circ...\circ F_1")$
is representable. 
\end{lem}

We shall denote the resulting universal object of $\on{PMon}^{\Ho}(\bA_1,\bA_{n+1})$
by $F_n\circ...\circ F_1$, and call it the pseudo composition of 
$F_1,...,F_n$.

\begin{lem}
The pseudo-composition of $F_1,...,F_n$ induces their
pseudo-composition as functors between plain DG categories
(see \secref{pseudo comp}). If $F_1,...,F_n$ are homotopy 
monoidal functors, then the pseudo composition is their homotopy composition.
\end{lem}

\sssec{}

Thus, we can introduce the 2-category, whose 0-objects are
homotopy monoidal categories, and 1-morphisms are the
categories $\on{HMon}(\bA_1,\bA_2)$. We shall denote this
2-category by $\DGMonCat$. 

\medskip

We can also consider the 2-category $\TrMonCat$ of triangulated
monoidal categories. We have an evident forgetful 2-functor
$\DGMonCat\to \TrMonCat$. 

\medskip

For a triangulated monoidal category, by its DG model we will
understand the fiber of the above functor.

\ssec{Actions}

Let $\bA$ be a DG pseudo monoidal category and $\bC$
another DG category. A (left) pseudo action of $\bA$
on $\bC$ is the following data:

For a finite ordered set $I$ and $X_I:I\to \bA$,
$Y',Y''\in \bC$ we must be given a complex
$$X_I,Y\mapsto \Hom^\bullet_{\bA,\bC}("X^\otimes_I\otimes Y'",Y''),$$
which functorially depends on all arguments.
The symbol $"X^\otimes_{I}\otimes Y'"$ stands for the
non-existing object $\underset{i\in I}\otimes\, X_i\otimes Y'\in \bC$.
These functors must be equipped with the following system of natural
transforations:

\medskip

For $I\twoheadrightarrow \{1,...,n\}$, $X_I\in \bA^I$, $Y_1,...,Y_n\in \bC$ 
we must be given a map
\begin{multline*}
\Hom_{\bA,\bC}^\bullet("X^\otimes_{I^{n}}\otimes Y_{n-1}",Y_{n})\bigotimes
\Hom_{\bA,\bC}^\bullet("X^\otimes_{I^{n-1}}\otimes Y_{n-2}",Y_{n-1})\bigotimes...\\
...\bigotimes \Hom_{\bA,\bC}^\bullet("X^\otimes_{I^1}\bigotimes Y_{1}",Y_{2})\to
\Hom_{\bA,\bC}^\bullet("X^\otimes_I\otimes Y_{1}",Y_{n}),
\end{multline*}
and for a surjection of finite ordered sets $I\twoheadrightarrow J$,
$X_I\in\bA^I$, $X_J\in \bA^J$, $Y',Y''\in \bC$
we must be given a map
$$\underset{j\in J}\otimes \Hom_\bA^\bullet("X^\otimes_{I^j}",X_j)
\bigotimes \Hom_{\bA,\bC}^\bullet("X^\otimes_J\otimes Y'",Y'')\to
\Hom_{\bA,\bC}^\bullet("X^\otimes_I\otimes Y'",Y'').$$
In addition, we need to be given a natural quasi-isomorphism corresponding
to $I\hookrightarrow I\cup \on{pt}$. These natural transformations must satisfy the
usual associativity axioms.

\medskip

Assume that the pseudo monoidal structure on $\bA$ is
a homotopy monoidal structure. We say that a pseudo action of 
$\bA$ on $\bC$ is a homotopy action if the data
of functors 
$$\{X_I,Y',Y''\}\mapsto
H^0\left(\Hom_{\bA,\bC}^\bullet("X^\otimes_I\otimes Y'",Y'')\right):
\Ho(\bA)^{I,op}\times \Ho(\bC)^{op}\times \Ho(\bC)\to \Vect_k$$
and natural transformations come from a (automatically uniquely
defined) monoidal action of $\Ho(\bA)$ on $\Ho(\bC)$.

\medskip

In a similar way one defines the notion of pseudo action
and homotopy action on the right. Any pseudo (resp., homotopy)
monoidal category carries a pseudo (resp., homotopy) action
on itself on both right and left; moreover, these two structures
commute in a natural sense.

\sssec{}   \label{functors between module categories}

If $\bA$ is a pseudo monoidal category with pseudo actions
on $\bC_1$ and $\bC_2$, a DG pseudo functor $F:\bC_1\to \bC_2$
compatible with the action is a functorial assignment for any
$X_I\in \bA^I$, $Y_1\in \bC_1,Y_2\in \bC_2$ of a complex
\begin{equation}  \label{funct data}
\Hom^\bullet_{\bA,\bC_1,\bC_2}("X^\otimes_I\otimes 
F(Y_1)",Y_2),
\end{equation}
where we should think of $"X^\otimes_I\otimes F(Y_1)"$ as the
corresponding non-existing object of $\bC_2$. (In the above formula
$I$ might be empty.)

\medskip

We must be given the following system of natural transformations.
Given three ordered finite sets $I_1$, $I_2$ and $I_3$,
$X_{I_j}\in \bA^{I_j}$, $j=1,2,3$, $Y'_1,Y''_1\in \bC_1$, $Y'_2,Y''_2\in \bC_2$,
consider the concatenation $I=I_3\cup I_2\cup I_1$. We should have a map 
\begin{multline}   \label{nat trans funct}
\Hom^\bullet_{\bA,\bC_1}("X^\otimes_{I_1}\otimes Y'_1",Y''_1)
\bigotimes \Hom_{\bA,\bC_1,\bC_2}^\bullet("X_{I_2}^\otimes \otimes
F(Y''_1)",Y'_2)\bigotimes \\
\bigotimes \Hom^\bullet_{\bA,\bC_2}("X^\otimes_{I_3}\otimes Y'_2",Y''_2)\to
\Hom^\bullet_{\bA,\bC_1,\bC_2}("X_{I}^\otimes \otimes F(Y'_1)",Y''_2).
\end{multline}

We should also be given natural transformations corresponding insertions
of unit objects. There natural transformations must satisfy the natural axioms
that we will not spell out here.

\medskip

We say that $F$ is a homotopy functor 
compatible with the action of $\bA$
if the data of $H^0\left(\Hom^\bullet_{\bA,\bC_1,\bC_2}
("X^\otimes_I\otimes F(Y_1)",Y_2)\right)$ comes from a functor $\Ho(\bC_1)\to \Ho(\bC_2)$,
compatible with the action of the monoidal category $\Ho(\bA)$.

\medskip

We say that $F$ is a quasi-equivalence, if the underlying functor
$\Ho(\bC_1)\to \Ho(\bC_2)$ is an equivalence.

\sssec{}   \label{DG cat of hom}

Pseudo functors
$\bC_1\to \bC_2$ compatible with the action of $\bA$ naturally form a DG category.
Namely, for two DG pseudo functors $F'$ and $F''$ we set
$\Hom_\bA^\bullet(F',F'')$ to be the be the sub-complex in
$$\underset{I;X_I,Y_1,Y_2}\Pi\, \Hom^{\bullet}_{\Comp_k}
\left(\Hom^\bullet_{\bA,\bC_1,\bC_2}("X^\otimes_I\otimes F''(Y_1)",Y_2),
\Hom^\bullet_{\bA,\bC_1,\bC_2}("X^\otimes_I\otimes F'(Y_1)",Y_2)\right),$$
that makes all diagrams corresponding to \eqref{nat trans funct} commute.

\medskip

Let us denote this DG category by $\on{PFunct}_{\bA}(\bC_1,\bC_2)$. 
We call an object of $\on{PFunct}_{\bA}(\bC_1,\bC_2)$ acyclic of all the 
complexes $\Hom^\bullet("X^\otimes_I\otimes Y_1",Y_2)$ are acyclic. The
resulting quotient traingulated category will be denoted
$$\bD(\on{PFunct}_{\bA}(\bC_1,\bC_2)):=:
\on{PFunct}^{\Ho}_{\bA}(\bC_1,\bC_2).$$

\medskip

When the pseudo action of $\bA$ on both $\bC_1$ and $\bC_2$ is
a homotopy action, we will denote the full subcategory of
$\bD(\on{PFunct}_{\bA}(\bC_1,\bC_2))$ formed by homotopy functors
by $\on{HFunct}_{\bA}(\bC_1,\bC_2)$.

\sssec{}

The structure of DG category on $\on{PFunct}_{\bA}(\bC_1,\bC_2)$
is part of a closed model category structure, where cofibrations are 
those maps $F'\Rightarrow F''$, for which all maps
$$\Hom^\bullet_{\bA,\bC_1,\bC_2}("X^\otimes_I\otimes F''(Y_1)",Y_2)\to
\Hom^\bullet_{\bA,\bC_1,\bC_2}("X^\otimes_I\otimes F'(Y_1)",Y_2)$$
are surjective.

\medskip

The model category structure on $\on{PFunct}_{\bA}(\bC_1,\bC_2)^{op}$
is akin to that on the category of DG modules over a DG associative algebra.

\medskip

A supply of fibrant objects in $\on{PFunct}_{\bA}(\bC_1,\bC_2)$ is provided
as in \secref{free functor} by the pair of adjoint functors
$$\on{PFunct_{gr}}_{\bA}(\bC_1,\bC_2) \leftrightarrows
\underset{n> 0}\Pi\, \on{PFunct_{gr}}(\bA^n\times \bC_1,\bC_2),$$
where the subscript "gr" stands for the 
graded-without-differential versions of the corresponding  categories.

\sssec{}   \label{DG models for action}

One defines compositions of homotopy functors compatible with
the action of $\bA$ following the pattern of \secref{pseudo comp}
and \secref{comp mon}. 

\medskip

Thus, given a homotopy monoidal category $\bA$, we can speak
about the 2-category $\DGMod(\bA)$. Its 0-objects are 
(essentially small)
DG categories, endowed with a homotopy action of $\bA$ and 
1-morphisms are the categories $\on{HFunct}_{\bA}(?,?)$.

\medskip

We can also have the 2-category $\TrMod(\Ho(\bA))$ of triangulated
categories equipped with an action of $\Ho(\bA)$. There exists
an evident forgetful 2-functor $\DGMod(\bA)\to \TrMod(\Ho(\bA))$.

\medskip

In what follows, for a 0-object $\bD\in \TrMod(\Ho(\bA))$, by its
DG model we will mean the fiber of the above map.

\ssec{Changing the acting category}  \label{changing acting}

Let $\bA_1,\bA_2$ be two pseudo monoidal categories, and let 
$F_\bA:\bA_1\to \bA_2$ be a pseudo monoidal functor. Let
$\bC$ be a DG category equipped with a homotopy action of 
$\bA_2$. In this subsection we shall construct the restriction
2-functor $\DGMod(\bA_2)\to \DGMod(\bA_1)$.

\sssec{}

Let $\bA_1$, $\bA_2$ be two DG pseudo monoidal categories, and
$\bC_1,\bC_2$ be two DG categories, endowed with pseudo
actions of $\bA_1$ and $\bA_2$, respectively. Let 
$F_\bA:\bA_1\to \bA_2$ a DG pseudo-monoidal functor.

\medskip

A DG pseudo functor $F_\bC:\bC_1\to \bC_2$ compatible with $F_\bA$
is a functorial assignment to $X_1{}_I\in \bA_1^I$, $Y_1\in \bC_1$,
$Y_2\in \bC_2$ of a complex
$$\Hom^\bullet_{\bA_1,\bC_1,\bC_2}
("X_1{}_I{}^\otimes\otimes F_\bC(Y_1)",Y_2),$$
endowed with the following system of natural transformations:

\medskip

For finite ordered sets $I,J,K,L$, a surjection
$J\twoheadrightarrow K$,
$X_1{}_I\in \bA_1^I$, $X_1{}_J\in \bA_1^J$, $X_1{}_L\in \bA_1^L$
$X_2{}_K\in \bA_2^K$
$Y'_1,Y''_1\in \bC_1$, $Y'_2,Y''_2\in \bC_2$, we need to be given a map
\begin{multline*}
\Hom^\bullet_{\bA_1,\bC_1}("X_1{}_L{}^\otimes \otimes Y'_1",Y''_1)\bigotimes \\
\bigotimes \Hom^\bullet_{\bA_1,\bC_1,\bC_2}
("X_1{}_I{}^\otimes \otimes F_\bC(Y''_1)",Y'_2)\bigotimes 
\left(\underset{k\in K}\otimes\, \Hom^\bullet_{\bA_1,\bA_2}("X_1{}_{J^k}{}^\otimes",
X_2{}_k)\right)\bigotimes\\
\bigotimes \Hom^\bullet_{\bA_2,\bC_2}("X_2{}_K{}^\otimes \otimes Y'_2",Y''_2)\to
\Hom^\bullet_{\bA_1,\bC_1,\bC_2}
("X_1{}_{J\cup I\cup L}^\otimes \otimes F_\bC(Y_1)",Y''_2).
\end{multline*}
These natural transformations are required to satisfy the natural axioms.

\medskip

If $\bA_1$, $\bA_2$ are homotopy monoidal categories, $F_\bA$
is a homotopy monoidal functor, and the pseudo actions of $\bA_i$ 
on $\bA_i$ for $i=1,2$ are homotopy actions, one says that $F_\bC$
is a homotopy functor, if the data of $H^0\left(\Hom^\bullet_{\bA_1,\bC_1,\bC_2}
("X_1{}_I{}^\otimes\otimes F_\bC(Y_1)",Y_2)\right)$ and natural transformations
comes from a functor $\Ho(\bC_1)\to \Ho(\bC_2)$, compatible with the
action of $\Ho(\bA_1)$ via $\Ho(F_\bA)$.

\sssec{}

Proceeding as in \secref{DG cat of hom}, for $\bA_1,\bA_2,F_\bA,\bC_1,\bC_2$
we introduce the DG category $\on{PFunct}_{F_\bA}(\bC_1,\bC_2)$, 
which has a closed model structure, and its homotopy category
$$\bD(\on{PFunct}_{F_\bA}(\bC_1,\bC_2)):=:
\on{PFunct}^{\Ho}_{F_\bA}(\bC_1,\bC_2),$$
which is a triangulated category. In the case of homotopy monoidal structures and 
actions, the latter contains a triangulated subcategory that consists of homotopy
functors, denoted
$\on{HFunct}_{F_\bA}(\bC_1,\bC_2)$.

\medskip

In addition, one has the functors:
$$\circ:\on{PFunct}^{\Ho}_{\bA_1}(\bC'_1,\bC_1) \times 
\on{PFunct}^{\Ho}_{F_\bA}(\bC_1,\bC_2)\to \on{PFunct}_{F_\bA}(\bC'_1,\bC_2)$$
and
$$\circ:\on{PFunct}^{\Ho}_{F_\bA}(\bC_1,\bC_2)\times
\on{PFunct}^{\Ho}_{\bA_2}(\bC_2,\bC'_2)\to \on{PFunct}_{F_\bA}(\bC_1,\bC'_2),$$
and in the case of homotopy monoidal structures, the functors
$$\circ:\on{HFunct}_{\bA_1}(\bC'_1,\bC_1) \times 
\on{HFunct}_{F_\bA}(\bC_1,\bC_2)\to \on{HFunct}_{F_\bA}(\bC'_1,\bC_2)$$
and
$$\circ:\on{HFunct}_{F_\bA}(\bC_1,\bC_2)\times
\on{HFunct}_{\bA_2}(\bC_2,\bC'_2)\to \on{HFunct}_{F_\bA}(\bC_1,\bC'_2),$$
defined by the procedure analogous to that of \secref{pseudo comp}
and \secref{comp mon}.

\sssec{}

Given pseudo monoidal categories $\bA_1,\bA_2$, a pseudo
monoidal functor $F_\bA$ and a category $\bC_2$ with a pseudo action of 
$\bA_2$, a DG category $\bC_{1,can}$ with a pseudo action of $\bA_1$
is called a restriction of $\bC_2$ with respect to $F_\bA$ if we are given an
object $F_{\bC,can}\in \on{PFunct}^{\Ho}_{F_\bA}(\bC_{1,can},\bC_2)$, such that
for any $\bC_1$ with a pseudo action of $\bA_1$, the functor 
$$G\mapsto F_{\bC,can}\circ G:\on{PFunct}^{\Ho}_{\bA_1}(\bC_1,\bC_{1,can})\to
\on{PFunct}^{\Ho}_{F_\bA}(\bC_1,\bC_2)$$
is an equivalence. By Yoneda's lemma, if a restriction exists, it is
canonically defined as a 0-obect in the appropriate 2-category of DG categories with a pseudo 
action of $\bA_1$, up to quasi-equivalence; we shall denote
by $\Res^{\bA_2}_{\bA_1}(\bC_2)$.

\begin{lem}
For any $\bC_2$ with a pseudo action of $\bA_2$, the restriction $\Res^{\bA_2}_{\bA_1}(\bC_2)$
exists. The pseudo functor 
$$F_{\bC_2,can}:\Res^{\bA_2}_{\bA_1}(\bC_2)\to \bC_2,$$
when regarded as a pseudo functor between plain DG categories,
is a quasi-equivalence.
\end{lem}

\sssec{}  \label{comp change}

Let $F'_\bA:\bA_1\to \bA_2$, $F''_\bA:\bA_2\to \bA_3$, $G_\bA:\bA_1\to \bA_3$
be pseudo monoidal functors between pseudo monoidal categories,
and let us be given a DG natural transformation $\phi_\bA:G_\bA\Rightarrow "F_\bA''
\circ F'_\bA"$. 

\medskip

Let $\bC_i$ be a DG category with a pseudo-action of $\bA_i$, $i=1,2,3$. Let
us be given pseudo functors $F'_\bC:\bC_1\to \bC_2$ and $F''_\bC:\bC_2\to \bC_3$,
compatible with $F'_\bA$ and $F''_\bA$, respectively.
 
\medskip

For a pseudo functor $G_\bC:\bC_1\to \bC_3$, compatible with $G_\bA$, we define 
the set
$$\Hom_{\on{PFunct}_{\phi_\bA}(\bC_1,\bC_3)}(G_\bC,"F''_\bC\circ F'_\bC").$$

Proceeding as in Sects. \ref{comp as colim} and \ref{comp mon as colim}, we define also
the set
$$\Hom_{\on{PFunct}^{\Ho}_{\phi_\bA}(\bC_1,\bC_3)}(G_\bC,"F''_\bC\circ F'_\bC").$$

\begin{lem} \label{comp different}
The functor
$$G_\bC\mapsto \Hom_{\on{PFunct}^{\Ho}_{\phi_\bA}(\bC_1,\bC_3)}
(G_\bC,"F''_\bC\circ F'_\bC")$$
on $\on{PFunct}^{\Ho}_{G_\bA}(\bC_1,\bC_3)$ is representable. The universal object,
denoted $F'_\bC\circ F''_\bC$ and called the pseudo composition of $F'_\bC$ and $F''_\bC$,
induces a pseudo composition at the level of plain categories. If the categories, actions 
and functors in question are homotopy monoidal, then the map
$$\Ho(\phi_\bC):\Ho(F''_\bC\circ F'_\bC)\Rightarrow \Ho(F''_\bC)\circ \Ho(F'_\bC)$$
is an isomorphism.
\end{lem}

\sssec{}

In the above setting, let us take $\bC_2:=\Res^{\bA_3}_{\bA_2}(\bC_3)$ with 
$F''_\bC$ being $F''_{\bC_3,can}$, and
$\bC_1:=\Res^{\bA_2}_{\bA_1}(\bC_2)$ with $F'_\bC$ being $F'_{\bC_2,can}$.
Assume that the arrow $\phi_\bA$ defines an isomorphism
$G_\bA\simeq F''_\bA\circ F'_\bA$.

\begin{lem}  \label{compose restrictions}
Under the above circumstances the 1-morphism $$\bC_1\to  \Res^{\bA_3}_{\bA_1}(\bC_3)$$ 
is a 1-isomorphism in the 2-category of DG categories with a pseudo 
action of $\bA_1$, up to quasi-equivalence. 

In particular, if the categories and actions are homotopy monoidal, then
the above 1-morphism is an isomorphism in $\DGMod(\bA_1)$.
\end{lem}

The upshot of the lemma is that we have a canonical equivalence in 
$\DGMod(\bA_1)$:
$$\Res^{\bA_3}_{\bA_1}(\bC_3)\simeq \Res^{\bA_2}_{\bA_1}(\Res^{\bA_3}_{\bA_2}(\bC_3)).$$

\ssec{}   \label{ind extensions}

Let $\bA$ be a DG pseudo monoidal category. We claim that
$\ua\bA$ naturally acquires a DG pseudo monoidal structure. 
Namely, for a finite order set $I$ and 
$X_i\in \ua\bA, i\in I$, $Y\in \ua\bA$,
given by
$$\{\underset{k_i\geq 0}\oplus\, X_i^{k_i},\Phi_i\} \text{ and }
Y=\{\underset{m\geq 0}\oplus\, Y^m,\Psi\}, \text{ respectively}$$ set
$$\Hom^\bullet("X_I^\otimes",Y):=
\underset{\alpha}\Pi\, \underset{m}\oplus\,
\Hom^\bullet("X^\otimes_{I_\alpha}",Y^m),$$
where $\alpha$ runs over the set of maps $I\to \BZ^{\geq 0}$,
each defining a map $I\overset{X_{I_\alpha}}\to 
\bA:i\mapsto X_i^{\alpha(i)}$.
The differential in the above complex is given using the maps
$\Phi$ and $\Psi$.

\medskip

If the DG pseudo monoidal structure on $\bA$ is a homotopy
monoidal structure, then so will be the case for $\ua\bA$.

\medskip

Similarly, if $\bA_1$ and $\bA_2$ are two DG pseudo monoidal categories
and $F:\bA_1\to \bA_2$ is a pseudo monoidal functor, it extends to 
a pseudo monoidal functor $\ua\bA{}_1\to \ua\bA{}_2$. If $F$ is a homotopy
monoidal functor between homotopy monoidal categories, so will be its
extension.

\sssec{}

Let $\bC$ be a DG category equipped with a DG pseudo
action of a pseudo monoidal category $\bA$. Then it
extends to a DG pseudo action of $\ua\bA$ on $\ua\bC$,
preserving the property of being a homotopy action.

\medskip

By a similar procedure, given a 1-morphism  
(resp., homotopy functor) $F:\bC_1\to \bC_2$ in $\DGMod(\bA)$
categories we extend it to a 1-morphism $\ua\bC{}_1\to \ua\bC{}_2$
in $\DGMod(\ua\bA)$.

\sssec{}   \label{action on subcat}

If $\bA$ is a DG pseudo monodal category, and $\bA'\subset \bA$
is a full DG subcategory, it automatically inherits a DG pseudo monodal
structure. If $\bA$ is a homotopy monoidal category, then $\bA'$
will be such if and only if $\Ho(\bA')$ is a monoidal subcategory of $\Ho(\bA)$.

\medskip

If $F:\bA_1\to \bA_2$ is a pseudo monoidal
functor and $\bA'_2\subset \bA_2$ a DG subcategory,
we obtain a pseudo monoidal functor $F':\bA_1\to \bA'_2$. If $F$
is a homotopy monoidal functor between homotopy monoidal categories,
and $\bA'_2$ is also a homotopy monoidal category, then
$F'$ is a homotopy monoidal functor if and only if $\Ho(F)$ sends $\Ho(\bA_1)$
to $\Ho(\bA'_2)$. Analogously to \secref{sub and nat}, this establishes
an equivalence between the category of 1-morphisms $\bA_1\to \bA'_2$
and the full subcategory of the category of those 1-morphisms $\bA_1\to \bA_2$
whose essential image on the homotopy level belongs to $\Ho(\bA'_2)$.

\medskip

By the same token if we have a pseudo action of $\bA$ on $\bC$
and $\bC'\subset \bC$ is a full DG subcategory, we have a 
pseudo action of $\bA$ on $\bC'$. If initially we had a homotopy
action of $\bA$ on $\bC$, then it will be the case for $\bC'$
if and only if the $\Ho(\bA)$-action on $\Ho(\bC)$ preserves
$\Ho(\bC')$.

\medskip

A similar discussion applies to DG pseudo functors
and homotopy functors $\bC_1\to \bC_2$ in $\DGMod(\bA)$.

\sssec{}

Let $\bA$ be a DG pseudo monoidal category, and $\bA'\subset \bA$
a full subcategory. By the construction of quotients in \secref{DG quotients}
and the above discussion, the DG category $\bA/\bA'$ acquires
a natural DG pseudo monoidal structure. If $\bA$ was a homotopy
monoidal category, then $\bA/\bA'$ will be such if and only if
$\Ho(\bA')\subset \Ho(\bA)$ is a monoidal ideal. 

\medskip

In the situation of a DG pseudo monoidal functor $F:\bA_1\to \bA_2$
we obtain a pseudo monoidal functor $\bA_1/\bA'_1\to \bA_2/\bA'_2$.
If the initial situaition was homotopy monoidal, then the latter functor
will be homotopy monoidal if and only if $\Ho(F)(\bA'_1)\subset \Ho(\bA'_2)$.

In particular, the canonical homotopy functor $\bA\to \bA/\bA'$ naturally
extends to a pseudo monoidal functor, which is homotopy monoidal 
if $\Ho(\bA')$ is an ideal.

A similar discussion applies to the situation when we have an action
of $\bA$ on $\bC$ and a DG subcategory $\bC'$. In particular,
we obtain:

\begin{lem}   \label{action on quotient} \hfill

\smallskip

\noindent(1)
Let $\bA$ be a homotopy monoidal category with a homotopy
action on a DG category $\bC$. Let $\bA'\subset \bA$,
$\bC'\subset \bC$ be DG subcategories. Assume that
$$\Ho(\bA')\times \Ho(\bC)\mapsto \Ho(\bC') \text{ and }
\Ho(\bA)\times \Ho(\bC')\mapsto \Ho(\bC').$$
Then we have a well-defined homotopy action of $\bA/\bA'$
on $\bC/\bC'$.

\smallskip

\noindent(2)
For a DG category $\bC_1$ with a homotopy action of $\bA/\bA'$
the following two categories are equivalent:

\smallskip

\noindent(a)
$\on{HFunct}_{\bA/\bA'}(\bC/\bC',\bC_1)$.

\smallskip

\noindent(b) The 
full subcategory of $\on{HFunct}_{\bA}(\bC,\Res^{\bA/\bA'}_{\bA}(\bC_1))$,
consisting of homotopy functors compatible with the action,
for which the underlying plain homotopy functor $\bC\to \bC_1$ 
factors through $\bC/\bC'$.
\end{lem}

\sssec{}   \label{weak quotient tensor}

In the sequel we will need a generalization of the above discussion
along the lines of \secref{weak to quotient}. Let $\bA$ be a DG pseudo-monoidal
category (resp., $\bC$ a DG category with a pseudo-action of $\bA$;
$F:\bC_1\to \bC_2$ a pseudo functor between such categories,
compatible with the pseudo-actions of $\bA$.)

\medskip

Let $\bA'\subset \bA$ (resp., $\bC'\subset \bC$; $\bC'_i\subset \bC_i$)
be DG subcategories. Suppose that the following conditions hold:

\begin{itemize}

\item
For $X_1,...,X_n\in \bA$, the functor on $\Ho(\bA)/\Ho(\bA')$ given by
$$X\mapsto \underset{f:X\to X'}{colim}\, 
H^0\left(\Hom^\bullet_{\bA}("X_1\otimes...\otimes X_n",X')\right)$$
is co-representable, and if one of the $X_i$'s belongs to $\bA'$,
then it equals zero. (The colimit is taken over the set of arrows $f$
with $\on{Cone}(f)\in \bA'$).

\item
For $X_1,...,X_n\in \bA$, $'Y\in \bC$, the functor on $\Ho(\bC)/\Ho(\bC')$ given by
$$''Y\mapsto \underset{f:{}''Y\to {}''Y'}{colim}\, 
H^0\left(\Hom^\bullet_{\bA,\bC}("X_1\otimes...\otimes X_n\otimes {}'Y",{}''Y')\right)$$
is co-representable, and if one of the $X_i$'s belongs to $\bA'$ or
$'Y$ belongs to $\bC'$, then it equals zero.

\item
For $X_1,...,X_n\in \bA$, $Y_1\in \bC_1$, the functor on 
$\Ho(\bC_2)/\Ho(\bC_2')$ given by
$$Y_2\mapsto \underset{f:Y_2\to Y_2'}{colim}\, 
H^0\left(\Hom^\bullet_{\bA,\bC_1,\bC_2}("X_1\otimes...\otimes X_n\otimes F(Y_1),Y'_2)\right)$$
is co-representable, and if one of the $X_i$'s belongs to $\bA'$ or
$Y_1$ belongs to $\bC_1'$, then it equals zero.

\end{itemize}

Then the construction of \secref{weak to quotient} endows $\bA/\bA'$ with
a structure of homotopy monoidal category, the category $\bC/\bC'$ with
a homotopy action of $\bA/\bA'$ and the homotopy
functor $\bC_1/\bC'_1\to \bC_2/\bC'_2$ with the structure of
compatibility with the homotopy action of $\bA/\bA'$.

\sssec{}

Finally, from \secref{Karoubian}, we obtain that if $\bA$ is a DG pseudo monoidal 
(resp., homotopy monoidal) category, then so is $\bA^{Kar}$.
Any DG pseudo monoidal (resp., homotopy monoidal)
functor $\bA_1\to \bA_2$ gives rise to a DG pseudo monoidal (resp., homotopy monoidal) functor $\bA_1^{Kar}\to \bA_2^{Kar}$, and similarly
for actions.

\section{Tensor products of categories}   \label{tensor products of categories}

\ssec{}

Let $\bC_1$ and $\bC_2$ be two DG categories. We form a non-pretriangulated 
DG category $(\bC_1\otimes \bC_2)^{non-pretr}$ to have as objects pairs
$X_1,X_2$ with $X_i\in \bC_i$ and 
$$\Hom((X_1,X_2),(Y_1,Y_2))=
\Hom_{\bC_1}(X_1,Y_1)\otimes \Hom_{\bC_2}(X_2,Y_2).$$
We define $\bC_1\otimes \bC_2$ as the strongly pre-triangulated
envelope of $(\bC_1\otimes \bC_2)^{non-pretr}$. 

\medskip

If $F_1:\bC_1\to \bC'_1$ and $F_2:\bC_2\to \bC'_2$ are 
1-morphisms in $\DGCat$, we have a well-defined 
1-morphism 
$$F_1\otimes F_2:\bC_1\otimes \bC_2\to \bC'_1\otimes \bC'_2.$$

\medskip

Note that for any DG category $\bC$, we have:
$$\bC\otimes \Comp^f_k\simeq \bC,$$
where $\Comp^f_k$ is the DG category of finite-dimensional complexes.

\sssec{}   \label{tensor free}

Let $\bA_1$ and $\bA_2$ be two DG categories equipped 
with a DG pseudo monoidal (resp., homotopy monoidal)
structure. Then their tensor product $\bA_1\otimes \bA_2$
is naturally a DG pseudo monoidal (resp., homotopy monoidal)
category.

\medskip

Indeed, for $(X^1_1,...,X^n_1)\in \bA_1$, $(X^1_2,...,X^n_2)\in \bA_2$,
$Y_1\in \bA_1$, $Y_2\in \bA_2$, we set
\begin{align*}
\Hom^\bullet_{\bA_1\otimes \bA_2}
\left("(X^1_1,X^1_2)\otimes...\otimes (X^n_1,X^n_2)",(Y_1,Y_2)\right):=\\
=\Hom^\bullet_{\bA_1}("X^1_1\otimes...\otimes X^n_1",Y_1)\otimes
\Hom^\bullet_{\bA_2}("X^1_2\otimes...\otimes X^n_2",Y_2),
\end{align*}
and this uniquely extends onto arbitrary objects of $\bA_1\otimes \bA_2$.

\medskip

Similarly, if in the above situation $\bA_1$ is endowed with
a DG pseudo action (resp., homotopy action) on $\bC^l_1$
and similarly for the pair $(\bA_2,\bC^l_2)$ we have a 
DG pseudo action (resp., homotopy action) of $\bA_1\otimes \bA_2$
on $\bC^l_1\otimes \bC^l_2$.

\medskip

As a particular case, we obtain that given 
a DG pseudo action (resp., homotopy action) of $\bA$ on $\bC^l$,
for an arbitrary DG category $\bC'$, the tensor product
$\bC^l\otimes \bC'$ carries a DG pseudo action (resp., homotopy action)
of $\bA$.

\ssec{}    \label{def ten prod}

Suppose now that $\bA$ is a DG category with a homotopy monoidal
structure. Let $\bC^r$ and $\bC^l$ be two DG categories equipped with
homotopy right and left actions of $\bA$, respectively. Our present
goal is to define a new DG category, which would be the tensor product
of $\bC^r$ and $\bC^l$ {\it over} $\bA$, denoted
\begin{equation} \label{DG ten prod}
\bC^r\underset{\bA}\otimes \bC^l.
\end{equation}

The construction given below was explained to us by J. Lurie.

\sssec{}

First, we define a DG category $\Bar(\bC^r,\bA,\bC^l)^{non-pretr}$. Its 
objects are 
\begin{equation} \label{object}
(Y^r,\ol{X}^n,Y^l):=(Y^r,\underset{n}{\underbrace{X^1,...,X^n}},Y^l),
\end{equation}
where $n \in \{0,1,2...\}$, $Y^r\in \bC^r$, $Y^l\in \bC^l$
and $X^i\in \bA$. In what follows, for a fixed $n$, we shall denote the 
corresponding subcategory by $\Bar^n(\bC^r,\bA,\bC^l)$.

\medskip

For an object $(Y_1^r,\ol{X}_1^n,Y^l_1)$ and an object
$(Y_2^r,\ol{X}_2^m,Y^l_2)$ and a non-decreasing map
$\phi:\{0,1,...,m\}\to \{0,1,...,n\}$ the set of $\phi$-morphisms
$$\Hom^\bullet_\phi((Y_1^r,\ol{X}_1^n,Y^l_1),(Y_2^r,\ol{X}_2^m,Y^l_2))$$
to be by definition
\begin{multline*}
\Hom^{\bullet}_{\bA,\bC^r}("Y^r_1\otimes X_1^1\otimes...\otimes
X^{\phi(0)}_1",Y^r_2)\otimes 
\Hom_{\bA}("X^{\phi(0)+1}_1\otimes...\otimes X^{\phi(1)}_1",X^1_2)\otimes...\\
\otimes \Hom^\bullet_{\bA}
("X^{\phi(i-1)+1}_1\otimes...\otimes X^{\phi(i)}_1",X^i_2)
\otimes ...\\ ...\otimes
\Hom^\bullet_{\bA}("X^{\phi(m-1)+1}_1\otimes...\otimes X^{\phi(m)}_1",X^m_2) 
\otimes \Hom^\bullet_{\bA,\bC^l}
("X^{\phi(m)+1}_1\otimes...\otimes X^n_1\otimes Y^l_1",Y^l_2),
\end{multline*}
where if $\phi(i-1)=\phi(i)$ in the corresponding term we make
an insertion of ${\bf 1}_{\bA}$.

\medskip

We define
$$\Hom^\bullet_{\Bar(\bC^r,\bA,\bC^l)^{non-pretr}}(
(Y_1^r,\ol{X}_1^n,Y^l_1),(Y_2^r,\ol{X}_2^m,Y^l_2)):=
\underset{\phi}\oplus\,
\Hom^\bullet_\phi((Y_1^r,\ol{X}_1^n,Y^l_1),(Y_2^r,\ol{X}_2^m,Y^l_2)),$$
as $\phi$ runs over the set of all non-decreasing maps
$\phi:\{0,1,...,m\}\to \{0,1,...,n\}$.

\medskip

For $\psi:\{0,1,...,k\}\to \{0,1,...,m\}$ and $(Y^r_3,\ol{X}^k_3,Y^l_3)$ 
we have a natural composition map,
\begin{multline*}
\Hom^\bullet_\phi((Y_1^r,\ol{X}_1^n,Y^l_1),(Y_2^r,\ol{X}_2^m,Y^l_2))
\otimes \Hom^\bullet_\psi((Y_2^r,\ol{X}_2^m,Y^l_2),(Y_3^r,\ol{X}_3^k,Y^l_3))\to \\
\to \Hom^\bullet_{\phi\circ \psi}((Y_1^r,\ol{X}_1^n,Y^l_1),(Y_3^r,\ol{X}_3^k,Y^l_3)).
\end{multline*}

\medskip

This defines a structure of (non-pretriangulated) DG category on
$\Bar(\bC^r,\bA,\bC^l)^{non-pretr}$. We denote by
$\Bar(\bC^r,\bA,\bC^l)$ its strongly pre-triangulated envelope.

\sssec{}

Let $(Y_1^r,\ol{X}_1^n,Y^l_1)$ and $(Y_2^r,\ol{X}_2^m,Y^l_2)$ be two
objects as above.
Let $f$ be an element of $\Hom^\bullet_\phi((Y_1^r,\ol{X}_1^n,Y^l_1),(Y_2^r,\ol{X}_2^m,Y^l_2))$
equal to the tensor product of 
\begin{align*}
& f^r\in \Hom^{\bullet}_{\bA,\bC^r}("Y^r_1\otimes X_1^1\otimes...\otimes
X^{\phi(0)}_1",Y^r_2),\\
& f^i\in
\Hom^\bullet_{\bA}
("X^{\phi(i-1)+1}_1\otimes...\otimes X^{\phi(i)}_1",X^i_2), i=1,...,m \\
& f^l\in \Hom^\bullet_{\bA,\bC^l}
("X^{\phi(m)+1}_1\otimes...\otimes X^n_1\otimes Y^l_1",Y^l_2).
\end{align*}

We say that $f$ is a {\it quasi-isomorphism} if all $f^i$ are cocycles
of degree $0$ that induce isomorphisms between the corresponding objects
on the homotopy level:
\begin{align*}
& Y^r_1\otimes X_1^1\otimes...\otimes
X^{\phi(0)}_1\to Y^r_2\in \Ho(\bC^r),\\
& X^{\phi(i-1)+1}_1\otimes...\otimes X^{\phi(i)}_1\to X^i_2 \in
\Ho(\bA),\\
& X^{\phi(m)+1}_1\otimes...\otimes X^n_1\otimes Y^l_1\to Y^l_2
\in \Ho(\bC^l).
\end{align*}

Let $\Ho(\bI_{\bC^r,\bA,\bC^l}))\subset 
\Ho(\Bar(\bC^r,\bA,\bC^l)$ be the triangulated subcategory
generated by the cones of all
quasi-isomorphisms. Let $\bI_{\bC^r,\bA,\bC^l}$ be its
preimage in $\Bar(\bC^r,\bA,\bC^l)$.

\medskip

We consider the quotient
$$\Bar(\bC^r,\bA,\bC^l)/\bI_{\bC^r,\bA,\bC^l},$$
as a 0-object of $\DGCat$. By definition, 
$$\bC^r\underset{\bA}\otimes \bC^l:=
\Bar(\bC^r,\bA,\bC^l)/\bI_{\bC^r,\bA,\bC^l}.$$
I.e., $\Ho\left(\bC^r\underset{\bA}\otimes \bC^l\right)\simeq
\Ho(\Bar(\bC^r,\bA,\bC^l)/\Ho(\bI_{\bC^r,\bA,\bC^l})$.

\sssec{}   \label{ten and funct}

Let $\bC_1^r,\bC_1^l$, $\bC_2^r,\bC_2^l$ be two pairs
of DG categories with homotopy actions of $\bA$.
It is clear from \secref{functors between module categories} 
that we obtain well-defined functors
$$\on{HFunct}_{\bA}(\bC_1^l,\bC_2^l)\times
\on{HFunct}_{\bA}(\bC_1^r,\bC_2^r)\to
\on{HFunct}\left(\Bar(\bC_1^r,\bA,\bC_1^l),\Bar(\bC_2^r,\bA,\bC_2^l)\right)$$
and 
$$\on{HFunct}_{\bA}(\bC_1^l,\bC_2^l)\times
\on{HFunct}_{\bA}(\bC_1^r,\bC_2^r)\to
\on{HFunct}\left(\bC^r_1\underset{\bA}\otimes \bC_1^l,
\bC_2^r\underset{\bA}\otimes \bC^l_2\right).$$

In particular, if $F^r:\bC_1^r\to \bC_2^r$ and $F^l:\bC_1^l\to \bC_2^l$
are quasi-equivalences, then so are the resulting functors
$$\Bar(\bC_1^r,\bA,\bC_1^l)\to \Bar(\bC_2^r,\bA,\bC_2^l) \text{ and }
\bC^r_1\underset{\bA}\otimes \bC_1^l\to
\bC_2^r\underset{\bA}\otimes \bC^l_2.$$

\sssec{}   \label{ten and mon}

Let $(\bC_1^l,\bC_1^r,\bA_1)$ and $(\bC_2^l,\bC_2^r,\bA_2)$
be two triples as above, $F_\bA:\bA_1\to \bA_2$ be a homotopy
monoidal functor, and $F^l:\bC_1^l\to \bC_2^l$, $F^r:\bC_1^r\to \bC_2^r$
be homotopy functors compatible with $F_\bA$. In this case, we have
well-defined 1-morphisms 
\begin{equation} \label{funct of ten}
\Bar(\bC_1^r,\bA_1,\bC_1^l)\to \Bar(\bC_2^r,\bA_2,\bC_2^l) \text{ and }
\bC_1^r\underset{\bA_1}\otimes \bC_1^l\to 
\bC_2^r\underset{\bA_2}\otimes \bC_1^l.
\end{equation}
In particular, this applies to the case when 
$$\bC^l_1=\Res^{\bA_2}_{\bA_1}(\bC^l_2) \text{ and }
\bC^r_1=\Res^{\bA_2}_{\bA_1}(\bC^r_2).$$

If the functors $F_\bA,F^l,F^r$ are quasi-equivalences, then so
are the 1-morphisms in \eqref{funct of ten}.

\ssec{}

Let us take $\bC^r:=\bA$ with the standard right action on itself. We claim:

\begin{prop} \label{ten with free}
There exists a canonical quasi-equivalence
$$\bA\underset{\bA}\otimes \bC^l\simeq \bC^l.$$
\end{prop}

\sssec{}   \label{iota funct}

We construct a DG functor 
$\iota:\bC^l\to \bA\underset{\bA}\otimes \bC^l$
as the composition
$$\bC^l\to \Bar^0(\bA,\bA,\bC^l)\hookrightarrow \Bar(\bA,\bA,\bC^l)\to 
\bC^l\underset{\bA}\otimes \bC^l,$$
where the first functor is
$$Y^l\mapsto ({\bf 1}_\bA,Y^l), \text{ where } 
{\bf 1}_\bA\in \bA\simeq \bC^r.$$

\sssec{}

We define a DG pseudo functor $F:\Bar(\bA,\bA,\bC^l)\to \bC^l$
as follows. For an object 
$(X^0,\ol{X}^n,Y^l):=(X^0,\underset{n}{\underbrace{X^1,...,X^n}},Y^l)$
of $\Bar^n(\bA,\bA,\bC^l)$
and $Y^l_1\in \bC^l$ we set
$$\Hom^\bullet\left("F(X^0,\ol{X}^n,Y^l)",Y^l_1\right):=
\Hom^\bullet_{\bA,\bC^l}("X^0\otimes X^1\otimes...\otimes X^n\otimes Y^l",Y_1^l).$$

This assignment is clearly a pseudo functor
$\Bar(\bA,\bA,\bC^l)^{non-pretr}\to \bC^l$, which uniquely extends
to a pseudo functor $\Bar(\bA,\bA,\bC^l)\to \bC^l$.

\begin{lem}
The above pseudo functor $F$ is a homotopy functor.
\end{lem}

\begin{proof}
By construction, it suffices to see that for $(X^0,\ol{X}^n,Y^l)\in
\Bar^n(\bA,\bA,\bC^l)$, the corresponding $\bC^l$-module is
corepresentable up to homotopy. But this is ensured by the fact that
the DG pseudo action of $\bA$ on $\bC^l$ was a homotopy action.
\end{proof}

We claim that the resulting 1-morphism $\Bar(\bA,\bA,\bC^l)\to \bC^l$
in $\DGCat$ gives rise to a 
1-morphism $F':\bA\underset{\bA}\otimes \bC^l\to \bC^l$. Indeed,
by \secref{univ ppty quotient}, it suffices to check that the functor
$$\Ho(F):\Ho(\bI_{\bA,\bA,\bC^l})\to \Ho(\bC^l)$$
is $0$, which follows from the definition of $\Ho(\bI_{\bA,\bA,\bC^l})$.
  
\sssec{}

We claim now that $\iota$ and $F'$ are mutually quasi-inverse
1-morphisms in $\DGCat$. 

\medskip

The fact that $F'\circ \iota\simeq
\on{Id}_{\bC^l}$ is evident from the construction. We will now show
that $\Ho(F')$ is the left adjoint of $\Ho(\iota)$. For $Y_1^l\in \Ho(\bC^l)$ and $Z\in \Ho(\bA\underset{\bA}\otimes \bC^l)$
consider the map
\begin{multline}  \label{iota F}
\Hom_{\Ho(\bA\underset{\bA}\otimes \bC^l)}
(Z,\Ho(\iota)(Y^l_1))\to
\Hom_{\Ho(\bC^l)}(\Ho(F')(Z),\Ho(F')\circ \Ho(\iota)(Y^l_1))\simeq \\
\simeq \Hom_{\Ho(\bC^l)}(\Ho(F')(Z),Y^l_1).
\end{multline}

We claim that this map is an isomorphism. To check this we can assume
that $Z$ comes from an object $(X^0,\underset{n}{\underbrace{X^1,...,X^n}},Y^l)$
of $\Bar^n(\bA,\bA,\bC^l)$. In this case we will construct a map inverse
to \eqref{iota F}. Namely, consider
\begin{multline*}
\Hom_{\Ho(\bC^l)}(\Ho(F')(Z),Y^l_1):=
H^0\left(\Hom^\bullet_{\bC^l}("X^0\otimes X^1\otimes...\otimes X^n\otimes Y^l",Y_1^l)\right)\to \\
\to H^0\left(\Hom^\bullet_{\phi}\left(({\bf 1}_\bA,X^0,
X^1,...,X^n,Y^l),({\bf 1}_\bA,Y_1^l)\right)\right)\hookrightarrow \\
\hookrightarrow 
H^0\left(\Hom^\bullet_{\Bar(\bA,\bA,\bC^l)}\left(({\bf 1}_\bA,X^0,
X^1,...,X^n,Y^l),({\bf 1}_\bA,Y_1^l)\right)\right)\simeq \\
\simeq 
\Hom_{\Ho(\Bar(\bA,\bA,\bC^l))}\left(({\bf 1}_\bA,X^0,
X^1,...,X^n,Y^l),({\bf 1}_\bA,Y_1^l)\right),
\end{multline*}
where $\phi$ is the map $0\in \{0\}\mapsto 0\in \{0,1,...,n+1\}$.

\medskip

We compose the above map with
\begin{multline*}
\Hom_{\Ho(\Bar(\bA,\bA,\bC^l))}\left(({\bf 1}_\bA,X^0,
X^1,...,X^n,Y^l),({\bf 1}_\bA,Y_1^l)\right)\to \\
\to
\Hom_{\Ho(\bA\underset{\bA}\otimes \bC^l)}\left(({\bf 1}_\bA,X^0,
X^1,...,X^n,Y^l),({\bf 1}_\bA,Y_1^l)\right)\simeq \\
\simeq
\Hom_{\Ho(\bA\underset{\bA}\otimes \bC^l)}\left((X^0,
X^1,...,X^n,Y^l),({\bf 1}_\bA,Y_1^l)\right),
\end{multline*}
the latter isomorphism is due to the canonical isomorphism 
between $({\bf 1}_\bA,X^0,X^1,...,X^n,Y^l)\in \Bar^{n+1}(\bA,\bA,\bC^l)$
and $(X^0,X^1,...,X^n,Y^l)\in \Bar^{n}(\bA,\bA,\bC^l)$ as objects
of $\Ho(\bA\underset{\bA}\otimes \bC^l)$.

\sssec{}

We are now ready to finish the proof of \propref{ten with free}. 
Consider the adjunction map
\begin{equation} \label{iota F adj}
\on{Id}_{\Ho(\bA\underset{\bA}\otimes \bC^l)}\to \Ho(\iota)\circ \Ho(F'),
\end{equation}
and it suffices to show that this map is an isomorphism. For that it
is sufficient to evaluate it on objects $Z$ of the form
$$(X^0,\underset{n}{\underbrace{X^1,...,X^n}},Y^l)\in
\Bar^n(\bA,\bA,\bC^l).$$

By construction, for such an object the map in \eqref{iota F adj} is
represented by the following "hut" in $\Bar(\bA,\bA,\bC^l)$:
$$(X^0,X^1,...,X^n,Y^l)\leftarrow
({\bf 1}_\bA,X^0,X^1,...,X^n,Y^l) \to ({\bf 1}_\bA,Y^l_1),$$
where $Y^l_1\in \bC^l$ is an object, whose image in $\Ho(\bC^l)$
is isomorphic to $X^0\otimes X^1\otimes...\otimes X^n\otimes Y^l$.
Since the arrow $\to$ in the above formula is a quasi-isomorphism,
the assertion follows.

\sssec{}

We note that the same argument proves the following generalization of 
\propref{ten with free}. Namely,
let $\bC_1$ be a DG category, and let us consider
$\bC_1\otimes \bA$, which has natural right homotopy module structure.
We have:
\begin{equation} \label{ten free gen}
(\bC_1\otimes \bA)\underset{\bA}\otimes \bC^l\simeq 
\bC_1\otimes \bC^l.
\end{equation}

As a particular case we obtain that for any two DG categories 
$\bC_1$ and $\bC_2$, which can be seen as acted on by the DG 
monoidal category $\Comp^f_k$, we have:
$$\bC_1\underset{\Comp_k^f}\otimes \bC_2\simeq \bC_1\otimes \bC_2.$$

\ssec{}

Let $\bA_1$ and $\bA_2$ be two homotopy monoidal categories. 
We have the tautological DG monoidal functors
$$\bA_1\to \bA_1\otimes \bA_2 \leftarrow \bA_2.$$
Assume that in the situation 
of \secref{def ten prod} the right action 
of $\bA=:\bA_2$ on $\bC^r$ has
been extended to a right action of $\bA^o_1\otimes \bA_2$,
where the superscript "$o$"
stands for the opposite homotopy monoidal structure.

\medskip

Then, by \secref{tensor free}, the DG category
$\Bar(\bC^r,\bA_2,\bC^l)$ carries an action
of $\bA_1$. 
The subcategory $\bI_{\bC^r,\bA_2,\bC^l}\subset
\Bar(\bC^r,\bA_2,\bC^r)$ has the property that its
image in $\Ho(\Bar(\bC^r,\bA_2,\bC^l))$ is preserved 
by the action of $\Ho(\bA_1)$. 

\medskip

Hence, by 
\lemref{action on quotient}, we obtain that
$\bC^{r}\underset{\bA_2}\otimes \bC^l$ is a well-defined
object of $\DGMod(\bA_1)$

\sssec{}

Suppose in addition that $\bC^r_1$ is a DG category, equipped
with a right homotopy action of $\bA_1$. From the constriction
we have the following natural 1-isomorphisms in $\DGCat$:
$$\Bar(\bC^r_1,\bA_1,\Bar(\bC^r,\bA_2,\bC^l))\simeq
\Bar((\bC^r_1,\bA_1,\bC^r),\bA_2,\bC^l)$$
and 
\begin{equation} \label{ass ten}
\bC^r_1\underset{\bA_1}\otimes (\bC^{r}\underset{\bA_2}\otimes \bC^l)
\simeq 
(\bA_1^{r}\underset{\bA_1}\otimes \bC^r)\underset{\bA_2}\otimes \bC^l.
\end{equation}

In other words, we have a well-defined objects of $\DGCat$:
$\Bar(\bC^r_1,\bA_1,\bC^r,\bA_2,\bC^l)$ and
$$\bC^r_1\underset{\bA_1}\otimes \bC^{r}\underset{\bA_2}\otimes \bC^l.$$

\sssec{}

In the above situation let us change the notations slightly and denote
$\bC^l_2:=\bC^l$ and $\bC:=\bC^r$. We have:

\begin{lem}   \label{ten with diag}
Under the above circumstances, we have a canonical 1-equivalence
in $\DGCat$:
$$\bC^r_1\underset{\bA_1}\otimes \bC\underset{\bA_2}\otimes \bC^l_2
\simeq 
(\bC^r_1\otimes \bC^l_2)\underset{\bA_1\otimes \bA^o_2}\otimes \bC.$$
\end{lem}

\begin{proof}

Consider the categories $$\Bar(\bC^r_1,\bA_1,\bC,\bA_2,\bC^l_2) \text{ and }
\Bar\left((\bC^r_1\otimes \bC^l_2),(\bA_1\otimes \bA^o_2),\bC\right).$$
We construct DG functors in both directions as follows:

For $\to$ we send
\begin{multline*}
(Y_1^r,X^1_1,...,X^n_1,Y,X^1_2,...,X^m_2,Y^l_2)\mapsto \\
\left((Y^r_1,Y^l_2),(X^1_1,{\bf 1}_{\bA_2}),...,(X^n_1,{\bf 1}_{\bA_2}),
({\bf 1}_{\bA_1},X^m_2),...,({\bf 1}_{\bA_1},X^1_2),Y\right).
\end{multline*}
For $\leftarrow$ we send
$$\left((Y^r_1,Y^l_2),(X^1_1,X^1_2),...,(X^k_1,X^k_2),Y\right)\mapsto
(Y^r_1,X^1_1,...,X^k_1,Y,X^k_2,...,X^1_2,Y^l_2).$$

These functors are easily seen to descend to 1-morphisms
$$\bC^r_1\underset{\bA_1}\otimes \bC\underset{\bA_2}\otimes \bC^l_2
\leftrightarrows
(\bC^r_1\otimes \bC^l_2)\underset{\bA_1\otimes \bA^o_2}\otimes \bC,$$
which are mutually quasi-inverse.

\end{proof}

\ssec{Induction}   

Let $F_\bA:\bA_1\to \bA_2$ be a homotopy monoidal functor
between homotopy monoidal categories. Consider
the 0-object
$$\Res^{\bA_2\otimes \bA_2^o}_{\bA_2\otimes \bA_1^o}(\bA_2)\in 
\DGMod(\bA_2\otimes \bA_1^o).$$

\sssec{}    \label{def induction}

For a DG category $\bC_1^l$ equipped with a homotopy action of
$\bA_1$ we define
$$\Ind^{\bA_2}_{\bA_1}(\bC_1^l):=\Res^{\bA_2\otimes \bA_2^o}_{\bA_2\otimes \bA_1^o}(\bA_2)\underset{\bA_1}\otimes \bC_1^l,$$
as an object of $\DGMod(\bA_2)$.

\medskip

Moreover, for a homotopy functor $\bC^l_1\to \wt\bC^l_1$ compatible
with an action of $\bA_1$, from \secref{ten and funct} we obtain
a homotopy functor 
$$\Ind^{\bA_2}_{\bA_1}(\bC^l_1)\to \Ind^{\bA_2}_{\bA_1}(\wt\bC^l_1).$$

In other words, the above construction defines a 2-functor
$\DGMod(\bA_1)\to \DGMod(\bA_2)$.

\sssec{}   \label{triv ind}

Let us consider a particular case when $\bA_2=\bA_1$. We claim
that the above 2-functor of induction is 1-isomorphic to the
identity functor. Namely, we claim that the map
$$\iota:\bC_1^l\to \bA_1\underset{\bA_1}\otimes \bC_1^l$$
of \secref{iota funct} has a natural structure of homotopy
functor compatible with the action of $\bA_1$. The above
map is a 1-isomorphism by \propref{ten with free}.

\sssec{}    \label{univ ppty base change}

Let $\bC_2^l$ be a 0-object of $\DGMod(\bA_2)$, and consider
the corresponding 0-object $\Res^{\bA_2}_{\bA_1}(\bC_2^l)\in 
\DGMod(\bA_1)$.

\begin{prop}    \label{induction adj}
There exists an equivalence of categories
$$\on{HFunct}_{\bA_1}(\bC_1^l,\Res^{\bA_2}_{\bA_1}(\bC_2^l))\simeq
\on{HFunct}_{\bA_2}(\Ind^{\bA_2}_{\bA_1}(\bC_1^l),\bC_2^l).$$
\end{prop}

\begin{proof}

We need to construct the adjunction 1-morphisms
\begin{equation}  \label{one adj}
\bC^l_1\mapsto \Res^{\bA_2}_{\bA_1}\left(\Ind^{\bA_2}_{\bA_1}(\bC^l_1)\right)
\end{equation}
and
\begin{equation}  \label{two adj}
\Ind^{\bA_2}_{\bA_1}\left(\Res^{\bA_2}_{\bA_1}(\bC_2^l)\right)\to \bC_2^l.
\end{equation}

The former follows by the functoriality of the tensor product construction
from the 1-morphism 
$$\bA_1\to \Res^{\bA_2\otimes \bA_1^o}_{\bA_1\otimes \bA_1^o}
\left(\Res^{\bA_2\otimes \bA_2^o}_{\bA_2\otimes \bA_1^o}(\bA_2)\right)
\simeq \Res^{\bA_2\otimes \bA_2^o}_{\bA_1\otimes \bA_1^o}(\bA_2)
\in \DGMod(\bA_1\otimes \bA_1^o),$$
which in turn results from the functor $F_{\bA}:\bA_1\to \bA_2$, viewed
as compatible with $F_\bA\otimes F_\bA$.

\medskip

The 1-morphism \eqref{two adj} results from \eqref{funct of ten}:
$$\Res^{\bA_2\otimes \bA_2^o}_{\bA_2\otimes \bA_1^o}(\bA_2)
\underset{\bA_1}\otimes \Res^{\bA_2}_{\bA_1}(\bC^l_2)\to
\bA_2\underset{\bA_2}\otimes \bC^l_2\simeq \bC^l_2.$$

\medskip

The fact that these 1-morphisms satisfy the adjunction property
follows from \eqref{ass ten}.

\end{proof}

\sssec{}  \label{transitivity of unduction}

Suppose that $\bA_3$ is a third homotopy monoidal category,
equipped with a homotopy monoidal functor $\bA_2\to \bA_3$.

\medskip

From \propref{induction adj} and \lemref{compose restrictions},
we obtain a 1-isomorphism of two 2-functors
$$\DGMod(\bA_1)\rightrightarrows \DGMod(\bA_3):
\Ind^{\bA_3}_{\bA_2}\circ \Ind^{\bA_2}_{\bA_1}\simeq \Ind^{\bA_3}_{\bA_1}.$$

\section{Adjunctions and tightness} \label{adjunctions and tightness}

\ssec{}   \label{upper and lower star}

In this subsection we will fix some notation.
Let $\bA$, $\bC^l$ and $\bC^r$ be as in \secref{def ten prod}. By
construction, we have a canonical 1-morphism
$$m_{\bC^r,\bC^l}:\bC^r\otimes \bC^l\to \bC^r\underset{\bA}\otimes \bC^l.$$
Consider the corresponding functors
$$m_{\bC^r,\bC^l}^*:\bD((\bC^r\times \bC^l)^{op}\mod)\to
\bD((\bC^r\underset{\bA}\otimes \bC^l)^{op}\mod)$$ and  its right adjoint
$$(m_{\bC^r,\bC^l})_*:\bD((\bC^r\underset{\bA}\otimes \bC^l)^{op}\mod)\to
\bD((\bC^r\times \bC^l)^{op}\mod).$$

We will also use the notation 
$$(Y^r,Y^l)\mapsto Y^r\underset{\bA}\otimes Y^l:=
m_{\bC^r,\bC^l}^*(Y^r,Y^l)$$
for $Y^r\in \bD(\bC^{r,op}\mod)$ and $Y^l\in
\bD(\bC^{l,op}\mod)$.

\medskip

A particular case of the above situation is when we have 
a homotopy monoidal category $\bA$ equipped with a homotopy
action on a DG category $\bC^l$. In particular, we have a homotopy
functor $$act:\bA\otimes \bC^l\to \bC^l.$$

We will consider the corresponding functor
$$act^*:\bD((\bA\times\bC^l)^{op}\mod)\to 
\bD(\bC^{l,op}\mod)$$ and its right adjoint 
$$act_*: \bD(\bC^{l,op}\mod) \to \bD((\bA\times \bC^l)^{op}\mod).$$

We will also use the notation
$$(X,Y)\mapsto X\underset{\bA}\otimes Y:=act^*(X,Y)$$
for $X\in \bD(\bA^{op}\mod)$, $Y\in \bD(\bC^{l,op}\mod)$.

\ssec{} \label{intr tightness}

Let now $F:\bC^l_1\to \bC^l_2$ be a homotopy functor compatible
with a homotopy action of $\bA$. The following diagram of functors
evidently commutes:
$$
\CD
\bD((\bA\times\bC_1^l)^{op}\mod)  @>{act^*_{\bC^l_1}}>> \bD(\bC_1^{l,op}\mod) \\
@V{(\on{Id}_\bA\times F)^*}VV   @V{F^*}VV   \\
\bD((\bA\times\bC_2^l)^{op}\mod)  @>{act^*_{\bC^l_2}}>> \bD(\bC_2^{l,op}\mod).
\endCD
$$

The next diagram, however, does {\it not} necessarily commute:
\begin{equation} \label{def up tight}
\CD
\bD((\bA\times\bC_1^l)^{op}\mod)  @>{act^*_{\bC^l_1}}>> \bD(\bC_1^{l,op}\mod) \\
@A{(\on{Id}_\bA\times F)_*}AA   @A{F_*}AA   \\
\bD((\bA\times\bC_2^l)^{op}\mod)  @>{act^*_{\bC^l_2}}>> \bD(\bC_2^{l,op}\mod).
\endCD
\end{equation}


However, we have the natural transformation:
\begin{equation} \label{upper tight}
act^*_{\bC^l_1}\circ (\on{Id}_\bA\times F)_*\to
F_*\circ act^*_{\bC^l_2}:\bD((\bA\times\bC_2^l)^{op}\mod) \rightrightarrows
\bD(\bC_1^{l,op}\mod).
\end{equation}

\medskip

We shall say that the functor $F$ is {\it tight} if \eqref{upper tight}
is an isomorphism. 

\ssec{}

Let $\bA,\bC^r,\bC^l_1,\bC^l_2,F:\bC^l_1\to \bC^l_2$ be as above.
The following diagram of functors tautologically commutes:
$$
\CD
\bD((\bC^r\times \bC_1^l)^{op}\mod)  @>{m_{\bC^r,\bC_1^l}^*}>> 
\bD((\bC^r\underset{\bA}\otimes \bC_1^l)^{op}\mod) \\
@V{(\on{Id}_{\bC^r}\times F)^*}VV    
@V{(\on{Id}_{\bC^r}\underset{\bA}\otimes F)^*}VV   \\
\bD((\bC^r\times \bC_2^l)^{op}\mod)   @>{m_{\bC^r,\bC_2^l}^*}>>
\bD((\bC^r\underset{\bA}\otimes \bC_2^l)^{op}\mod).
\endCD
$$

However, the diagram
\begin{equation} \label{upper tightness diagram}
\CD
\bD((\bC^r\times \bC_1^l)^{op}\mod)  @>{m_{\bC^r,\bC_1^l}^*}>> 
\bD((\bC^r\underset{\bA}\otimes \bC_1^l)^{op}\mod) \\
@A{(\on{Id}_{\bC^r}\times F)_*}AA    
@A{(\on{Id}_{\bC^r}\underset{\bA}\otimes F)_*}AA  \\
\bD((\bC^r\times \bC_2^l)^{op}\mod)   @>{m_{\bC^r,\bC_2^l}^*}>>
\bD((\bC^r\underset{\bA}\otimes \bC_2^l)^{op}\mod)
\endCD
\end{equation}
{\it does not} a priori commute. However, we have a natural transformation:

\begin{equation} \label{upper proj form ten}
m_{\bC^r,\bC_1^l}^*\circ (\on{Id}_{\bC^r}\times F)_*\to 
(\on{Id}_{\bC^r}\underset{\bA}\otimes F)_*\circ m_{\bC^r,\bC_2^l}^*.
\end{equation}

\begin{prop}  \label{tight prop}
The following conditions are equivalent:

\smallskip

\noindent(a) $F$ is tight.

\smallskip

\noindent(b) For any $\bC^r$ the natural transformation
\eqref{upper proj form ten} is an isomorphism.

\end{prop}

\begin{proof}

For (b)$\Rightarrow$(a) let us take $\bC^r=\bA$. Then the
diagram \eqref{upper tightness diagram} coincides with
\eqref{def up tight}. 

\medskip

For (b)$\Rightarrow$(a) consider first the diagram
$$
\CD
\bD((\bC^r\times \bC_1^l)^{op}\mod)  @>{\iota_{\bC^r,\bC^l_1}^*}>> 
\bD(\Bar(\bC^r,\bA,\bC_1^l)^{op}\mod) \\
@A{(\on{Id}_{\bC^r}\times F)_*}AA    
@A{\Bar(\on{Id}_{\bC^r},\on{Id}_\bA,F)_*}AA  \\
\bD((\bC^r\times \bC_2^l)^{op}\mod)   @>{\iota_{\bC^r,\bC^l_2}^*}^*>>
\bD(\Bar(\bC^r,\bA,\bC_2^l)^{op}\mod),
\endCD
$$
where $\iota_{\bC^r,\bC^l}$ denotes the canonical functor
$\bC^r\otimes \bC^l\to \Bar(\bC^r,\bA,\bC_1^l)$. It is easy to see
that the above diagram commutes. 

\medskip

Hence, it is enough to show that the following diagram
$$
\CD
\bD(\Bar(\bC^r,\bA,\bC_1^l)^{op}\mod) @>>> 
\bD((\bC^r\underset{\bA}\otimes \bC_1^l)^{op}\mod) \\
@A{\Bar(\on{Id}_{\bC^r},\bA,F)_*}AA  @A{(\on{Id}_{\bC^r}\underset{\bA}\times F)_*}AA \\
\bD(\Bar(\bC^r,\bA,\bC_1^2)^{op}\mod) @>>>
\bD((\bC^r\underset{\bA}\otimes \bC_2^l)^{op}\mod) 
\endCD
$$
is commutative, where the horizontal arrows are the ind-limits of the
tautological projections of the "Bar" categories to the tensor products.

\medskip

The latter is equivalent to the fact that the functor 
$(\on{Id}_{\bC^r}\underset{\bA}\times F)_*$ sends 
$\ua\bI{}_{\bC^r,\bA,\bC^l_1}$ to $\ua\bI{}_{\bC^r,\bA,\bC^l_2}$,
which in turn follows from the tightness condition.

\end{proof} 

\begin{cor} \label{total tightness}
Let $\bC^r_i$, $\bC^l_i$, $i=1,2$ be two pairs of categories as above,
and $F^r:\bC^r_1\to \bC^r_2$, $F^l:\bC^l_1\to \bC^l_2$ be homotopy
functors. Assume that the functor $F^r$ is tight.
Then the following diagram of functors commutes:
$$
\CD
\bD\left((\bC_1^r\underset{\bA}\otimes \bC^l_1)^{op}\mod\right) 
@>{(\on{Id}_{\bC_1^r}\underset{\bA}\otimes F^l)^*}>>
\bD\left((\bC_1^r\underset{\bA}\otimes \bC^2_1)^{op}\mod\right) \\
@A{(F^r\underset{\bA}\otimes \on{Id}_{\bC_1^l})_*}AA   
@A{(F^r\underset{\bA}\otimes \on{Id}_{\bC_2^l})_*}AA  \\
\bD\left((\bC_2^r\underset{\bA}\otimes \bC^l_1)^{op}\mod\right)  
@>{(\on{Id}_{\bC_2^r}\underset{\bA}\otimes F^l)^*}>>
\bD\left((\bC_2^r\underset{\bA}\otimes \bC^l_2)^{op}\mod\right).
\endCD
$$
\end{cor}

\ssec{}

Let us consider some examples of the above situation.

\sssec{}

Let $\bA_1$ and $\bA_2$ be two homotopy monoidal categories
and $F_\bA:\bA_1\to \bA_2$ a homotopy monoidal functor between them.
We say that $F_\bA$ is tight if the canonical 1-morphism $\bA_1\to
\Res^{\bA^o_2}_{\bA^o_1}(\bA^o_2)$ is tight as a functor between categories,
acted on the {\it right} by $\bA_1$. Recall also that as plain DG categories 
$\Res^{\bA^o_2}_{\bA^o_1}(\bA^o_2)\simeq \bA_2^o$.

\medskip

Consider the functors 
$$F_\bA^*:\bD(\bA_1^{op}\mod)\to \bD(\bA_2^{op}\mod) \text{ and }
F_\bA{}_*:\bD(\bA_2^{op}\mod)\to \bD(\bA_1^{op}\mod).$$

From \propref{tight prop}
we obtain the following:

\begin{cor}  \label{tight mon}
Let $\bC_1^l$ be a category equipped with a homotopy action of $\bA_1$
on the left. Assume that $F_\bA$ is tight. 

\smallskip

\noindent(1) The following diagram of functors
is commutative:
$$
\CD
\bD((\bA_1\otimes \bC_1^l)^{op}\mod)   @>{act^*_{\bC^l}}>>
\bD(\bC_1^{l,op}\mod)   \\
@A{(F_\bA\times \on{Id}_{\bC_1^l})_*}AA 
@A{(F_{\bC_1^l})_*}AA     \\
\bD((\bA_2\otimes \bC_1^l)^{op}\mod)
@>{act^*_{\Ind^{\bA_2}_{\bA_1}(\bC_1^l)}\circ 
(\on{Id}_{\bA_2}\otimes F_{\bC_1^l})}>>
\bD((\Ind^{\bA_2}_{\bA_1}(\bC_1^l))^{op}\mod),
\endCD
$$
where $F_{\bC_1^l}$ denotes the canonical 1-morphism $\bC_1^l\to 
\Ind^{\bA_2}_{\bA_1}(\bC_1^l)$.

\smallskip

\noindent(2)
For $Y_1,Y_2\in \bC_1^l$, $X\in \bA_2$,
$$\Hom_{\Ho(\Ind^{\bA_2}_{\bA_1}(\bC^l_1))}
(F_{\bC_1^l}^*(Y_1),X\underset{\bA_2}\otimes F_{\bC_1^l}^*(Y_2))
\simeq
\Hom_{\bD(\bC_1^{l,op}\mod)}(Y_1,F_\bA{}_*(X)\underset{\bA_1}\otimes Y_2).$$
\end{cor}

\begin{cor}   \label{ind low tight}
Let $F_\bA:\bA_1\to \bA_2$, be as in \corref{tight mon}. 
Let $'\bC_1^l,{}''\bC_1^l$ be DG categories equipped with
a homotopy action of $\bA_1$.
Let $F:{}'\bC_1^l\to {}''\bC_1^l$ be a homotopy functor. Then the
diagram of functors
$$
\CD
\bD((\Ind^{\bA_2}_{\bA_1}({}'\bC^l_1))^{op}\mod)  
@>{(\Ind^{\bA_2}_{\bA_1}(F))^*}>>  
\bD((\Ind^{\bA_2}_{\bA_1}({}''\bC^l_1))^{op}\mod)  \\
@V{(F_{{}'\bC^l_1})_*}VV    @V{(F_{{}''\bC^l_2})_*}VV  \\
\bD(({}'\bC^l_1)^{op}\mod) @>{F^*}>>   \bD(({}''\bC^l_1)^{op}\mod) 
\endCD
$$
commutes.
\end{cor}

\sssec{}

We will say that $\bA$ {\it has a tight diagonal}
if the tensor product homotopy functor
$m_\bA:\bA\otimes \bA\to \bA$, considered as a homotopy
functor between categories endowed with a homotopy action of
$\bA\otimes \bA^o$ is tight.

\medskip

Assume that $\bA$ has a tight diagonal. We obtain that for
any $\bC^l$ and $\bC^r$ as above the following 
diagram of functors commutes:
$$
\CD
\bD\left((\bC^r\otimes \bC^l\otimes \bA\otimes \bA)^{op}\mod\right)   
@>{(act_{\bC^r}\otimes act_{\bC^l})^*}>>  \bD((\bC^r\otimes \bC^l)^{op}\mod) \\
@A{(\on{Id}_{\bC^r\times \bC^l}\otimes m_\bA)_*}AA   
@A{(m_{\bC^r,\bC^l})_*}AA   \\
\bD((\bC^r\otimes \bC^l\otimes \bA)^{op}\mod) @>>>
\bD((\bC^r\underset{\bA}\otimes \bC^l)^{op}\mod),
\endCD
$$
where the lower horizontal arrow is induced by either of the two
compositions
$$
\CD
\bC^r\otimes \bA\otimes  \bC^l  @>{\on{Id}_{\bC^r}\otimes act_{\bC^l}}>>
\bC^r\otimes \bC^l \\
@V{act_{\bC^r}\otimes\on{Id}_{\bC^l}}VV    @V{m_{\bC^r,\bC^l}}VV   \\
\bC^r\otimes \bC^l  @>{m_{\bC^r,\bC^l}}>>  \bC^r\underset{\bA}\otimes \bC^l.
\endCD
$$

Let $\on{Diag}_{\bA}$ denote the object of $\bD((\bA\otimes \bA)^{op}\mod)$
equal to $(m_\bA)_*({\bf 1}_\bA)$. We obtain:
\begin{cor}  \label{tight diagonal}
If $\bA$ has a tight diagonal, for $Y^l_1,Y^l_2\in \bC^l$, $Y^r_1,Y^r_2\in \bC^r$
we have:
\begin{multline*}
\Hom_{\Ho(\bC^r\underset{\bA}\otimes \bC^l)}
\left(Y^r_1\underset{\bA}\otimes Y^l_1,Y^r_2\underset{\bA}\otimes Y^l_2\right)
\simeq \\
\simeq \Hom_{\bD((\bC^r\times \bC^l)^{op}\mod)}\left((Y^r_1,Y^l_1),
(\on{Diag}_{\bA})\underset{\bA\otimes \bA}
\otimes (Y^r_2,Y^l_2)\right).
\end{multline*}
\end{cor}

\sssec{}   \label{dir image}

Let us assume now that in the situation of \propref{tight prop}, the functor
$F_*$ sends $\Ho(\bC^l_2)$, regarded a full subcategory of $\bD(\bC^{l,op}_2\mod)$ to
$\Ho(\bC^l_1)$, regarded as a subcategory of $\bD(\bC^{l,op}_1\mod)$.
By \lemref{action on quotient}, the resulting functor, denoted $G:
\Ho(\bC^l_2)\to \Ho(\bC^l_1)$ naturally lifts to a 1-morphism in $\DGMod(\bA)$.

We can then consider the functor
$$(\on{Id}_{\bC^r}\underset{\bA}\otimes G):\bC^r\underset{\bA}\otimes \bC_2^l\to
\bC^r\underset{\bA}\otimes \bC_1^l,$$
and its ind-extension
$$(\on{Id}_{\bC^r}\underset{\bA}\otimes G)^*:
\bD((\bC^r\underset{\bA}\otimes \bC_2^l)^{op}\mod)\to
\bD((\bC^r\underset{\bA}\otimes \bC_1^l)^{op}\mod).$$

\begin{prop}
Suppose $F$ is tight. Then at the triangulated level,
the functor $(\on{Id}_{\bC^r}\underset{\bA}\otimes G)^*$ is the right adjoint of
$(\on{Id}_{\bC^r}\underset{\bA}\otimes F)^*$, i.e., we have an isomorphism
of functors at the triangulated level:
$$(\on{Id}_{\bC^r}\underset{\bA}\otimes G)^*\simeq (\on{Id}_{\bC^r}\underset{\bA}\otimes F)_*.$$. 
\end{prop}

\begin{proof}

We have an evidently defined 2-morphism $\on{Id}_{\ua\bC{}^l_1}\to
F_*\circ F^*$ in $\DGMod(\bA)$, and by \secref{action on subcat},
also a 2-morphism $\on{Id}_{\bC^l_1}\to G\circ F$. The latter gives
rise to a 2-morphism
$$\on{Id}_{\bC^r\underset{\bA}\otimes \bC_1^l}\to
(\on{Id}_{\bC^r}\underset{\bA}\otimes G)\circ (\on{Id}_{\bC^r}\underset{\bA}\otimes F),$$
and to a morphism
$$\on{Id}_{\bD((\bC^r\underset{\bA}\otimes \bC_1^l)^{op}\mod)}\to
(\on{Id}_{\bC^r}\underset{\bA}\otimes G)^*\circ (\on{Id}_{\bC^r}\underset{\bA}\otimes F)^*,$$

Thus, for $\wt{X_1}\in \bD((\bC^r\underset{\bA}\otimes \bC_1^l)^{op}\mod)$ and
$\wt{X_2}\in \bD((\bC^r\underset{\bA}\otimes \bC_1^l)^{op}\mod)$ we obtain a map
$$\Hom((\on{Id}_{\bC^r}\underset{\bA}\otimes F)^*(\wt{X_1}),\wt{X_2})\to
\Hom(\wt{X_1}, (\on{Id}_{\bC^r}\underset{\bA}\otimes G)^*(\wt{X_2}),$$
and we have to show that the latter is an isomorphism. 

\medskip

With no restriction of generality we can assume that $$\wt{X_1}\in \Ho(\bC^r\underset{\bA}\otimes \bC_1^l)
\text{ and } \wt{X_2}\in \Ho(\bC^r\underset{\bA}\otimes \bC_2^l).$$ Further, we can assume that
$\wt{X_2}$ is of the form $X^r\underset{\bA}\otimes X_2$ with $X^r\in \Ho(\bC^r)$ and
$X_2\in \Ho(\bC^l_2)$. 

\medskip

In the latter case the desired isomorphism follows from \propref{tight prop}.

\end{proof}

\ssec{Rigidity}   \label{rigidity}

Here is a way to insure that any functor $F$ is tight.
We shall say that
$\bA$ is rigid of the triangulated monoidal category $\Ho(\bA)$ has this
property. I.e., if there exists a self anti-equivalence of $\Ho(\bA):X\mapsto X^\vee$,
and maps ${\bf 1}_\bA\to X\otimes X^\vee$ and $X^\vee\otimes X\to {\bf 1}_\bA$
such that the two compositions
$$X\to X\underset{\bA}\otimes X^\vee\underset{\bA}\otimes X\to X \text{ and }
X^\vee\to X^\vee\underset{\bA}\otimes X\underset{\bA}\otimes X^\vee\to X^\vee$$
are the identity maps in $\Ho(\bA)$.

\begin{lem}  \label{rigid adjunction}
For $\bC$ endowed with a homotopy action of $\bA$ and $X\in \bA$
as above the functor $act_\bC^*(X,?)$ is the right adjoint of $act_\bC^*(X^\vee,?)$.
\end{lem}

We have:

\begin{prop}   \label{rigid is tight}
If $\bA$ is rigid, then any functor $F:\bC^l_1\to \bC^l_2$ is 
tight.
\end{prop}

\begin{proof}
It is sufficient to show that the map
$$X\underset{\bA}\otimes F_*(Y)\to F_*(X \underset{\bA}\otimes Y)$$
is an isomorphism for any $X\in \Ho(\bA)$, $Y\in \bD(\bC_2^{l,op}\mod)$.

\medskip

In this case, we we will construct the inverse map to the one above.
By \lemref{rigid adjunction}, constructing a map
$$F_*(X\underset{\bA}\otimes Y)\to X \underset{\bA}\otimes F_*(Y)$$
is equivalent to constructing a map 
$$X^\vee \underset{\bA}\otimes F_*(X\underset{\bA}\otimes Y)\to
F_*(Y).$$
The latter equals the composition
$$
X^\vee \underset{\bA}\otimes F_*(X\underset{\bA}\otimes Y)\to
F_*\left(X^\vee \underset{\bA}\otimes (X\underset{\bA}\otimes Y)\right)\simeq
F_*\left((X^\vee\underset{\bA}\otimes X)\underset{\bA}\otimes Y\right)\to
F_*(Y).$$

\end{proof}

\section{$\bbt$-structures: a reminder}   \label{t-structures: a reminder}

\ssec{}  \label{bdd subcat}

Recall the notion of t-structure on a triangulated category.  Given
a t-structure on $\bD$ we will use the standard notations:
$$\bD^+:=\underset{k}\cup\, \bD^{\geq k},\,\,
\bD^-:=\underset{k}\cup\, \bD^{\leq k},\,\, \bD^b:=\bD^+\cap \bD^-.$$

\sssec{}

If $\bD$ is a triangulated category equipped with a t-structure,
and $\bD'\subset \bD$ is a full triangulated subcategory, we
will say that $\bD'$ is compatible with the t-structure, if it is
preserved by the truncation functors. In this case, $\bD'$
inherits a t-structure: it is the unique t-structure for which the
inclusion functor is exact.

\sssec{}    \label{t colim}

Let $\bD$ be a co-complete triangulated category equipped
with a DG model and a t-structure. We say that the t-structure is 
compatible with colimits if for for every homotopy $I$-object $X_I$ with
$X_i\in \bD^{\leq 0}$ (resp., $X_i\in \bD^{\geq 0}$) for all $i\in I$, we have
$hocolim(X_I)\in \bD^{\leq 0}$ (resp., $hocolim(X_I)\in \bD^{\geq 0}$).

\sssec{}

The following assertion generalizes \lemref{generate colimit}. 
Suppose that $\bD$ is co-complete, is equipped with a DG
model and a t-structure.  Assume that the t-structure is
compatible with colimits.

\medskip

Let $\bD'\subset \bD$ be triangulated
subcategory that generates $\bD$. Note that we are not 
assuming that $\bD'$ is compatible with the t-structure.

\begin{lem}  \label{colimit generate t}
Under the above circumstances every object of 
$\bD^{\leq 0}$ (resp., $\bD^{\geq 0}$) can be represented as
a homotopy colimit of a homotopy $I$-object $X_I$ such that the image
of every $X_i$ in $\bD$ is of the form $\tau^{\leq 0}(X'_i)$
(resp., $\tau^{\geq 0}(X'_i)$) for $X'_i\in \bD'$.
\end{lem}

\ssec{}   \label{compact gener t}

Let us recall a general construction of t-structures on a 
co-complete triangulated category. This construction was explained
to us by J.~Lurie.

\medskip

Let $\bD$ be a co-complete triangulated category. Let $X_a\in \bD$ be a 
collection of compact objects, indexed by some set $A$. \footnote{As was 
explained to us by J.~Lurie, for what follows one does not 
in fact need to require that $X_a$ be compact, if some general
set-theoretic assumption on $\bD$ is satisfied.}

\begin{lem}  \label{t general}  
Under the above circumstances, there exists a unique t-structure
on $\bD$, such that $\bD^{>0}$ consists of all objects $Y$ such that
$\Hom(X_a[k],Y)=0$ for all $k\geq 0$. In this case $\bD^{\leq 0}$
is the minimal subcategory of $\bD$, stable under extensions and
direct sums that contains the objects $X_a[k]$, $k\geq 0$.
\end{lem}

We will call t-structures that arise by the procedure of the above
lemma {\it compactly generated}.
Tautologically, we have:

\begin{lem} \label{comp gen criter}
Let $\bD$ be a co-complete triangulated category equipped with a t-structure,
and with $\bD^c$ (the subcategory of compact objects) essentially small.
Then the t-structure is compactly generated if and only if 
$$X\in \bD^{>0} \Leftrightarrow  \Hom(X',X)=0,\,\,\forall X\in \bD^c\cap \bD^{\leq 0}.$$
\end{lem}

\medskip

The following results immediately from \lemref{compact and limit}:

\begin{lem}
Let $\bD$ be a co-complete triangulated category equipped with
a DG model and a compactly generated t-structure. Then this
t-structure is compatible with colimits.
\end{lem}

\ssec{}

Let us recall that whenever we have a triangulated category $\bD$ equipped with
a t-structure and a DG model, we have an exact functor:
\begin{equation} \label{from der to tr bdd}
\bD^b(\obC)\to \bD^b,
\end{equation}
where $\obC:=\Heart(\bD)$, equipped with a DG model.

\sssec{}  \label{functor der to tr}

Let us recall the construction. Let $\bC$ be a DG category such that
$\bD=\Ho(\bC)$. The DG model for $\bD^b(\Heart(\bD))$ is the standard
one, resulting from the identification
$\Ho(\bC^b(\Heart(\bD)))/\Ho(\bC^b_{acylc}(\Heart(\bD)))$ (see
\secref{rigid quotients}).

\medskip

Consider the following DG category, denoted $\bC^{double}$. Its objects are
finite diagrams
$$\{X^{-\infty}=X^{\geq -n}\supset X^{\geq -n+1}\supset...\supset
X^{\geq m}\supset X^{\geq m+1}=0\}$$
for some $m,n\in \BZ^{\geq 0}$, $X^{\geq i}\in \bC$, such that for each $i$ we are 
given a splitting
$$X^{\geq i}\simeq X^{\geq i+1}\oplus X^i,\,\, X^i\in \bC$$
as functors $\bC^{op}\to \Vect_k^{\BZ}$ (where $\Vect^{\BZ}$ denotes
the category of $\BZ$-graded vector spaces). We require that the image 
of each $X^i$ in $\Ho(\bC)$ belong to $\Heart(\bD)[-i]$. Morphisms 
between $X^\bullet$ and $Y^\bullet$ are compatible families of maps
$X^{\geq i}\to Y^{\geq i}$.

\medskip

We have the evident forgetful functor $\bC^{double}\to \bC$. In addition,
the boundary map for the t-structure defines a DG functor 
$\bC^{double}\to \bC^b(\Heart(\bD))$.

\medskip

Let $\bC^{double}_{acycl}$ be the preimage of $\bC_{acycl}(\Heart(\bD))$
under $\bC^{double}\to \bC^b(\Heart(\bD))$.

\begin{lem}  \hfill  \label{d acycl to acycl}

\smallskip

\noindent{\em(a)}
The functor $$\Ho\left(\bC^{double}\right)/
\Ho\left(\bC_{acycl}^{double}\right)\to \Ho\left(\bC^b(\Heart(\bD))\right)/
\Ho\left(\bC^b_{acycl}(\Heart(\bD))\right)\simeq \bD^b(\Heart(\bD))$$
is an equivalence.

\smallskip

\noindent{\em(b)}
The functor $\Ho\left(\bC^{double}\right)\to \Ho(\bC)$ factors
as 
$$\Ho\left(\bC^{double}\right)\to \Ho\left(\bC^{double}\right)/
\Ho\left(\bC_{acycl}^{double}\right)\to
\Ho(\bC).$$

\end{lem}

The desired 1-morphism is obtained from the diagram
$$\bC^b(\Heart(\bD))/\bC^b_{acycl}(\Heart(\bD)) \overset{\sim}\leftarrow
\bC^{double}/\bC_{acycl}^{double}\to \bC.$$




\sssec{}

Let us recall the necessary and sufficient conditions for the functor
of \eqref{from der to tr bdd} to be fully faithful:

\begin{lem}  \label{der vs tr}
The following conditions are equivalent:

\smallskip

\noindent(1) The functor \eqref{from der to tr bdd} is fully faithful.

\smallskip

\noindent(2) The functor \eqref{from der to tr bdd} is equivalence.

\smallskip

\noindent(3) The functor \eqref{from der to tr bdd}  induces an
isomorphism $\on{Ext}^i_{\obC}(X',X)\to
\Hom_{\bD}(X',X[i])$ for any $X,X'\in \obC$.

\smallskip

\noindent(4) For any $X,X'$ as above, $i>0$ and an element $\alpha\in
\Hom_{\bD}(X',X[i])$ there exists a surjection $X'_1\twoheadrightarrow X'$
in $\obC$, such that the image of $\alpha$ in 
$\Hom_{\bD}(X'_1,X[i])$ vanishes.

\noindent(5) For any $X,X'$ as above, $i>0$ and an element $\alpha\in
\Hom_{\bD}(X',X[i])$ there exists an injection $X\hookrightarrow X_1$
in $\obC$, such that the image of $\alpha$ in 
$\Hom_{\bD}(X,X_1[i])$ vanishes.

\end{lem}

\ssec{}   \label{proof gen equiv}

In this subsection we will prove \propref{gen equiv}. First, we note that 
conditions Cat(i,ii) and Funct(i,ii) imply that $\sT$ is fully faithful.

\sssec{$\sT$ is left-exact}

Let $X$ be an object of $\bD_1^{\geq 0}$. Using conditions 
Cat(i,ii) for $\bD_1$, Cat(a) for $\bD_1$ and $\bD_2$, 
\lemref{colimit generate t} for $\bD_1$, Funct(i) and 
\lemref{functor and limit} we conclude that it is sufficient to show
that $\sT(\tau^{\geq 0}(X))$ is in $\bD_2^{\geq 0}$ for
$X\in \bD_1^c$.

\medskip

By condition Funct(c), we know that for $X$ as above,
$\sT(\tau^{\geq 0}(X))\in \bD_2^+$. Let $k$ be the minimal
integer such that $H^k\left(\sT(\tau^{\geq 0}(X))\right)\neq 0$.
Assume by contradiction that $k<0$.

\medskip

By Funct(d), we can find $X'\in \bD_1^{\leq 0}$ with a non-zero map
$\sT(X')\to H^k\left(\sT(\tau^{\geq 0}(X))\right)$. I.e., we obtain
a non-zero map
$$\sT(X'[-k])\to \sT(\tau^{\geq 0}(X)).$$

However, this contradicts the fact that $\sT$ is fully faithful. 

\sssec{$\sT$ induces an equivalence $\bD_1^+\to \bD_2^+$}

It is easy to see that $\sT$ admits a right adjoint. We shall denote it
by $\sS$. Fully faithfulness of $\sT$ means that the composition
$$\on{Id}_{\bD_2}\to \sS\circ\sT$$
is an isomorphism. Hence, it is sufficient to show that for
$Y\in \bD_2^+$, the adjunction map
$\sT(\sS(Y))\to Y$ is an isomorphism.

\medskip

Being a right adjoint of a right-exact functor (condition Funct(b)),
$\sS$ is left-exact. In particular, it maps $\bD_2^+$ to $\bD_1^+$. 
We have: $\sT(\sS(Y))\in \bD_2^+$, by the already established
left-exactness of $\sT$. Hence, 
$Y'=\on{Cone}(\sT(\sS(Y))\to Y)\in \bD_2^+$. By the fully faithfulness of $\sT$,
we have $\sS(Y')=0$. However, condition Funct(d) implies that
$\sS$ is conservative on $\bD_2^+$. Hence, $Y'=0$.

\sssec{End of the proof}

It remains to show that $\sT$ is essentially surjective. This
is equivalent to the fact that the image of $\sT$ generates
$\bD_2$. Using Cat(b), it is enough to show
that $\bD_2^+$ is in the image of $\Psi$. However,
this has been established above.

\section{Tensor products and $\bbt$-structures}   \label{tensor products and t-structures}

\ssec{$\bbt$-structure on the tensor product}  \label{intr t structure on ten}

Let $\bA$ be a DG category equipped with a homotopy
monoidal structure. Set $\bD_\bA:=\Ho(\ua\bA)$. We assume that 
$\bD_\bA$ is equipped with a t-structure. We assume that this
t-structure is compactly generated (see \secref{compact gener t}), and  
the tensor product functor
$$\bD_\bA\times \bD_\bA\to \bD_\bA$$
is right-exact, and ${\bf 1}_{\bA}\in \Heart(\bD_\bA)$.

\sssec{}

Let $\bC^l$ and $\bC^r$ be DG categories, equipped with homotopy
actions of $\bA$ on the left and on the right, respectively. 

\medskip

Let $\bD^l:=\Ho(\ua\bC^l)$ and $\bD^r:=\Ho(\ua\bC^r)$ be
equipped with t-structures. We are assume that these t-structures 
are compactly generated. 

\medskip

We assume also that the action functors
$$act_{\bC^l}^*:\bD_\bA\times \bD^l\to \bD^l \text{ and }
act_{\bC^r}^*:\bD^r\times \bD_\bA\to \bD^r$$
are right-exact.

\sssec{}   \label{t structure ten prod}

Consider the DG category $\bC^r\underset{\bA}\otimes \bC^l$, and the
triangulated category
$$\bD^t:=\bD\left((\bC^r\underset{\bA}\otimes \bC^l)^{op}\mod\right)=:
\Ho\left(\underset{\longrightarrow}{\bC^r\underset{\bA}\otimes \bC^l}\right).$$

We define a t-structure on $\bD^t$ 
by the procedure of \lemref{t general} with $\bD^{t,\leq 0}$ being
generated by objects of the form 
$Y^r\underset{\bA}\otimes Y^l\in \Ho(\bC^r\underset{\bA}\otimes \bC^l)$
with $Y^r\in \Ho(\bC^r)\cap \bD^{r,\leq 0}$ and 
$Y^l\in \Ho(\bC^l)\cap \bD^{l,\leq 0}$. 

\sssec{}

Consider the particular case when $\bC^r=\bA$. By \propref{ten with free},
we have an equivalence $\bD^t\simeq \bD^l$, and by construction 
the t-structure on the LHS equals the given one on the RHS.

\sssec{}

The following question appears to be natural, but we do not know
how to answer it:

\medskip

Assume that the t-structures on $\bD_\bA$, $\bD^l$ and $\bD^r$ are 
compatible with the subcategories $\bD^{l,c}\simeq (\Ho(\bC^l))^{Kar}$ and 
$\bD^{r,c}\simeq (\Ho(\bC^r))^{Kar}$.
Under what conditions is the above t-structure on $\bD^t$ 
compatible with $\bD^{t,c}\simeq 
\Ho(\bC^r\underset{\bA}\otimes \bC^l)^{Kar}$?

\ssec{Flatness}

Let $\bA$, $\bC^l$, $\bC^r$ be as above. We will say that an object
$Y^l\in \Heart(\bD^l)$ is $\bA$-flat if the functor
\begin{equation} \label{mult by X}
\bD_\bA\to \bD^l:X\mapsto X\underset{\bA}\otimes Y
\end{equation}
is exact. (A priori, this functor is right-exact.)

\begin{prop}   \label{flat exact}
Assume that $\bA$ has a tight diagonal. Let $Y^l\in \Heart(\bD^l)$
be flat. Then the functor
$$\bD^r\to \bD^t:
Y^r\mapsto Y^r\underset{\bA}\otimes Y^l$$
is exact.
\end{prop}

\begin{proof}

A priori, the functor in question is right-exact. To prove the left-exactness,
we have to show that for $Y^r\in \bD^{r,\geq 0}$,
$Y^r_1\in \bD^{r,<0}$ and
$Y^l_1\in \bD^{l,<0}$, we have
$$\Hom_{\bD^t}
(Y^r_1\underset{\bA}\otimes Y^l_1,Y^r\underset{\bA}\otimes Y^l)=0.$$

By \corref{tight diagonal}, the LHS of the above expression can be rewritten
as
$$\Hom_{\bD((\bC^r\times \bC^l)^{op}\mod)}
\left((Y^r_1,Y^l_1),\on{Diag}_\bA\underset{\bA\otimes \bA}\otimes
(Y^r,Y^l)\right).$$
Thus, it is sufficient to show that the object
$$\on{Diag}_\bA\underset{\bA\otimes \bA}\otimes
(Y^r,Y^l)\in \bD((\bC^r\times \bC^l)^{op}\mod)$$
is $\geq 0$.

\medskip

By the flatness assumption on $Y^l$, the functor
$$(\on{Id}\times act(?,Y^l))^*:
\bD((\bC^{r}\times \bA)^{op}\mod)\to \bD((\bC^r\times \bC^l)^{op}\mod),$$
given by $(Y^r,X)\mapsto (Y^r,X\underset{\bA}\otimes Y^l)$, is exact.
We have
$$\on{Diag}_\bA\underset{\bA\otimes \bA}\otimes
(Y^r,Y^l)\simeq (\on{Id}\times act(?,Y^l))^*\left(
\on{Diag}_\bA\underset{\bA\otimes \bA}\otimes  (Y^r,{\bf 1}_\bA)\right).$$
Thus, it remains to see that
$$\on{Diag}_\bA\underset{\bA\otimes \bA}\otimes  (Y^r,{\bf 1}_\bA)\in
\bD((\bC^{r}\times \bA)^{op}\mod)$$
is $\geq 0$.

\medskip

However, since $\bA$ has a tight diagonal, the latter object is 
isomorphic to $(act_{\bC^r})_*(Y^r)$, and the assertion follows.

\end{proof}

\ssec{Base change}

We will mostly consider a particular case of the above situation, where
$\bC^r=\Res^{\bA_1}_{\bA}(\bA_1)$, where $\bA_1$ is another homotopy monoidal
category, equipped with a homotopy monoidal functor
$F:\bA\to \bA_1$. We will assume that the t-structure on $\bD_{\bA_1}$
satisfies the assumptions of
\secref{intr t structure on ten}. In particular, the functor $F$ is right-exact.

\sssec{}

Denote $\bC^l_1:=\Ind^{\bA_1}_\bA(\bC^l)$. We will call this category the base 
change of $\bC^l$ with respect to $F$. We will call the above t-structure
on $\bD^l_1:=\bD(\bC_1^{l,op}\mod)$ the base-changed t-structure. 

\medskip

The following results from the definitions:

\begin{lem}
The functor
$$(F_{\bC^l})^*:\bD^l\to \bD^l_1$$
is right-exact. 
\end{lem}

Hence, by adjunction, the functor
$$(F_{\bC^l})_*:\bD^l_1\to \bD^l$$
is left-exact.

\medskip

We shall say that $F$ is flat if it is exact. From \propref{flat exact} we obtain:
\begin{cor}   \label{flat base change}
Assume that $F$ is flat. Assume also that $\bA$ has a tight diagonal.
Then the functor $(F_{\bC^l})^*$ is exact.
\end{cor}

\ssec{Affiness}   \label{affine functor}

Let $F$ be as above. We will
say that $F$ is affine if the functor
$$F_*:\bD_{\bA_1}\to \bD_\bA$$
is exact and conservative.

\begin{prop}    \label{affine exact}
Assume that $F$ is affine and tight. Then  for $\bC^l$ as above, the functor
$(F_{\bC^l})_*$ is also exact and conservative.
\end{prop}

\begin{proof}

The fact that $F_*$ is conservative is equivalent to the fact that
$F(\Ho(\bA))$ generates $\bD_{\bA_1}$. By
construction, this implies that the image of $(F_{\bC^l})^*$ generates 
$\bD^l_1$. The latter is equivalent to $(F_{\bC^l})_*$ 
being conservative.

\medskip

It remains to show that $(F_{\bC^l})_*$ is right-exact. By the definition
of the t-structure, we have to show that 
$$(F_{\bC^l})_*(X_1\underset{\bA}\otimes Y^l)\in \bD^{l,\leq 0}$$
for $X_1\in \bD_{\bA_1}^{\leq 0}$ and $Y^l\in \bD^{l,\leq 0}$. 

\medskip

By \corref{tight mon}, we have:
$$(F_{\bC^l})_*(X_1\underset{\bA_1}\otimes Y^l)\simeq F_*(X_1)
\underset{\bA}\otimes Y^l.$$
However, by assumption, $F_*(X_1)\in \bD_\bA^{\leq 0}$,
implying our assertion.

\end{proof}

\ssec{Tensor product of abelian categories}   \label{ten prod abelian}

In this subsection we will be concerned with the following situation:

\medskip

Let $\obA$ be a abelian category equipped with a monoidal structure.
We will assume that the following conditions hold:

\noindent(A*) $\obA$ is a Grothendieck category,
and that the action functor $\obA\times \obA\to \obA$ is right-exact, 
and commutes with direct sums.

\medskip

Let $\obC^l$ be another abelian category, equipped with a monoidal action 
$\obA$. We will assume that the following conditions hold:

\medskip

\noindent(C*) 
$\obC^l$ is also a Grothendieck category,
and that the action functor $\obA\times \obC^l\to \obC^l$ is right-exact, 
and commutes with direct sums.

\medskip

Let $\obC^r$ be another abelian category with a right action of $\obA$, satisfying
(C*).

\sssec{}

We shall say that a Grothendieck abelian category $\obC^t$ is the tensor product
$\obC^r\underset{\obA}\otimes \obC^l$ if it has the following universal property:

\medskip

\noindent For any Grothendieck abelian category $\obC'$, the category of 
right-exact functors $Q^t:\obC^t\to \obC'$ that commute with direct sums
is equivalent to the category of bi-additive functors
$$Q^{r,l}:\obC^r\times \obC^l\to \obC'$$
that are right-exact and commute with direct sums in both arguments,
and equipped with functorial isomorphisms
$$Q^{r,l}(Y^r\underset{\obA}\otimes X,Y^l)\simeq
Q^{r,l}(Y^r,X\underset{\obA}\otimes Y^l),\,\, X\in \obA,Y^r\in \obC^r,Y^l\in \obC^l.$$

\sssec{}

Let us consider a special case of the above situation, when $\obC^r$ is itself
an abelian monoidal category $\obA_1$ satisfying (A*),
and the action of $\obA$ on $\obA_1$ comes from a right-exact monoidal
functor $\obA\to \obA_1$.

\medskip

Assume that for $\obC^l$ as above, the category 
$\Ind^{\obA_1}_{\obA}(\obC^l):=\obA_1\underset{\obA}\otimes \obC^l$
exists. By the universal property, $\Ind^{\obA_1}_{\obA}(\obC^l)$
carries an action of $\obA_1$, satisfying (C*).

\medskip

Let now $\obC^l_1$ be another abelian category endowed with an action of
$\obA_1$, satisfying (C*). We have:

\begin{lem}
The category of functors $\obC^l\to \obC^l_1$ that are compatible with
the action of $\obA$, and are right-exact and commute with direct sums
is equivalent to the category of functors $\Ind^{\obA_1}_{\obA}(\obC^l)\to \obC^l_1$
that are compatible with the action of $\obA_1$, and are right 
exact and commute with direct sums.
\end{lem}

\sssec{}
Consider a specific example of the above situation. Let
$\obC$ be a Grothenieck abelian category, and $A$ a commutative
algebra that maps to the center of $\obC$. Then the abelian
monoidal category $\obA:=A\mod$ acts on $\obC$.

\medskip

Let $A\to A'$ be a homomorphism of commutative algebras, and set
$\obA_1:=A_1\mod$. 

\medskip

Under the above circumstances, the category $\Ind^{\obA_1}_{\obA}(\obC)$
exists and can be described as follows. Its objects are objects $X\in \obC$,
endowed with an additional action of $A_1$, such that the two actions of
$A$ on $X$ coincide. Morphisms are $\obC$-morphisms $X_1\to X_2$
that intertwine the $A_1$-actions.

\ssec{}   \label{affine diag}

We shall now study the compatibility of the notions of tensor products 
in the DG and abelian settings.

\medskip

Let now $\bA$, $\bC^l$, $\bC^r$ be as in \secref{intr t structure on ten}. Then the 
abelian categories $\obA=\Heart(\ua\bA)$, $\obC^l=\Heart(\ua\bC^l)$ and $\obC^r=\Heart(\ua\bC^r)$
satisfy the conditions of \secref{ten prod abelian}.

\medskip

We will assume that $\bA$ has a tight diagonal. In addition, we will assume
that has $\bA$ {\it has an affine diagonal}, by which we mean that the
functor
$$(m_\bA)_*:\Ho(\ua\bA)\times \Ho(\ua\bA)\to \Ho(\ua\bA)$$
is exact (a priori, it is only left-exact).

\begin{prop}   \label{heart of base change}
Under the above circumstances, we have:
$$\Heart(\bD^t)\simeq \obC^r\underset{\obA}\otimes \obC^r.$$
\end{prop}

\begin{proof}

The universal right-exact functor
$$\obC^r\times \obC^l\to \obC^t:=\Heart(\bD^t)$$
comes from the functor
$$(m_{\bC^r,\bC^l})_*:\bD^r\times \bD^l\to \bD^t.$$

Conversely, given a right-exact functor $Q^{r,l}:
\obC^r\times \obC^l\to \obC'$ we produce the
functor $Q^t:\obC^t\to \obC'$ as follows. First,
we tautologically extend $Q^{r,l}$ to a functor
$$'Q^{r,l}:\Ho^{\leq 0}(\underset{\longrightarrow}{\bC^r\otimes \bC^l})\simeq
\bD^{\leq 0}((\bC^r\times \bC^l)^{op}\mod)
\to \obC'$$
that vanishes on $\bD^{\leq -1}((\bC^r\times \bC^l)^{op}\mod)$.

\medskip

Note that, by construction, for every object $Z\in\bD^{t,\leq 0}$
there exist $Y^r\in \obC^r$, $Y^l\in \obC^l$ and maps
$Y^r\underset{\bA}\otimes Y^l\to Z$ that induces a surjection
$$H^0\left(Y^r\underset{\bA}\otimes Y^l\right)\to H^0(Z).$$

Hence, in order to construct $Q^t$ it suffices to define a functorial
map
\begin{equation} \label{bilinear}
\Hom_{\bD^{t,\leq 0}}(Y_1^r\underset{\bA}\otimes Y^l_1,
Y^r_2\underset{\bA}\otimes Y^l_2)\to
\Hom_{\obC'}({}'Q^{r,l}(Y_1^r,Y^l_1),{}'Q^{r,l}(Y_2^r,Y^l_2)).
\end{equation}

We rewrite the LHS of \eqref{bilinear} as
$$\Hom_{\Ho(\underset{\longrightarrow}{\bC^r\otimes \bC^l})}
((Y_1^r,Y_1^l),\on{Diag}_\bA\underset{\bA\otimes \bA}\otimes (Y_2^r,Y_2^l)).$$
By assumption, the object $\on{Diag}_\bA$ belongs to
$\bD^{\leq 0}((\bA\otimes \bA)^{op}\mod)$. Hence, it suffices to construct a 
functorial isomorphism
$${}'Q^{r,l}\left(\on{Diag}_\bA\underset{\bA\otimes \bA}\otimes (Y^r,Y^l)\right)\to
{}'Q^{r,l}(Y^r,Y^l).$$

However, the latter follows from the fact that
for any $X^{r,l}\in \Ho^{\leq 0}(\bA\times \bA)$ we have:
$${}'Q^{r,l}\left(X^{r,l}\underset{\bA\otimes \bA}\otimes (Y^r,Y^l)\right)\simeq
{}'Q^{r,l}(Y^r,m_\bA^*(X^{r,l})\underset{\bA}\otimes Y^l)\simeq
{}'Q^{r,l}(m_\bA^*(X^{r,l})\underset{\bA}\otimes Y^r,Y^l).$$

\end{proof}

\ssec{}

Let $\bA$, $\bC^r$, $\bC^l$ be as in \secref{affine diag}.
We will assume that the triangulated
categories $\bD^l$, $\bD^r$ and $\bD_\bA$ satisfy the equivalent conditions of
\lemref{der vs tr}. We are interested in the following question: under what 
circumstances will $\bD^t$ also satisfy these conditions?

\medskip

\begin{lem}  \label{abelian vs derived base change}
Assume that the following condition holds: for every $Y^l\in \obC^l, Y^r\in \obC^r$
and $i>0$ there exist surjections $Y_1^l\twoheadrightarrow Y^l$, $Y_1^r\twoheadrightarrow Y^r$
in $\obC^l$ and $\obC^r$, respectively, such that the map
$$H^{-i}\left(Y^r_1\underset{\bA}\otimes Y^l_1\right)\to
H^{-i}\left(Y^r\underset{\bA}\otimes Y^r\right)$$
is zero. Then $\bD^t$ satisfies the equivalent conditions
of \lemref{der vs tr}.
\end{lem}

\begin{proof}

We will show that $\bD^t$ satisfies condition (4) of \lemref{der vs tr}. Since every object
of $\obC^t$ receives a surjection from some $H^0\left(Y^r\underset{\bA}\otimes Y^l\right)$
for $Y^l,Y^r$ as above, it is enough to show that for any $Z\in \obC^t$, $i>0$
and $\alpha\in \Hom_{\bD^t}\left(H^0\left(Y^r\underset{\bA}\otimes Y^l\right),Z[i]\right)$
there exist $Y'^l,Y'^r$ such that the image $\alpha'$ of $\alpha$ in
$$\Hom_{\bD^t}\left(H^0\left(Y'^r\underset{\bA}\otimes Y'^l\right),Z[i]\right)$$
vanishes.

\medskip

Consider the map 
$Y^r\underset{\bA}\otimes Y^l\to H^0(Y^r\underset{\bA}\otimes Y^l)$,
and let $\beta$ be the image of $\alpha$ in
$$\Hom_{\bD^t}(Y^r\underset{\bA}\otimes Y^l,Z[i])\simeq
\Hom_{\bD((\bC^r\times \bC^l)^{op}\mod)}\left((Y^r,Y^l),
\on{Diag}_\bA\underset{\bA\otimes \bA}\otimes Z[i]\right).$$
Since $\bC^l$ and $\bC^r$ satisfy condition (4) of \lemref{der vs tr},
and 
$$\on{Diag}_\bA\underset{\bA\otimes \bA}\otimes Z\simeq
(m_{\bC^r,\bC^l})_*(Z)\in \bD^{\leq 0}((\bC^r\times \bC^l)^{op}\mod),$$
there exist surjections $Y''^r\to Y^r$ and $Y''^l\to Y^l$, such 
that the image $\beta''$ of $\beta$ in
$$\Hom_{\bD^t}(Y''^r\underset{\bA}\otimes Y''^l,Z[i])\simeq
\Hom_{\bD((\bC^r\times \bC^l)^{op}\mod)}\left((Y''^r,Y''^l),
\on{Diag}_\bA\underset{\bA\otimes \bA}\otimes Z[i]\right)$$
vanishes.

\medskip

Let $\alpha''$ be the image of $\alpha$ in
$\Hom_{\bD^t}\left(H^0(Y''^r\underset{\bA}\otimes Y''^l),Z[i]\right)$.
By construction its image in 
$$\Hom_{\bD^t}(Y''^r\underset{\bA}\otimes Y''^l,Z[i])$$
vanishes; so $\alpha''$
is the image of some class 
$$\gamma\in \Hom_{\bD^t}\left(\tau^{\leq -1}(Y''^r\underset{\bA}\otimes Y''^l),Z[i-1]\right)\simeq
\Hom_{\bD^t}\left(\tau^{\geq -i+1,\leq -1}(Y''^r\underset{\bA}\otimes Y''^l),Z[i-1]\right).$$

By the assumption of the lemma, there exist surjections
$Y_1^r\to Y''^r$ and $Y_1^r\to Y''^l$, so that the restriction of $\gamma$ to
$\Hom_{\bD^t}\left(\tau^{\geq -i+1,\leq -1}(Y_1^r\underset{\bA}
\otimes Y_1^l),Z[i-1]\right)$
is such that its further restriction to
$\Hom_{\bD^t}\left(H^{-i+1}(Y_1^r\underset{\bA}\otimes Y_1^l),Z[i-1]\right)$
vanishes. So, the above restriction comes from a class 
$$\gamma_1\in 
\Hom_{\bD^t}\left(\tau^{\geq -i+2,\leq -1}(Y_1^r\underset{\bA}\otimes Y_1^l),Z[i-1]\right).$$
By induction, we find sequences of surjections
$$Y^r_{i-1}\to...\to Y^r_1 \text{ and }
Y^l_{i-1}\to...\to Y^l_1$$
so that the restriction of $\gamma$ to
$\Hom_{\bD^t}\left(\tau^{\geq -i+j,\leq -1}(Y_j^r\underset{\bA}\otimes Y_j^l),Z[i-1]\right)$
vanishes. This implies that the restriction of $\alpha''$ to
$\Hom_{\bD^t}\left(H^0(Y_{i-1}^r\underset{\bA}\otimes Y_{i-1}^l),Z[i]\right)$
vanishes. I.e., the objects $Y'^r=Y_{i-1}^r$, $Y'^l=Y_{i-1}^l$ have the desired properties.

\end{proof}

\section{Categories over stacks}  \label{categories over stacks}

\ssec{}

Let $\CY$ be an Artin stack. We assume that $\CY$ is of finite type, i.e., can
be covered with a smooth map by an affine scheme of finite type.
In addition, we shall assume that the diagonal morphism $\CY\to \CY\times \CY$
is affine. 

\medskip

An example of such a stack is when $\CY=Z/G$, where $Z$ is a scheme of finite type
and $G$ an affine algebraic group, acting on it. In fact, all algebraic
stacks in this paper will be of this form.

\sssec{}

The abelian categories $\Coh(\CY)$ and $\QCoh(\CY)$ have an evident meaning.
We consider the DG categories of complexes
$\bC^b(\Coh(\CY))$, $\bC(\QCoh(\CY))$, and the corresponding subcategories
of acyclic complexes
$$\bC^b_{acycl}(\Coh(\CY))\subset \bC^b(\Coh(\CY)),\,\,
\bC_{acycl}(\QCoh(\CY))\subset \bC(\QCoh(\CY)).$$

Denote the corresponding quotients, regarded as 0-objects in $\DGCat$ by
$\bC^b(\CY)$ and $\bC(\CY)$, respectively (see \secref{rigid quotients}). 
Their homotopy categories identify with
\begin{multline*}
\bD^b(\Coh(\CY)):=\Ho\left(\bC^b(\Coh(\CY))\right)/
\Ho\left(\bC^b_{acycl}(\Coh(\CY))\right), \\
\bD(\QCoh(\CY)):=\Ho\left(\bC(\QCoh(\CY))\right)/
\Ho\left(\bC_{acycl}(\QCoh(\CY))\right),
\end{multline*}
respectively.

\medskip

As in the case of schemes, one shows that the natural functor
$$\bD^b(\Coh(\CY))\to \bD(\QCoh(\CY))$$
is fully faithful, and its essential image consists of objects
with coherent cohomologies.

\sssec{}   \label{ass *}

We shall now make the following additional assumptions on $\CY$: 

\medskip

\noindent(*a) \hskip0.5cm {\it The functor of global sections on $\bD(\QCoh(\CY))$ is of finite cohomological
dimension.}

\medskip

\noindent(*b) \hskip0.5cm {\it Every coherent sheaf on $\CY$ admits a non-zero map from a locally
free coherent sheaf.} 

\medskip

In the example $\CY=Z/G$, assumption (*a) is satisfied e.g. if the ground field is of characteristic $0$.
Assumption (*b) is satisfied e.g. if $Z$ admits
a G-equivariant locally closed embedding into the projective space.

\medskip

\noindent{\it Remark.} Assumption (*a) implies that the structure sheaf $\CO_\CY$
is a compact object in $\bD(\QCoh(\CY))$.

\medskip

Assumption (*b) expresses the fact
that $\CY$ carries enough vector bundles. It 
implies that every object $\CM\in \bD^-(\Coh(\CY))$ can be represented
by a (possible infinite to the left) complex whose terms are locally free. This complex
can be chosen to be finite if $\CM\in \bD^b(\Coh(\CY))$, and $H^i(\CM)$ are 
of finite Tor-dimension. 

\sssec{}

Let $\bC^b(\Coh^{loc.free}(\CY))$ be the DG category of bounded
complexes of locally free sheaves,
and let $\bC_{acycl}^b(\Coh^{loc.free}(\CY))\subset \bC^b(\Coh^{loc.free}(\CY))$ be the
subcategory of acyclic complexes. Let $\bC^{perf}(\CY)$ be the corresponding
quotient, regarded as a 0-object of $\DGCat$; denote 
$\bD^{perf}(\Coh(\CY)):=\Ho(\bC^{perf}(\CY))$.

\medskip

We will use the following general assertion:

\begin{lem}  \label{adapted replacement}
Let $\bD_1\subset \bD\supset \bD'$ be triangulated categories, and set
$\bD'_1:=\bD_1\cap \bD'$. Consider the functor
\begin{equation} \label{prime and one}
\bD_1/\bD'_1\to \bD/\bD'.
\end{equation}
Assume that in the above situation one of the following two conditions holds:

\smallskip

\noindent{\em(1)}
For every object $X\in \bD$ there exists an
object $X_1\in \bD_1$ and an arrow $X_1\to X$ with
$\on{Cone}(X_1\to X)\in \bD'$.

\smallskip

\noindent{\em(2)}
For every object $X\in \bD$ there exists an
object $X_1\in \bD_1$ and an arrow $X\to X_1$ with
$\on{Cone}(X\to X_1)\in \bD'$.

\medskip

Then the functor \eqref{prime and one} is an equivalence.
\end{lem}

\medskip

Assumption (*b) and \lemref{adapted replacement} imply that the natural functor
$$\bD^{perf}(\Coh(\CY))\to \bD^b(\Coh(\CY))$$ is
fully faithful, and its essential image consists of objects, whose cohomologies have 
a finite Tor-dimension. In particular, the above functor is an equivalence if $\CY$
is smooth.

\medskip

As in the case of schemes which have enough locally free sheaves, one shows:

\begin{lem}
The functor 
$$\bD^{perf}(\Coh(\CY))\to \bD(\QCoh(\CY))$$
extends to an equivalence between the ind-completion of $\bD^{perf}(\Coh(\CY))$
and $\bD(\QCoh(\CY))$.
\end{lem}

\ssec{}  \label{perf on stack}

The category $\bC^b(\Coh^{loc.free}(\CY))$ is naturally
a DG monoidal category, and $$\Ho\left(\bC_{acycl}^b(\Coh^{loc.free}(\CY))\right)\subset
\Ho\left(\bC^b(\Coh^{loc.free}(\CY))\right)$$ is an ideal. Therefore, by \secref{action on subcat}, 
$\bC^{perf}(\CY)$ is a 0-object of $\DGMonCat$, which is a DG model for the triangulated
monoidal category $\bD^{perf}(\Coh(\CY))$.

\medskip

The category $\bD^{perf}(\Coh(\CY))$ is rigid (see \secref{rigidity}). The
category $\bD(\QCoh(\CY))$ has a natural t-structure, and the functor
of tensor product is right-exact. I.e., the above DG model of $\bD^{perf}(\Coh(\CY))$
satisfies the assumptions of \secref{intr t structure on ten}.

\medskip

In addition, the assumption that $\CY$ has an affine diagonal implies
that $\bD^{perf}(\Coh(\CY))$ has an affine diagonal, see \secref{affine diag}.

\sssec{}   \label{over}

By a triangulated category {\it over} the stack $\CY$ we will understand
a triangulated category, equipped with a DG model, which is a 0-object of 
\begin{equation} \label{DG Y}
\DGMod\left(\bC^{perf}(\CY)\right).
\end{equation}

For two triangulated categories over
$\CY$, by a functor (resp., natural transformation) {\it over} $\CY$ we will mean 
a 1-morphism (resp., 2-morphism) between the corresponding objects in
the 2-category \eqref{DG Y}.

\sssec{}   \label{base change}

Let $f:\CY_1\to \CY_2$ be a morphism between stacks. It defines an evident
DG monoidal functor $f^*:\bC(\Coh^{loc.free}(\CY_2))\to \bC(\Coh^{loc.free}(\CY_1))$,
and by \lemref{action on quotient}, it gives rise to a 1-morphism in $\DGMonCat$,
$$f^*:\bC^{perf}(\CY_2)\to \bC^{perf}(\CY_1),$$
providing a model for the pull-back functor $f^*:\bD^{perf}(\QCoh(\CY_2))\to \bD^{perf}(\QCoh(\CY_1))$.

\medskip

If $\bD_2$ is a triangulated category over $\CY_2$, we will denote by
$\CY_1\underset{\CY_2}\times \bD_2$ the induction 
$$\on{Ind}^{\bC^{perf}(\CY_1)}_{\bC^{perf}(\CY_2)}(\bC_2),$$
(see \secref{def induction}), where $\bC_2$ is a DG model of $\bD_2$.
We will refer to $\CY_1\underset{\CY_2}\times \bD_2$ as the base change
of $\bD_2$ with respect to $f$. We have a tautological functor
$$(f^*\underset{\CY}\times \on{Id}_{\bD_2}):
\bD_2\to \CY_1\underset{\CY_2}\times \bD_2.$$

\medskip

Since $\CY_1\underset{\CY_2}\times \bD_2$ is defined as an object of $\DGCat$,
its ind-completion (see \secref{big and small, triang}) is well-defined; we will denote it
by $\CY_1\underset{\CY_2}\arrowtimes \bD_2$. By a slight abuse of notation,
we will denote by the same character $(f^*\underset{\CY}\times \on{Id}_{\bD_2})$ 
the functor
$$\ua\bD{}_2\to \CY_1\underset{\CY_2}\arrowtimes \bD_2.$$

\sssec{}

Assume that $\ua\bD{}_2$ is equipped with a t-structure, satisfying the
conditions of \secref{intr t structure on ten}. Then the category 
$\CY_1\underset{\CY_2}\arrowtimes \bD_2$ acquires a t-structure, such that
the functor $(f^*\underset{\CY}\times \on{Id}_{\bD_2})$ is right-exact.

\ssec{}

Consider now the following situation. Let $f:\CY_1\to \CY_2$ be a map of stacks
as before, and let $g:\CY'_2\to \CY_2$ be another map. We assume that all three
are Artin stacks of finite type that satisfy assumptions (*a) and (*b) of \secref{ass *}.

\medskip

Assume also that the following holds: 

\medskip

\noindent(**) \hskip1cm{\it The tensor product 
$\CO_{\CY_1}\overset{L}{\underset{\CO_{\CY_2}}\otimes}\CO_{\CY'_2}$ is acyclic
off cohomologically degree $0$.}

\medskip

This ensures that the naive fiber product
$\CY_1\underset{\CY_2}\times \CY'_2$ coincides with the derived one.

\sssec{}

Set $\CY'_1:=\CY_1\underset{\CY_2}\times \CY'_2$. Let $f'$ denote the map $\CY_1'\to \CY'_2$. 
The pull-back functor $f'{}^*:\bD^{perf}(\Coh(\CY'_2))\to \bD^{perf}(\Coh(\CY'_1))$,
which coincides with the derived pull-back by virtue of assumption (**), is 
naturally a 1-morphism between categories over $\CY_2$. Hence, by \secref{univ ppty base change}
we can consider the induced 1-morphism
$$\Ind^{\CY_1}_{\CY_2}(f'{}^*):\CY_1\underset{\CY_2}\times \bD^{perf}(\Coh(\CY'_2))\to 
\bD^{perf}(\Coh(\CY'_1))$$
of categories over $\CY_1$. By ind-extension we obtain a functor
\begin{equation} \label{fiber product}
\Ind^{\CY_1}_{\CY_2}(f'{}^*):\CY_1\underset{\CY_2}\arrowtimes \bD^{perf}(\Coh(\CY'_2))\to 
\bD(\QCoh(\CY'_1)).
\end{equation}

\medskip

We shall now make the following additional assumption: 

\medskip

\noindent(***) \hskip1cm{\it  Every quasi-coherent sheaf 
on $\CY'_1$ admits a non-zero map from a sheaf of the form
\begin{equation} \label{ext product}
\CM_1\underset{\CO_{\CY_2}}\otimes \CM'_2,
\end{equation}
where $\CM_1$ (resp., $\CM'_2$) is a locally free coherent sheaf on $\CY_1$
(resp., $\CY'_2$).}

\medskip

\begin{prop}  \label{base change stacks}
Under the above circumstances the functor $\Ind^{\CY_1}_{\CY_2}(f'{}^*)$ of
\eqref{fiber product} is an exact equivalence of categories.
\end{prop}

To prove the proposition, we are going to check that the conditions of \propref{gen equiv} 
hold. Conditions Cat(i,ii), Cat(a,b) and Funct(i,ii) evidently hold. The functor
in question is by construction right-exact, so Funct(b) holds as well.

\sssec{Condition Funct(a)}

Let is verify Funct(a), i.e., the fact that the functor in question is fully
faithful. Since objects of the form 
$$\CM_1\underset{\bD^{perf}(\Coh(\CY_2))}\otimes\CM'_2,\,\,
\CM_1\in \bD^{perf}(\Coh(\CY_1)),\,\, \CM'_2\in
\bD^{perf}(\Coh(\CY'_2))$$
generate $\CY_1\underset{\CY_2}\times \bD^{perf}(\Coh(\CY'_2))$, it suffices to show
that the functor $\Ind^{\CY_1}_{\CY_2}(f'{}^*)$ induces an isomorphism on $\Hom$'s
between such objects.

\medskip

Note that assumption (**) is equivalent to the fact that
$$\Ind^{\CY_1}_{\CY_2}(f'{}^*)
\left(\CM_1\underset{\bD^{perf}(\Coh(\CY_2))}\otimes\CM'_2\right)
\simeq \CM_1\underset{\CO_{\CY_2}}\otimes \CM'_2.$$

Thus, we have to verify that 
\begin{multline}   \label{compare hom}
\Hom_{\CY_1\underset{\CY_2}\times \bD^{perf}(\Coh(\CY'_2))}
\left(\CM_1\underset{\bD^{perf}(\Coh(\CY_2))}\otimes\CM'_2, 
\wt{\CM_1}\underset{\bD^{perf}(\Coh(\CY_2))}\otimes \wt{\CM}'_2\right)\to  \\
\to \Hom_{\bD^{perf}(\Coh(\CY'_1))}\left(\CM_1\underset{\CO_{\CY_2}}\otimes \CM'_2,
\wt{\CM}_1\underset{\CO_{\CY_2}}\otimes \wt{\CM}'_2\right)
\end{multline}
is an isomorphism. Using \lemref{rigid adjunction}, we can assume that $\CM_1\simeq
\CO_{\CY_1}$.

\medskip

By \corref{tight mon}, the LHS of \eqref{compare hom} identifies with
$$\Hom_{\bD(\QCoh(\CY'_2))}\left(\CM'_2,f_*(\wt{\CM}_1)\underset{\CO_{\CY_2}}
\otimes 
\wt{\CM}'_2\right).$$

\medskip

By adjunction and the projection formula,
\begin{multline*}
\Hom_{\bD^{perf}(\Coh(\CY'_1))}\left(f'{}^*(\CM'_2),
\wt{\CM}_1\underset{\CO_{\CY_2}}\otimes \wt{\CM}'_2\right)\simeq
\Hom_{\bD^{perf}(\Coh(\CY'_2))}
\left(\CM'_2,f'_*(\wt{\CM}_1\underset{\CO_{\CY_2}}\otimes \wt{\CM}'_2)\right)\simeq \\
\simeq
\Hom_{\bD^{perf}(\Coh(\CY'_2))}
\left(\CM'_2,f_*(\wt{\CM}_1)\underset{\CO_{\CY_2}}\otimes 
\wt{\CM}'_2\right),
\end{multline*}
implying the desired isomorphism.

\sssec{End of proof of \propref{base change stacks}}

Note that assumption (***) implies condition Funct(d). Moreover,  
it implies that every object in $\bC^-(\QCoh(\CY'_1))$
is quasi-isomorphic to a complex whose terms are of the form \eqref{ext product}.
This observation, combined with the fully-faithfulness of $\sT:=\Ind^{\CY_1}_{\CY_2}(f'{}^*)$,
imply that the above functor induces an equivalence $\bD_1^-\to \bD_2^-$. 

\medskip

We claim that this implies the left-exactness of $\sT$, and in 
particular, condition Funct(c). Indeed, for $X\in \bD_1^{\geq 0}$ consider 
$\tau^{\leq 0}(\sT(X))$.
If this object is non-zero, it is of the form $\sT(X')$ for some $X'\in \bD_1^{<0}$,
and by fully faithfulness $\Hom_{\bD_1}(X',X)\neq 0$, which is a contradiction.

\qed

\ssec{}  \label{over affine}

Let $A$ be an associative algebra over the ground field. Consider the DG category 
of finite complexes of free $A$-modules of finite rank, denoted 
$\bC^b(A-mod^{free,fin.rk.})$. 
Note that $\Ho\left(\ua\bC^b(A\mod^{free,fin.rk.})\right)$
identifies with $\bD(A\mod)$--the usual derived category of $A$-modules.

\medskip

The Karoubian envelope of $\Ho(\bC^b(A\mod^{free,fin.rk.}))$ is
the full subcategory of $\bD(A\mod)$ consisting of complexes quasi-isomorphic
to finite complexes of projective finitely generated $A$-modules; we will denote
this category by $\bD^{perf}(A\mod)$; it comes equipped with a natural DG model.

\sssec{}

Assume now that $A$ is commutative. Let $Z=\Spec(A)$; this is an affine scheme, 
possibly of infinite type. 

\medskip

In this case $\bC^b(A-mod^{free,fin.rk.})$ has a naturally a DG monoidal category structure. 
Then $\bD^{perf}(\Coh(Z)):=\bD^{perf}(A\mod)$ admits a natural lifting to a 0-object of 
$\DGMonCat$, denoted $\bC^{perf}(Z)$.

\medskip

By a triangulated category $\bD$ over $Z$ we will mean an object of
$\DGMod\left(\bC^{perf}(Z)\right)$.

\medskip

For example, if $\bC$ is a DG category with a map $Z\to \End(\on{Id}_{\bC})$
and such that $\Ho(\bC)$ is Karoubian, it gives rise to a well-define object
of $\DGMod\left(\bC^{perf}(Z)\right)$. 

\sssec{}

Let $\CY$ be an Artin stack as in \secref{ass *}, and let us be given a map
$g:Z\to \CY$. We will assume that this map factors $Z\to Z'\to \CY$, where
$Z'$ is a scheme of finite type. The the functor $g^*$ defines a 1-morphism 
$\bC^{perf}(\CY)\to \bC^{perf}(Z)$ in $\DGMonCat$.

\medskip

For a triangulated category $\bD_\CY$ over $\CY$, we shall
denote by $$Z\underset{\CY}\times \bD_\CY$$
the base change of $\CY$ with respect to the 1-morphism 
$g^*$.

\sssec{}  \label{cart prod inf type}

We conclude this subsection by the following observation used in the main body 
of the paper.  

\medskip  

Let $\CY_1\to \CY_2$ be a representable map of stacks, and let $Z$ be an affine
scheme of infinite type, endowed with a flat morphism $g$ to $\CY_2$. We will
assume that $Z$ can be represented as $\underset{\longleftarrow}{lim}\, Z_i$,
with the transition maps $Z_j\to Z_i$ flat, and such that $g$ factors through
a flat morphism $g_i:Z_I\to \CY_2$ for some $i$.
Consider the scheme $\CY_1\underset{\CY_2}\times Z$. 

\begin{prop}  \label{base change inf type}
Under the above circumstances there exists an exact equivalence
$$Z\underset{\CY_2}\arrowtimes \bD^{perf}(\Coh(\CY_1))\to
\bD(\QCoh(\CY_1\underset{\CY_2}\times Z)).$$
\end{prop}

The proof is obtained by applying \propref{gen equiv}, repeating the
argument proving \propref{base change stacks}.

\ssec{}

In this subsection we will assume that $f:\CY_1\to \CY_2$ is a regular closed immersion.

\sssec{}   \label{closed imm 1}

By \lemref{action on quotient}, we obtain that the direct image functor $f_*$
gives rise to a well-defined 1-morphism in $\DGMod(\bC^{perf}(\CY_2))$:
$$f_*:\bC^{perf}(\CY_1)\to \bC^{perf}(\CY_2).$$

\medskip

Let now $\bD_2$ be a category over $\CY_2$. We obtain a
1-morphism over $\CY_2$
$$(f_*\underset{\CY_2}\times \on{Id}_{\bD_2}):
\CY_1\underset{\CY_2}\times \bD_2\to \bD_2.$$ 

From \secref{dir image} we obtain:

\begin{lem}   \label{dir im dir im}
The functor $(f_*\underset{\CY_2}\times \on{Id}_{\bD_2})$ is the right adjoint of
$$(f^*\underset{\CY_2}\times \on{Id}_{\bD_2}):
\bD_2\to \CY_1\underset{\CY_2}\times \bD_2$$
at the triangulated level.
\end{lem}

\sssec{}  \label{closed imm 3}

Note that $f_!:=f_*$ admits a right adjoint at the triangulated level,
which by \lemref{action on quotient} give rise to a 1-morphism in 
$\DGMod(\bC^{perf}(\CY_2))$. Moreover, by \secref{dir image}, for any category $\bD_2$ 
over $\CY_2$,
we obtain a pair of functors
$$(f_*\underset{\CY_2}\times \on{Id}_{\bD_2}):\CY_1\underset{\CY_2}\times \bD\leftrightarrows 
\bD:(f^!\underset{\CY_2}\times \on{Id}_{\bD_2}),$$
which are mutually adjoint at the triangulated level. 

\medskip

Note that at the triangulated level we have an isomorphism
$$f^!(\CO_{\CY_2})\simeq f^*(\CO_{\CY_2})\underset{\CO_{\CY_2}}
\otimes \Lambda^n(\on{Norm}_{\CY_1/\CY_2})[-n]\in \bC^{perf}(\CY_1).$$

However, since $\bC^{perf}(\CY_2)$ is a free 0-object of $\DGMod(\bC^{perf}(\CY_2))$,
generated by $\CO_{\CY_2}$, by \propref{induction adj}, we obtain that this isomorphism
takes place also at the level of 1-morphisms in $\DGMod(\bC^{perf}(\CY_2))$.

Hence, we obtain:

\begin{cor} \label{! and * on D}
For a category $\bD_2$ over $\CY_2$, we have an isomorphism:
$$(f^!\underset{\CY_2}\times \on{Id}_{\bD_2})\simeq
(f^*\underset{\CY_2}\times \on{Id}_{\bD_2})\circ (\Lambda^n(\on{Norm}_{\CY_1/\CY_2})
\underset{\CO_{\CY_2}} \otimes ?)[-n].$$
\end{cor}

\sssec{}   \label{closed imm 2}

Let $\bD_2$ be as above, and let us assume that it is equipped with a
t-structure, satisfying the conditions of 
\secref{intr t structure on ten}, and consider the corresponding t-structure on
$\CY_1\underset{\CY_2}\arrowtimes \bD_2$.

\medskip

Let us assume now that $f$ is a regular closed immersion of codimension $k$.
The following assertion is used in the main body of the paper, and follows
immediately from \corref{! and * on D}.

\begin{lem}   \label{closed imm bdd}
Under the above circumstances the functor 
$$(f^*\underset{\CY_2}\times \on{Id}_{\bD_2}): 
\ua\bD{}_2\to \CY_1\underset{\CY_2}\arrowtimes \bD_2$$
has a cohomological amplitude bounded by $k$.
\end{lem}

\ssec{}

Returning to the general context of Artin stacks, assume that 
a morphism $f:\CY_1\to \CY_2$ is representable.

\sssec{}  \label{coherent direct image}

Let us call a quasi-coherent sheaf on 
$\CY_1$ {\it adapted} if the higher cohomologies
of its direct image on $\CY_2$ vanish. Let $\bC_{adapt}(\QCoh(\CY_1))$ denote
the full DG subcategory of $\bC(\QCoh(\CY_1))$ consisting of complexes of
adapted quasi-coherent sheaves. 

\medskip

We will make the following assumption on the morphism $f$:
any coherent sheaf $\CY_2$ can be embedded into
an adapted  quasi-coherent sheaf.

\medskip

We claim that under the above circumstances, we can lift the direct image functor
$f_*:\bD(\QCoh(\CY_1))\to \bD(\QCoh(\CY_2))$ to
a 1-morphism in $\DGMod(\bC^{perf}(\CY_2))$.

\medskip

Indeed $f$ is of finite type, the above assumption implies that any object of
$\bC(\QCoh(\CY_1))$ admits a quasi-isomorphism into an object of
$\bC_{adapt}(\QCoh(\CY_1))$. Hence, by \lemref{adapted replacement}, 
the functor
\begin{multline*}
\Ho\left(\bC_{adapt}(\QCoh(\CY_1))\right)/
\Ho\left(\bC_{adapt,acycl}(\QCoh(\CY_1))\right)\to \\
\Ho\left(\bC(\QCoh(\CY_1))\right)/
\Ho\left(\bC(\QCoh_{acycl}(\CY_1))\right)=:\bD(\QCoh(\CY_1))
\end{multline*}
is an equivalence, where
$$\bC_{adapt,acycl}(\QCoh(\CY_1)):=
\bC_{adapt}(\QCoh(\CY_1))\cap \bC(\QCoh_{acycl}(\CY_1)).$$
Since, 
$$f_*(\bC_{adapt,acycl}(\QCoh(\CY_1)))\subset \bC_{acycl}(\QCoh(\CY_2)),$$
the lifting of $f_*$ to the DG level follows from \lemref{action on quotient}.

\medskip

Moreover, we claim that the adjunction morphism
$\on{Id}_{\bD(\QCoh(\CY_2))}\to f_*\circ f^*$
at the level of triangulated categories can also be lifted
to a 2-morphism in $\DGMod(\bC^{perf}(\CY_2))$.
This follows again from the fact that $\bC^{perf}(\CY_2)$ is free
as a 1-object of $\DGMod(\bC^{perf}(\CY_2))$, so by 
\lemref{induction adj}, this 2-morphism is determined by the
map
$$\CO_{\CY_2}\to f_*\circ f^*(\CO_{\CY_2})\in \bD(\QCoh(\CY_2)).$$

\sssec{} \label{proper morphism adjunctions}

Finally, let us consider the following situation. Assume that the map $f$ is proper and smooth.
We define the functor $f_?:\bD^{perf}(\Coh(\CY_1))\to \bD^{perf}(\Coh(\CY_2))$ by
$$f_?(\CN):=f_*(\CN\otimes \Omega^n_{\CY_1/\CY_2})[n],$$
where $n$ is the relative dimension. 

\medskip

By \secref{coherent direct image}, $f_?$
lifts to a 1-morphism in $\DGMod(\bC^{perf}(\CY_2))$.

\medskip

Note that the assumptions on $f$ imply that $f_?$ admits a right 
adjoint at the triangulated level. As in \secref{closed imm 3}, 
we obtain that there exists a 1-morphism
$f^?:\DGMod(\bC^{perf}(\CY_2))\to \DGMod(\bC^{perf}(\CY_1))$
in $\DGMod(\bC^{perf}(\CY_2))$, such that
$(f_?,f^?)$ form a pair of mutually adjoint functors at the triangulated
level. As in \secref{coherent direct image}, the adjunction morphism
$$f_?\circ f^?\to \on{Id}_{\bD^{perf}(\Coh(\CY_2))}$$
lifts to a 2-morphism in $\DGMod(\bC^{perf}(\CY_2))$.

\medskip

Finally, we note that, as in {\it loc. cit.}, from the usual Serre duality, 
we obtain a 2-isomorphism in $\DGMod(\bC^{perf}(\CY_2))$:
\begin{equation} \label{Serre expected}
f^?\simeq f^*.
\end{equation}

\newpage

\centerline{\bf \large Part IV: Triangulated categories arising in representation theory}

\bigskip

\section{A renormalization procedure}  \label{a renormalization procedure}

\ssec{}      \label{renorm without t}

Let $\bD$ be a co-complete triangulated category, equipped
with a DG model. Let $X_a,a\in A$ be a collection of objects in $\bD$.
Starting with this data we will construct a pair of new triangulated 
categories, both equipped with models.

\sssec{}

Let $\bD^f$ be the Karoubian envelope of the triangulated subcategory of
$\bD$ {\it strongly generated} by the objects $X_a$.
I.e., this is the smallest full triangulated subcategory of
$\bD$, containing all these objects and closed under direct summands. 
By \secref{rigid quotients}, $\bD^f$, and its embedding into $\bD$, come equipped
with models.

\medskip

Set $\bD_{ren}$ be the ind-completion of $\bD^f$, i.e.,
$\ua\bD^f$, see \secref{big and small, triang}.
This is a co-complete triangulated category,
which is also equipped with a model. When thinking of
$\bD^f$ as a subcategory of $\bD_{ren}$, we will sometimes
denote it also by $\bD^f_{ren}$.

\sssec{}

We claim that there exists
a pair of functors 
$$\Psi:\bD_{ren}\to \bD \text{ and }
\Phi:\bD\to \bD_{ren},$$
such that $\Phi$ is the right adjoint of $\Psi$. 

\medskip

Indeed, let $\bD=\Ho(\bC)$,
we set $\bC^f\subset \bC$ be the preimage of $\bD^f\subset \bD$.
Then $\bD_{ren}\simeq\Ho(\ua\bC^f)$.

\medskip

The functor 
$$\Phi: \bD\simeq \Ho(\bC)
\to \bD((\bC^f)^{op}\mod)\simeq \bD_{ren}$$ 
is defined tautologically. Namely, an object $Y\in \Ho(\bC)$
maps to the $(\bC^f)^{op}$-module 
$$X\mapsto \Hom^\bullet_{\bC}(X,Y).$$

\medskip

The restriction of $\Psi$ to $\bD^f_{ren}\subset \bD_{ren}$
is by definition the tautological embedding $\bD^f_{ren}=\bD^f\hookrightarrow \bD$.
For $X\in \bD^f$, the adjunction
$$\Hom_{\bD}(\Psi(X),Y)\simeq \Hom_{\bD_{ren}}(X,\Phi(Y))$$
is evident.

\medskip

Since $\bD^f_{ren}$ generates $\bD_{ren}$, the functor
$\Psi$ extends canonically onto the entire $\bD_{ren}$ by
the adjunction property. 

\sssec{}  \label{functor Xi}

Suppose for a moment that the category $\bD$ was itself generated by its
subcategory $\bD^c$ of compact objects. \footnote{This assumption is {\it not} satisfied
in the examples for which the notion of $\bD_{ren}$ is developed here, i.e.,
D-modules on infinite-dimensional schemes, or Kac-Moody representations.
However, it is satisfied in the example of the derived category of quasi-coherent
sheaves on a singular algebraic variety.}

\medskip

Assume that $\bD^c\subset \bD^f$. The ind-extension of the above embedding
defines a functor, that we shall denote $\Xi:\bD\to \bD_{ren}$. It is easy to see
that $\Xi$ is the left adjoint of $\Psi$. In addition, $\Xi$ is fully faithful, which implies
that the adjunction map $\on{Id}_\bD\to \Psi\circ \Xi$ is an isomorphism.

\medskip

It is easy to see that $\Psi$ is an equivalence if and only if the objects 
$X_a$, $a\in A$ generate $\bD$ and are compact.

\sssec{}   \label{criter for renorm plus}
Assume that in the set-up of \secref{renorm without t}, $\bD$ is
equipped with a t-structure.   
As usual, let us denote $\bD^+:=\underset{n}\cup\, \bD^{\geq -n}$,
$\bD^-:=\underset{n}\cup\, \bD^{\leq n}$, and $\bD^b:=\bD^+\cap \bD^-$.
Assume that $\bD^f\subset \bD^+$.

\medskip

Assume that the following condition holds:

\medskip

\noindent(*)  {\it There exists an exact and conservative functor $\sF:\bD\to \Vect$
(conservative means $\sF(Y[n])=0\,\, \forall n\in \BZ\,\Rightarrow\, Y=0$), which
commutes with colimits, and a filtered inverse system 
$\{\sZ_\sk\},\sk\in \sK$ of objects from $\bD^f$ such that 
for $Y\in \bD^+$ we have a functorial isomorphism
$$\sF(Y)\simeq \underset{k}{colim}\, \Hom_\bD(\sZ_\sk,Y)=0.$$}

\medskip

\begin{prop} \label{modifying plus}
Under the above circumstances, the adjunction map
$\Psi\circ \Phi|_{\bD^+}\to \on{Id}_{\bD^+}$ is an isomorphism.
\end{prop}

\begin{proof}

For $Y\in \bD^{\geq 0}$, let $X$ denote $\Phi(Y)$. By \lemref{generate colimit}, 
$X\simeq \underset{I}{hocolim}(X_I)$ for some set $I$ and a homotopy $I$-object $X_I$ 
of a DG model of $\bD_{ren}$, with all $X_i$ being in $\bD^f_{ren}$. 
Denote $Y_I=\Psi(X_I)$, so $\Psi\circ \Phi(Y)\simeq \underset{I}{hocolim}(Y_I)$.

\medskip

It suffices to show that the arrow
$$\sF(\underset{I}{hocolim}(Y_I))\to \sF(Y)$$ is 
an isomorphism.

\medskip

We have
$$\sF(\underset{I}{hocolim}(Y_I))\simeq \underset{i}{colim}\, \sF(Y_i)\simeq
\underset{i,k}{colim}\, \Hom_\bD(\sZ_\sk,Y_i),$$
since $\sF$ commutes with colimits and $Y_i\in \bD^+$, and
$$\sF(Y)\simeq \underset{k}{colim}\, \Hom_\bD(\sZ_\sk,Y)\simeq
\underset{k}{colim}\, \Hom_{\bD_{ren}}(\sZ_\sk,X)\simeq
\underset{k,i}{colim}\, \Hom_{\bD_{ren}}(\sZ_\sk,X_i),$$
and the isomorphism is manifest.

\end{proof}

\ssec{}   \label{renorm with t}

Assume again that in the set-up of \secref{renorm without t}, $\bD$ is
equipped with a t-structure. 
\footnote{We are grateful to Jacob Lurie and Amnon Neeman
for help with the material in this subsection.} 

\begin{prop}  \label{D plus}
Assume that the adjunction map $\Psi\circ \Phi(Y)\to Y$
is an isomorphism for $Y\in \bD^+$ (in particular,
the functor $\Phi$ restricted to $\bD^+$ is fully faithful).

\smallskip

\noindent{\em(a)} 
There exists a unique t-structure on $\bD_{ren}$ such that the
functor $\Phi$ induces an exact equivalence $\bD^+\to \bD_{ren}^+$.

\smallskip

\noindent{\em(b)} The functor $\Psi$ is exact with respect the t-structure
of point (a).

\end{prop}

\begin{proof}

The requirement on the t-structure implies that $\bD_{ren}^{> 0}$
equals the essential image of $\Phi|_{\bD^{> 0}}$. Hence, 
$\bD_{ren}^{\leq 0}$, being its left orthogonal, consists of
$$\{X\in \bD_{ren},\,|\, \Psi(X)\in \bD^{\leq 0}\}.$$

To prove (a) we need to show that every $X\in \bD_{ren}$ admits
a truncation triangle. Consider the map $X\to \Phi(\tau^{>0}(\Psi(X)))$.
We have $\Phi(\tau^{>0}(\Psi(X)))\in \bD_{ren}^{> 0}$, and it remains
to see that 
$$\on{Cone}\left(X\to \Phi(\tau^{>0}(\Psi(X)))\right)[-1]\in \bD_{ren}^{\leq 0}.$$
We have
$$\Psi\left(\on{Cone}\left(X\to \Phi(\tau^{>0}(\Psi(X)))[-1]\right)\right)\simeq
\on{Cone}\left(\Psi(X)\to \tau^{>0}(\Psi(X))\right)[-1]\simeq
\tau^{\leq 0}(\Psi(X)),$$
as required.

\medskip

Point (b) of the proposition follows from the construction. 

\end{proof}

\noindent{\it Remark.} Assume for a moment that the t-structure on $\bD$
has the property that $\underset{k\geq 0}\cap\, \bD^{\leq -k}=0$, and
$\Phi$ is fully faithful. Then we obtain $\bD$ is the triangulated quotient of 
$\bD_{ren}$ by the subcategory of acyclic objects, i.e., 
$\underset{k\geq 0}\cap\, \bD_{ren}^{\leq -k}$.

\ssec{}   \label{almost compact}

Let $\bD$ be a co-complete triangulated category equipped
with a DG model and a t-structure. Assume that the t-structure is
compatible with colimits. 

\medskip

We say that an object $Y\in \bD$ is almost compact if for any $k$
and a homotopy $I$-object $X_I$, the map
$$\underset{i\in I}{colim}\, \Hom_\bD(Y[k],X_i)\to
\Hom_\bD(Y[k],hocolim(X_i))$$
is an isomorphism, {\it provided that} $X_i\in \bD^{\geq 0}$ for all $i\in I$.

\sssec{}

Let $\bD$ be a triangulated category as in \secref{renorm with t}, satisfying
the assumption of \propref{D plus}.

\begin{prop}  \label{D plus almost compact}
Assume that the t-structure on $\bD$ is compatible with colimits, and 
that the objects $X_a\in \bD$ are almost compact. Then the t-structure on
$\bD_{ren}$ is compatible with colimits.
\end{prop}

\begin{proof}

It follows from \lemref{functor and limit} that if $X_I$ is a homotopy $I$-object
of (a DG model of) $\bD_{ren}$ with $X_i\in \bD_{ren}^{\leq 0}$,
then $hocolim(X_I)\in \bD_{ren}^{\leq 0}$. This does not require the
objects $X_a$ to be almost compact.

\medskip

Assume now that $X_i\in \bD_{ren}^{\geq 0}$. We have 
$\Psi(X_i)\in \bD^{\geq 0}$ and $hocolim(\Psi(X_I))\in \bD^{\geq 0}$.
Thus, it suffices to show that the map
\begin{equation} \label{Phi colim}
hocolim(X_I) \to \Phi(hocolim(\Psi(X_I)))
\end{equation}
is an isomorphism.

\medskip

The next assertion results from the definitions.

\begin{lem}  \label{ind lim on plus}
Assume in the circumstances of \propref{D plus} that the objects
$X_a$ are almost compact. Then, if $Y_I$ is a homotopy $I$-object
of (a DG model of) $\bD$ with $Y_i\in \bD^{\geq 0}$, the natural map
$$hocolim(\Phi(Y_I))\to \Phi(hocolim(Y_I))$$
is an isomorphism.
\end{lem}

Applying the lemma to $Y_I=\Psi(X_I)$, we obtain
$$\Phi(hocolim(\Psi(X_I)))\simeq hocolim (\Phi(\Psi(X_I)))\simeq hocolim(X_I).$$

\end{proof}

\noindent{\it Remark.} It is easy to see that under the above circumstances,
the t-structure on $\bD_{ren}$ is compatible with colimits if {\it and only if} each $X_a$
is almost compact in $\bD$.

\ssec{}   \label{criter for t renorm}

Let $\bD$ be again a category as in \secref{renorm with t}. We would
like now to give a criterion for when the t-structure on it is compactly
generated. 

\medskip

Assume that there exists an inverse family $\{\sZ_k\}$ as in \secref{criter for renorm plus},
with the following additional properties: 

\begin{itemize}

\item(i) $\sZ_\sk\in \bD^{\leq 0}$ and for $\sk_2\geq \sk_1$,  
$\on{Cone}(Z_{\sk_2}\to Z_{\sk_1})\in \bD^{< 0}$.

\item(ii) Any $Y\in \bD^f$ belongs to $\bD^b$ and if
$Y\in \bD^{\leq 0}\cap \bD^f$, there exists an object $\sZ$, equal
to a finite direct sum of $\sZ_k$'s, and a map
$\sZ\to Y$ with $\on{Cone}(\sZ\to Y)\in \bD^{< 0}$.

\end{itemize}

\begin{prop}  \label{renorm comp gen}
Under the above circumstances, the t-structure on $\bD_{ren}$
is compactly generated.
\end{prop} 

\begin{proof}

By \lemref{comp gen criter}, we have to show that if $X\in \bD_{ren}$ satisfies
$\Hom_{\bD_{ren}}(Y,X)$ for all $Y\in \bD^f_{ren}\cap \bD^{\leq 0}_{ren}$,
then $X\in \bD_{ren}^{>0}$.

\medskip

The proof of \propref{modifying plus} shows that $\sF\circ \Psi(X[n])=0$ for $n>0$,
and hence $\Psi(X)\in \bD^{>0}$. Consider the object 
$X':=\on{Cone}(X\to \Phi\circ \Psi(X))[-1]$. It satisfies the same assumption as $X$, 
and also $\Psi(X')=0$. We claim that any such
object equals $0$. 

\medskip

Indeed, suppose, by contradiction, that $X'\neq 0$.
Let $n$ be the minimal integer such that $\Hom_{\bD_{ren}}(Y[-n],X')\neq 0$
for $Y\in \bD^f_{ren}\cap \bD^{\leq 0}_{ren}$. By assumption, $n\geq 0$,
and it exists since $\bD^f_{ren}$ generates $\bD_{ren}$. Consider
a map $\sZ\to Y$ as in (ii). By the minimality assumption on $n$,
the map $\Hom_{\bD_{ren}}(Y[-n],X')\to \Hom_{\bD_{ren}}(\sZ[-n],X')$ is 
injective. So, there exists an index $k'$ and a non-zero element in
$\Hom_{\bD_{ren}}(\sZ_{k'}[-n],X')$. By (i), for $k_2\geq k_1$, he map
$$\Hom_{\bD_{ren}}(\sZ_{k_1}[-n],X')\to \Hom_{\bD_{ren}}(\sZ_{k_2}[-n],X')$$
is injective. Hence, we obtain that $$\underset{\sk\in \sK}{colim}\, 
\Hom_{\bD_{ren}}(\sZ_\sk[-n],X')\neq 0.$$
However, by assumption (*),
$$\underset{\sk\in \sK}{colim}\, 
\Hom_{\bD_{ren}}(\sZ_\sk[-n],X')\simeq \underset{\sk\in \sK}{colim}\, 
\Hom_{\bD}(\sZ_\sk[-n],\Psi(X'))\simeq \sF\circ \Psi(X[n]),$$
which is a contradiction.

\end{proof}

\ssec{}

Let us consider an example of the situation described in \secref{renorm with t}.
(Another example relevant to representations of Kac-Moody algebras will
be considered in \secref{rep theory}).

\sssec{}   \label{good abelian}

Let $\obC$ be a Grothendieck abelian category, i.e., it is closed
under inductive limits, and the functor of the inductive limit
over a filtered index category is exact. We take $\bD:=\bD(\obC)$ be the 
usual derived category of $\obC$. 

\medskip

Let $\obC^f\subset \obC$
be a small abelian subcategory, satisfying the following two conditions:

\begin{itemize}

\item(a) Every object of $\obC$ can be presented as an inductive limit of objects
the $\obC^f$. 

\item(b) Every object $X\in \obC^f$ is almost compact as an object of $\bD$,
i.e., the functors $\on{Ext}^i(X_a,?)$ commute with filtered inductive limits
for $i=0,1,2,...$

\end{itemize}

Note that condition (a), combined with (b) for just $i=0,1$ imply 
that $\obC$ identifies with the category $\on{Ind}(\obC^f)$.

\sssec{}

We take the objects $X_a\in \bD$ to be the objects from $\obC^f$,
and let us form the corresponding categories 
$\bD^f=\bD^f_{ren}\subset \bD_{ren}$.
Note that we have a natural equivalence:
$$\bD^f_{ren}\simeq \bD^b(\obC^f).$$

\medskip

We claim that the conditions of \propref{D plus} hold. Indeed, we only have
to check that the adjunction
\begin{equation} \label{plus adjunction}
\Psi\circ \Phi(Y)\to Y
\end{equation}
is an isomorphism for $Y\in \bD^{\geq 0}$.

\begin{lem}
Every object $Y\in \bD^+$ can be represented as $hocolim(Y_I)$,
where $Y_i\in \bD^b(\obC^f)$.
\end{lem}

This lemma, combined with \lemref{ind lim on plus} reduces 
\eqref{plus adjunction} to the case of $Y\in \bD^f$, for which
it follows from the definitions.

\medskip

In addition, we claim that the t-structure on $\bD_{ren}$ is
compactly generated. Indeed,
by \lemref{colimit generate t}, it is enough to show that for any
$X\in \bD_{ren}^f$, the object $\tau^{\leq 0}(X)$ belongs to the subcategory
generated by extensions, direct sums and non-positive shifts by $\obC^f$.
However, $\tau^{\leq 0}(X)\in \bD^{b,\leq 0}(\obC^f)$.

\medskip

\noindent{\it Remark.} As was explained to us by A.~Neeman, the
above construction reproduces one of \cite{Kr}. Namely, one can show 
that $\bD_{ren}$ is equivalent to the homotopy category of complexes 
of injective objects in $\obC$. 

\sssec{}

Let us consider two specific examples of the situation described
in \secref{good abelian}. Let $\CY$ be a strict ind-scheme of 
ind-finite type. I.e., $\CY$ is a union $\underset{i\geq 0}\cup\, \CY_i$,
where $\CY_i$ are schemes of finite type, and the maps
$\CY_i\to \CY_{i+1}$ are closed embeddings. We let
$\obC$ be either 
$$\QCoh(\CY) \text{ or } \fD(\CY)\mod.$$
The corresponding categories $\obC^f$ identify with
$$2\on{-}colim\, \Coh(\CY_i) \text{ and }
2\on{-}colim\, \fD^f(\CY_i)\mod,$$
respectively, where $\fD^f(\cdot)\mod$ denotes the category of finitely generated
(i.e., coherent) D-modules over a scheme of finite type. We shall denote the
corresponding categories by $\bD_{ren}(\QCoh(\CY)$ and
$\bD_{ren}(\fD(\CY)\mod)$, respectively.

\medskip

In both cases, the derived functor of global sections 
$\Gamma:\bD^f\to \bD(\Vect_k)$ admits a DG model and gives rise to
a functor
$$\Gamma:\bD_{ren}\to \bD(\Vect_k).$$

\medskip

Note that the latter functor does not, in general, factor through $\bD_{ren}
\twoheadrightarrow \bD$. 

\medskip

For example, let assume that each is $\CY_i$ smooth and projective,
and $\dim(\CY_{i+1})>\dim(\CY_i)$. Let $\obC=\QCoh(\CY)$. We take $X\in \bD_{ren}$
to be the dualizing complex $K_\CY$, which identifies with $\underset{i}{colim}\, K_{\CY_i}$.
The image of $K_\CY$ in $\bD(\QCoh(\CY))$ is zero, however,
$\Gamma(\CY,K_\CY)\simeq k$.

\section{The derived category of Kac-Moody modules}   \label{rep theory}

\ssec{}

Let $\hg_\kappa\mod$ be the abelian category of modules over the Kac-Moody
algebra $\hg$ at level $\kappa$. Let $\bD(\hg_\kappa\mod)$ be the usual
derived category of $\hg_\kappa\mod$, i.e., the triangulated quotient of the
homotopy category of complexes of objects $\hg_\kappa\mod$ by the subcategory
of acyclic ones. By \secref{rigid quotients}, $\bD(\hg_\kappa\mod)$ naturally
comes equipped with a DG model. 

\sssec{}

By construction, $\bD(\hg_\kappa\mod)$ is co-complete, and is equipped
with a t-structure compatible with colimits. The difficulty in working with $\bD(\hg_\kappa\mod)$ is 
that it is not generated by compact objects.

\medskip

For $i\geq 0$ let us denote by $\BV_{\kappa,i}\in \hg_\kappa\mod$
the induced module $\Ind^{\hg_\kappa}_{\fg\otimes (t^i\cdot \BC[[t]])}(\BC)$.
By \cite{FG2}, Proposition 23.12, these objects are almost compact
(see \secref{renorm with t}) in $\bD(\hg_\kappa\mod)$, but they are not compact.

\sssec{}

We are now going to apply a renormalization procedure described
in \secref{renorm without t} and obtain a better behaved triangulated category:

\bigskip

\noindent{\it 
We take $\bD^f(\hg_\kappa\mod)$ to be the Karoubian envelope of 
the subcategory of $\bD(\hg_\kappa\mod)$ {\it strongly generated} by the objects $\BV_{\kappa,i}$, $i=0,1,...$.}

\bigskip

By \secref{renorm without t}, we obtain a triangulated 
category that we shall denote
$\bD_{ren}(\hg_\kappa\mod)$, equipped with a DG model, which is co-complete,
and endowed with a pair of mutually adjoint functors
$$\Psi:\bD_{ren}(\hg_\kappa\mod)\leftrightarrows 
\bD(\hg_\kappa\mod):\Phi.$$

We claim that the conditions of Sections \ref{criter for renorm plus} and 
\ref{criter for t renorm} are satisfied. Indeed, we take $\sF$ to be the usual 
forgetful functor, and $\sZ_k:=\BV_{\kappa,k}$. Thus we obtain that 
$$\Psi\circ \Phi|_{\bD^+(\hg_\kappa\mod)}\simeq \on{Id}_{\bD^+(\hg_\kappa\mod)}.$$
Moreover, $\bD_{ren}(\hg_\kappa\mod)$ has a compactly generated t-structure, 
in which $\Psi$ exact and $\Phi|_{\bD^+(\hg_\kappa\mod)}$ exact, and we have an equivalence
$$\bD^+_{ren}(\hg_\kappa\mod)\leftrightarrows 
\bD^+(\hg_\kappa\mod).$$
The kernel of $\Psi$ is the subcategory
of $\bD_{ren}(\hg_\kappa\mod)$, consisting of acyclic objects with
respect to this t-structure.

\medskip

\noindent{\it Remark.} We do not know whether the t-structure on 
$\bD(\hg_\kappa\mod)$ (or, equivalently, on $\bD_{ren}(\hg_\kappa\mod)$) induces a
t-structure on the subcategory $\bD^f(\hg_\kappa\mod)$. \footnote{Most probably, it does not.}
I.e., the formalism of \secref{good abelian} is a priori not applicable in this case.

\ssec{The critical level case}

Let us specialize to the case $\kappa=\kappa_\crit$. 
Let $\fZ_\fg$ denote the center of $\hg_\crit\mod$. This is a topological
commutative algebra isomorphic to the inverse limit of $\fZ^i_\fg$
(see \cite{FG2}, Sect. 7.1), where each $\fZ^i_\fg$ is isomorphic
to a polynomial algebra on infinitely many variables, and the ideals
$\on{ker}(\fZ^{i_2}_\fg\to \fZ^{i_1}_\fg)$ are finitely presented and
regular. By \cite{BD}, Theorem 3.7.9, the action of $\fZ_\fg$
on $\BV_{\crit,i}$ factors through $\fZ^i_\fg$, and $\BV_{\crit,i}$ 
is flat as a $\fZ^i_\fg$-module.

\sssec{}   \label{sZ}

Let $\sZ$ be a discrete quotient algebra of $\fZ_\fg$, such that for
some (=any) $i$ the map $\fZ_\fg\to \sZ$ factors through 
$\fZ^i_\fg\to \sZ$, with the ideal $\on{ker}(\fZ^{i}_\fg\to \sZ)$ being
finitely presented and regular. 

\medskip

Let $\hg_\crit\mod_\sZ$ be the full abelian subcategory of
$\hg_\crit\mod$, consisting of modules on which the action
of the center $\fZ_\fg$ factors through $\fZ$. Let $\bD(\hg_\crit\mod_\sZ)$
be the usual derived category of this abelian category; it is co-complete
and has a t-structure compatible with colimits, and naturally
comes equipped with a DG model. 

\medskip

We have a tautological functor
$$(\iota^\sZ_\hg)_*:\bD(\hg_\crit\mod_\sZ)\to \bD(\hg_\crit\mod).$$

\sssec{}   \label{good category sZ}

We shall now define a renormalized version of the category $\bD(\hg_\crit\mod_\sZ)$,
denoted $\bD_{ren}(\hg_\crit\mod_\sZ)$:

\bigskip

\noindent{\it We define
$\bD^f(\hg_\crit\mod_\sZ)$ to consist of those objects of $\CM\in \bD(\hg_\crit\mod_\sZ)$,
for which $(\iota^{\sZ}_\hg)_*(\CM)\in  \bD^f(\hg_\crit\mod)$.}

\bigskip

By \secref{renorm without t}, we obtain a category which we shall denote $\bD_{ren}(\hg_\crit\mod_\sZ)$,
and a pair of mutually adjoint functors
$$\Psi^\sZ:\bD_{ren}(\hg_\crit\mod_\sZ)\leftrightarrows
\bD(\hg_\crit\mod_\sZ):\Phi^\sZ.$$

\medskip 

We claim that conditions of \secref{criter for renorm plus} and \secref{criter for t renorm}
are satisfied. Indeed, we take $\sF$ to be again the forgetful functor, and we take $\sZ_k$
to be the objects
$$\BV^\sZ_{\crit,i}:=\BV_{\crit,i}\underset{\fZ_\fg^i}\otimes \sZ$$
for $i$ such that $\sZ$ is a quotient of $\fZ^i_\fg$. (Note that
$\BV^\sZ_{\crit,i}$ is an object of the abelian category 
$\hg_\crit\mod_\sZ$.) 

\medskip

Thus, we obtain that 
$\Psi^\sZ\circ \Phi^\sZ|_{\bD^+(\hg_\crit\mod_\sZ)}\to \on{Id}_{\bD^+(\hg_\crit\mod_\sZ)}$,
and that $\bD_{ren}(\hg_\crit\mod_\sZ)$ acquires a compactly generated t-structure,
such that the functors $\Psi^\sZ$ and $\Phi^\sZ_{\bD^+(\hg_\crit\mod_\sZ)}$ are exact, and
induce mutually quasi-inverse equivalences
$$\Psi^\sZ:\bD^+_{ren}(\hg_\crit\mod)\leftrightarrows 
\bD^+(\hg_\crit\mod):\Phi^\sZ.$$

\sssec{}

By construction, the category $\bD(\hg_\crit\mod_\sZ)$ is realized
as a quotient
$$\Ho\left(\bC(\hg_\crit\mod_\sZ)\right)/\Ho\left(\bC_{acycl}(\hg_\crit\mod_\sZ)\right),$$
so by \secref{action on quotient}, it lifts to an object of
$\DGMod\left(\bC^b(\sZ\mod^{free,fin.rk.})\right)$, see \secref{over affine}.
Since $\bD(\hg_\crit\mod_\sZ)$ is Karoubian, this structure extends to that
of triangulated category over $\Spec(\sZ)$.

\medskip

Hence, the category $\bD^f(\hg_\crit\mod_\sZ)$, which is also Karoubian, inherits this
structure.

\ssec{Changing the central character}

Let $\sZ'$ be another discrete and regular quotient of $\fZ_\fg$, such that the
projection $\fZ_\fg \twoheadrightarrow \sZ'$ factors through $\sZ$. We have
a regular closed immersion
$$\Spec(\sZ')\hookrightarrow \Spec(\sZ)$$
that we shall denote by $\iota^{\sZ',\sZ}$.

\sssec{}

The functor 
$$(\iota^{\sZ',\sZ})^*:\bD^{perf}(\sZ\mod)\to \bD^{perf}(\sZ'\mod)$$
lifts naturally to a 1-morphism in $\DGMonCat$, which we denote by the same
character. Since $\iota^{\sZ',\sZ}$ is a regular embedding, the adjoint functor
$$\iota^{\sZ',\sZ}_*:\bD(Z'\mod)\to \bD(Z\mod)$$
sends $\bD^{perf}(\sZ'\mod)$ to $\bD^{perf}(\sZ\mod)$.

\medskip

Consider the base-changed category
\begin{equation} \label{base change Z}
\Spec(\sZ')\underset{\Spec(\sZ)}\times \bD^f(\hg_\crit\mod_\sZ).
\end{equation}
This is a triangulated category over $\Spec(\sZ')$. Consider also
its ind-completion
$$\Spec(\sZ')\underset{\Spec(\sZ)}\arrowtimes \bD^f(\hg_\crit\mod_\sZ).$$

\medskip

By \secref{dir image}, the functors $(\iota^{\sZ',\sZ})^*$ and $(\iota^{\sZ',\sZ})_*$
induce a pair of mutually adjoint functors, denoted
$$(\iota^{\sZ',\sZ}\underset{\Spec(\sZ)}\otimes \on{Id})^*,\,\,
(\iota^{\sZ',\sZ}\underset{\Spec(\sZ)}\otimes \on{Id})_*,$$
respectively:
$$\bD^f(\hg_\crit\mod_\sZ)
\leftrightarrows 
\Spec(\sZ')\underset{\Spec(\sZ)}\times \bD^f(\hg_\crit\mod_\sZ)$$
and
$$\bD_{ren}(\hg_\crit\mod_\sZ)
\leftrightarrows 
\Spec(\sZ')\underset{\Spec(\sZ)}\arrowtimes \bD^f(\hg_\crit\mod_\sZ)$$

The functor $(\iota^{\sZ',\sZ}\underset{\Spec(\sZ)}\otimes \on{Id})^*$ is right-exact. 
By \propref{affine exact}, the functor $(\iota^{\sZ',\sZ}\underset{\Spec(\sZ)}\otimes \on{Id})_*$
is exact and conservative.

\sssec{}

Consider now the tautological functor
\begin{equation} \label{naive i g}
(\iota^{\sZ',\sZ}_\hg)_*:\bD(\hg_\crit\mod_{\sZ'})\to \bD(\hg_\crit\mod_{\sZ}).
\end{equation}
It is naturally equipped with a DG model, and as such is
compatible with the action of $\bC^b(\sZ\mod^{free,fin.rk.})$, and hence
is a functor between categories over $\Spec(\sZ)$.

\medskip

By definition, the above functor sends $\bD^f(\hg_\crit\mod_{\sZ'})$ to
$\bD^f(\hg_\crit\mod_{\sZ})$. Hence, by \secref{action on subcat}, we obtain
a functor 
$$(\iota^{\sZ',\sZ}_\hg)_*:\bD^f(\hg_\crit\mod_{\sZ'})\to \bD^f(\hg_\crit\mod_{\sZ})$$
between triangulated categories over $\Spec(\sZ)$. We shall denote by
$$(\iota^{\sZ',\sZ}_{\hg,ren})_*:\bD_{ren}(\hg_\crit\mod_{\sZ'})\to 
\bD_{ren}(\hg_\crit\mod_{\sZ})$$
its ind-extension.

\sssec{}  \label{restr central char 1}

The functor \eqref{naive i g} admits a right adjoint, denoted $(\iota^{\sZ',\sZ}_\hg)^*$. 
The following
assertion is established in \cite{FG2}, Lemma 7.5:

\begin{lem}  \label{i g well behaved}
For $\sZ'$, $\sZ$ as above we have a commutative diagram of functors
$$
\CD
\bD(\hg_\crit\mod_{\sZ})   @>{(\iota^{\sZ',\sZ}_\hg)^*}>>   \bD(\hg_\crit\mod_{\sZ'})  \\
@VVV   @VVV   \\
\bD(\sZ\mod)  @>{(\iota^{\sZ',\sZ})^*}>> \bD(\sZ'\mod).
\endCD
$$ 
\end{lem}

The functor $(\iota^{\sZ,\sZ'}_\hg)^*$ is naturally equipped with a DG model, 
and as such is compatible with the action of $\bC^b(\sZ\mod^{free,fin.rk.})$, and hence
is a functor between categories over $\Spec(\sZ)$.

\medskip

\begin{lem}
The functor 
$$(\iota^{\sZ',\sZ}_\hg)^*:\bD(\hg_\crit\mod_\sZ)\to \bD(\hg_\crit\mod_{\sZ'})$$
sends $\bD^f(\hg_\crit\mod_\sZ)$ to $\bD^f(\hg_\crit\mod)$.
\end{lem}

\begin{proof}

For $\CM\in \bD(\hg_\crit\mod_\sZ)$,
$$(\iota^{\sZ'}_\hg)_*\circ (\iota^{\sZ',\sZ}_\hg)^*(\CM)\simeq
(\iota^{\sZ}_\hg)_*\left((\iota^{\sZ',\sZ}_\hg)_*\circ (\iota^{\sZ',\sZ}_\hg)^*(\CM)\right)
\simeq (\iota^{\sZ}_\hg)_*\left(\sZ'\underset{\sZ}\otimes \CM\right),$$
and the assertion follows from the fact that $\sZ'$ admits a finite resolution 
by locally free $\sZ$-modules.

\end{proof}

Thus, we obtain a 1-morphism
$$(\iota^{\sZ',\sZ}_\hg)^*:\bD^f(\hg_\crit\mod_\sZ)\to \bD^f(\hg_\crit\mod_{\sZ'})$$
over $\Spec(\sZ)$. 

\sssec{} \label{restr central char 2}

By \secref{univ ppty base change}, the functor $(\iota^{\sZ',\sZ}_\hg)^*$ gives
rise to a functor
\begin{equation} \label{Z to Z'}
(\iota^{\sZ',\sZ}\underset{\Spec(\sZ)}\times
\iota^{\sZ,\sZ'}_\hg)^*:
\Spec(\sZ')\underset{\Spec(\sZ)}\times \bD^f(\hg_\crit\mod_\sZ) \to
\bD^f(\hg_\crit\mod_{\sZ'})
\end{equation}
and by ind-extension a functor
\begin{equation} \label{Z to Z' ind}
(\iota^{\sZ',\sZ}\underset{\Spec(\sZ)}\times
\iota^{\sZ,\sZ'}_\hg)^*:
\Spec(\sZ')\underset{\Spec(\sZ)}\arrowtimes \bD^f(\hg_\crit\mod_\sZ) \to
\bD_{ren}(\hg_\crit\mod_{\sZ'})
\end{equation}

\begin{prop}   \label{prop restr central char}
The functors \eqref{Z to Z'} and \eqref{Z to Z' ind} are fully faithful.
\end{prop}

\begin{proof}

It suffices to prove the assertion concerning the functor \eqref{Z to Z'}. Moreover,
it is easy to see that it suffices to show that for $\CM_1,\CM_2\in 
\bD^f(\hg_\crit\mod_\sZ)$, the functor $(\iota^{\sZ',\sZ}\underset{\Spec(\sZ)}\times
\iota^{\sZ,\sZ'}_\hg)^*$ induces an isomorphism
\begin{multline*}
\Hom_{\Spec(\sZ')\underset{\Spec(\sZ)}\times \bD^f(\hg_\crit\mod_\sZ)}
\left((\iota^{\sZ',\sZ}\underset{\Spec(\sZ)}\otimes \on{Id})^*(\CM_1),
(\iota^{\sZ',\sZ}\underset{\Spec(\sZ)}\otimes \on{Id})^*(\CM_2)\right)\to \\
\to
\Hom_{\bD_{ren}(\hg_\crit\mod_{\sZ'})}\left(\iota^{\sZ,\sZ'}_\hg)^*(\CM_1),
\iota^{\sZ,\sZ'}_\hg)^*(\CM_2)\right).
\end{multline*}

We rewrite the LHS using \corref{tight mon}(2) as
\begin{multline*}
\Hom_{\bD^f(\hg_\crit\mod_\sZ)}\left(\CM_1,
(\iota^{\sZ',\sZ}\underset{\Spec(\sZ)}\otimes \on{Id})_*\circ 
(\iota^{\sZ',\sZ}\underset{\Spec(\sZ)}\otimes \on{Id})^*(\CM_2)\right)\simeq \\
\simeq \Hom_{\bD^f(\hg_\crit\mod_\sZ)}(\CM_1,\sZ'\underset{\sZ}\otimes \CM_2),
\end{multline*}
and the RHS using \lemref{i g well behaved} as
$$\Hom_{\bD^f(\hg_\crit\mod_\sZ)}(\CM_1,(\iota^{\sZ,\sZ'}_\hg)_*\circ 
(\iota^{\sZ,\sZ'}_\hg)^*(\CM_2))\simeq 
\Hom_{\bD^f(\hg_\crit\mod_\sZ)}(\CM_1,\sZ'\underset{\sZ}\otimes \CM_2),$$
implying the desired isomorphism.

\end{proof}

\ssec{}

The functor \eqref{Z to Z' ind} that appears in \propref{prop restr central char}
is not an equivalence of categories. We shall now repeat the manipulation 
of \secref{turning Ups} and turn it into an equivalence by modifying the LHS.

\sssec{}

Consider the functor 
$$(\iota^\sZ_\hg)_*:\bD^f(\hg_\crit\mod_\sZ)\to \bD^f(\hg_\crit\mod)$$
and its ind-extension 
$$(\iota^{\sZ',\sZ}_{\hg,ren})_*:\bD_{ren}(\hg_\crit\mod_\sZ)\to 
\bD_{ren}(\hg_\crit\mod).$$

\begin{prop}   \label{i inf cons}
The functor $(\iota^{\sZ}_{\hg,ren})_*$ is exact, and is conservative 
when restricted to $\bD^+_{ren}(\hg_\crit\mod_\sZ)$.
\end{prop}

The proof will be based on the following lemma, established in 
\cite{FG2}, Proposition 23.11:

\begin{lem}  \label{ext as dir lim} 
For $\CM_1,\CM_2\in \bD^+(\hg_\crit\mod_\sZ)$ with $\CM_1$ almost compact,
the natural map
\begin{multline*}
\underset{i}{colim}\, \Hom_{\bD(\hg_\crit\mod_{\fZ^i_\fg})}
\left((\iota^{\sZ,\fZ_\fg^i}_\hg)_*(\CM_1),(\iota^{\sZ,\fZ_\fg^i}_\hg)_*(\CM_2)\right)
\to \\
\to \Hom_{\bD(\hg_\crit\mod)}\left((\iota^{\sZ}_\hg)_*(\CM_1),
(\iota^{\sZ}_\hg)_*(\CM_2)\right)
\end{multline*} 
is an isomorphism.
\end{lem}

\begin{proof}

The right-exatness of $(\iota^{\sZ}_{\hg,ren})_*$ follows by the definition
of the t-structures on both sides from the following diagram
$$
\CD
\bD_{ren}(\hg_\crit\mod_\sZ)    @>(\iota^{\sZ}_{\hg,ren})_*>>
\bD_{ren}(\hg_\crit\mod)  \\
@V{\Psi^\sZ}VV    @V{\Psi}VV  \\
\bD(\hg_\crit\mod_\sZ)    @>(\iota^{\sZ}_{\hg})_*>>
\bD(\hg_\crit\mod),
\endCD
$$
which is commutative by construction.

\medskip

To prove that $(\iota^{\sZ}_{\hg,ren})_*$ is left-exact and conservative on 
$\bD^+_{ren}(\hg_\crit\mod_\sZ)$, it suffices to show
that the following diagram also commutes:
$$
\CD
\bD_{ren}(\hg_\crit\mod_\sZ)    @>(\iota^{\sZ}_{\hg,ren})_*>>
\bD_{ren}(\hg_\crit\mod)  \\
@A{\Phi^\sZ}AA    @A{\Phi}AA  \\
\bD^+(\hg_\crit\mod_\sZ)    @>(\iota^{\sZ}_{\hg})_*>>
\bD^+(\hg_\crit\mod).
\endCD
$$

\medskip

In other words, we have to show the following: 

\medskip 

\noindent Let $\CN$ be an object
of $\bD^+(\hg_\crit\mod_\sZ)$, represented as $\underset{A}{hocolim}\, \CM_A$
for some set $A$, such that for $a\in A$, $\CM_a\in \bD^f(\hg_\crit\mod_\sZ)$, 
and so that the arrow
$$\underset{a}{colim}\, \Hom_{\bD(\hg_\crit\mod_\sZ)}(\CM',\CM_a)\to
\Hom_{\bD(\hg_\crit\mod_\sZ)}(\CM',\CN)$$
is an isomorphism for any $\CM'\in \bD^f(\hg_\crit\mod_\sZ)$.
We have to show that in this case the arrow
$$\underset{a}{colim}\, \Hom_{\bD(\hg_\crit\mod)}\left(\CM'',
(\iota^{\sZ}_{\hg})_*(\CM_a)\right)\to
\Hom_{\bD(\hg_\crit\mod)}\left(\CM'',
(\iota^{\sZ}_{\hg})_*(\CN)\right)$$
is also an isomorphism for any $\CM''\in \bD^f(\hg_\crit\mod)$.

\medskip

By \lemref{ext as dir lim}, the latter would follow once we show that 
\begin{multline*}
\underset{a}{colim}\, \Hom_{\bD(\hg_\crit\mod_{\fZ^i_\fg})}
\left(\CM'',
(\iota^{\sZ,\fZ^i_\fg}_{\hg})_*(\CM_a)\right)\to \\
\to \Hom_{\bD(\hg_\crit\mod_{\fZ^i_\fg})}\left(\CM'',
(\iota^{\sZ,\fZ^i_\fg}_{\hg})_*(\CN)\right)
\end{multline*}
is an isomorphism for any $i \gg 0$; in particular $i$ is such that we
can consider $\CM''$ as an object of $\bD^f(\hg_\crit\mod_{\fZ^i_\fg})$.

\medskip

We rewrite both sides of the above expression as

\begin{multline*}
\underset{a}{colim}\, \Hom_{\bD(\hg_\crit\mod_{\sZ})}
\left((\iota^{\sZ,\fZ^i_\fg}_{\hg})^*(\CM''),\CM_a\right)
\to  \\
\to \Hom_{\bD(\hg_\crit\mod_{\sZ})}
\left((\iota^{\sZ,\fZ^i_\fg}_{\hg})^*(\CM''),\CN\right).
\end{multline*}

The assertion follows now from the fact that 
$(\iota^{\sZ,\fZ^i_\fg}_{\fg})^*(\CM'')\in \bD^f(\hg_\crit\mod_{\sZ})$,
established above.

\end{proof}

\sssec{}   \label{restr central char 3}

We define a new functor

\begin{equation} \label{Z to Z' plus}
(\iota^{\sZ',\sZ}\underset{\Spec(\sZ)}\times
\iota^{\sZ,\sZ'}_\hg)^{*,+}:
\left(\Spec(\sZ')\underset{\Spec(\sZ)}\arrowtimes \bD^f(\hg_\crit\mod_\sZ)\right)^+ \to
\bD^+_{ren}(\hg_\crit\mod_{\sZ'})
\end{equation}
as follows.

For $\CM\in \Spec(\sZ')\underset{\Spec(\sZ)}\arrowtimes \bD^f(\hg_\crit\mod_\sZ)$,
which is $\geq i$, we set
$$(\iota^{\sZ',\sZ}\underset{\Spec(\sZ)}\times
\iota^{\sZ,\sZ'}_\hg)^{*,+}(\CM):=\tau^{\geq j}\left((\iota^{\sZ',\sZ}\underset{\Spec(\sZ)}\times
\iota^{\sZ,\sZ'}_\hg)^*(\CM)\right)$$
for some/any $j<i$. The independence of the choice of $j$ is assured by \propref{i inf cons},
since we have an isomorphism of functors:
\begin{multline} \label{two iotas}
(\iota^{\sZ'}_\hg)_* \circ (\iota^{\sZ',\sZ}\underset{\Spec(\sZ)}\times
\iota^{\sZ,\sZ'}_\hg)^*\simeq 
(\iota^{\sZ}_\hg)_*\circ (\iota^{\sZ',\sZ}\underset{\Spec(\sZ)}\otimes \on{Id})_*: \\
\Spec(\sZ')\underset{\Spec(\sZ)}\arrowtimes \bD^f(\hg_\crit\mod_\sZ)
\to \bD_{ren}(\hg_\crit\mod),
\end{multline}
and the latter functor is exact.

\medskip

As in \thmref{Upsilon +} one shows that the functor $(\iota^{\sZ',\sZ}\underset{\Spec(\sZ)}\times
\iota^{\sZ,\sZ'}_\hg)^{*,+}$ is an equivalence of categories. 

\medskip

Let $\bD^f(\hg_\crit\mod_{\sZ',\sZ})$ be the full subcategory of 
$\Spec(\sZ')\underset{\Spec(\sZ)}\arrowtimes \bD^f(\hg_\crit\mod_\sZ)$
consisting of objects $\CM$, such that 
$$(\iota^{\sZ',\sZ}\underset{\Spec(\sZ)}\otimes \on{Id})_*(\CM)\in 
\bD^f(\hg_\crit\mod_{\sZ}).$$
Let $\bD_{ren}(\hg_\crit\mod_{\sZ',\sZ})$ denote its ind-completion.

\medskip

Consider the restriction of the functor $(\iota^{\sZ',\sZ}\underset{\Spec(\sZ)}\times
\iota^{\sZ,\sZ'}_\hg)^{*,+}$ to $\bD^f(\hg_\crit\mod_{\sZ',\sZ})$. The isomorphism
\eqref{two iotas} implies that the image of this functor belongs to 
$\bD^f(\hg_\crit\mod_{\sZ'})$. Let $(\iota^{\sZ',\sZ}\underset{\Spec(\sZ)}\times
\iota^{\sZ,\sZ'}_\hg)^*_{ren}$ denote the ind-extension
$$\bD_{ren}(\hg_\crit\mod_{\sZ',\sZ})\to \bD^f(\hg_\crit\mod_{\sZ'}).$$

\begin{prop}
The functor $(\iota^{\sZ',\sZ}\underset{\Spec(\sZ)}\times
\iota^{\sZ,\sZ'}_\hg)^*_{ren}$ is an exact equivalence of categories.
\end{prop}

We omit the proof as it essentially repeats the proof of 
\thmref{equiv flags vs grass}.

\ssec{From D-modules to $\hg$-modules}

\sssec{}   \label{cat on Fl}

For any level $\kappa$ we consider the abelian category 
$\fD(\Fl^{\on{aff}}_G)_\kappa\mod$, and its derived category 
$\bD(\fD(\Fl^{\on{aff}}_G)_\kappa\mod)$.

Let $\fD^{fg}(\Fl^{\on{aff}}_G)_\kappa\mod\subset \fD(\Fl^{\on{aff}}_G)_\kappa\mod$ be the 
abelian subcategory of finitely generated D-modules. This pair of
categories satisfies the conditions of \secref{good abelian}. We obtain
the renormalized category $\bD_{ren}(\fD(\Fl^{\on{aff}}_G)_\kappa\mod)$,
$$\bD^f_{ren}(\fD(\Fl^{\on{aff}}_G)_\kappa\mod):= 
\bD^f(\fD(\Fl^{\on{aff}}_G)_\kappa\mod)\simeq \bD^b(\fD^{fg}(\Fl^{\on{aff}}_G)_\kappa\mod)$$ and
$$\bD_{ren}(\fD(\Fl^{\on{aff}}_G)_\kappa\mod)\simeq 
\Ind\left(\bD^f(\fD(\Fl^{\on{aff}}_G)_\kappa\mod)\right).$$

\medskip

Our present goal is to construct a functor
$$\Gamma_\Fl:\bD^f(\fD(\Fl^{\on{aff}}_G)_\kappa\mod)\to \bD^f(\hg_\kappa\mod),$$
and its ind-extension
$$\Gamma_\Fl:\bD_{ren}(\fD(\Fl^{\on{aff}}_G)_\kappa\mod)\to 
\bD_{ren}(\hg_\kappa\mod).$$

In order to do this we will use a particular DG model for the
category $\bD^f(\fD(\Fl^{\on{aff}}_G)_\kappa\mod)$.

\sssec{}   \label{initial DG}

Let $\bC^b(\fD^{fg}(\Fl^{\on{aff}}_G)_\kappa\mod)$ be the category consisting of finite
complexes of objects from $\fD^f(\Fl^{\on{aff}}_G)_\kappa\mod$. Let
$\bC_{acycl}^b(\fD^{fg}(\Fl^{\on{aff}}_G)_\kappa\mod)\subset \bC^b(\fD^{fg}(\Fl^{\on{aff}}_G)_\kappa\mod)$
be the DG subcategory of acyclic complexes. By definition, $\bD^f(\fD(\Fl^{\on{aff}}_G)_\kappa\mod)$ is the triangulated quotient
$$\Ho\left(\bC^b(\fD^{fg}(\Fl^{\on{aff}}_G)_\kappa\mod)\right)/
\Ho\left(\bC^b_{acycl}(\fD^{fg}(\Fl^{\on{aff}}_G)_\kappa\mod)\right),$$
which by \secref{rigid quotients} endows $\bD^f(\fD(\Fl^{\on{aff}}_G)_\kappa\mod)$ 
with a DG model. This is the standard DG model for 
$\bD^f(\fD(\Fl^{\on{aff}}_G)_\kappa\mod)$. 

\medskip

Let now $'\bC^b(\fD^{fg}(\Fl^{\on{aff}}_G)_\kappa\mod)$ be a full DG subcategory
of $\bC(\fD(\Fl^{\on{aff}}_G)_\kappa\mod)$ that consists of finite complexes
$\CF^\bullet$, with cohomologies belonging to $\bD^{fg}(\fD(\Fl^{\on{aff}}_G)_\kappa\mod)$, 
and such that each $\CF^k$ is supported on a finite-dimensional subscheme of $\Fl^{\on{aff}}_G$.
Let $'\bC^b_{acycl}(\fD^{fg}(\Fl^{\on{aff}}_G)_\kappa\mod)\subset {}
'\bC^b(\fD^{fg}(\Fl^{\on{aff}}_G)_\kappa\mod)$ be the DG subcategory of acyclic 
complexes.

\medskip

We have a canonical 1-morphism in $\DGCat$:
\begin{multline*}
\bC^b(\fD^{fg}(\Fl^{\on{aff}}_G)_\kappa\mod)/\bC^b_{acycl}(\fD^{fg}(\Fl^{\on{aff}}_G)_\kappa\mod)\to \\
\to {}'\bC^b(\fD^{fg}(\Fl^{\on{aff}}_G)_\kappa\mod)/{}'\bC^b_{acycl}(\fD^{fg}(\Fl^{\on{aff}}_G)_\kappa\mod).
\end{multline*}
It is easy to see that it induces an isomorphism on the level of
homotopy categories. We obtain an equivalence
$$\bD^f(\fD(\Fl^{\on{aff}}_G)_\kappa\mod)\simeq 
\Ho\left({}'\bC^b(\fD^{fg}(\Fl^{\on{aff}}_G)_\kappa\mod)\right)/
\Ho\left({}'\bC^b_{acycl}(\fD^{fg}(\Fl^{\on{aff}}_G)_\kappa\mod)\right),$$
which equips $\bD^f(\fD(\Fl^{\on{aff}}_G)_\kappa\mod)$ with a DG model within
the same equivalence class.

\sssec{}

Let now
$$'\bC^b_{adapt}(\fD(\Fl^{\on{aff}}_G)_\kappa\mod)\subset {}'\bC^b(\fD^{fg}(\Fl^{\on{aff}}_G)_\kappa\mod)$$
be a DG subcategory, whose objects are complexes $\CF^\bullet$
such that for each $k$
$$R^i\Gamma_\Fl(\CF^k)=0,\,\, \forall i>0.$$

Set
$$'\bC^b_{adapt,acycl}(\fD^{fg}(\Fl^{\on{aff}}_G)_\kappa\mod):='\bC^b_{adapt}(\fD^{fg}(\Fl^{\on{aff}}_G)_\kappa\mod)
\cap {}'\bC^b_{acycl}(\fD^{fg}(\Fl^{\on{aff}}_G)_\kappa\mod).$$
We have a canonical 1-morphism:
\begin{multline*}
'\bC^b_{adapt}(\fD^{fg}(\Fl^{\on{aff}}_G)_\kappa\mod)/{}
'\bC^b_{adapt,acycl}(\fD^{fg}(\Fl^{\on{aff}}_G)_\kappa\mod) \to \\
\to {}'\bC^b(\fD^{fg}(\Fl^{\on{aff}}_G)_\kappa\mod)/
{}'\bC^b_{acycl}(\fD^{fg}(\Fl^{\on{aff}}_G)_\kappa\mod).
\end{multline*}

\begin{lem}  \label{adapt is equiv}
The functor
\begin{multline*}
\Ho\left({}'\bC^b_{adapt}(\fD^{fg}(\Fl^{\on{aff}}_G)_\kappa\mod)\right)/{}
\Ho\left({}'\bC^b_{adapt,acycl}(\fD^{fg}(\Fl^{\on{aff}}_G)_\kappa\mod)\right) \to \\
\to 
\Ho\left({}'\bC^b(\fD^{fg}(\Fl^{\on{aff}}_G)_\kappa\mod)\right)/
\Ho\left({}'\bC^b_{acycl}(\fD^{fg}(\Fl^{\on{aff}}_G)_\kappa\mod)\right)
\end{multline*}
is an equivalence.
\end{lem}

\begin{proof}

The assertion of the lemma follows from \lemref{adapted replacement}(2),
since every D-module $\CF$ over a finite-dimensional subscheme $\CY$ of $\Fl^{\on{aff}}_G$
admits a Cech resolution with respect to an affine cover of $\CY$.

\end{proof}

Thus, we have an equivalence
\begin{multline*}
\bD^f(\fD(\Fl^{\on{aff}}_G)_\kappa\mod)\simeq \\
\Ho\left({}'\bC^b_{adapt}(\fD^{fg}(\Fl^{\on{aff}}_G)_\kappa\mod)\right)/{}
\Ho\left({}'\bC^b_{adapt,acycl}(\fD^{fg}(\Fl^{\on{aff}}_G)_\kappa\mod)\right),
\end{multline*}
which equips $\bD^f(\fD(\Fl^{\on{aff}}_G)_\kappa\mod)$ with yet another DG
model within the same equivalence class. It is that latter
DG model that we will use to construct the functor of sections.

\sssec{}  \label{construct functor}

We have a DG functor
$$\Gamma_\Fl:{}'\bC^b_{adapt}(\fD^{fg}(\Fl^{\on{aff}}_G)_\kappa\mod)\to
\bC(\hg_\kappa\mod),$$
obtained by restriction from the evident functor
$$\Gamma_\Fl:\bC(\fD(\Fl^{\on{aff}}_G)_\kappa\mod)\to \bC(\hg_\kappa\mod).$$

By construction, 
$$\Gamma\left({}'\bC^b_{adapt,acycl}(\fD^{fg}(\Fl^{\on{aff}}_G)_\kappa\mod)\right)\subset
\bC_{acycl}(\hg_\kappa\mod),$$
which endows the (usual) derived functor
\begin{equation} \label{Gamma on f}
\Gamma_\Fl:\bD^f(\fD(\Fl^{\on{aff}}_G)_\kappa\mod)\to \bD(\hg_\kappa\mod)
\end{equation}
with a DG model.

\begin{prop}  \label{prop Gamma on f}
The image of the functor \eqref{Gamma on f} belongs to
$\bD^f(\hg_\kappa\mod)$.
\end{prop}

Since the DG model on $\bD^f(\hg_\kappa\mod)$ is inherited from
that on $\bD(\hg_\kappa\mod)$, we obtain that the desired functor
$$\Gamma^f:\bD^f(\fD(\Fl^{\on{aff}}_G)_\kappa\mod)\to \bD^f(\hg_\kappa\mod),$$
equipped with a DG model.

\sssec{Proof of \propref{prop Gamma on f}}

It is enough to show that for a finitely generated D-module
$\CF$ on $\Fl$, the object $\Gamma_\Fl(\CF)\in \bD(\hg_\kappa\mod)$
belongs to $\bD^f(\hg_\kappa\mod)$. 

\medskip

Let $\CF$ be supported over a closed finite-dimensional subscheme 
$\CY\subset \Fl^{\on{aff}}_G$. Then $\CF$ admits a finite resolution by D-modules of 
the form $\Ind^{\fD(\Fl^{\on{aff}}_G)_\kappa}_{\CO(\Fl^{\on{aff}}_G)}(\CM)$, where 
$\CM$ is a coherent sheaf on $\CY$. Thus, we can assume that $\CF$
has this form.

\medskip

Let $n>0$ be sufficiently large so that 
$$\fg\otimes (t^n\cdot \BC[[t]])\subset \on{Lie}(I_y)$$
for all $y\in \CY$, where $I_y$ denotes the conjugate of
the Iwahori subgroup $I$ corresponding to a point
$y\in \Fl^{\on{aff}}_G\simeq G\ppart/I$. Consider the vector bundle
$I_\CY/\fg\otimes (t^n\cdot \BC[[t]])$ over $\CY$, whose
fiber at $y\in \CY$ is $\on{Lie}(I_y)/\fg\otimes (t^n\cdot \BC[[t]])$.
This vector bundle carries an action of the group $G(t^n\cdot \BC[[t]])$.
For an integer $i$ consider its $i$-th exterior power
$\Lambda^i(I_\CY/\fg\otimes (t^n\cdot \BC[[t]]))$, and the 
$G(t^n\cdot \BC[[t]])$-module
\begin{equation} \label{before ind}
\Gamma\left(\CY,\Lambda^i(I_\CY/\fg\otimes (t^n\cdot \BC[[t]]))\underset{\CO_\CY}
\otimes \CM\right).
\end{equation}
Consider the $\hg_\kappa$-module
$$\BV_{\kappa,n,i}(\CY,\CM):=\Ind^{\hg}_{\fg\otimes (t^n\cdot \BC[[t]])}
\left(\Gamma\left(\CY,\Lambda^i(I_\CY/\fg\otimes (t^n\cdot \BC[[t]]))\underset{\CO_\CY}
\otimes \CM\right)\right).$$

\medskip

The relative version of the Chevalley complex construction gives rise to a map
of $\hg_\kappa$-modules
$$\BV^i_{\kappa,n}(\CY,\CM)\to \BV^{i-1}_{\kappa,n}(\CY,\CM),$$
such that the composition
$$\BV^i_{\kappa,n}(\CY,\CM)\to \BV^{i-1}_{\kappa,n}(\CY,\CM)\to
\BV^{i-2}_{\kappa,n}(\CY,\CM)$$
vanishes. Denote the resulting complex
$\BV^\bullet_{\kappa,n}(\CY,\CM)$.

\begin{lem}  
Assume that for all $i$ and $j>0$ 
$$R^j\Gamma\left(\CY,\Lambda^i(I_\CY/\fg\otimes (t^n\cdot \BC[[t]]))\underset{\CO_\CY}
\otimes \CM\right)=0.$$
Then $\BV^\bullet_{\kappa,n}(\CY,\CM)$ is quasi-isomorphic to 
$\Gamma_\Fl\left(\Ind^{\fD(\Fl^{\on{aff}}_G)_\kappa}_{\CO(\Fl^{\on{aff}}_G)}(\CM)\right)$.
\end{lem}

This implies the assertion of the proposition: 

\medskip

Indeed, resolving $\CM$ by coherent sheaves, we can assume that the vanishing 
condition of the above lemma holds. Hence, it suffices to see that the $\hg_\kappa$-modules
$\BV^i_{\kappa,n}(\CY,\CM)$ belong to $\bD^f(\hg_\kappa\mod)$. 

\medskip

However, since $\CY$ is proper, the $G(t^n\cdot \BC[[t]])$-module
\eqref{before ind} is finite-dimensional, and hence admits a finite
filtration with trivial quotients. Hence, $\BV^i_{\kappa,n}(\CY,\CM)$ 
admits a finite filtration with quotients isomorphic to
$\BV_{\kappa,n}$.

\qed

\ssec{Sections at the critical level}   \label{sections to finite}

Let us take now $\kappa=\kappa_\crit$. Our goal is to show that the functor
$\Gamma_\Fl$ of \secref{cat on Fl} factors through a functor
$$\Gamma_\Fl:\bD^f(\fD(\Fl^{\on{aff}}_G)_\crit\mod)\to 
\bD^f(\hg_\crit\mod_{\fZ^\nilp_\fg}).$$

\medskip

First, we recall that by \cite{FG2}, Sect. 7.19, for any
$\CF\in \fD(\Fl^{\on{aff}}_G)_\crit\mod$, the individual cohomologies
$R^i\Gamma_\Fl(\CF)$ belong to the full subcategory
$$\hg_\crit\mod_{\fZ^\nilp_\fg}\subset \hg_\crit\mod.$$

\medskip

Hence, the procedure of \secref{construct functor} defines
a functor
\begin{multline*}
\bD^f(\fD(\Fl^{\on{aff}}_G)_\crit\mod)\simeq \\
\simeq\Ho\left({}'\bC^b_{adapt}(\fD^{fg}(\Fl^{\on{aff}}_G)_\crit\mod)\right)/{}
\Ho\left({}'\bC^b_{adapt,acycl}(\fD^{fg}(\Fl^{\on{aff}}_G)_\crit\mod)\right)\to \\
\to 
\Ho\left(\bC(\hg_\kappa\mod_{\fZ^\nilp_\fg})\right)/
\Ho\left(\bC_{acycl}(\hg_\kappa\mod_{\fZ^\nilp_\fg})\right)\simeq
\bD(\hg_\crit\mod_{\fZ^\nilp_\fg}),
\end{multline*}
equipped with a DG model.  

\medskip

Thus, it remains to see that on the triangulated level, its essential
image is contained in $\bD^f(\hg_\crit\mod_{\fZ^\nilp_\fg})$. The
latter results from \propref{prop Gamma on f} and the definition
of the latter category (see \secref{good category sZ}).

\section{DG model for the \cite{AB} action}  \label{DG model for the AB action}

\ssec{}

In this subsection we will discuss several different, but equivalent, DG models
for the category $\bD^f(\fD(\Fl^{\on{aff}}_G)_\crit\mod)$ that are needed to upgrade
to the DG level various triangulated functors from the main body of
the paper.

\sssec{}

Recall the category $'\bC(\fD(\Fl^{\on{aff}}_G)_\crit\mod)$
and its subcategory of acyclic objects, denoted 
$'\bC_{acycl}(\fD(\Fl^{\on{aff}}_G)_\crit\mod)$.

\medskip

Let 
$$\bC^b_{aff}(\fD^{fg}(\Fl^{\on{aff}}_G)_\crit\mod)\subset {}'\bC^b(\fD^{fg}(\Fl^{\on{aff}}_G)_\crit\mod)$$
be the DG subcategory consisting of complexes $\CF^\bullet$ with
the additional condition that each $\CF^i$ is a direct sum
of D-modules, each being the direct image of a D-module on a
finite-dimensional locally closed {\it affine} sub-scheme of $\Fl^{\on{aff}}_G$.
(Note that $\bC^b_{aff}(\fD^{fg}(\Fl^{\on{aff}}_G)_\crit\mod)$ is contained is the
subcategory $\bC_{adapt}(\fD(\Fl^{\on{aff}}_G)_\crit\mod)$, introduced
earlier.)

\medskip

Set
$$\bC^b_{aff,acycl}(\fD^{fg}(\Fl^{\on{aff}}_G)_\crit\mod):=
\bC^b_{aff}(\fD^{fg}(\Fl^{\on{aff}}_G)_\crit\mod)\cap {}'\bC_{acycl}(\fD(\Fl^{\on{aff}}_G)_\crit\mod).$$

The proof of \lemref{adapt is equiv} shows that the natural functor
\begin{multline*}
\Ho\left(\bC^b_{aff}(\fD^{fg}(\Fl^{\on{aff}}_G)_\crit\mod)\right)/
\Ho\left(\bC^b_{aff,acycl}(\fD^{fg}(\Fl^{\on{aff}}_G)_\crit\mod)\right)\to \\
\to \Ho\left({}'\bC(\fD(\Fl^{\on{aff}}_G)_\crit\mod)\right)/
\Ho\left({}'\bC_{acycl}(\fD(\Fl^{\on{aff}}_G)_\crit\mod)\right)
\end{multline*}
is an equivalence. 

\medskip

Hence, 
$$\bC^b_{aff}(\fD^{fg}(\Fl^{\on{aff}}_G)_\crit\mod)/\bC^b_{aff,acycl}(\fD^{fg}(\Fl^{\on{aff}}_G)_\crit\mod)$$
defines a different, but equivalent, DG model for 
$\bD^f(\fD(\Fl^{\on{aff}}_G)_\crit\mod)$. 

\medskip

The above DG model gives a different, but again equivalent,
DG model for the functor $\Gamma_\Fl:\bD^f(\fD(\Fl^{\on{aff}}_G)_\crit\mod)\to
\bD^f(\hg_\crit\mod_\nilp)$.

\sssec{}    \label{cat on Fl: fine DG model}

Let
$$\bC^b_{aff,ind}(\fD^{fg}(\Fl^{\on{aff}}_G)_\crit\mod)\subset
\bC^b_{aff}(\fD^{fg}(\Fl^{\on{aff}}_G)_\crit\mod)$$
be the DG subcategory, consisting of complexes $\CF^\bullet$, where
we impose the additional condition that each $\CF^i$ is of the form
$\Ind^{\fD(\Fl^{\on{aff}}_G)_\crit}_{\CO_{\Fl^{\on{aff}}_G}}(\CM)$ for a quasi-coherent
sheaf $\CM$ on $\Fl^{\on{aff}}_G$.

\medskip

Set
\begin{multline*}
\bC^b_{aff,ind,acycl}(\fD^{fg}(\Fl^{\on{aff}}_G)_\crit\mod):=\\
\bC^b_{aff,ind}(\fD^{fg}(\Fl^{\on{aff}}_G)_\crit\mod)\cap 
\bC^b_{aff,acycl}(\fD^{fg}(\Fl^{\on{aff}}_G)_\crit\mod).
\end{multline*}

\begin{lem}
The natural functor
\begin{multline*}
\Ho\left(\bC^b_{aff,ind}(\fD^{fg}(\Fl^{\on{aff}}_G)_\crit\mod)\right)/
\Ho\left(\bC^b_{aff,ind,acycl}(\fD^{fg}(\Fl^{\on{aff}}_G)_\crit\mod)\right)\to \\
\to
\Ho\left(\bC^b_{aff}(\fD^{fg}(\Fl^{\on{aff}}_G)_\crit\mod)\right)/
\Ho\left(\bC^b_{aff,acycl}(\fD^{fg}(\Fl^{\on{aff}}_G)_\crit\mod)\right)
\end{multline*}
is an equivalence.
\end{lem}

\begin{proof}

It suffices to show that any D-module equal to the direct
image from a (sufficiently small) finite-dimensional
affine subscheme $\CY$ of $\Fl^{\on{aff}}_G$ admits a finite
left resolution consisting of induced D-modules. 
Since the operation of induction commutes with that of direct 
image under a locally closed map, it suffices to construct such 
a resolution in the category of D-modules on $\CY$ itself.

\medskip

We can assume that $\CY$ is small enough so that it is
contained in an open ind--subscheme of $\Fl^{\on{aff}}_G$, isomorphic
to the ind-affine space $\BA^\infty$; then $\CY\subset 
\BA^n\subset \BA^\infty$. The required resolution is
given by the De Rham complex on $\BA^n$.

\end{proof}

Thus, we can use
$$\bC^b_{aff,ind}(\fD^{fg}(\Fl^{\on{aff}}_G)_\crit\mod)/\bC^b_{aff,ind,acycl}(\fD^{fg}(\Fl^{\on{aff}}_G)_\crit\mod)$$
as yet another DG model for the category $\bD^f(\fD(\Fl^{\on{aff}}_G)_\crit\mod)$
and the functor $\Gamma_\Fl$. In what follows, it will be this model
that we will use to perform our constructions.

\ssec{}   \label{AB action: model}

We shall now define a homotopy action of 
$$\bC^{free}(\Coh(\tN/\cG))/\bC^b_{acycl}(\Coh^{free}(\tN/\cG))$$
on 
$$\bC^b_{aff,ind}(\fD^{fg}(\Fl^{\on{aff}}_G)_\crit\mod)/\bC^b_{aff,ind,acycl}(\fD^{fg}(\Fl^{\on{aff}}_G)_\crit\mod).$$

\medskip

We will use the paradigm of \secref{weak quotient tensor}. First, we note:

\begin{lem}    \label{conv acyclic}
Assume that $\CF$ is a D-module on $\Fl^{\on{aff}}_G$, which is induced
from a quasi-coherent sheaf, and isomorphic to the direct
image of a D-module from a finite-dimensional affine subscheme. Then
$$H^j(\CF\star J_\cla)=0,\,\, \forall j\neq 0 \text{ and } \cla\in \cLambda.$$
\end{lem}

\begin{proof}

Write $\cla=\cla_1-\cla_2$ with $\cla_i\in \cLambda^*$. 

\medskip

Recall that
the functor $?\star J_{\cmu}$ is right-exact for $\cmu\in \cLambda^+$,
since in this case $J_{\cmu}=j_{\cmu,*}$. The functor $?\star J_{-\cmu}$,
being the right (but in fact also left) adjoint of $J_\cmu$, is left-exact.

\medskip

Let us first show that $H^j(\CF\star J_\cla)=0$ for $j>0$. For that
it suffices to see that $\CF\star J_{-\cla_2}\in 
\bD^{\leq 0}(\fD(\Fl^{\on{aff}}_G)_\crit\mod)$. However, this is true for
any convolution $\CF\star \CF'$ for $\CF'\in \fD(\CY)\mod^I$,
where $\CY$ is an ind-scheme with an action of $G\ppart$,
provided that $\CF$ is somorphic to the direct
image of a D-module from a finite-dimensional affine subscheme.

\medskip

Let us now show that $H^j(\CF\star J_\cla)=0$ for $j<0$. For that
it suffices to see that $\CF\star J_{\cla_1}\in 
\bD^{\geq 0}(\fD(\Fl^{\on{aff}}_G)_\crit\mod)$. But this is again true
for any convolution $\CF\star \CF'$ as above, provided that
$\CF$ is induced.

\end{proof}

\sssec{}

The sought-for pseudo-action of $\bC^b(\Coh^{free}(\tN/\cG))$ on 
$\bC^b_{aff,ind}(\fD^{fg}(\Fl^{\on{aff}}_G)_\crit\mod)$ is defined as follows: for a collection
of objects $\CM^\bullet_1,...,\CM^\bullet_n\in \bC^b(\Coh^{free}(\tN/\cG))$,
$\CF^\bullet_1,\CF^\bullet_2\in \bC^b_{aff,ind}(\fD^{fg}(\Fl^{\on{aff}}_G)_\crit\mod)$ we set
\begin{multline} \label{defn of action}
\Hom^\bullet
_{\bC^b(\Coh^{free}(\tN/\cG)),\bC^b_{aff,ind}(\fD^{fg}(\Fl^{\on{aff}}_G)_\crit\mod)}
("\CM^\bullet_1\otimes...\otimes \CM^\bullet_n\otimes 
\CF^\bullet_1",\CF^\bullet_2):=\\
\Hom_{\bC(\fD(\Fl^{\on{aff}}_G)_\crit\mod)}\left(\CF^\bullet_1\star 
\sF(\CM^\bullet_1\otimes...\otimes\CM^\bullet_n),\CF^\bullet_2\right),
\end{multline}
where 
\begin{equation} \label{repr complex}
\CF^\bullet_1\star 
\sF(\CM^\bullet_1\otimes...\otimes\CM^\bullet_n) 
\end{equation}
is regarded as
a complex of D-modules on $\Fl^{\on{aff}}_G$ (which does not in general belong to
$\bC^b_{aff,ind}(\fD^{fg}(\Fl^{\on{aff}}_G)_\crit\mod)$), obtained by {\it term-wise}
application of the convolution functor.

\medskip

We claim that the the required co-representability and vanishing conditions of
\secref{weak quotient tensor} hold with respect to the subcategories
$$\bC^b_{acycl}(\Coh^{free}(\tN/\cG))\subset
\bC^b(\Coh^{free}(\tN/\cG))$$
and
$$\bC^b_{aff,ind,acycl}(\fD^{fg}(\Fl^{\on{aff}}_G)_\crit\mod)\subset 
\bC^b_{aff,ind}(\fD^{fg}(\Fl^{\on{aff}}_G)_\crit\mod).$$

Indeed, by \lemref{conv acyclic}, for fixed
$\CM^\bullet_1,...,\CM^\bullet_n$ and $\CF^\bullet_1$ the 
co-representing object in
\begin{multline*}
\Ho\left(\bC^b_{aff,ind}(\fD^{fg}(\Fl^{\on{aff}}_G)_\crit\mod)\right)/
\Ho\left(\bC^b_{aff,ind,acycl}(\fD^{fg}(\Fl^{\on{aff}}_G)_\crit\mod)\right)\simeq \\
\simeq \bD^f(\fD(\Fl^{\on{aff}}_G)_\crit\mod)
\end{multline*}
is represented by the complex \eqref{repr complex}.

\medskip

Thus, we obtain a required homotopy action of $$\bC^b(\Coh^{free}(\tN/\cG))/\bC^b_{acycl}(\Coh^{free}(\tN/\cG))$$
on 
$$\bC^b_{aff,ind}(\fD^{fg}(\Fl^{\on{aff}}_G)_\crit\mod)/\bC^b_{aff,ind,acycl}(\fD^{fg}(\Fl^{\on{aff}}_G)_\crit\mod).$$

\ssec{}   \label{sections: model}

We shall now upgrade the functor $\Gamma_\Fl$ to a DG functor,
compatible with the action of $$\bC^b(\Coh^{free}(\tN/\cG))/\bC^b_{acycl}(\Coh^{free}(\tN/\cG)).$$

\medskip

To do so, it is sufficient to construct the corresponding structure 
on $\Gamma_\Fl$ as a homotopy functor between the DG categories
$$\bC^b_{aff,ind}(\fD^{fg}(\Fl^{\on{aff}}_G)_\crit\mod)/
\bC^b_{aff,ind,acycl}(\fD^{fg}(\Fl^{\on{aff}}_G)_\crit\mod)$$
and
$$\bC(\hg_\crit\mod_\nilp)/\bC_{acycl}(\hg_\crit\mod_\nilp).$$

\sssec{}

We first consider the pseudo-functor
$$\bC^b_{aff,ind}(\fD^{fg}(\Fl^{\on{aff}}_G)_\crit\mod)\to \bC(\hg_\crit\mod_\nilp)$$
as categories with a pseudo-action of $\bC^b(\Coh^{free}(\tN/\cG))$.
For $\CM^\bullet_1,...,\CM^\bullet_n\in \bC^b(\Coh^{free}(\tN/\cG))$,
$\CF^\bullet\in \bC^b_{aff,ind}(\fD^{fg}(\Fl^{\on{aff}}_G)_\crit\mod)$ and
$\CV^\bullet\in \bC(\hg_\crit\mod_\nilp)$, we set 
\begin{multline} \label{defn of DG Gamma}
\Hom_{\bC^b(\Coh^{free}(\tN/\cG)),
\bC^b_{aff,ind}(\fD^{fg}(\Fl^{\on{aff}}_G)_\crit\mod),\bC(\hg_\crit\mod_\nilp)} \\
\left("\CM_1^\bullet\otimes...\otimes \CM_n^\bullet\otimes 
\Gamma_\Fl(\CF^\bullet)",
\CV^\bullet\right):= 
\Hom_{\bC(\hg_\crit\mod_\nilp)}
\left(\Gamma_\Fl(\CF^\bullet\star 
\sF(\CM^\bullet_1\otimes...\otimes\CM^\bullet_n)),\CV^\bullet\right),
\end{multline}
where again 
\begin{equation} \label{repr complex of g-mod}
\Gamma_\Fl(\CF^\bullet\star 
\sF(\CM^\bullet_1\otimes...\otimes\CM^\bullet_n))
\end{equation}
is by definition obtained by applying the functor $\Gamma_\Fl$
term-wise to \eqref{repr complex} (with $\CF_1^\bullet$ replaced
by $\CF_2^\bullet$) as a complex of D-modules on $\Fl^{\on{aff}}_G$.

\medskip

The construction of the required natural transformations has
been carried out in \secref{a-e}.

\sssec{}

Now we claim that the co-representability conditions of 
\secref{weak quotient tensor} hold with respect to the 
subcategories:

$$\bC^b_{acycl}(\Coh^{free}(\tN/\cG))\subset
\bC^b(\Coh^{free}(\tN/\cG)),$$
$$\bC^b_{aff,ind,acycl}(\fD^{fg}(\Fl^{\on{aff}}_G)_\crit\mod)\subset 
\bC^b_{aff,ind}(\fD^{fg}(\Fl^{\on{aff}}_G)_\crit\mod),$$
and
$$\bC_{acycl}(\hg_\crit\mod_\nilp)\subset \bC(\hg_\crit\mod_\nilp).$$

Indeed, for $\CM^\bullet_1,...,\CM^\bullet_n$,
$\CF^\bullet$ as above, the required object of $\bD(\hg_\crit\mod_\nilp)$ is the image of the
complex \eqref{repr complex of g-mod}. 

\medskip

Thus, we have constructed the required 1-morphism
\begin{multline*}
\bC^b_{aff,ind}(\fD^{fg}(\Fl^{\on{aff}}_G)_\crit\mod)/
\bC^b_{aff,ind,acycl}(\fD^{fg}(\Fl^{\on{aff}}_G)_\crit\mod)\to \\
\to 
\bC(\hg_\crit\mod_\nilp)/\bC_{acycl}(\hg_\crit\mod_\nilp),
\end{multline*}
in $\DGMod\left(\bC^b(\Coh^{free}(\tN/\cG))/
\bC^b_{acycl}(\Coh^{free}(\tN/\cG))\right)$.

\ssec{}   \label{ups: model}

Recall the functor $$\wt\Upsilon:\bD^f(\fD(\Fl^{\on{aff}}_G)_\crit\mod)\to
\on{pt}/\cB\underset{\on{pt}/\cG}\times \bD^f(\fD(\Gr^{\on{aff}}_G)_\crit\mod).$$

We will now upgrade it to the DG level, in a way compatible with
the homotopy action of 
$\bC^b(\Coh^{free}(\tN/\cG))/\bC^b_{acycl}(\Coh^{free}(\tN/\cG))$
on both sides. This amounts to constructing a 1-morphism
\begin{multline}  \label{model for Ups}
\bC^b_{aff,ind}(\fD^{fg}(\Fl^{\on{aff}}_G)_\crit\mod)/
\bC^b_{aff,ind,acycl}(\fD^{fg}(\Fl^{\on{aff}}_G)_\crit\mod)\to \\
\to
\bC(\on{pt}/\cB\underset{\on{pt}/\cG}\times \fD(\Gr^{\on{aff}}_G)_\crit\mod)/
\bC_{acycl}(\on{pt}/\cB\underset{\on{pt}/\cG}\times \fD(\Gr^{\on{aff}}_G)_\crit\mod).
\end{multline}
in $\DGMod\left(\bC^b(\Coh^{free}(\tN/\cG))/
\bC^b_{acycl}(\Coh^{free}(\tN/\cG))\right)$.

\medskip

Proceeding as above, we first define a pseudo-functor
$$\bC^b_{aff,ind}(\fD^{fg}(\Fl^{\on{aff}}_G)_\crit\mod)\to
\bC(\on{pt}/\cB\underset{\on{pt}/\cG}\times \fD(\Gr^{\on{aff}}_G)_\crit\mod)$$
with a compatibility data with respect to the action of
$\bC^b(\Coh^{free}(\tN/\cG))$.

\medskip

For $\CM^\bullet_1,...,\CM^\bullet_n\in \bC^b(\Coh^{free}(\tN/\cG))$,
$\CF^\bullet\in \bC^b_{aff,ind}(\fD^{fg}(\Fl^{\on{aff}}_G)_\crit\mod)$ and
$$'\CF^\bullet\in \bC(\on{pt}/\cB\underset{\on{pt}/\cG}\times \fD(\Gr^{\on{aff}}_G)_\crit\mod),$$ 
we set
\begin{multline} \label{defn of Ups}
\Hom^\bullet_{\bC^b(\Coh^{free}(\tN/\cG)),
\bC^b_{aff,ind}(\fD^{fg}(\Fl^{\on{aff}}_G)_\crit\mod),
\bC(\on{pt}/\cB\underset{\on{pt}/\cG}\times \fD(\Gr^{\on{aff}}_G)_\crit\mod)} \\
\left("\CM_1^\bullet\otimes...\otimes \CM_n^\bullet\otimes 
\wt\Upsilon(\CF^\bullet)",
{}'\CF^\bullet\right):= \\
\Hom_{\bC(\on{pt}/\cB\underset{\on{pt}/\cG}\times \fD(\Gr^{\on{aff}}_G)_\crit\mod)}
\left(\CF^\bullet\star 
\sF(\CM^\bullet_1\otimes...\otimes\CM^\bullet_n)\star J_{2\rho}\star \CW,
{}'\CF^\bullet\right),
\end{multline}
where
\begin{equation}
\CF^\bullet\star 
\sF(\CM^\bullet_1\otimes...\otimes\CM^\bullet_n)\star J_{2\rho}\star \CW\in 
\bC(\on{pt}/\cB\underset{\on{pt}/\cG}\times \fD(\Gr^{\on{aff}}_G)_\crit\mod)
\end{equation}
is given by term-wise convolution.

\medskip

The co-representability condition of \secref{weak quotient tensor}
is satisfied because for any D-module $\CF$ appearing
as a term of an object of $\bC^b_{aff,ind}(\fD^{fg}(\Fl^{\on{aff}}_G)_\crit\mod)$
and $\CM\in \Coh^{free}(\tN/\cG)$, the convolution
$$\CF\star \sF(\CM)\star \CF'\in 
\bD(\on{pt}/\cB\underset{\on{pt}/\cG}\times \fD(\Gr^{\on{aff}}_G)_\crit\mod)$$
is acyclic off cohomological degree $0$ for any
$\CF'\in \on{pt}/\cB\underset{\on{pt}/\cG}\times \fD(\Gr^{\on{aff}}_G)_\crit\mod^I$.

\medskip

This defines the desired 1-morphism \eqref{model for Ups}.

\sssec{}   \label{wt i p: DG model}

Finally, we remark that the above construction defines also the lifting
of the natural transformation of \lemref{wt i p} to a 2-morphism over
$\on{pt}/\cG$, once we lift the functor
$$\on{co-Ind}:\bD(\on{pt}/\cB\underset{\on{pt}/\cG}\times \fD(\Gr^{\on{aff}}_G)_\crit\mod)
\to \bD(\fD(\Gr^{\on{aff}}_G)_\crit\mod)$$
to a 1-morphism over $\on{pt}/\cG$, by the procedure of 
\secref{proper morphism adjunctions}.

\section{The $I^0$-equivariant situation}  \label{the I-equivariant situation}

\ssec{} \label{situation}

Let $\bD$ be either of the categories $\bD(\hg_\kappa\mod)$, 
or $\bD(\fD(\Fl^{\on{aff}}_G)_\crit\mod)$
(the latter being considered with the {\it old} t-structure).

\medskip

We define $\bD^{I^0,+}\subset \bD^+$ to be the
full triangulated subcategory, consisting of complexes whose cohomologies
are strongly $I^0$-equivariant objects of the corresponding abelian category,
see \cite{FG2}, Sect. 20.11.

\medskip

The goal of this subsection is to define the $I^0$-equivariant version of
the corresponding category $\bD_{ren}$. This will be done in the following
abstract set-up.

\sssec{}   \label{cond for I}

Let $\bD^f\subset \bD$ be as in \secref{renorm with t}, so that the conditions
of \propref{D plus} hold. 

\medskip

Let $\bD_1^+\subset\bD^+$ be a full triangulated 
subcategory. We assume that the following conditions hold:

\begin{itemize}

\item(1) The tautological functor $\on{emb}:\bD_1^+\to \bD^+$ admits a right
adjoint, denoted $\on{Av}:\bD^+\to \bD_1^+$, such that the composition
$\on{emb}\circ \on{Av}:\bD^+\to \bD^+$ is left-exact.

\item(2) For every $X\in \bD$ belonging to a strongly generating set of
objects of $\bD^f$ there exists an inverse
family $X\to \{...\to X^2\to X^1\}$ with $X^k\in \bD_1^+\cap \bD^f$ such that
for any $Z\in \bD_1^+$, the arrow
$$colim\, \Hom_{\bD_1^+}(X^k,Z)\to \Hom_{\bD^+}(X,Z)$$ 
is an isomorphism.

\end{itemize}

Denote $\bD^f_1:=\bD^f_{1,ren}:=\bD_1^+\cap \bD^f$; being a triangulated
subcategory of $\bD$, it acquires a DG model. Hence, its ind-completion,
denoted $\bD_{1,ren}$, is well-defined. It comes equipped with a 
functor (which is also equipped with a DG model) 
$$\on{emb}_{ren}:\bD_{1,ren}\to \bD_{ren},$$
which is fully faithful, sends compact objects to compact ones, and
commutes with direct sums. Hence, $\on{emb}_{ren}$ admits
a right adjoint, denoted $\on{Av}_{ren}$, which also
commutes with direct sums.

\begin{prop}   \label{abstract equivariant}
Under the above circumstances we have:

\medskip

\noindent(a) The functor $\on{emb}_{ren}\circ \on{Av}_{ren}:\bD_{ren}\to \bD_{ren}$
is left-exact.

\medskip

\noindent(b) The category $\bD_{1,ren}$ acquires a unique t-structure, for which
the functor $\on{emb}_{ren}$ is exact. (The functor $\on{Av}_{ren}$ is then
automatically left-exact.)

\medskip

\noindent(c) The (mutually adjoint) functors $\Psi:\bD^+_{ren}\leftrightarrows \bD^+:\Phi$
send the categories $\bD_{1,ren}^+\subset \bD^+_{ren}$ and $\bD_1^+\subset \bD^+$
to one another. (We shall denote the resulting pair of mutually adjoint functors by 
$\Psi_1,\Phi_1$, respectively.)

\medskip

\noindent(d) We have the isomorphisms of functors
$$\on{Av}_{ren}\circ \Phi\simeq \Phi_1\circ \on{Av}:\bD^+\to 
\bD_{1,ren}^+ \text{ and }
\on{Av}\circ \Psi\simeq \Psi_1\circ \on{Av}_{ren}:\bD^+_{ren}\to \bD^+_1.$$

\end{prop}

\sssec{Proof of \propref{abstract equivariant}}

First, we claim that there exists a natural transformation:
\begin{equation} \label{Loc and Phi}
\on{emb}_{ren}\circ \on{Av}_{ren}\circ \Phi\to \Phi\circ \on{emb}\circ \on{Av}
\end{equation}
between functors $\bD^+\to \bD_{ren}$. To construct it, it is sufficient to
construct a natural transformation 
\begin{equation} \label{Loc and Phi,Psi}
\Psi\circ \on{emb}_{ren}\circ \on{Av}_{ren}\circ \Phi\to \on{emb}\circ \on{Av}:\bD^+\rightrightarrows \bD.
\end{equation}

For $Y\in \bD^+$ consider the maps from the LHS of \eqref{Loc and Phi,Psi} to the distinguished
triangle
$$\on{emb}\circ \on{Av}(Y)\to Y\to \on{Cone}(\on{emb}\circ \on{Av}(Y)\to Y).$$
We have a canonical map $$\Psi\circ \on{emb}_{ren}\circ \on{Av}_{ren}\circ \Phi(Y)\to Y.$$
Hence, in order to construct the morphism in \eqref{Loc and Phi,Psi}, it suffices to show
that for any $Y'\in \on{ker}(\on{Av})$ and any $X\in \bD_{1,ren}$, we have
$$\Hom_{\bD}(\Psi\circ \on{emb}_{ren}(X),Y')=0,$$
which follows from the definitions.

\medskip

Now, we claim that the natural transformation \eqref{Loc and Phi} is an isomorphism. To check it,
it suffices to show that for any $Y\in \bD^+$ and $X\in  \bD^f$ as in Condition (2) of \secref{cond for I}, the map
\begin{equation} \label{Loc and Phi, Hom}
\Hom_{\bD_{ren}}(X,\on{emb}_{ren}\circ \on{Av}_{ren}\circ \Phi(Y))\to
\Hom_{\bD_{ren}}(X,\Phi\circ \on{emb}\circ \on{Av}(Y))
\end{equation}
is an isomorphism. Let us write $\on{Av}_{ren}\circ \Phi(Y)$ as $colim\, Z_i$, $Z_i\in \bD_1^f$, i.e.,
for every $X_1\in \bD_1^f$ the arrow
$$\underset{i}{colim}\, \Hom_{\bD_1^f}(X_1,Z_i)\to \Hom_{\bD}(X_1,Y)$$
is an isomorphism. Let $X^k$ be the corresponding inverse system for $X$. Then the LHS of
\eqref{Loc and Phi, Hom} identifies with
\begin{multline*}
\underset{i}{colim}\, \Hom_{\bD^f}(X,Z_i)\simeq \underset{i}{colim}\,\underset{k}{colim}\,
\Hom_{\bD^f_1}(X^k,Z_i)\simeq \\
\simeq \underset{k}{colim}\,\underset{i}{colim}\,
\Hom_{\bD^f_1}(X^k,Z_i)\simeq 
\underset{k}{colim}\, \Hom_{\bD}(X^k,Y).
\end{multline*}
The RHS of \eqref{Loc and Phi, Hom} identifies with
$$\Hom_\bD(X,\on{emb}\circ \on{Av}(Y))\simeq \underset{k}{colim}\, 
\Hom_{\bD_1}(X^k,\on{emb}\circ \on{Av}(Y))\simeq \underset{k}{colim}\, \Hom_{\bD}(X^k,Y),$$
implying our assertion.

\medskip

The isomorphism \eqref{Loc and Phi}  readily implies point (a) of the proposition. Indeed, for 
$X\in \bD_{ren}^{\geq 0}$, write $X=\Phi(Y)$ for $Y\in \bD^{\geq 0}$. We have:
$$\on{emb}_{ren}\circ \on{Av}_{ren}(X)\simeq \Phi\left(\on{emb}\circ \on{Av}(Y)\right),$$
and as $\on{emb}\circ \on{Av}(Y)\in \bD^{\geq 0}$ (by assumption), the assertion 
follows.

\medskip

Point (b) is a formal corollary of point (a). Indeed, we claim that for $Z\in \bD_{1,ren}$,
the terms of the distinguished triangle
\begin{equation} \label{truncate emb}
\tau^{\leq 0}(\on{emb}_{ren}(Z))\to \on{emb}_{ren}(Z)\to \tau^{>0}(\on{emb}_{ren}(Z))
\end{equation}
belong to the essential image of $\on{emb}_{ren}$, which implies our assertion. To prove
the claim we compare the distinguished triangle \eqref{truncate emb} with
\begin{equation} \label{truncate emb Av}
\on{emb}_{ren}\circ \on{Av}_{ren}
\left(\tau^{\leq 0}(\on{emb}_{ren}(Z))\right)\to \on{emb}_{ren}(Z)\to 
\on{emb}_{ren}\circ 
\on{Av}_{ren}\left(\tau^{>0}(\on{emb}_{ren}(Z))\right).
\end{equation}

\medskip

The adjunction map $\on{emb}_{ren}\circ \on{Av}_{ren}\to \on{Id}$ gives rise to 
a map of triangles \eqref{truncate emb Av}  $\to$ \eqref{truncate emb}. On the other hand,
since $\on{emb}_{ren}\circ \on{Av}_{ren}\left(\tau^{>0}(\on{emb}_{ren}(Z))\right)\in \bD_{ren}^{>0}$
and $\tau^{\leq 0}(\on{emb}_{ren}(Z))\in \bD_{ren}^{\leq 0}$, we have a unique map
\eqref{truncate emb}  $\to$ \eqref{truncate emb Av}. Moreover, the composition
\eqref{truncate emb}  $\to$ \eqref{truncate emb Av} $\to$ \eqref{truncate emb} equals the
identity map. This implies that the terms of \eqref{truncate emb} 
are direct summands of the terms of \eqref{truncate emb Av}. However, a direct summand
of an object in the essential image of $\on{emb}_{ren}$ itself belongs to the
essential image of $\on{emb}_{ren}$.

\medskip

Points (c) and (d) of the proposition follow formally from (a) and (b)
and \eqref{Loc and Phi}. \qed








\ssec{}   \label{proof I lemma b}

We are going to apply \propref{abstract equivariant} 
to $\bD$ being one of the categories
of \secref{situation} and $\bD_1=\bD^{I^0}$. 

Note that in the case of $\bD=\bD(\fD(\Fl^{\on{aff}}_G)_\crit\mod)$, this
would establish \lemref{I^0 monodromic D-modules nice} and \propref{prop Iw and t}(b).

\sssec{}

Let us show that conditions (1) and (2) of \secref{cond for I} 
hold. This will be done in the following context:

\medskip

Let $\bD$ be the derived category of an abelian category $\CC$ 
acted on by $G\ppart$ of \cite{FG2}, Sect. 22.1. Note that
condition (1) of \secref{cond for I} is given by \cite{FG2}, Sect. 20.10. Let $H$ be a group
sub-scheme of $G[[t]]$. 

\medskip

Let $\bD^+(\CC)^{w,H}$ and $\bD^+(\CC)^{s,H}:=\bD^+(\CC)^{H}$ be the 
corresponding weak and strong equivariant categories, respectively. Let 
$$\bD^+(\CC)^{s,H}\overset{\on{emb}^{s,w}}\longrightarrow \bD^+(\CC)^{w,H} 
\overset{\on{emb}^{w}}\longrightarrow \bD^+(\CC)$$
be the corresponding functors, and let $\on{emb}:=\on{emb}^s:=\on{emb}^{w}\circ \on{emb}^{s,w}$.

\medskip

Let $\bD^f(\CC)\subset \bD(\CC)$ be a full subcategory, contained in $\bD^b(\CC)$,
such that every object of $\bD^f(\CC)$ is strongly equivariant with respect
to some congruence subgroup $G(t^n\cdot \BC[[t]])$. 

\medskip

We will make the following additional assumption, satisfied for the categories 
appearing in \secref{situation}:

\medskip

\noindent$(\star)$\hskip0.5cm {\it The category $\bD^f(\CC)$ is strongly generated
by objects that belong to the essential image of the functor 
$\bD^+(\CC)^{w,H}\to \bD^+(\CC)$.}

\medskip

This assumption is satisfied for both examples under consideration.

\medskip

For any $X$ as above we will construct a family of objects $X_k$ that satisfy
condition (2) of \secref{cond for I}. 

\sssec{}

Indeed, let $X=\on{emb}^{w}(X_1)$ with $X_1\in \bD^{\leq 0}(\CC)^{w,H}$. Let $X_1$
be strongly equivariant with respect to $G(t^n\cdot \BC[[t]])$. 

\medskip

Let $H_n$ be the group $H/H\cap G(t^n\cdot \BC[[t]])$, $\fh_n:=\on{Lie}(H_n)$.
Consider the $H_n$-module $\on{Fun}(H_n)$ and let us represent it as a union of
finite-dimensional modules $\on{Fun}(H_n)^k$. Let $(\on{Fun}(H_n)^k)^*$
be the dual representations.

\medskip

Then the desired objects $X^k$ are given by
$$X^k:=\Lambda^\bullet(\fh_n)\otimes \left(X_1\otimes (\on{Fun}(H_n)^k)^*\right),$$
where $X_1\otimes (\on{Fun}(H_n)^k)^*$ is regarded as an object of 
$\bD(\CC)^{w,H}$, and for $Y\in \bD(\CC)^{w,H}$
$$Y\mapsto \Lambda^\bullet(\fh_n)\otimes Y$$
is the (homological) Chevalley complex of $\fh_n$ with coefficients in $Y$,
which by \cite{FG2}, Sect. 20.10 is the left adjoint functor to $\on{emb}^{s,w}$,
restricted to the strongly $G(t^n\cdot \BC[[t]])$-equivariant subcategory.

\ssec{}   \label{proof I lemma c}

In this subsection we will prove point (c) of \propref{prop Iw and t}, i.e.,
that the subcategory $\bD_{ren}(\fD(\Fl^{\on{aff}}_G)_\crit\mod)^{I^0}$
is compatible with the new t-structure on
$\bD_{ren}(\fD(\Fl^{\on{aff}}_G)_\crit\mod)$.

\sssec{}

Let us consider a t-structure on $\bD_{ren}(\fD(\Fl^{\on{aff}}_G)_\crit\mod)^{I^0}$, 
denoted $new'$ by letting $\bD^{\leq 0_{new'}}_{ren}(\fD(\Fl^{\on{aff}}_G)_\crit\mod)^{I^0}$
be generated by
\begin{equation} \label{D minus I}
\bD_{ren}(\fD(\Fl^{\on{aff}}_G)_\crit\mod)^{I^0}\cap 
\bD^{\leq 0_{new}}_{ren}(\fD(\Fl^{\on{aff}}_G)_\crit\mod)\cap 
\bD^f_{ren}(\fD(\Fl^{\on{aff}}_G)_\crit\mod).
\end{equation}

It is sufficient to show that the functor
$$\on{emb}_{ren}:\bD_{ren}(\fD(\Fl^{\on{aff}}_G)_\crit\mod)^{I^0}\to
\bD_{ren}(\fD(\Fl^{\on{aff}}_G)_\crit\mod)$$
is exact in the new t-structures. 
The above functor is evidently right-exact. Hence, it remains to show the following:

\medskip

\noindent(*) If $\CF$ is an object of $\bD_{ren}(\fD(\Fl^{\on{aff}}_G)_\crit\mod)^{I^0}$
such that $\Hom(\CF',\CF)=0$ for
any $\CF'$ belonging to \eqref{D minus I}, then $\Hom(\CF_1,\CF)=0$ for any
$$\CF_1\in \bD^{\leq 0_{new}}_{ren}(\fD(\Fl^{\on{aff}}_G)_\crit\mod)\cap 
\bD^f_{ren}(\fD(\Fl^{\on{aff}}_G)_\crit\mod).$$

\sssec{}

For an element $w$ of the affine Weyl group, let $j_w$ denote the embedding 
of the corresponding $I$-orbit $\Fl^{\on{aff}}_{w,G}\subset \Fl^{\on{aff}}_G$. We shall say that $w$ is right-maximal
if $w$ is the element of maximal length in its right coset with respect to the
finite Weyl group. We have the following assertion (\cite{AB}, Lemma 15):

\begin{lem} 
For any fixed $\CF_1\in \bD^f(\fD(\Fl^{\on{aff}}_G)_\crit\mod)$, for all $\cla$ sufficiently
deep inside the dominant cone, $j_w^!\left(\CF_1\star j_{-\cla,*}\right)=0$
unless $w$ is right-maximal.
\end{lem}

Let $\CF$ and $\CF_1$ be as in (*). Suppose by contradiction that we have a non-zero morphism
$\CF_1\to \CF$. Let $\cla$ be as in the above lemma. Since $-\star j_{-\cla,!}$
is fully faithful, then the morphism 
\begin{equation} \label{bad morphism}
\CF_1\star j_{-\cla,!}\to \CF\star j_{-\cla,!}
\end{equation}
is non-zero either. Then there exists an element $w$ in the affine Weyl group, such that
the morphism
\begin{equation} \label{bad morphism w}
\CF_1\star j_{-\cla,!}\to j_w{}_*\left(j_w^!(\CF\star j_{-\cla,!})\right)
\end{equation}
is non-zero. By the choice of $\cla$, we obtain that $w$ must be right-maximal.
By the assumption on $\CF$, the LHS in \eqref{bad morphism w} is $\leq 0$ 
in the old t-structure. Hence, to obtain a contradiction, it suffices to show that
the RHS in\eqref{bad morphism w} is $>0$ in the old t-structure, or, which is
the same, that $j_w^!(\CF\star j_{-\cla,!})$ is $>0$ as an object of the
derived category of twisted D-modules on $\Fl^{\on{aff}}_{w,G}$.

\medskip

Since $\CF$ is $I^0$-equivariant, the object $j_w^!(\CF\star j_{-\cla,!})$ is 
an extension (in fact, a direct sum) of copies of the shifted constant D-module
on $\Fl^{\on{aff}}_{w,G}$; let $j_{w,!}$ denote the !-extension of the latter
onto $\Fl^{\on{aff}}_G$. Thus, it is enough to show that
$\Hom(j_{w,!}[k],\CF\star j_{-\cla,!})=0$ for $k\geq 0$. However, this follows
from the assumption on $\CF$ and the next assertion:

\begin{lem}
If $w$ is right-maximal, then the object $j_{w,!}\in \bD^f(\fD(\Fl^{\on{aff}}_G)_\crit\mod)^{I^0}$
belongs to $\bD^{\leq 0_{new}}_{ren}(\fD(\Fl^{\on{aff}}_G)_\crit\mod)$.
\end{lem}

\begin{proof}

Indeed, for $w$, which is right-maximal, we have $j_{w,!}\star j_{-\cla,!}\simeq j_{w\cdot (-\cla),!}$.

\end{proof}

\ssec{}   \label{proof I lemma d}

In order to prove \propref{Iw fully faithful} we will consider the following general context.

\medskip

Let $F:\bA\to \bA'$ be a 1-morphism in $\DGMonCat$. Let $G:\bC_1\to \bC$ be 
a 1-morphism in $\DGMod(\bA)$. Assume that the functor
$\Ho(G):\Ho(\bC_1)\to \Ho(\bC)$ is fully faithful. 

\medskip

Set $\bC'_1=\Ind^{\bA'}_\bA(\bC_1)$ and $\bC'=\Ind^{\bA'}_\bA(\bC)$.
Let $G'$ denote the resulting 1-morphism $\bC'_1\to \bC'$ in
$\DGMod(\bA')$, and 
$$F_{\bC_1}:\bC_1\to \bC'_1 \text{ and } F_{\bC}:\bC\to \bC'$$
be the corresponding morphisms in $\DGCat$.

\begin{prop}  \label{Iw fully faithful abs}
Assume that $\Ho(\bA)$ is rigid. We have:

\smallskip

\noindent{\em(1)} The functor $\Ho(G'):\Ho(\bC'_1)\to \Ho(\bC')$ is
fully faithful.

\smallskip

\noindent{\em(2)} The natural transformation $G^*\circ F_{\bC_1}{}_*\to
F_{\bC}{}_*\circ G'{}^*$ is an isomorphism.

\smallskip

\noindent{\em(3)} If $F_*:\Ho(\bA')\to \Ho(\bA)$ is conservative, then
the essential image of $\Ho(\ua\bC{}'_1)$ in $\Ho(\ua\bC')$
under $G'{}^*$ equals the pre-image of $G^*(\Ho(\ua\bC{}_1))\subset \Ho(\ua\bC)$
under $F_\bC{}_*$.

\end{prop}

\begin{proof}

Point (1) follows from \corref{tight mon}(2). Point (2) follows from
\corref{ind low tight}. 

To prove point (3), we note that the functor $F_\bC$ is also conservative
(see \propref{affine exact}), hence, it suffices to show that the natural transformation
$$G^*\circ G_*\circ F_{\bC}{}_*\to F_{\bC}{}_*\circ G'{}^*\circ G'_*$$
is an isomorphism. However, this follows from point (2). 

\end{proof}

\sssec{}  \label{Iw fully faithful, proof bis}

To prove \lemref{Iw fully faithful} we apply \propref{Iw fully faithful abs}
to $\bA$ and $\bA'$ being DG models of
$\bD^{perf}(\Coh(\tN/\cG))$ and $\bD^{perf}(\Coh(\nOp))$, respectively,
and $\bC_1\subset \bC$ being DG models of
$$\bD^f(\fD(\Gr^{\on{aff}}_G)_\crit\mod)^{I^0}\subset
\bD^f(\fD(\Gr^{\on{aff}}_G)_\crit\mod).$$

\sssec{Proof of \propref{prop Iw and t}(d)}   \label{proof prop Iw and t d}

Assume now that in the context of \secref{proof I lemma d},
the categories $\Ho(\ua\bC), \Ho(\ua\bA), \Ho(\ua\bA')$ are equipped with 
t-structures, satisfying the assumptions of \secref{intr t structure on ten}. 
Consider the resulting t-structures on $\Ho(\ua\bC')$.

\medskip

The next assertion follows immediately from \propref{Iw fully faithful abs}(2)
and \propref{affine exact}:

\begin{cor}  \label{Iw fully faithful abs t}
Assume that the subcategory $\Ho(\ua\bC{}_1)\subset \Ho(\ua\bC)$
is compatible with the t-structure. Assume also that the functor $F$
is affine. Then the subcategory $\Ho(\ua\bC'{}_1)\subset \Ho(\ua\bC')$
is also compatible with the t-structure.
\end{cor}

We apply the above corollary to the same choice of the categories
as in \secref{Iw fully faithful, proof bis}, and the new t-structure on
$\bD_{ren}(\fD(\Gr^{\on{aff}}_G)_\crit\mod)$. The required compatibility
with the t-structure is insured by \propref{prop Iw and t}(c).

\end{document}